\documentclass[a4paper,12pt,reqno]{amsart}

\usepackage{amsmath, amsfonts, amssymb, amsthm, mathrsfs, mathtools, mathalfa}
\usepackage{comment}
\usepackage{yfonts}
\usepackage{enumerate}
\usepackage{enumitem}
\usepackage{stmaryrd}
\usepackage{fancyhdr}
\usepackage{bm}
\usepackage[utf8]{inputenc}
\usepackage[pdfencoding=auto, hyperfootnotes=false]{hyperref}
\usepackage{tikz,tikz-cd}
\usepackage[hang,flushmargin]{footmisc}
\usepackage{graphicx}
\usepackage{xr}
\usepackage[dvistyle, colorinlistoftodos]{todonotes}

\makeatletter
\def\@tocline#1#2#3#4#5#6#7{\relax
  \ifnum #1>\c@tocdepth 
  \else
    \par \addpenalty\@secpenalty\addvspace{#2}%
    \begingroup \hyphenpenalty\@M
    \@ifempty{#4}{%
      \@tempdima\csname r@tocindent\number#1\endcsname\relax
    }{%
      \@tempdima#4\relax
    }%
    \parindent\z@ \leftskip#3\relax \advance\leftskip\@tempdima\relax
    \rightskip\@pnumwidth plus4em \parfillskip-\@pnumwidth
    #5\leavevmode\hskip-\@tempdima
      \ifcase #1
       \or\or \hskip 1em \or \hskip 2em \else \hskip 3em \fi%
      #6\nobreak\relax
    \dotfill\hbox to\@pnumwidth{\@tocpagenum{#7}}\par
    \nobreak
    \endgroup
  \fi}
\makeatother

\newtheorem{theorem}{Theorem}[section]
\newtheorem{lemma}[theorem]{Lemma}

\newtheorem{corollary}[theorem]{Corollary}
\newtheorem{proposition}[theorem]{Proposition}

\theoremstyle{definition}
\newtheorem{defn}[theorem]{Definition}
\newtheorem{remark}[theorem]{Remark}
\newtheorem{example}[theorem]{Example}

\newcommand{\mc}{\mathcal}

\newcommand{\mb}{\mathbb}

\newcommand{\wh}{\widehat}
\newcommand{\wt}{\widetilde}

\newcommand{\ud}{\,\mathrm{d}}
\newcommand{\id}{\mathrm{id}}

\DeclareMathOperator{\aut}{Aut}
\DeclareMathOperator{\ab}{Z}
\DeclareMathOperator{\Bo}{B}
\DeclareMathOperator{\tran}{\Theta}

\DeclareMathOperator{\Supp}{Supp}
\DeclareMathOperator{\supp}{supp}

\DeclareMathOperator{\coup}{\mathsf{Cg}}
\DeclareMathOperator{\q}{c}
\DeclareMathOperator{\ns}{X}
\DeclareMathOperator{\nss}{Y}
\DeclareMathOperator{\co}{\circ\hspace{-0.02 cm}}
\DeclareMathOperator{\cu}{C}

\DeclareMathOperator{\cs}{s}
\DeclareMathOperator{\upmod}{\perp\!\!\!\perp}

\newcommand*{\sbr}[1]{\scalebox{0.8}{$(#1)$}}
\newcommand*{\db}[1]{\llbracket #1\rrbracket}

\begin{document}

\title[On cubic couplings]{Nilspace factors for general uniformity seminorms, cubic exchangeability and limits}

\author{Pablo Candela}
\address{Universidad Aut\'onoma de Madrid and ICMAT\\ Ciudad Universitaria de Cantoblanco\\ Madrid 28049\\ Spain}
\email{pablo.candela@uam.es}

\author{Bal\'azs Szegedy}
\address{MTA Alfr\'ed R\'enyi Institute of Mathematics\\ 
Re\'altanoda utca 13-15\\
Budapest, Hungary, H-1053}
\email{szegedyb@gmail.com}
\begin{abstract}
We study a class of measure-theoretic objects that we call \emph{cubic couplings}, on which there is a common generalization of the Gowers norms and the Host--Kra seminorms. Our main result yields a complete structural description of cubic couplings, using nilspaces. We give three applications. Firstly, we describe the characteristic factors of Host--Kra type seminorms for measure-preserving actions of countable nilpotent groups. This yields an extension of the structure theorem of Host and Kra. Secondly, we characterize sequences of random variables with a property that we call \emph{cubic exchangeability}. These are sequences indexed by the infinite discrete cube, such that for every integer $k\geq 0$ the joint distribution's marginals on affine subcubes of dimension $k$ are all equal. In particular, our result gives a description, in terms of compact nilspaces, of a related exchangeability property considered by Austin, inspired by a problem of Aldous. Finally, using nilspaces we obtain limit objects for sequences of functions on compact abelian groups (more generally on compact nilspaces) such that the densities of certain patterns in these functions converge. The paper thus proposes a measure-theoretic framework on which the area of higher-order Fourier analysis can be based, and which yields new applications of this area in a unified way in ergodic theory and arithmetic combinatorics.\vspace{-0.5cm}
\end{abstract}
\date{}
\vspace*{-0.1cm}
\maketitle
\tableofcontents
\vspace*{-1cm}
\section{Introduction}

\noindent A fruitful interaction between the areas of combinatorics and ergodic theory was initiated in the 1970s by Furstenberg's proof of Szemer\'edi's theorem on arithmetic progressions \cite{FurstSzem}. In the last two decades, this interaction has intensified, thanks especially to the emergence of analogous key tools and methods in these areas. A central example is given by the uniformity norms introduced by Gowers in arithmetic combinatorics, in his seminal work on Szemer\'edi's theorem \cite{GSz}, and by how these norms found ergodic theoretic analogues in the uniformity seminorms introduced by Host and Kra \cite{HK}. Each side of this analogy has led to a major topic of research, and these topics have been in conversation ever since. On one side there is the study and use of the basic harmonics of a function on a compact abelian group that are characteristic for each uniformity norm, a topic now known as higher-order Fourier analysis. An important result here is the inverse theorem for the Gowers norms (\cite[Theorem 1.3]{GTZ}, \cite[Theorem 2]{Szegedy:HFA}). As stated in \cite[\S 3.3]{GreenICM}, a principal objective in this topic is to find new proofs of this theorem that provide further conceptual clarification; for more information on this topic we refer to the survey \cite{GHFA}. On the other side, there is the study of characteristic factors for uniformity seminorms,  and the related structural approach to the analysis of multiple ergodic averages. This direction, propelled in particular by work of Host and Kra, has attracted numerous contributions by many authors; for further information we refer to the book \cite{HKbook} and also to the survey \cite{Fra}.

This paper proposes an approach that enables a unified analysis of uniformity seminorms in ergodic theory and arithmetic combinatorics. Such a unification has been hoped for as part of the development of the above topics, as  expressed for instance in \cite[end of Chapter 17]{HKbook}. The  approach in this paper is based on the study of measure-theoretic objects that we call \emph{cubic couplings}. 

A cubic coupling on a probability space $\varOmega=(\Omega,\mc{A},\lambda)$ consists of a sequence of probability measures, where the $n$-th measure is a coupling\footnote{The notion of a coupling is recalled below in Definition \ref{def:coup}. Here it means that for each $v\in \{0,1\}^n$ the image of the $n$-th measure under the coordinate projection $\Omega^{\{0,1\}^n}\to \Omega$, $\omega\mapsto \omega_v$ is the measure $\lambda$.} of $\lambda$ defined on the product measurable space $(\Omega^{\{0,1\}^n},\mc{A}^{\{0,1\}^n})$ (where $\mc{A}^{\{0,1\}^n}$ denotes the product $\sigma$-algebra $\bigotimes_{v\in \{0,1\}^n}\mc{A}$), and the measures satisfy three axioms. For the present introductory purposes, let us describe these axioms informally. The first one, called the \emph{consistency axiom}, states that for every $n$ and every injective discrete-cube morphism\footnote{As in \cite{CamSzeg}, we call $\phi:\{0,1\}^m\to\{0,1\}^n$ a morphism if $\phi$ extends to an affine homomorphism $\mb{Z}^m\to\mb{Z}^n$.} $\phi:\{0,1\}^m\to\{0,1\}^n$, the image of the $n$-th measure under the map $\Omega^{\{0,1\}^m}\to\Omega^{\{0,1\}^n}$ induced by $\phi$ is equal to the $m$-th measure. Next, the \emph{ergodicity axiom} states that the measure on $\Omega^{\{0,1\}}=\Omega\times \Omega$ is the product measure $\lambda\times \lambda$. Finally, the \emph{conditional independence axiom} states that for every $n$, for any faces $F_1$, $F_2$ of codimension 1 in the cube $\{0,1\}^n$ with $F_1\cap F_2\neq \emptyset$, the two sub-$\sigma$-algebras of $\mc{A}^{\{0,1\}^n}$ generated by the projections $\Omega^{\{0,1\}^n}\to\Omega^{F_i}$ are conditionally independent relative to the $\sigma$-algebra generated by the projection $\Omega^{\{0,1\}^n}\to\Omega^{F_1\cap F_2}$. 

Cubic couplings are described above as \emph{sequences} of measures, but note that by the consistency axiom we can view a cubic coupling as a \emph{single} measure\footnote{This viewpoint is useful in Section \ref{sec:exchange}; see Remark \ref{rem:infinitecc}.}; indeed we can view the measures in the sequence as marginals of a single coupling of $\lambda$ defined on $\Omega^{\db{\mb{N}}}$, where by $\db{\mb{N}}$ we denote the infinite discrete cube, that is, the set of elements of $\{0,1\}^\mb{N}$ that have only finitely many coordinates equal to $1$.
 
We leave the formal definition of a cubic coupling for the sequel (see Definition \ref{def:cc}), but let us illustrate this concept straightaway with examples that are actually key objects of study in the two topics mentioned above.

In arithmetic combinatorics, the example in question consists of the Haar measures on the groups of standard cubes of increasing dimension in a compact abelian group $\ab$ (for a basic discussion of these cubes, see \cite[\S 2.1]{Cand:Notes1}). More precisely, the $n$-th measure in this cubic coupling is the Haar probability measure on the group of $n$-cubes \vspace{-0.1cm}
\begin{equation}\label{eq:cubes}
\cu^n(\ab):=\big\{\q=\big(x+v\sbr{1}\,h_1+\cdots +v\sbr{n}\,h_n\big)_{v\in \{0,1\}^n}: x,h_1,\ldots,h_n\in \ab\big\}\leq \ab^{\{0,1\}^n}. \vspace{-0.1cm}
\end{equation}
Let us recall that for a bounded measurable function $f:\ab\to \mb{R}$, if we integrate the function $\q\mapsto \prod_{v\in \{0,1\}^n} f\big(\!\q(v)\big)$ over $\cu^n(\ab)$, and take the $2^n$-th root of the result, then we obtain the Gowers $U^n$ norm of $f$, denoted by $\|f\|_{U^n}$.

In ergodic theory, the example in question is the sequence of measures $\mu^{[n]}$ constructed by Host and Kra in \cite[\S 3]{HK} for ergodic measure-preserving systems.

The main result of this paper is a characterization of the structure of a general cubic coupling on a Borel probability space, using objects the study of which began recently in connection with the analysis of uniformity norms, namely compact nilspaces. These spaces, introduced by the second-named author in joint work with Antol\'in Camarena \cite{CamSzeg}, offer a useful common generalization of compact abelian groups and nilmanifolds (see also the treatments of nilspaces in \cite{Cand:Notes1,Cand:Notes2,GMV1,GMV2,GMV3}). A compact nilspace $\ns$ is naturally equipped with a sequence of probability measures, the $n$-th term in the sequence being the Haar measure $\mu_{\cu^n(\ns)}$ on the set of $n$-cubes $\cu^n(\ns)$ (this is detailed in \cite[\S 2.2]{Cand:Notes2}). Every compact nilspace with this sequence of measures is a cubic coupling; see Proposition \ref{prop:nilspace-cc}.

Apart from these examples generated by nilspaces, there is also the trivial example consisting of an \emph{independent} cubic coupling, formed by taking the powers $\lambda^{\{0,1\}^n}$ (in the sense of the product measure) of the measure $\lambda$. 

The main result of this paper, Theorem \ref{thm:MeasInvThmGenIntro}, tells us that, more generally, a cubic coupling is a combination of the above constructions, in some natural sense which involves the concept of relative independence (we defer the discussion of this concept to Definition \ref{def:indoverfact}). The formal statement of the result uses the following notation.

Given a map $f:X\to Y$ between two sets $X,Y$, and given another set $S$, we use the power notation $f^S$ to denote the map from the Cartesian power $X^S$ to $Y^S$ defined by $f^S\big((x_v)_{v\in S}\big) = \big(f(x_v)\big)_{v\in S}$. It is also convenient for the sequel to introduce the shorter notation $\db{n}$ to denote the discrete $n$-cube $\{0,1\}^n$ (this simplifies notations,  especially when these cubes appear as superscripts). We can now state our main result.
\begin{theorem}\label{thm:MeasInvThmGenIntro}
Let $(\mu_n)_{n\geq 0}$ be a cubic coupling on a Borel probability space $\varOmega$. Then there is a compact nilspace $\ns$ and a measure-preserving map $\gamma: \Omega\to \ns$ such that for each $n$ the map $\gamma^{\db{n}}$ is measure-preserving from $\big(\Omega^{\db{n}}, \mu_n\big)$ to  $\big(\ns^{\db{n}}, \mu_{\cu^n(\ns)}\big)$. Furthermore, for each $n$ the coupling $\mu_n$ is relatively independent over the factor generated by $\gamma^{\db{n}}$.
\end{theorem}
\noindent Our first application of Theorem \ref{thm:MeasInvThmGenIntro} is a description of the characteristic factors for natural generalizations of the Host--Kra seminorms. We define these generalizations for any measure-preserving action of a countable nilpotent group on a Borel probability space, and our application describes the corresponding characteristic factors as compact nilspaces acted upon by their translation groups; see Theorem \ref{thm:ergthyapp}. This yields the following generalization of the celebrated structure theorem of Host and Kra \cite[Theorem 10.1]{HK}.
\begin{theorem}\label{thm:HKgenintro}
Let $G$ be a finitely generated nilpotent group acting ergodically on a Borel probability space $\varOmega$. Then, for each positive integer $k$, the $k$-th Host--Kra factor of the system $(\varOmega,G)$ is isomorphic to an inverse limit of $k$-step nilsystems.
\end{theorem}
\noindent The notion of Host--Kra factor used here is given in Definition \ref{def:HKfactor}, and extends \cite[Definition 4.1]{HK}. Actually, these factors and their corresponding seminorms can be defined for any filtration on $G$ (in Theorem \ref{thm:HKgenintro}, the underlying filtration is the lower central series), and we describe these factors in this more general setting; see Theorem \ref{thm:HKgen}.

The only other results in the direction of Theorems  \ref{thm:ergthyapp} and \ref{thm:HKgen}, apart from \cite[Theorem 10.1]{HK}, are those of Bergelson, Tao and Ziegler in the abelian setting \cite{BTZ}, which focus on actions of infinite-dimensional vector spaces $\mb{F}_p^{\infty}$. The possibility of structural results such as our Theorem \ref{thm:ergthyapp}, concerning nilpotent measure-preserving group actions, was evoked in \cite[p.\ 1540]{BTZ}. (In the setting of topological dynamics, compact nilspaces also appear in recent structure theorems related to group actions, in \cite{GGY}.)

In the analysis of limits of multiple ergodic averages, one of the main strategies is based on using uniformity seminorms to control such averages, and on analyzing characteristic factors for these seminorms; see \cite[\S 2.6]{Fra}. (Other strategies include that of  Ziegler in \cite{Ziegler}, which isolates characteristic factors in a different way.) Theorems \ref{thm:ergthyapp}, \ref{thm:HKgen} contribute to an  extension of this strategy to nilpotent group actions, by enabling a reduction of the problem, for a family of averages including those in \cite{HK}, to the analysis of these averages on  nilspaces, or even on nilmanifolds (when Theorem \ref{thm:HKgen} is applicable). This reduction is not treated in this paper; see Remark \ref{rem:ergavs}.

Our second application concerns the theory of exchangeable random variables. Broadly speaking, this theory aims to describe the structure of the joint distribution of a sequence of random variables, assuming that the distribution has certain prescribed symmetries. The original definition of exchangeability states that a sequence of random variables $(X_v)_{v\in I}$ is exchangeable if the joint distribution is invariant under all permutations of finite subsets of the index set $I$. If $I$ is countably infinite, then a characterization of such distributions is given by the classical theorem of de Finetti, which describes such a distribution as a convex combination of distributions of i.i.d.\ random variables \cite{dF}. Weaker notions of exchangeability, involving fewer symmetries, yield various extensions or analogues of de Finetti's theorem, and the resulting subject is rich in connections with other areas, including combinatorics and ergodic theory; see \cite{Austin,D&J,Fra2}. Despite these connections, and the importance of this subject within probability theory, complete characterizations of exchangeable distributions are known only for a few variants of the exchangeability property. Beyond de Finetti's theorem, principal results of this type are the Aldous--Hoover theorem \cite{Aldous,Hoover} and its extension by Kallenberg \cite{Kal}. 

In this paper we study joint distributions of sequences of random variables that are indexed by the infinite discrete cube $\db{\mb{N}}$. For such a distribution, we assume a  property that we call \emph{cubic exchangeability}, which says that for every $k\geq 0$ the marginal distributions on $k$-dimensional affine subcubes of $\db{\mb{N}}$ are all equal; see Definition \ref{def:cubexch}. Related properties have been studied before. In \cite[\S 16]{Aldous2}, Aldous considered a weaker property, namely invariance under the group $\aut(\db{\mb{N}})$ of symmetries of $\db{\mb{N}}$ (we detail this group in Remark \ref{rem:exrels}). Aldous asked for a description of measures with this property. In \cite{Austin2}, Austin showed that such a description is difficult, in that these $\aut(\db{\mb{N}})$-exchangeable measures (or \emph{cube-exchangeable} measures, as they were called in \cite{Austin2}) form a Poulsen simplex. However, in \cite[\S 5.3]{Austin2} it was noted that a stronger version of $\aut(\db{\mb{N}})$-exchangeability, requiring invariance under the whole group $\textrm{Aff}(\mb{F}_2^\infty)$ of \emph{affine} symmetries of $\db{\mb{N}}$ (viewing $\db{\mb{N}}$ as $\mb{F}_2^\infty$), is also natural. This motivated the problem of describing measures with the latter property. As explained in Remark \ref{rem:exrels}, these properties are related as follows:\vspace{0.2cm}\\
$\textrm{Aff}(\mb{F}_2^\infty)$-exchangeability \cite{Austin2} $\;\,\Rightarrow\;$ cubic exchangeability $\;\,\Rightarrow\;$ $\aut(\db{\mb{N}})$-exchangeability \cite{Aldous2}.

\noindent Our main result in this direction is a description of cubic exchangeable measures (and thereby  of $\textrm{Aff}(\mb{F}_2^\infty)$-exchangeable measures) using nilspaces. Note that nilmanifolds appear in other exchangeability contexts, for instance in Frantzikinakis's work \cite{Fra2}. We explain our result using the following general construction of cubic exchangeable systems of random variables.\footnote{We say that a system or sequence of random variables is cubic exchangeable if its joint distribution is.} Let $\ab$ be a compact abelian group and let $\Bo$ be a standard Borel space. Let $\mc{P}(\Bo)$ denote the set of Borel probability measures on $\Bo$, equipped with its standard Borel structure \cite[p.\ 113]{Ke}. Let $m:\ab\to \mc{P}(\Bo)$ be a Borel function. Let $x,h_1,h_2,\dots$ be i.i.d.\ random elements of $\ab$ chosen according to the Haar probability. For $v\in\db{\mb{N}}$, let $X_v=m(x+v\sbr{1}h_1+v\sbr{2}h_2+\cdots)$ (this sum has only finitely many non-zero terms, by definition of $\db{\mb{N}}$). If we look at a finite-dimensional affine subcube of $\db{\mb{N}}$, then, for the corresponding subcollection of $(X_v)_{v\in\db{\mb{N}}}$, the elements $x+v\sbr{1}h_1+v\sbr{2}h_2+\cdots$ form a subcube of one of the cubes in \eqref{eq:cubes}. Using this, it is seen that $(X_v)_{v\in\db{\mb{N}}}$ is a cubic exchangeable $\mc{P}(\Bo)$-valued sequence. Now, in a second round of randomization, for each $v$ independently we choose $Y_v\in \Bo$ with probability distribution $X_v$. This yields a cubic exchangeable system $(Y_v)_{v\in\db{\mathbb{N}}}$ of $\Bo$-valued random variables. We denote the joint distribution of $(Y_v)_{v\in\db{\mathbb{N}}}$ by $\zeta_{\ab,m}$. This construction can be generalized, replacing $\ab$ by a compact nilspace. Our result describes cubic exchangeability in terms of this construction.
\begin{theorem}\label{IntMainexch}
Let $\Bo$ be a standard Borel space. Then every cubic exchangeable probability measure on $\Bo^{\db{\mb{N}}}$ is a convex combination of measures of the form $\zeta_{\ns,m}$, where $\ns$ is a compact nilspace and $m:\ns\to\mc{P}(\Bo)$ is a Borel function.
\end{theorem}
\noindent It also follows from our results that the measures of the form $\zeta_{\ns,m}$ are extreme points in the convex set of cubic exchangeable measures on $\Bo^{\db{\mb{N}}}$. In particular, we show that one can detect whether a cubic exchangeable measure is such an extreme point by using a type of independence property (see Theorem \ref{thm:mainexch1}). As a consequence, we obtain that if $\Bo$ is a compact Polish space then the set of such extremal cubic exchangeable measures on $\Bo^{\db{\mb{N}}}$ is closed with respect to weak limits. This fact leads to the following third main application in this paper, which concerns arithmetic combinatorics.

Given a compact abelian group $\ab$ and a bounded Borel function $f:\ab\to\mb{C}$, and given a type of additive pattern in $\ab$ (any type of configuration determined by a system of integer linear forms), we can talk about the \emph{density} of such patterns in $f$ to refer to the integral of $f$ over the group of such patterns, using the Haar probability on this group. In particular, if $f$ is the indicator function of a Borel set $A\subset \ab$, and the patterns are $k$-term arithmetic progressions, say, then the  integral in question is indeed the density (or probability), among all $k$-term progressions in $\ab$, of those progressions that are included in $A$.  Another example of such a density is the $2^n$-th power of the Gowers norm $\|f\|_{U^n}$, where the additive patterns involved are the $n$-cubes described in \eqref{eq:cubes}.

In keeping with various central questions in arithmetic combinatorics, it is natural to study notions of convergence for sequences of such functions, based on the convergence of the densities of certain additive patterns in the functions. More precisely, if $S$ is a set of additive patterns (in other words $S$ is a collection of systems of integer linear forms) and $(f_i:\ab_i\to\mb{C})_{i\in \mb{N}}$ is a sequence of uniformly bounded measurable functions on compact abelian groups, then we say that the sequence is \emph{$S$-convergent} if for every pattern in $S$ the density of this pattern in $f_i$ converges as $i\to\infty$. It is then also natural to seek a so-called \emph{limit object} for such a convergent sequence, i.e.\ a fixed space with a function defined on it such that the limits of the densities in the sequence can be expressed exactly as certain integrals involving this function. When $S$ is the set of additive patterns given by systems of linear forms of compexity 1 (as per the definition of complexity from \cite{GW}), a complete limit theory with appropriate limit objects was worked out in \cite{Szegedy:Lim}. Other such limits were obtained for a different set of patterns in \cite{Szegedy:HFA}. Our results in this paper are  related to the ones in \cite{Szegedy:HFA}, as we use a similar set of patterns. However, here we are able to extend the results to functions on compact nilspaces. Our main theorem on this topic provides a limit object for a uniformly bounded sequence of functions $(f_i:\ns_i\to\mb{C})$ on compact nilspaces $\ns_i$, assuming the convergence of densities of certain patterns that we call \emph{cubic patterns} (see Definition \ref{def:cp}). The limit object is a measurable function on a compact nilspace, with the values of the function being probability measures on $\mb{C}$ (see Theorem \ref{thm:limob}); this is thus a natural analogue in arithmetic combinatorics of limit objects from the theory of convergent sequences of graphs and hypergraphs \cite{Lovasz}.

Let us briefly describe another application of our main result, concerning the inverse theorem for the Gowers norms. The proof of this theorem given by the second-named author in \cite{Szegedy:HFA} works with ultraproducts of finite abelian groups. Some  arguments in this paper use some of the key ideas from \cite{Szegedy:HFA}, but the tools developed here do not assume any group structure on the underlying probability space. As a consequence, the deduction of the inverse theorem in \cite{Szegedy:HFA} can be carried out similarly, but at a more general level, using the main results from this paper. This yields in particular the inverse theorem for the Gowers norms stated in \cite[Theorem 2]{Szegedy:HFA}, but it also gives an extension of this theorem in which the initial bounded Borel function $f$ can be defined not just on a compact abelian group, but more generally on a compact nilspace. The proof of this theorem can be summarized as follows: supposing for a contradiction the existence of a sequence of functions violating the conclusion of the inverse theorem, one  takes an ultraproduct of this sequence, in which one can then locate a separable factor that yields a cubic coupling, and the result then follows by applying Theorem \ref{thm:MeasInvThmGenIntro}. Thus the core of this proof of the inverse theorem is Theorem \ref{thm:MeasInvThmGenIntro}. Since this proof requires background on the separate topic of analysis on ultraproducts, we shall detail this application in separate work.

Finally, let us add a few remarks about the organization of the paper. Section \ref{sec:prelim} gathers tools from probability theory. Some of these are new (in particular in \S \ref{subsec:condlattice}), while others may be more familiar to probabilists. In any case, it is certainly viable to skim through Section \ref{sec:prelim} on a first reading, as the core of the paper consists really of Section \ref{sec:cc}, in which cubic couplings are introduced, and of Section \ref{sec:structhm}, where the main results on the structure of cubic couplings are obtained. The applications are treated in sections \ref{sec:ergapps} to \ref{sec:limits}.

\vspace{0.5cm}

\noindent \textbf{Acknowledgements.} We are very grateful to the anonymous referee and to Diego Gonz\'alez-S\'anchez for careful reading and many useful remarks that helped to improve this paper. The first-named author received funding from the Spanish Ministerio de Ciencia e Innovaci\'on project MTM2017-83496-P. The second-named author received funding from the European Research Council under the European Union's Seventh Framework Programme (FP7/2007-2013)/ERC grant agreement 617747. The research was partially supported by the MTA R\'enyi Institute Lend\"ulet \emph{Limits of Structures} Research Group.

\section{Measure-theoretic preliminaries}\label{sec:prelim}

\noindent This section gathers the concepts and results from measure theory needed for the sequel.  

\subsection{Some basic notions}\hfill \medskip\\
\noindent Among the results in this subsection, some are well-known (belonging to the folklore) or appear elsewhere in the literature. In these cases we  refer to the relevant sources or give the proofs in the appendix, in order to enable a lighter first reading of this subsection.

Let $(\Omega,\mc{A},\lambda)$ be a probability space. When the $\sigma$-algebra $\mc{A}$ and the probability measure $\lambda$ are clear from the context, we write $\varOmega$ instead of $(\Omega,\mc{A},\lambda)$. Given a family $\mc{F}$ of subsets of $\Omega$, we denote by $\sigma(\mc{F})$ the $\sigma$-algebra generated by $\mc{F}$, that is, the smallest $\sigma$-algebra (relative to inclusion) among the $\sigma$-algebras that include $\mc{F}$.
\begin{defn}[Join of $\sigma$-algebras]
Given $\sigma$-algebras $\mc{A}$, $\mc{B}$ on a set $\Omega$, the \emph{join} of $\mc{A}$ and $\mc{B}$ is the $\sigma$-algebra $\mc{A}\vee\mc{B} := \sigma(\mc{A} \cup \mc{B})$. 
\end{defn}
\noindent Given a Polish\footnote{A topological space is \emph{Polish} if it is separable and metrizable by means of a complete metric.} space $X$ and $\mc{A}$-measurable functions $f,g:\Omega\to X$ (relative to the Borel $\sigma$-algebra on $X$), we write  $f=_\lambda g$ to mean that $f,g$ are equal $\lambda$-almost everywhere, i.e.\ $\lambda(\{\omega\in\Omega: f(\omega)\neq g(\omega)\})=0$. (The assumption that $X$ is Polish ensures that the set $\{f\neq g\}$ here is measurable; see \cite[Lemma 6.4.2 and Example 6.4.3]{Boga2}.) For $p\in [1,\infty]$ and a probability space $\varOmega=(\Omega,\mc{A},\mu)$, we denote the corresponding $L^p$ space by $L^p(\varOmega)$ (see \cite[\S 4.1]{Boga1}).\footnote{Strictly speaking the elements of $L^p(\varOmega)$ are \emph{equivalence classes} of measurable functions $f$ with $\int_\Omega |f|^p\ud\lambda<\infty$, under the relation $=_\lambda$, but we shall take part in the common abuse of terminology whereby a \emph{function} $f$ is declared to be in $L^p(\varOmega)$ if $\int_\Omega |f|^p\ud\lambda<\infty$.} We also use variants of this notation when a particular component of $\varOmega$ needs to be emphasized and the other components are clear, for instance the notation $L^p(\lambda)$ or $L^p(\mc{A})$. We denote by $\mc{U}^p(\mc{A})$ the unit ball of $L^p(\mc{A})$.

We use the following approximation result many times (for a proof see Lemma \ref{lem:pisysapproxApp}). 

\begin{lemma}\label{lem:pisysapprox}
Let $1\leq p <\infty$, let $(\mc{B}_i)_{i=1}^n$ be a sequence of sub-$\sigma$-algebras of $\mc{A}$, and let $\mc{B}=\bigvee_{i=1}^n \mc{B}_i$. Let $\mc{R}$ denote the set of functions on $\Omega$ of the form $\omega\mapsto \prod_{i=1}^n f_i(\omega)$, where $f_i\in \mc{U}^\infty(\mc{B}_i)$ for all $i$. Then for every $f\in L^p(\mc{B})$, and every $\epsilon>0$, there exists a finite linear combination $g$ of functions in $\mc{R}$ such that $\|f-g\|_{L^p}\leq\epsilon$.
\end{lemma}
\noindent  When we need to specify the $\sigma$-algebras involved in $\mc{R}$, we write $\mc{R}\big((\mc{B}_i)_{i=1}^n\big)$.

Let us recall the following definition of conditional independence \cite[p.\ 30]{Meyer}, \cite[\S 7]{Yan}.

\begin{defn}[Conditional independence of two sub-$\sigma$-algebras relative to a third one]\label{def:relcondindep}
Let $(\Omega,\mc{A},\lambda)$ be a probability space, and let $\mc{B}_0,\mc{B}_1,\mc{B}$ be sub-$\sigma$-algebras of $\mc{A}$. We say that $\mc{B}_0,\mc{B}_1$ are \emph{conditionally independent relative to} $\mc{B}$ if for every bounded $\mc{B}_0$-measurable function $f_0$ and bounded $\mc{B}_1$-measurable function $f_1$, we have
$\mb{E}(f_0 f_1| \mc{B}) =_\lambda \mb{E}( f_0 | \mc{B})\; \mb{E}( f_1| \mc{B})$. 
\end{defn}
\noindent Recall also that $\mc{B}_0$, $\mc{B}_1$ are said to be \emph{independent} if for every function $f_0\in L^\infty(\mc{B}_0)$ and $f_1\in L^\infty(\mc{B}_1)$ we have $\mb{E}(f_0 f_1) = \mb{E}( f_0 )\, \mb{E}( f_1)$; equivalently if $\lambda(A_0\cap A_1)= \lambda(A_0)\,\lambda(A_1)$ for every $A_0\in \mc{B}_0$, $A_1\in \mc{B}_1$. In some contexts we may have to clarify what is the measure with respect to which the $\sigma$-algebras are independent; in this case we say they are \emph{independent in $\lambda$}. This notion of independence is the special case of Definition \ref{def:relcondindep} with  $\mc{B}=\{\emptyset,\Omega\}$.

In Definition \ref{def:relcondindep} we are fully rigorous by using the notation  $=_\lambda$. To avoid overloading the notation, when the measure $\lambda$ is clear from the context we shall often omit it from such equalities between conditional expectations (especially from Subsection \ref{subsec:couplings} onwards).

Let us recall also the following useful result (for a proof see \cite[p.\ 30, Theorem 51]{Meyer}).
\begin{theorem}\label{thm:condindep}
Let $(\Omega,\mc{A},\lambda)$ be a probability space, and let $\mc{B}_0$, $\mc{B}_1$, $\mc{B}$ be sub-$\sigma$-algebras of $\mc{A}$. Then $\mc{B}_0$, $\mc{B}_1$ are conditionally independent relative to $\mc{B}$ if and only if the following statement holds\textup{:} for every function $f\in L^1(\mc{B}_1)$ we have $\mb{E}(f|\mc{B}_0\vee \mc{B})=_\lambda \mb{E}(f| \mc{B})$.
\end{theorem}
\begin{remark}\label{rem:integtobded}
In Theorem \ref{thm:condindep} the equivalence still holds if we replace $L^1(\mc{B}_1)$ with $L^\infty(\mc{B}_1)$. This follows from the density of $L^\infty(\mc{B}_1)$ in $L^1(\mc{B}_1)$ \cite[Lemma 4.2.1]{Boga1} and the fact that conditional expectation is a contraction on $L^1(\mc{B}_1)$ \cite[Theorem 10.1.5 (5)]{Boga2}.
\end{remark}
\noindent We use mostly a special case of Definition \ref{def:relcondindep} where $\mc{B}$ is given by the following operation.

\begin{defn}[Meet of sub-$\sigma$-algebras]
Let $(\Omega,\mc{A},\lambda)$ be a probability space. For sub-$\sigma$-algebras  $\mc{B}_0$, $\mc{B}_1$ of $\mc{A}$, we denote by $\mc{B}_0\wedge_\lambda \mc{B}_1$ the sub-$\sigma$-algebra of $\mc{A}$ consisting of all sets $A\in \mc{A}$ such that there exist $B_0\in \mc{B}_0$ and $B_1\in \mc{B}_1$ satisfying $\lambda(A\Delta B_0)=\lambda(A\Delta B_1)=0$.
\end{defn}
\noindent When the ambient measure $\lambda$ is clear, we omit it from the notation, writing just $\mc{B}_0 \wedge \mc{B}_1$. It is readily shown that $\mc{B}_0 \wedge \mc{B}_1$ is indeed a sub-$\sigma$-algebra of $\mc{A}$ (see Lemma \ref{lem:meetsubalg}).

Given a $\sigma$-algebra $\mc{A}'$ on a set $\Omega'$ and a function $f:\Omega \to \Omega'$, we denote by $f^{-1}(\mc{B}')$ the \emph{preimage $\sigma$-algebra} (or \emph{preimage of $\mc{B}'$ under $f$}), that is $f^{-1}(\mc{B}')=\{f^{-1}(S):S\in \mc{B}'\}$.
\begin{remark}
The $\sigma$-algebra $\mc{B}_0\wedge_\lambda \mc{B}_1$ clearly includes the intersection $\sigma$-algebra $\mc{B}_0 \cap \mc{B}_1$, but this inclusion may be strict. For instance, consider $\Omega=[0,1]\times [0,1]= [0,1]^{\{0,1\}}$ with the product $\sigma$-algebra $\mc{B}\otimes \mc{B}$ where $\mc{B}$ is the Borel $\sigma$-algebra on the interval $[0,1]$. For $i=0,1$ let $\mc{B}_i=p_i^{-1}(\mc{B})$ where $p_i:[0,1]^2\to [0,1]$, $(\omega_0,\omega_1)\mapsto \omega_i$. Then $\mc{B}_0\cap \mc{B}_1=\{\emptyset, \Omega\}$. Let $D$ denote the diagonal $\{(\omega_0,\omega_1)\in \Omega:\omega_0=\omega_1\}$ and $\lambda$ the probability measure on $[0,1]^2$ defined as the image of the Lebesgue measure on $[0,1]$ under the map $t\mapsto (t,t)$ (in particular $\lambda(D)=1$). Then for every $A\in \mc{B}\otimes \mc{B}$ we have $\lambda\big(A\Delta p_i^{-1}(p_i(A\cap D))\big)=0$, for $i=0,1$. It follows that $\mc{B}_0\wedge_\lambda \mc{B}_1=\mc{B}\otimes \mc{B}$.
\end{remark}
\noindent The above example shows also that we can have $\mc{B}_0\wedge_\lambda \mc{B}_1\not\subset \mc{B}_i$ for $i=0,1$. However, we do have inclusion up to null sets in general, in the following sense. Recall that,  for sub-$\sigma$-algebras $\mc{B},\mc{B}'$ of $\mc{A}$, the relation of inclusion up to null sets, denoted by $\subset_\lambda$, is defined by declaring that $\mc{B}\subset_\lambda \mc{B}'$ if for every $A\in \mc{B}$ there exists $A'\subset \mc{B}'$ such that $\lambda(A\Delta A')=0$. We clearly have $\mc{B}_0\wedge_\lambda \mc{B}_1\subset_\lambda \mc{B}_i$ for $i=0,1$. We write $\mc{B}=_\lambda \mc{B}'$ to mean that $\mc{B}\subset_\lambda \mc{B}'$ and $\mc{B}'\subset_\lambda \mc{B}$. Let us record the following basic fact about the relation $\subset_\lambda$ (for a proof, see Lemma \ref{lem:nestexpApp}).
\begin{lemma}\label{lem:nestexp}
Let $(\Omega,\mc{A},\lambda)$ be a probability space, and let $\mc{B}$, $\mc{B}'$ be sub-$\sigma$-algebras of $\mc{A}$ with $\mc{B}\subset_\lambda\mc{B}'$. Then for every integrable function $f:\Omega\to \mb{R}$ we have $\mb{E}(\mb{E}(f|\mc{B}')|\mc{B})=_\lambda \mb{E}(f|\mc{B})$, and also $\mb{E}(f|\mc{B}')=_\lambda \mb{E}(f|\mc{B}' \vee \mc{B})$. 
\end{lemma}
\noindent We may also use the notation $\subset_\lambda$ with \emph{sets} $A,B\in \mc{A}$, writing $A \subset_\lambda B$ to mean that $\lambda(A\setminus B)=0$. We write $A=_\lambda B$ to mean that $A \subset_\lambda B$ and $B \subset_\lambda A$, i.e.\ $\lambda(A\Delta B)=0$.

The special case of Definition \ref{def:relcondindep} that we shall use is the following.

\begin{defn}[Conditional independence of two sub-$\sigma$-algebras]\label{def:condindep}
Let $\varOmega=(\Omega,\mc{A},\lambda)$ be a probability space and let $\mc{B}_0$, $\mc{B}_1$ be sub-$\sigma$-algebras of $\mc{A}$. We say that $\mc{B}_0$, $\mc{B}_1$ are \emph{conditionally independent in} $\lambda$ (or \emph{in} $\varOmega$), and we write $\mc{B}_0 \upmod_\lambda \mc{B}_1$, if $\mc{B}_0$, $\mc{B}_1$ are conditionally independent relative to $\mc{B}_0\wedge_\lambda \mc{B}_1$ as per Definition \ref{def:relcondindep}.
\end{defn}
\noindent When the ambient measure $\lambda$ is clear, we omit it from the notation and terminology, writing just $\mc{B}_0 \upmod \mc{B}_1$ and saying that $\mc{B}_0$, $\mc{B}_1$ are conditionally independent.

The following result characterizes conditional independence in terms of conditional expectation, and we use it many times in the sequel.
\begin{proposition}\label{prop:condindep}
Let $(\Omega,\mc{A},\lambda)$ be a probability space, and let $\mc{B}_0,\mc{B}_1$ be sub-$\sigma$-algebras of $\mc{A}$. Then $\mc{B}_0\upmod \mc{B}_1$ holds if and only if, for every bounded measurable function $f:\Omega\to\mb{R}$, the following equation is satisfied for $i=0$ or, equivalently, for $i=1$\textup{:} 
\begin{equation}\label{eq:condindep}
\mb{E}( \,\mb{E}(f|\mc{B}_i)\, |\mc{B}_{1-i}) =_\lambda \mb{E}( f | \mc{B}_0\wedge \mc{B}_1).
\end{equation}
\end{proposition}
\noindent This result appears in a similar form (stated for standard probability spaces) in \cite[Theorem 9]{Yan}; we include a proof in the appendix for completeness (see Proposition \ref{prop:condindepApp}). As in Theorem \ref{thm:condindep}, we may replace ``bounded" by ``integrable" in this proposition.

Note that \eqref{eq:condindep} implies that the conditional expectation operators for $\mc{B}_0$ and $\mc{B}_1$ commute. This motivates the terminology from \cite{Yan} which says that $\mc{B}_0$, $\mc{B}_1$ are \emph{stochastically commuting} if $\mc{B}_0\upmod \mc{B}_1$. We shall stick with the conditional independence terminology, motivated by the relation of this notion with Definition \ref{def:relcondindep}. There is also a useful interpretation of this notion in terms of certain subspaces of a Hilbert space being perpendicular. To detail this we use the following fact (for a proof see Lemma \ref{lem:intermeetApp}).
\begin{lemma}\label{lem:intermeet}
Let $(\Omega,\mc{A},\lambda)$ be a probability space, let $\mc{B}_0$, $\mc{B}_1$ be sub-$\sigma$-algebras of $\mc{A}$, and let $1\leq p<\infty$. Then $L^p(\mc{B}_0)\cap L^p(\mc{B}_1)= L^p(\mc{B}_0\wedge \mc{B}_1)$.
\end{lemma}
\noindent Recall that $L^2(\mc{B}_0)$ and $L^2(\mc{B}_1)$ are closed subspaces of the Hilbert space $L^2(\mc{A})$, and the expectation operator $f\mapsto \mb{E}(f|\mc{B}_i)$ is the orthogonal projection onto $L^2(\mc{B}_i)$ (see \cite[Chapter 5, \S 3]{Furst}). Then, by \eqref{eq:condindep} and Lemma \ref{lem:intermeet}, conditional independence of $\mc{B}_0$, $\mc{B}_1$ means that projection from one of these subspaces to the other is the same as projection to the intersection of these subspaces. This yields the intuition that $\mc{B}_0\upmod\mc{B}_1$ holds when $L^2(\mc{B}_0)$ and $L^2(\mc{B}_1)$ are in a sense perpendicular. This intuition is illustrated further by the following result, which we also use repeatedly in the sequel.

\begin{lemma}\label{lem:condindepequiv}
Let $(\Omega,\mc{A},\lambda)$ be a probability space, and let $\mc{B}_0$, $\mc{B}_1$ be sub-$\sigma$-algebras of $\mc{A}$. Then for $\mc{B}_0 \upmod \mc{B}_1$ to hold it is necessary and sufficient to have that every bounded $\mc{B}_0$-measurable function $f:\Omega\to \mb{R}$ such that $\mb{E}(f|\mc{B}_0\wedge \mc{B}_1)=_\lambda 0$ satisfies $\mb{E}(f|\mc{B}_1) =_\lambda 0$.
\end{lemma}

\begin{proof}
The necessity follows from \eqref{eq:condindep} and the fact that $f=_\lambda\mb{E}(f|\mc{B}_0)$. To prove the sufficiency, we let $f$ be any integrable function $\Omega\to\mb{R}$ and we show that \eqref{eq:condindep} holds. Let $g$ be a function equal to $\mb{E}(f|\mc{B}_0)-\mb{E}(f|\mc{B}_0\wedge \mc{B}_1)$ almost everywhere. Note that since $\mc{B}_0\wedge \mc{B}_1\subset_\lambda \mc{B}_0$, it follows from Lemma \ref{lem:nestexp} that $\mb{E}(f|\mc{B}_0\wedge \mc{B}_1)$ is almost-surely equal to a $\mc{B}_0$-measurable function. Hence this is also true of $g$, i.e.\ there is a $\mc{B}_0$-measurable function $h$ such that $g=_\lambda h$. By linearity of conditional expectation and the first  equality in Lemma \ref{lem:nestexp}, we have $\mb{E}(h|\mc{B}_0\wedge\mc{B}_1)=_\lambda \mb{E}(g|\mc{B}_0\wedge\mc{B}_1) =_\lambda \mb{E}(\mb{E}(f|\mc{B}_0)|\mc{B}_0\wedge\mc{B}_1)- \mb{E}(f|\mc{B}_0\wedge \mc{B}_1)=_\lambda 0$. By our assumption we therefore have $\mb{E}(h|\mc{B}_1)=_\lambda 0$. Hence $ \mb{E}(\mb{E}(f|\mc{B}_0)|\mc{B}_1) -\mb{E}(f|\mc{B}_0\wedge \mc{B}_1)=_\lambda \mb{E}(g|\mc{B}_1)=_\lambda \mb{E}(h|\mc{B}_1)=_\lambda 0$, and \eqref{eq:condindep} follows.
\end{proof}
\noindent We shall need to handle interactions between joins and meets of sub-$\sigma$-algebras. One result on this is the following (recorded here mainly for illustration; for a proof see Lemma \ref{lem:halfdistribApp}).
\begin{lemma}\label{lem:halfdistrib}
Let $(\Omega,\mc{A},\lambda)$ be a probability space and $\mc{B}_1,\mc{B}_2,\mc{B}_3$ be sub-$\sigma$-algebras of $\mc{A}$. Then $(\mc{B}_1\vee \mc{B}_2)\wedge \mc{B}_3 \supset (\mc{B}_1\wedge \mc{B}_3)\vee (\mc{B}_2\wedge \mc{B}_3)$. The opposite inclusion can fail.
\end{lemma}
\noindent While the inclusion in this lemma cannot be reversed in general, we can replace it with equality up to null sets in some situations, assuming conditional independence. This is the case in the following useful result, which can be seen as a special case of the \emph{modular law} from lattice theory.
\begin{lemma}\label{lem:modlaw}
Let $(\Omega,\mc{A},\lambda)$ be a probability space, let $\mc{B}$ and $\mc{C}$ be sub-$\sigma$-algebras of $\mc{A}$ satisfying $\mc{B}\upmod \mc{C}$, and let $\mc{B}_1$ be a sub-$\sigma$-algebra of $\mc{B}$. Then $(\mc{C} \vee \mc{B}_1) \wedge \mc{B}\; =_\lambda \;(\mc{C} \wedge \mc{B}) \vee \mc{B}_1$.
\end{lemma}
\noindent For a proof see Lemma \ref{lem:modlawApp}. A similar result appears in \cite[Corollary 16]{Yan2}.

We conclude this subsection with a few remarks on more specific types of probability spaces. From the next subsection onward, most of the key results from measure theory that we shall use (and, therefore, our main results in Section \ref{sec:structhm} themselves) can be established under the assumption that the probability spaces in question are \emph{standard probability spaces} (also called Lebesgue--Rokhlin spaces \cite[\S 9.4]{Boga2}). This is the case for instance in the result concerning the topological properties of coupling spaces, Proposition \ref{prop:coupspace}. Aiming for our main results to be applicable to any standard probability space is natural, given that these results are intended in particular for applications in ergodic theory. However, for the use of certain tools it is more convenient to work instead with the following closely related probability spaces.

\begin{defn}[Borel probability spaces]
A measurable space $(\Omega,\mc{A})$ is a \emph{standard Borel space} if there is a Polish topology $\tau$ on $\Omega$ such that $\mc{A}$ is the Borel $\sigma$-algebra $\sigma(\tau)$. A probability space $(\Omega,\mc{A},\lambda)$ is a \emph{Borel probability space} if $(\Omega,\mc{A})$ is a standard Borel space and $\lambda$ is a probability measure on $\mc{A}$. 
\end{defn}
\noindent To obtain our main results for general standard probability spaces, there will be no loss in assuming in several places that the spaces we work with are Borel probability spaces, because on one hand (as we detail in Section \ref{sec:structhm}) our main results are invariant under mod 0 isomorphisms of probability spaces, and on the other hand every standard probability space is mod 0 isomorphic to a Borel probability space (indeed this can be taken as a definition of standard probability spaces; see \cite[Definition 6.8]{Eisner&al}).

Situations in which Borel probability spaces are especially convenient for us include those where we have to work with disintegrations of measures. For these spaces we have a usefully simple form of the measure disintegration theorem; see \cite[(17.35), ii)]{Ke}.\footnote{The disintegration results valid for more general standard probability spaces come with less convenient additional technicalities, such as the fact that the $\sigma$-algebras on the fibres cannot be guaranteed to be almost all equal to the original $\sigma$-algebra; see for instance \cite[Example 1.2]{Pachl}.} This usefulness is illustrated by the following result, which we shall apply several times.
\begin{lemma}\label{lem:fibmeaspres}
For $i=1,2$ let $(\Omega_i,\mc{A}_i,\lambda_i)$ be a Borel probability space, let $\mu$ be a Borel measure on $(\Omega=\Omega_1\times\Omega_2,\,\mc{A}=\mc{A}_1\otimes \mc{A}_2)$, suppose that the projection $f_i:\Omega\to \Omega_i$ is measure-preserving \textup{(}i.e.\ $\mu\co f_i^{-1}=\lambda_i$\textup{)} for $i=1,2$, and that $f_1^{-1}(\mc{A}_1)$, $f_2^{-1}(\mc{A}_2)$ are independent in $\mu$. Let $(\mu_\omega)_{\omega\in \Omega_2}$ be a disintegration of $\mu$ relative to $f_2$. Then for $\lambda_2$-almost every $\omega$ the restriction $f_1:f_2^{-1}(\omega)\to \Omega_1$ is still measure-preserving \textup{(}i.e.\ $\mu_\omega\co f_1^{-1}=\lambda_1$\textup{)}.
\end{lemma}

\begin{proof}
First we claim that for an arbitrary fixed set $B\in \mc{A}_1$, for $\lambda_2$-almost every $\omega\in \Omega_2$ we have $\mu_\omega\co f_1^{-1}(B)=\lambda_1(B)$. To see this, fix any $C\in\mc{A}_2$ and note that the disintegration implies that $\mu(f_1^{-1}(B) \cap f_2^{-1}(C))=\int_{\Omega_2} 1_C(\omega)\, \mu_\omega(f_1^{-1}(B))\,  \ud\lambda_2(\omega)$. Since $f_1^{-1}(\mc{A}_1)$, $f_2^{-1}(\mc{A}_2)$ are independent, we have $\mu(f_1^{-1}(B) \cap f_2^{-1}(C)) = \mu(f_1^{-1}(B)) \; \mu(f_2^{-1}(C))  = \lambda_1(B)\,\lambda_2(C)$. We have thus shown that for every such set $C$ we have $\int_C  \mu_\omega(f_1^{-1}(B))\,  \ud\lambda_2(\omega) =\lambda_1(B)\,\lambda_2(C)$. This implies that the function $f:\omega \mapsto \mu_\omega(f_1^{-1}(B))$ equals the constant $\lambda_1(B)$ for $\lambda_2$-almost every $\omega$. Indeed, otherwise $\lambda_2(\{\omega: |f(\omega)-\lambda_1(B)|>\frac{1}{n}\})>0$ for some $n\in\mb{N}$, and then there would be $C\in \mc{A}_2$ such that $|\int_C \big(f(\omega)-\lambda_1(B)\big) \ud \lambda_2(\omega)|\geq \lambda_2(C)/n>0$, a contradiction (we must be able to take $C$ to be one of the sets $\{\omega: f(\omega)>\lambda_1(B)+\frac{1}{n}\}$, $\{\omega: f(\omega)<\lambda_1(B)-\frac{1}{n}\}$). This proves our claim. 

Now we apply this claim to each term of a sequence $(B_i)_{i\in \mb{N}}$ of sets in $\mc{A}_1$ that is closed under finite intersections and generates $\mc{A}_1$. The existence of such a sequence is clear when $\Omega_1$ is countable, and when it is uncountable the standard Borel space $(\Omega_1,\mc{A}_1)$ is Borel isomorphic to the interval $[0,1]$ with the Borel $\sigma$-algebra (see \cite[Theorem (15.6)]{Ke}), so in this case we can let the $B_i$ be the sets corresponding under this isomorphism to the open intervals in $[0,1]$ with rational end points. For each $B_i$, by the previous paragraph there is $C_i\in \mc{A}_2$ with $\lambda_2(C_i)=0$ and such that $\mu_\omega\co f_1^{-1}=\lambda_1$ for all $\omega\in \Omega_2\setminus C_i$. The set $D=\cup_{i\in \mb{N}} C_i$ is then a $\lambda_2$-null set such that for every $\omega\in \Omega_2\setminus D$, for every $i\in \mb{N}$ we have $\mu_\omega\co f_1^{-1}(B_i)=\lambda_1(B_i)$, whence by \cite[Lemma 1.9.4]{Boga1} we have $\mu_\omega\co f_1^{-1}(B)= \lambda_1(B)$ for every $B\in \mc{A}_1$.
\end{proof}
To close the subsection let us recall the following standard fact (for a proof see \cite{Prat}).
\begin{lemma}[Doob property of Polish spaces]\label{lem:Doob}
Let $(\Omega,\mc{A})$, $(\Omega',\mc{A}')$ be measurable  spaces, let $p:\Omega\to \Omega'$ be measurable, and let $X$ be a Polish space. For every $p^{-1}(\mc{A}')$-measurable function $f:\Omega\to X$ there is an $\mc{A}'$-measurable function $f':\Omega'\to X$ such that $f= f'\co p$. 
\end{lemma}
\noindent In particular, if $\lambda$ is a probability measure on $(\Omega,\mc{A})$ and $g:\Omega\to \mb{C}$ is a bounded $\mc{A}$-measurable function, then $\mb{E}\big(g|p^{-1}(\mc{A}')\big)$ can be regarded as a function on $\Omega'$, i.e.\ there is a bounded $\mc{A}'$-measurable function $f':\Omega'\to\mb{C}$ such that $\mb{E}(g|p^{-1}(\mc{A}'))=_\lambda f'\co p$. A simple but useful consequence is that if $(\Omega_i,\mc{A}_i,\lambda_i)$, $i=1,2$ are probability spaces and $\phi:\Omega_1\to\Omega_2$ is measure-preserving (i.e.\ $\phi$ is $(\mc{A}_1,\mc{A}_2)$-measurable and $\lambda_1\co\phi^{-1}=\lambda_2$), then for every sub-$\sigma$-algebra $\mc{B}\subset\mc{A}_2$ and every function $f\in L^1(\mc{A}_2)$ we have
\begin{equation}\label{eq:exprel}
\mb{E}_{\lambda_2}(f|\mc{B})\co \phi =_{\lambda_1} \mb{E}_{\lambda_1}(f\co \phi\,|\,\phi^{-1}\mc{B}).
\end{equation}

\subsection{Couplings}\label{subsec:couplings}\hfill \medskip \\
Given sets $T\subset S$ and a Cartesian product of sets $\prod_{v\in S} X_v$, we denote by $p_T$ the projection $\prod_{v\in S} X_v \to \prod_{v\in T} X_v$, $(x_v)_{v\in S}\mapsto (x_v)_{v\in T}$. (When $T=\{w\}$ we write $p_w$ rather than $p_{\{w\}}$.) 

Given probability spaces $\varOmega_v=(\Omega_v,\mc{A}_v,\lambda_v)$, $v\in S$, we denote by $\prod_{v\in S} (\Omega_v,\mc{A}_v)$ the product measurable space, consisting of the Cartesian product $\prod_{v\in S} \Omega_v$ and the product $\sigma$-algebra $\bigotimes_{v\in S} \mc{A}_v=\bigvee_{v\in S} p_v^{-1}(\mc{A}_v)$.
\begin{defn}[Coupling]\label{def:coup}
Let $S$ be a set and for each $v\in S$ let $\varOmega_v=(\Omega_v,\mc{A}_v,\lambda_v)$ be a probability space. A \emph{coupling} of the probability spaces $(\varOmega_v)_{v\in S}$ (or of the measures $\lambda_v$) is a measure $\mu$ on $\prod_{v\in S} (\Omega_v,\mc{A}_v)$ such that for each $v\in S$ we have $\mu\co p_v^{-1}=\lambda_v$. When $\varOmega_v=\varOmega$ for every $v\in S$, we call $\mu$ a \emph{self-coupling of $\varOmega$} (or \emph{of} $\lambda$) \emph{indexed by} $S$.
\end{defn}
\noindent In this paper $S$ denotes a finite set, except in certain clearly indicated places where it can also denote a countably infinite set (for instance in Section \ref{sec:exchange}). Note that if every $\varOmega_v$ is a Borel probability space then, for every coupling $\mu$ of these spaces, the probability space $(\prod_{v\in S} \Omega_v,\bigotimes_{v\in S} \mc{A}_v,\mu)$ is Borel (since $\prod_{v\in S} (\Omega_v,\mc{A}_v)$ is standard Borel \cite[p.\ 75]{Ke}).

In our analysis of couplings, the following functions play a key role.
\begin{defn}\label{def:multilin}
Let $\mu$ be a coupling of $(\varOmega_v)_{v\in S}$. Let $F=(f_v:\Omega_v\to\mb{C})_{v\in S}$ be a system of bounded measurable functions. Then we define
\begin{equation}\label{eq:cupconst2}
\xi(\mu,F):= \int_{\prod_{v\in S}\Omega_v}\;\prod_{v\in S} f_v\co p_v\;\ud\mu.
\end{equation}
\end{defn}
\noindent Note that the $L^1$-norm of each function in $F$ controls the function $\xi(\mu,\cdot):F\mapsto \xi(\mu,F)$, more precisely, for every $w\in S$ we have $|\xi(\mu,F)|\leq \|f_w\|_{L^1(\lambda_w)}\prod_{v\in S\setminus\{w\}}\|f_v\|_{L^\infty(\lambda_v)}$. We can use the functions $\xi$ to define a topology on a set of couplings, as follows.

\begin{defn}[Coupling space]\label{def:couptop}
Let $\varOmega$ be a probability space and let $S$ be a set. We denote by $\coup(\varOmega,S)$ the topological space consisting of the set of self-couplings of $\varOmega$ indexed by $S$ and the initial topology generated by the functions $\xi(\cdot,F):\mu \mapsto \xi(\mu,F)$, for systems $F=(f_v)_{v\in S}$ of bounded measurable functions $f_v:\Omega_v\to \mb{C}$.
\end{defn}
\noindent The following result gives a property of coupling spaces that is crucial for the sequel.
\begin{proposition}\label{prop:coupspace}
Let $S$ be a finite set and let $\varOmega$ be a Borel or standard probability space. Then $\coup(\varOmega,S)$ is a non-empty convex\footnote{This convexity property involves the vector-space structure on the set of signed measures on $(\Omega,\mc{A})^S$, of which $\coup(\varOmega,S)$ is a subset. The convexity property states that for every Borel probability measure $\nu$ on $\coup(\varOmega,S)$ we have that $\int_{\coup(\varOmega,S)} \;\mu\; \ud\nu(\mu)$ is a measure in $\coup(\varOmega,S)$.} compact Polish space.
\end{proposition}
\noindent Proofs of similar results appear in the literature, for instance in work of Kellerer \cite[Proposition 1.2]{Kell}, and of Furstenberg \cite[Lemma 5.2]{FurstSzem} (where couplings are called \emph{standard measures}). We include a proof of Proposition \ref{prop:coupspace} in the appendix (see Proposition \ref{prop:coupspaceApp}).

We now turn to several constructions of new couplings out of given ones, and other  useful properties of couplings.

\begin{defn}[Factor coupling]\label{def:factorcoup}
Let $\big(\varOmega_v=(\Omega_v,\mc{A}_v,\lambda_v)\big)_{v\in S}$ be a system of probability spaces, and let $\mu$ be a coupling of this system. A \emph{factor} of $\mu$ is a coupling obtained by restricting $\mu$ to a product $\sigma$-algebra $\bigotimes_{v\in S} \mc{B}_v$ where $\mc{B}_v$ is a sub-$\sigma$-algebra of $\mc{A}_v$ for each $v\in S$. If $\varOmega_v=\varOmega=(\Omega,\mc{A},\lambda)$ for all $v\in S$ and $\mu\in \coup(\varOmega,S)$, then given a sub-$\sigma$-algebra $\mc{B}\subset \mc{A}$ we write $_{\mc{B}|}{\mu}$ to denote the factor coupling of $\mu$ corresponding to $\mc{B}$. 
\end{defn}
\noindent Note that for a general standard probability space $(\Omega,\mc{A},\lambda)$, for a sub-$\sigma$-algebra $\mc{B}$ the probability space $(\Omega,\mc{B},\lambda|_\mc{B})$ may not be standard, for it may not be separable in the sense of Rokhlin (as defined in \cite[\S 9.4]{Boga2} for instance). However, there exists a standard probability space $(\Omega',\mc{B}',\lambda')$ and a measure-preserving map $\phi:\Omega\to\Omega'$ such that $\phi^{-1}(\mc{B}')=\mc{B}$ (see \cite[Theorem 57]{KalCut}). 
Similar facts hold for a standard Borel space $(\Omega,\mc{A})$. Indeed it follows from \cite[Corollary (15.2)]{Ke} that the only sub-$\sigma$-algebra of $\mc{A}$ that makes $\Omega$ a standard Borel space is $\mc{A}$ itself. However, if $\mc{B}$ is a countably generated sub-$\sigma$-algebra of $\mc{A}$, then by \cite[(14.16), (18.20)]{Ke} there is a standard Borel space $(\Omega',\mc{B}')$ and a Borel map $f:\Omega\to \Omega'$ such that $\mc{B}=f^{-1}(\mc{B}')$, and then $(\Omega',\mc{B}', \mu\co f^{-1})$ is a Borel probability space.

\begin{defn}[Relative independence over a factor]\label{def:indoverfact}
Let $\big(\varOmega_v=(\Omega_v,\mc{A}_v,\lambda_v)\big)_{v\in S}$ be a system of probability spaces, let $\mu$ be a coupling of this system, and let $\mu'$ be a factor of $\mu$ corresponding to $\sigma$-algebras $\mc{B}_v\subset \mc{A}_v$,  $v\in S$. We say that $\mu$ is \emph{relatively independent over} $\mu'$ if for every system $G=(g_v:\Omega_v\to \mb{C})_{v\in S}$ of bounded measurable functions, the system $G'=\big(\mb{E}(g_v|\mc{B}_v)\big)_{v\in S}$ satisfies $\xi(\mu,G)=\xi(\mu,G')$.
\end{defn}
\noindent In particular, if $\mu \in \coup(\varOmega,S)$ is relatively independent over a factor $\mu'=\prescript{}{\mc{B}|}{\mu}$, then the multilinear map $G\mapsto \xi(\mu,G)$ is uniquely determined by $\mu'$. This notion agrees with that of a \emph{conditional product measure} from \cite[see Lemma 9.1]{FurstSzem}.
\begin{remark}\label{rem:indoverfact}
We have $\mu$ relatively independent over $\mu'$ if and only if for every $w\in S$ and every system $G=(g_v)_{v\in S}$ of functions $g_v\in L^\infty(\mc{A}_v)$ with $\mb{E}(g_w|\mc{B}_w)=0$, we have $\xi(\mu,G)=0$. This equivalence follows from a basic argument using the linearity of the map $G\mapsto \xi(\mu,G)$ in each entry $g_v$, and the decomposition of any $\mc{A}$-measurable function $g$ as the sum $\mb{E}(g|\mc{B}_v) + \big(g-\mb{E}(g|\mc{B}_v)\big)$.
\end{remark}
\begin{defn}[Subcouplings along subsets]
Given $\mu\in\coup(\varOmega,S)$, and a set $T\subset S$, the \emph{subcoupling of $\mu$ along} $T$, denoted by $\mu_T$, is the image measure $\mu\co p_T^{-1}\in \coup(\varOmega,T)$.
\end{defn}
\noindent Note that the $\sigma$-algebra on which $\mu_T$ is defined is the power $\sigma$-algebra $\bigotimes_{v\in T}\mc{A}$. From now on we denote such a power $\sigma$-algebra by $\mc{A}^T$. We shall often need to handle preimages of such $\sigma$-algebras $\mc{A}^T$ under projections $p_T:\Omega^S\to \Omega^T$. We denote this sub-$\sigma$-algebra $p_T^{-1}(\mc{A}^T)$ of $\mc{A}^S$ by $\mc{A}^S_T$. Thus $\mc{A}^S_T$ is the sub-$\sigma$-algebra of $\mc{A}^S$ consisting of sets whose indicator functions depend only on coordinates indexed by $T$. Note that $\mc{A}^S_T \subset \mc{A}^S_{T'}$ whenever $T\subset T'$, that $\mc{A}^S_\emptyset$ is the trivial $\sigma$-algebra on $\Omega^S$, and that $\mc{A}^S_S$ is just $\mc{A}^S$. When $T$ is a singleton $\{v\}$ we denote the $\sigma$-algebra $\mc{A}^S_{\{v\}}$ simply by $\mc{A}^S_v$ (this notation will not clash with previous notations $\mc{A}_v$ above, because from now on we only consider self-couplings of $\varOmega$, so the ambient $\sigma$-algebra $\mc{A}$ is the same for every $v\in S$).
\begin{defn}[Subcouplings along injections]\label{def:subcoupembeds}
Let $\mu\in\coup(\varOmega,S)$ and let $\tau: R \to S$ be an injection. The \emph{subcoupling of $\mu$ along $\tau$} is the coupling $\mu_\tau \in \coup(\varOmega,R)$ obtained as follows: in the coupling $\mu_{\tau(R)}$ along  $\tau(R)\subset S$, each $v\in \tau(R)$ is renamed $w=\tau^{-1}(v)\in R$.
\end{defn}
\noindent We often consider two couplings that are equal up to renaming  the indices. Let us formalize this as follows.

\begin{defn}[Isomorphism of couplings]\label{def:coupiso}
Let $\mu \in \coup(\varOmega,S)$ and $\mu' \in \coup(\varOmega,S')$. We call a bijection $\sigma : S \to S'$ an \emph{isomorphism} of $\mu$ and $\mu'$ if $\mu'_\sigma= \mu$. If there is such an isomorphism we say that $\mu$ and $\mu'$ are \emph{isomorphic} and write $\mu \cong \mu'$, or $\mu \cong_\sigma \mu'$ if we wish to specify the isomorphism.
\end{defn}
\noindent Thus for instance the couplings $\mu_\tau$ and $\mu_{\tau(R)}$ in Definition \ref{def:subcoupembeds} are isomorphic.

In general a coupling with index set $S$ is not determined by subcouplings on two subsets $T_1$, $T_2$ with $T_1\cup T_2=S$. However, under certain additional conditions it \emph{is} determined, and the following result gives a useful example.

\begin{lemma}\label{lem:coupdeterm}
Let $\mu\in \coup(\varOmega,S)$ and let $T_1,T_2\subset S$ with $T_1\cup T_2=S$. Suppose that $\mc{A}^S_{T_2} =_\mu \mc{A}^S_{T_1\cap T_2}$. Then $\mu$ is uniquely determined by $\mu_{T_1},\mu_{T_2}$.
\end{lemma}
\begin{proof}
Let $F=(f_v)_{v\in S}$ be a system of bounded $\mc{A}$-measurable functions. We  show that $\xi(\mu,F)$ is uniquely determined by the subcouplings $\mu_{T_1},\mu_{T_2}$. The function $g:=\prod_{v\in T_2} f_v\co p_v$ on $\Omega^S$ satisfies clearly $g= g'\co p_{T_2}$ for an $\mc{A}^{T_2}$-measurable function $g'$ (we can just take $g'=\prod_{v\in T_2} f_v\co p_v$ where now $p_v$ is defined on $\Omega^{T_2}$). Fix any $\epsilon>0$, and note that since by assumption $L^2(\mc{A}^{T_2},\mu_{T_2})=L^2(\mc{A}^{T_2}_{T_1\cap T_2},\mu_{T_2})$, by Lemma \ref{lem:pisysapprox} we can approximate $g'$ within $\epsilon$ in $L^2(\mu_{T_2})$ by $\sum_{i\in [m]} h_i$ where $h_i=\prod_{v\in T_1\cap T_2} g_{i,v}\co p_v$ for some bounded $\mc{A}$-measurable functions $g_{i,v}$ (note that this is done entirely in the known coupling $\mu_{T_2}$). For each $i\in [m]$ define the system $G_i=(g'_{i,v})_{v\in T_1}$ by $g'_{i,v}=f_v$ for $v\in T_1\setminus T_2$ and $g'_{i,v}=g_{i,v}$ for $v\in T_1\cap T_2$. It follows that $|\xi(\mu,F)-\sum_{i\in [m]}\xi(G_i,\mu_{T_1})|\leq \epsilon$, where each $\xi(G_i,\mu_{T_1})$ involves only the known couplings $\mu_{T_1}$, $\mu_{T_2}$. Since $\epsilon>0$ was arbitrary, the result follows.
\end{proof}
\noindent Recall from Definition \ref{def:condindep} the notation $\mc{B}_0 \upmod_{\mu} \mc{B}_1$ for conditional independence. We now use this to define a related notion for subsets of the index set of a self-coupling.
\begin{defn}[Conditionally independent index sets]\label{def:condindepsets}
Let $\mu\in\coup(\varOmega,S)$ and let $T_1,T_2\subset S$. We say that $T_1,T_2$ are \emph{conditionally independent in} $\mu$, and write $T_1 ~\bot_\mu ~ T_2$, if we have $\mc{A}^S_{T_1} \upmod_{\mu} \mc{A}^S_{T_2}$ and
\begin{equation}\label{eq:condindepsets}
\mc{A}^S_{T_1}\wedge_\mu \mc{A}^S_{T_2} \, =_\mu\, \mc{A}^S_{T_1\cap T_2}.
\end{equation}
\end{defn}
\noindent As for previous notations, when the ambient coupling $\mu$ is clear  we just write $T_1~\bot~T_2$. Let us note the following useful equivalent definition of the relation $\bot$.

\begin{lemma}\label{lem:botsuff}
Let $\mu \in \coup(\varOmega,S)$ and let $T_1, T_2 \subset S$. Then $T_1~\bot~T_2$ if and only if for every bounded $\mc{A}^S_{T_1}$-measurable $f:\Omega^S\to\mb{C}$ there is a $\mc{A}^S_{T_1\cap T_2}$-measurable function $h$ such that $\mb{E}(f|\mc{A}^S_{T_2})=_\mu h$. 
\end{lemma}
In particular for every $T\subset S$ and $F\subset T$, we have $F~\bot~T$.
\begin{proof}
To see the forward implication, note that from Definition \ref{def:condindep} and Proposition \ref{prop:condindep} it follows that $\mb{E}(f|\mc{A}^S_{T_2})=_\mu \mb{E}(f|\mc{A}^S_{T_1}\wedge\mc{A}^S_{T_2})$, and by Definition \ref{def:condindepsets} we have $\mc{A}^S_{T_1}\wedge\mc{A}^S_{T_2} =_\mu \mc{A}^S_{T_1\cap T_2}$. In particular $\mc{A}^S_{T_1}\wedge\mc{A}^S_{T_2} \subset_\mu \mc{A}^S_{T_1\cap T_2}$ and so by Lemma \ref{lem:nestexp} we deduce that $\mb{E}(f|\mc{A}^S_{T_1}\wedge\mc{A}^S_{T_2})$ has a $\mc{A}^S_{T_1\cap T_2}$-representative under $=_\mu$. But then $\mb{E}(f|\mc{A}^S_{T_2})$ also has such a representative $h$, as claimed.

 For the backward implication, let $\mc{B}_0=\mc{A}^S_{T_1}$, $\mc{B}_1=\mc{A}^S_{T_2}$, and let $h$ be an $\mc{A}^S_{T_1\cap T_2}$-measurable representative of $\mb{E}(f|\mc{B}_1)$. In particular $\mb{E}(\mb{E}(f|\mc{B}_1)|\mc{B}_0)=_\mu\mb{E}(h|\mc{B}_0)$. Since $\mc{A}^S_{T_1\cap T_2}\subset \mc{B}_0\cap \mc{B}_1$, we have that $h$ is both $\mc{B}_0$-measurable and $\mc{B}_0\wedge \mc{B}_1$-measurable, so $\mb{E}(h|\mc{B}_0)=_\mu h =_\mu \mb{E}(h|\mc{B}_0\wedge\mc{B}_1)$, so $\mb{E}(\mb{E}(f|\mc{B}_1)|\mc{B}_0)=_\mu \mb{E}(h|\mc{B}_0\wedge\mc{B}_1) =_\mu \mb{E}(\mb{E}(f|\mc{B}_1)|\mc{B}_0\wedge\mc{B}_1)$. But the latter is $=_\mu \mb{E}(f|\mc{B}_0\wedge\mc{B}_1)$, by Lemma \ref{lem:nestexp}, so \eqref{eq:condindep} holds, whence $\mc{B}_0 \upmod \mc{B}_1$. On the other hand, for every set $B \in \mc{A}^S_{T_1}  \wedge \mc{A}^S_{T_2}$ we have $\mb{E}(1_B|\mc{A}^S_{T_2})\,=_\mu\,1_B$, but by assumption we also have that $\mb{E}(1_B|\mc{A}^S_{T_2})$ is $\mc{A}^S_{T_1\cap T_2}$-measurable. Therefore $\mc{A}^S_{T_1}  \wedge \mc{A}^S_{T_2} \subset_{\mu} \mc{A}^S_{T_1\cap T_2}$. Since we also clearly have $\mc{A}^S_{T_1}  \wedge \mc{A}^S_{T_2} \supset \mc{A}^S_{T_1\cap T_2}$, the result follows. 
\end{proof}

\begin{remark}\label{rem:botinsubcoup}
Note that if $V_1,V_2\subset S$ satisfy $V_1\,\bot_{\mu_T}\, V_2$ for $\mu\in\coup(\varOmega,S)$ and some set $T\subset S$ with $T\supset V_1\cup V_2$, then we have $V_1\,\bot_{\mu_{T'}}\, V_2$ for every $T'\subset S$ with $T'\supset V_1\cup V_2$.
\end{remark}

\begin{example}\label{ex:U2coup}
Let us illustrate some of the previous definitions and results with a basic example of a coupling familiar in arithmetic combinatorics. Consider the probability space consisting of a compact abelian group $\ab$ equipped with the Haar probability measure $\lambda$, let $S=\{0,1\}^2$, and let $\mu\in \coup(\ab,S)$ be supported on the group of standard $2$-cubes $G=\{x=(x_{00},x_{10},x_{01},x_{11}): x_{00}-x_{10}= x_{01}-x_{11}\}\leq \ab^S$ and equal to the Haar probability measure on $G$. (This is a coupling since each projection $p_v$ is a continuous surjective homomorphism $G\to \ab$.) The sets $T_1=\{00,10\}$, $T_2=\{00,01\}$  satisfy $T_1\,\bot_\mu\, T_2$. Indeed, letting $\mc{A}$ be  the Borel $\sigma$-algebra on $\ab$, let $f$ be any bounded $\mc{A}^S$-measurable function, and consider the function $a_{T_1}$ defined on $\ab^{S}$ by $a_{T_1}(x) = \int_{\ab} f(x_{00},x_{10},x_{01}+k,x_{11}+k)\ud\lambda(k)$. Let $a_{T_1}'$ be the function defined on $\ab^{S}$ by  $a_{T_1}'(x_{00},x_{10},x_{01},x_{11})=a_{T_1}(x_{00},x_{10},x_{00},x_{10})$. Note that $a_{T_1}'$ is $\mc{A}_{T_1}^S$-measurable (it depends only on $x_{00},x_{10}$), and that $a_{T_1}'=_\mu a_{T_1}$, since $\mu$ is supported on $G$ and $a_{T_1}'(x)= a_{T_1}(x)$ for all $x\in G$. For any $g\in L^\infty(\mc{A}^S_{T_1})$, let $g'$ be the $\mc{A}^{T_1}$-measurable function such that $g=_\mu g'\co p_{T_1}$ (given by Lemma \ref{lem:Doob}). Then, using the parametrization $x=(y,y+h,y+k,y+h+k)$, $y,h,k\in \ab$ for $x\in G$, we have
\begin{eqnarray*}
\int f g\ud\mu \! & = \! & \int_{\ab^3} f(y,y+h,y+k,y+h+k)\;g(y,y+h,y+k,y+h+k) \ud\lambda^3(y,h,k) \\
\! & = \! & \int_{\ab^2} g'(y,y+h) \Big(\int_{\ab} f(y,y+h,y+k,y+h+k) \ud\lambda\Big) \ud\lambda^2(y,h)\; = \int a_{T_1}\, g\ud\mu.
\end{eqnarray*}
Hence $a_{T_1}$ is a representative of $\mb{E}(f|\mc{A}^S_{T_1})$ under $=_\mu$. Similarly, a representative of $\mb{E}(f|\mc{A}^S_{T_2})$ is the function $x\mapsto \int_{\ab} f(x_{00},x_{10}+h,x_{01},x_{11}+h)\ud\lambda(h)$, and yet another similar argument shows that the function $x\mapsto \int_{\ab^2} f(x_{00},x_{10}+h,x_{01}+k,x_{11}+h+k)\ud\lambda(h,k)$ represents $\mb{E}(f|\mc{A}^S_{T_1\cap T_2})$. From this it is seen clearly that the composition of the operators $\mb{E}(\cdot|\mc{A}^S_{T_1})$, $\mb{E}(\cdot|\mc{A}^S_{T_2})$ is $\mb{E}(\cdot|\mc{A}^S_{T_1\cap T_2})$ (in particular these operators commute), whence $T_1\,\bot_\mu\, T_2$ holds indeed, by Lemma \ref{lem:botsuff}. Now let us take instead $T_1=\{00\}$ and $T_2=S\setminus T_1$. A similar consideration of the operators $\mb{E}(\cdot|\mc{A}^S_{T_1})$, $\mb{E}(\cdot|\mc{A}^S_{T_2})$ reveals that they still commute, whence we still have $\mc{A}^S_{T_1} \upmod_{\mu} \mc{A}^S_{T_2}$ (see \cite[Theorem 9]{Yan}). However, now we do not have $T_1\,\bot_\mu\, T_2$, because $\mc{A}^S_{T_1} \wedge \mc{A}^S_{T_2}$ is not the trivial $\sigma$-algebra $\mc{A}^S_{T_1\cap T_2}=\mc{A}^S_{\emptyset}$. Indeed, note that for any character $\chi\in\wh{\ab}$, on the group $G$ we have $\chi\co p_{00}=(\chi\co p_{10})(\chi\co p_{01})(\overline{\chi}\co p_{11})$, so $\chi\co p_{00}\in L^\infty(\mc{A}^S_{T_1} \wedge \mc{A}^S_{T_2})$, which shows that $\mc{A}^S_{T_1} \wedge \mc{A}^S_{T_2}$ is indeed non-trivial.
\end{example}
\begin{defn}[Conditionally independent system of sets]\label{def:condindepsys}
Let $\mu\in \coup(\varOmega,S)$. We say that a system $(T_i)_{i\in [k]}$ of subsets of $S$ is \emph{conditionally independent} if for every $R_1,R_2\subseteq [k]$  we have $\big(\bigcup_{j\in R_1} T_i\big)~\bot~\big(\bigcup_{j\in R_2}T_i\big)$.
\end{defn}
\noindent An example of this property is given by a 3-dimensional generalization of Example \ref{ex:U2coup}, letting $\mu \in\coup(\ab,\{0,1\}^3)$ be similarly given by the Haar measure on the group of standard 3-cubes on $\ab$, and letting $(T_1,T_2,T_3)$ be the system of the 2-dimensional faces of $\{0,1\}^3$ containing the point $0^3$. It can be checked directly that this system is conditionally independent (for example by computing what the various expectation operators are, as in Example \ref{ex:U2coup}). This is also established more generally in Section \ref{sec:cc} (see Remark \ref{rem:strongaxiom}).

The following notion enables us to ``glue" together two couplings by identifying parts of their index sets, in such a way that the two index sets become conditionally independent in the new coupling. The definition uses the following notation: given two finite sets $S$, $S'$, two subsets $T\subseteq S$, $T'\subseteq S'$, and a bijection $\sigma:T\to T'$, we denote by $S\cup_\sigma S'$ the set obtained by first taking the disjoint union of $S$ and $S'$ and then identifying every $t\in T$ with $\sigma(t)$ (thus $|S\cup_\sigma S'|=|S|+|S'|-|T|$). 
\begin{defn}[Conditionally independent coupling along a bijection]\label{def:condindcoup}\hfill\\
Let $\varOmega$ be a Borel or standard probability space, and let $\mu\in\coup(\varOmega,S)$, $\mu'\in\coup(\varOmega,S')$, $T\subseteq S$, $T'\subseteq S'$. Let $\sigma: T\to T'$ be a bijection such that $\mu_T\cong_\sigma \mu'_{T'}$, and let $U=S\cup_\sigma S'$. The \emph{conditionally independent coupling of $\mu$, $\mu'$ along} $\sigma$ is the unique coupling $\mu'' \in\coup(\varOmega,S\cup_\sigma S')$ such that $S\bot_{\mu''}S'$.
\end{defn}

\noindent This definition requires the following justification.
\begin{lemma}\label{lem:condindcoup}
The coupling $\mu'' \in\coup(\varOmega,S\cup_\sigma S')$ in Definition \ref{def:condindcoup} exists and is unique.
\end{lemma}
\begin{proof}
Using the notation in the definition, let $B_1\in \mc{A}^{S\setminus T}$, $B_2\in \mc{A}^{T}$, $B_3\in \mc{A}^{S'\setminus T'}$, and let $U=S\cup_\sigma S'$.  We have that $B_1\times B_2\times B_3\subseteq \Omega^U$, and $\mc{A}^U$ is generated by such sets. Since $\mb{E}_\mu(1_{B_1}\co p_{S\setminus T}|\mc{A}^S_T)$ is $\mc{A}^S_T$-measurable, it can be regarded as an $\mc{A}^T$-measurable function on $\Omega^T$ (by Lemma \ref{lem:Doob}). Similarly $\mb{E}_{\mu'}(1_{B_3}\co p_{S'\setminus T'}|\mc{A}^{S'}_{T'})$ can be regarded as an $\mc{A}^T$-measurable function on $\Omega^T$ (since $\mu_T\cong_\sigma \mu'_{T'}$). We can therefore define (abusing the notation)
\begin{equation}\label{eq:cid}
\mu''(B_1\times B_2\times B_3) := \int_{\Omega^T} \mb{E}_\mu(1_{B_1}\co p_{S\setminus T}|\mc{A}^S_T)\;\; 1_{B_2} \;\; \mb{E}_{\mu'}(1_{B_3}\co p_{S'\setminus T'}|\mc{A}^{S'}_{T'})\, \ud\mu_T.
\end{equation}
This formula implies that $\mu''(B_1\times B_2\times B_3)$ is additive in each entry $B_1,B_2,B_3$, which implies that it satisfies property (i) in Definition \ref{def:premeas}. Property (ii) from that definition clearly holds as well, so the existence of the coupling $\mu''$ follows from Lemma \ref{lem:coupcorresp}. 

To prove that $S~\bot_{\mu''}~S'$, let us first show that, by \eqref{eq:cid}, for every function $f\in L^\infty(\mc{A}_S^U)$ and $f'\in L^\infty(\mc{A}_{S'}^U)$, we have
\begin{equation}\label{eq:cid2}
\int_{\Omega^U} f\; f' \ud \mu''= \int_{\Omega^U} \mb{E}(f|\mc{A}_T^U) \; \mb{E}(f'|\mc{A}_T^U) \ud\mu''.
\end{equation}
This can be deduced by approximating $f$ in $L^2(\mc{A}_S^U)$ by simple functions involving sets of the form $B_1\times B_2 \times \Omega^{S'\setminus T'}$, similarly for $f'$ with sets of the form $\Omega^{S\setminus T}\times B_2\times B_3$, and applying \eqref{eq:cid} to intersections of such sets. Now, if $g\in L^\infty(\mc{A}_S^U)$, then applying \eqref{eq:cid2} with $f=\overline{f'}=\mb{E}(g|\mc{A}_{S'}^U)$ we deduce that $\|\mb{E}(g|\mc{A}_{S'}^U)\|_{L^2(\mu'')}=\|\mb{E}(g|\mc{A}_T^U)\|_{L^2(\mu'')}$, which implies that $\mb{E}(g|\mc{A}_{S'}^U)$ is $\mc{A}_T^U$-measurable, and then $S~\bot_{\mu''}~S'$ follows by Lemma \ref{lem:botsuff}.

To see that $\mu''$ is unique, suppose that $\nu\in \coup(\varOmega,S\cup_\sigma S')$ satisfies $S~\bot_{\nu}~S'$. Then given any sets $B_1$, $B_2$, $B_3$ as above, we have $\int_{\Omega^U} (1_{B_1}\co p_{S\setminus T})(1_{B_2}\co p_T )(1_{B_3}\co p_{S'\setminus T'}) \ud\nu=$ $\int_{\Omega^U} \mb{E} (1_{B_1} \co p_{S\setminus T}|\mc{A}^U_{S'}) (1_{B_2}\co p_T) 1_{B_3}\co p_{S'\setminus T'}\ud\nu$, where $\mb{E}(1_{B_1} \co p_{S\setminus T}|\mc{A}^U_{S'})=\mb{E}(1_{B_1} \co p_{S\setminus T}|\mc{A}^U_T)$ since $S~\bot_{\nu}~S'$, so the last integral is $\int_{\Omega^U} \mb{E}(1_{B_1}\co p_{S\setminus T}|\mc{A}^U_{T}) 1_{B_2}\co p_T \mb{E}(1_{B_3}\co p_{S'\setminus T'} |\mc{A}^U_T)\ud\nu$, and this yields the right side of \eqref{eq:cid}. We have thus shown that $\nu( B_1\times B_2 \times B_3) = \mu''( B_1\times B_2 \times B_3)$, and so $\nu=\mu''$ by uniqueness in Carath\'eodory's extension theorem.
\end{proof}
\noindent The next definition and result will be used in Section \ref{sec:ergapps} for applications in ergodic theory.
\begin{defn}
Let $\varOmega=(\Omega,\mc{A},\lambda)$ be a probability space, and let $\mu\in \coup(\varOmega,S)$. For each $v\in S$ let $\theta_v$ be a measure-preserving transformation on $\Omega$. We denote by $\theta$ the corresponding transformation $\Omega^S\to\Omega^S$, $(\omega_v)_{v\in S}\mapsto (\theta_v(\omega_v))_{v\in S}$.
\end{defn}
\noindent The following result enables us to view a measure-preserving group action as a family of continuous maps from a coupling space to itself.
\begin{lemma}\label{lem:mpcoupact}
Let $(\theta_v)_{v\in S}$ be a system of measure-preserving transformations on $\Omega$ and let $\theta$ be the corresponding transformation $\Omega^S\to\Omega^S$. Then the map $\coup(\varOmega,S)\to \coup(\varOmega,S)$, $\mu\mapsto \mu\co \theta^{-1}$ is continuous.
\end{lemma}
\begin{proof}
Fix any system $F=(f_v)_{v\in S}$ of bounded measurable functions on $\Omega$, and let $F'$ denote the system $(f_v\co\theta_v)_{v\in S}$. Then for every $\mu\in \coup(\varOmega,S)$, the assumption that each $\theta_v$ is measure-preserving ensures that $\mu\co \theta^{-1}\in \coup(\varOmega,S)$, and the functions from \eqref{eq:cupconst2} generating the topology on $\coup(\varOmega,S)$ satisfy $\xi(\mu\co \theta^{-1},F)=\xi(\mu,F')$. Let $(\mu_n)_{n\geq 1}$ be a sequence in $\coup(\varOmega,S)$ with $\mu_n\to \mu_0$. Then the last equality implies that $\mu_n\co \theta^{-1}\to \mu_0\co \theta^{-1}$, and the result follows.
\end{proof}

\subsection{Closed properties in a coupling space}\hfill\\
The results in this subsection identify certain useful closed subsets of a general coupling space (for a general probability space $\varOmega$).
\begin{lemma}\label{lem:indepclosed}
Let $S$ be a finite set, and let $T_1,T_2\subset S$ be disjoint sets. Let $Q$ be the set of couplings $\mu \in \coup(\varOmega,S)$ such that $\mc{A}_{T_1}^S$ and $\mc{A}_{T_2}^S$ are independent in $\mu$. Then $Q$ is a closed set in $\coup(\varOmega,S)$.
\end{lemma}
\begin{proof}
We can describe this independence property in terms of equations involving the functions $\xi$ from Definition \ref{def:couptop}. More precisely, the property holds if and only if we have $\xi(\mu,F)=\xi(\mu,F_1)\,\xi(\mu,F_2)$ for every system $F_1$ of functions $f_{1,v}\in L^\infty(\mc{A})$ with $f_{1,v}=1$ for $v\not\in T_1$, every system $F_2$ of functions $f_{2,v}\in L^\infty(\mc{A})$ with $f_{2,v}=1$ for $v\not\in T_2$, and $F$ the system with $f_v=f_{1,v}$ if $v\in T_1$, with $f_v=f_{2,v}$ if $v\in T_2$, and $f_v=1$ otherwise. For every such  system $F_1$ and $F_2$, the set of couplings satisfying  $\xi(\mu,F)=\xi(\mu,F_1)\xi(\mu,F_2)$ is closed (by continuity of the functions $\xi(\cdot,F)$ in general). Since $Q$ is the intersection of all these sets, the result follows.
\end{proof}

\begin{lemma}\label{lem:pindownclosed}
Let $S$ be a finite set, let $T\subset S$, and let $\nu\in\coup(\varOmega,T)$. Let $Q$ be the set of couplings $\mu \in \coup(\varOmega,S)$ such that $\mu_T=\nu$. Then $Q$ is a closed set in $\coup(\varOmega,S)$. 
\end{lemma}

\begin{proof}
The property $\mu_T=\nu$ holds if and only if we have $\xi(\mu,F)=\xi(\nu,F')$  for every system $F$ of functions $f_v\in L^\infty(\mc{A})$ with $f_v=1$ for $v\not\in T$, where $F'=(f_v)_{v\in T}$. For every fixed $F$, the condition $\xi(\mu,F)=\xi(\nu,F')$ defines a closed set of couplings $\mu\in\coup(\varOmega,S)$. The set $Q$ is the intersection of all such sets over all such systems $F$, so it is closed.
\end{proof}

\begin{lemma}\label{lem:botclosed}
Let $S$ be a finite set, let $T_1,T_2\subset S$ be such that $T_1\cap T_2=\{w\}$ for some $w\in S$, and let $\nu\in\coup(\varOmega,T_1)$. Let $Q$ be the set of couplings $\mu \in \coup(\varOmega,S)$ such that $\mu_{T_1}=\nu$ and $T_1\,\bot_\mu\, T_2$. Then $Q$ is a closed set in $\coup(\varOmega,S)$.
\end{lemma}

\begin{proof}
We can describe the properties defining $Q$ again in terms of equations involving the functions $\xi$. For $i=1,2$ let $F_i$ be a system of functions $f_{i,v}\in L^\infty(\varOmega)$ for $v\in T_i$, and let $F$ be the function system with $f_v=f_{1,v}$ for $v\in T_1\setminus\{w\}$, with $f_v=f_{2,v}$ for $v\in T_2\setminus\{w\}$, with $f_w=f_{1,w}f_{2,w}$, and with $f_v=1$ otherwise. Let $H$ be the function system with $h_v=1$ for $v\in S\setminus T_2$, with $h_v=f_{2,v}$ for $v\in T_2\setminus\{w\}$, and $h_w=\mb{E}(\prod_{v\in T_1} f_{1,v}\co p_v | \mc{A}_w^S)\, f_{2,w}$. It follows from Lemma \ref{lem:botsuff} and approximation by rank-1 functions (Lemma \ref{lem:pisysapprox}) that $T_1\,\bot_\mu\, T_2$ holds if and only if we have $\xi(\mu,F)=\xi(\mu,H)$ for every such $F$. Let $Q'=\{\mu\in\coup(\varOmega,S): \mu_{T_1}=\nu\}$, which is a closed set by Lemma \ref{lem:pindownclosed}. Note that if $\mu\in Q'$ then the function $h_w$ (and therefore $H$) does not change as $\mu$ varies. We have that $Q$ is the set of couplings $\mu\in Q'$ such that $\xi(\mu,F)=\xi(\mu,H)$ for every system $F$. For a single $F$, the fact that $H$ does not change as $\mu$ varies implies that the last equation defines a closed subset of $Q'$. The set $Q$ is the intersection of all these closed subsets of $Q'$, so it is closed.
\end{proof}

\begin{remark}
One may wonder whether the property $T_1\, \bot_\mu\, T_2$ always defines a closed set of couplings $\mu\in \coup(\varOmega,S)$ for $T_1,T_2\subset S$. It turns out that this is not true, as shown by the following example. Let $\theta$ be a mixing invertible measure-preserving transformation on $\varOmega=(\Omega,\mc{A},\lambda)$. For each $n\in \mb{N}$, let $\mu_n\in \coup(\varOmega, [3])$ be the image of $\lambda$ under the map $x\mapsto (\theta^n x,x,\theta^n x)$. It is readily seen that $\{1,2\}\,\bot_{\mu_n}\, \{2,3\}$ for every $n$ (in fact we have $\mc{A}^{[3]}_{\{1,2\}}=_{\mu_n}\mc{A}^{[3]}_{\{2,3\}}=_{\mu_n}\mc{A}^{[3]}_{\{2\}}$). Furthermore, the mixing property implies that $\mu_n$ converges to the coupling $\mu$ defined as the image of the product measure $\lambda\times \lambda$ on $\Omega^2$ under the map $(x,y)\mapsto (x,y,x)$. However, we do not have $\{1,2\}\,\bot_\mu\, \{2,3\}$ (we still have $\mc{A}^{[3]}_{\{1,2\}}=_{\mu}\mc{A}^{[3]}_{\{2,3\}}=_{\mu} \mc{A}^{[3]}_{\{1,2\}}\wedge \mc{A}^{[3]}_{\{2,3\}}$, but the latter $\sigma$-algebra is strictly larger than $\mc{A}^{[3]}_{\{2\}}$).
\end{remark}

\begin{lemma}\label{lem:relindclosed}
Let $S$ be a finite set and for each $v\in S$ let $\mc{B}\sbr{v}$ be a sub-$\sigma$-algebra of $\mc{A}$. Let $Q$ be the set of couplings $\mu\in\coup(\varOmega,S)$ such that $\mu$ is relatively independent over its factor corresponding to $\bigotimes_{v\in S} \mc{B}\sbr{v}$. Then $Q$ is a closed set in $\coup(\varOmega,S)$.
\end{lemma}

\begin{proof}
As mentioned in Remark \ref{rem:indoverfact}, we have that $\mu$ is relatively independent over this factor if and only if for every system $F=(f_v)_{v\in S}$ such that $\mb{E}(f_v|\mc{B}\sbr{v})=0$ for some $v$, we have $\xi(\mu,F)=0$.  For every such system, the set of couplings $\mu$ with $\xi(\mu,F)=0$ is closed. Since $Q$ is the intersection of all these sets, the result follows.
\end{proof}

\subsection{Localization}\hfill\\
\noindent We now turn to properties of couplings that involve measure disintegration. To handle disintegrations and related tools in a convenient way, in this subsection we assume that $\varOmega$ is a Borel probability space. The main type of disintegration that we use applies to a coupling $\mu\in \coup(\varOmega,S)$ relative to a projection $p_T:\Omega^S\to \Omega^T$ and the subcoupling $\mu_T$. A reference for this result is \cite[(17.35) ii)]{Ke}. 

First we want to ensure that in such a disintegration almost all the fibre measures are couplings in $\coup(\varOmega,S\setminus T)$. This will be shown to hold when $T$ is of the following kind.
\begin{defn}[Local set]
Let $\mu\in\coup(\varOmega,S)$. We say that a set $T\subset S$ is $\mu$-\emph{local} (or \emph{local in} $\mu$) if for every $v\in S\setminus T$ the $\sigma$-algebras $\mc{A}^S_v$ and $\mc{A}^S_T$ are independent in $\mu$.
\end{defn}
\noindent Note that the family of local subsets of $S$ is closed under intersection. We also have the following fact, which is a straightforward consequence of Lemma \ref{lem:indepclosed}.
\begin{lemma}
For a fixed $T\subset S$, the couplings in $\coup(\varOmega,S)$ in which $T$ is local form a closed set in $\coup(\varOmega,S)$.
\end{lemma}
\noindent The following lemma ensures the property relative to disintegrations mentioned above.

\begin{lemma}\label{lem:localize}
Let $\mu\in\coup(\varOmega,S)$ and let $T\subset S$ be local in $\mu$. Then there is a Borel measurable function $f_{\mu,T}:\Omega^T\to \coup(\varOmega,S\setminus T)$, $x\mapsto \mu_x$ such that for every function $f\in L^\infty(\mc{A}^S)$ we have $\int_{\Omega^S} f \ud\mu= \int_{\Omega^T} \int_{\Omega^{S\setminus T}} f \ud\mu_x\ud\mu_T$. Any other function $g$ with the same properties as $f_{\mu,T}$ satisfies $g=_{\mu_T} f_{\mu,T}$.
\end{lemma}
\begin{proof}
By \cite[(17.35) ii)]{Ke} there is a Borel function $x\mapsto \mu_x$ from $\Omega^T$ to the space of Borel measures on $\Omega^{S\setminus T}$, such that $(\mu_x)_{x\in \Omega^T}$ is a disintegration of $\mu$ relative to $p_T: \Omega^S\to \Omega^T$ and $\mu_T=\mu \co p_T^{-1}$, and such that any other such function agrees with this one $\mu_T$-almost surely. We have $\mu_x \in \coup(\varOmega,S\setminus T)$ for $\mu_T$-almost every $x$, indeed for each $v\in S\setminus T$ we see that $\mu_x\co p_v^{-1}=\lambda$,  by applying  Lemma \ref{lem:fibmeaspres} with $\Omega_1=\Omega$, $f_1=p_v$, $\Omega_2=\Omega^T$, and $f_2=p_T$. Letting $E$ denote a $\mu_T$-null set such that for every $x \in \Omega^T\setminus E$ we have $\mu_x \in \coup(\varOmega,S\setminus T)$, we can now define an appropriate function $f_{\mu,T}$ by fixing some arbitrary $\nu\in \coup(\varOmega,S\setminus T)$ and setting $f_{\mu,T}(x)=\mu_x$ for $x\in \Omega\setminus E$ and $f_{\mu,T}(x)=\nu$ otherwise. 
\end{proof}
\noindent Next we define an operation that will play a key role in the definition of certain topological spaces using couplings. 
\begin{defn}[Localization]\label{def:loc}
Let $\mu\in\coup(\varOmega,S)$ and let $T\subset S$ be a $\mu$-local set. The $T$-\emph{localization} of $\mu$ is the measurable function $f_{\mu,T}: \Omega^T\to \coup(\varOmega,S\setminus T))$, $x\mapsto \mu_x$ defined (uniquely up to a change on a $\mu_T$-null set) in Lemma \ref{lem:localize}.
\end{defn}
\noindent When the coupling $\mu$ is clear from the context, we shall write $f_T$ rather than $f_{\mu,T}$. We can use the $T$-localization to define a probability measure on the compact space $\coup(\varOmega,S\setminus T)$, namely the image measure $\mu_T\co f_{\mu,T}^{-1}$. This construction is important for the sequel, because it enables us to define certain topological spaces that will turn out to be the compact nilspaces involved in our main results in Section \ref{sec:structhm}. These spaces will be defined to be the \emph{supports} of measures of the form $\mu_T\co f_{\mu,T}^{-1}$. Let us recall here the notion of the support of a Borel measure (see also for instance \cite[Proposition 7.2.9]{Boga2}).
\begin{defn}[Support of a regular Borel measure]
Given a regular Borel measure $\mu$ on a topological space $X$, the \emph{support} of $\mu$ is the closed set $\Supp(\mu)=\{x\in X: \textrm{for every open set $U\ni x$ we have $\mu(U)>0$}\}$.
\end{defn}
\noindent We use the notation $\Supp$ to distinguish this from the purely set-theoretic notion of the support of a complex-valued function $f$ on a set $X$, that is $\supp(f)=\{x\in X:f(x)\neq 0\}$. 
\begin{remark}\label{rem:dualfns}
The localization construction, when applied in particular to cubic couplings (discussed in Section \ref{sec:cc}), can be seen to yield a common generalization of constructions that have played important roles both in arithmetic combinatorics and in ergodic theory, and that are centered on the notion of \emph{dual functions}. To see an example from arithmetic combinatorics, consider again the coupling $\mu\in\coup(\ab,\{0,1\}^2)$ from Example \ref{ex:U2coup}, with Haar measure $\lambda$, and let $T=\{00\}$. Then the $T$-localization of $\mu$ assigns to each $x\in \ab$ a coupling $\mu_x\in \coup(\ab, \{10,01,11\})$, which is determined by the constants $\xi(\mu_x,(f_{10},f_{01},f_{11}))=\int_{\ab^2} f_{10}(x+z_1) f_{01}(x+z_2) \overline{f_{11}(x+z_1+z_2)}\ud\lambda^2(z_1,z_2)$, for bounded Borel functions $f_{10},f_{01},f_{11}$ on $\ab$. The functions $x\mapsto \xi(\mu_x,(f_{10},f_{01},f_{11}))$ are the $U^2$ \emph{dual-functions} on $\ab$ as defined in \cite[(6.3)]{GTprimes}, and the same construction for the cubes $\{0,1\}^n$ with $n>2$ yields higher-order $U^n$ dual functions. Dual functions can also be defined on nilmanifolds (see \cite[Chapter 12, \S 3.2]{HKbook}), and in this setting again they can be seen as special cases of the above construction, when this is applied to the Haar measures on cubes on the nilmanifold. Another example, from ergodic theory, is given by the dual functions defined in \cite[(35) and $(\mc{B}_k)$]{HK}, which can also be seen as special cases of the above construction, when it is applied to the measures $\mu^{[n]}$ from \cite{HK}.
\end{remark}

\noindent When using disintegrations of couplings, it can be very useful to know that some given property of the coupling is inherited by almost every fibre measure in the disintegration. The next lemma ensures this for the property of conditional independence of index sets. 
\begin{lemma}\label{lem:fibrebotindep}
Let $\mu\in \coup(\varOmega,S)$, let $T_1,T_2$ be subsets of $S$ with $T_1\cap T_2=\{w\}$, let $v\in T_1\cup T_2$, and suppose that $T_1~\bot_{\mu}~T_2$. Let $(\mu_x)_{x\in \Omega}$ be a disintegration of $\mu$ relative to $p_v:\Omega^S\to \Omega$ and $\lambda$. Then for $\lambda$-almost every $x$ we have $(T_1\setminus\{v\})~\bot_{\mu_x}~(T_2\setminus\{v\})$.
\end{lemma}
\begin{proof}
We first prove the case in which $v\neq w$. In this case we can assume without loss of generality that $v\in T_1\setminus T_2$. Let $S'=S\setminus\{v\}$ and $T_1'=T_1\setminus\{v\}$. We will prove that for $\lambda$-almost every $x\in \Omega$, for every bounded $\mc{A}^{S'}_{T_1'}$-measurable function $g':\Omega^{S'}\to\mb{C}$ and bounded $\mc{A}^{S'}_{T_2}$-measurable function $h':\Omega^{S'}\to\mb{C}$, we have \vspace{-0.2cm}
\begin{equation}\label{eq:fibrebotindep}
\int_{\Omega^{S'}} g'\, h' \ud\mu_x = \int_{\Omega^{S'}} \mb{E}_{\mu_x}(g'|\mc{A}^{S'}_{w}) \, h' \ud\mu_x.\vspace{-0.2cm}
\end{equation}
Thus we will have $\mb{E}_{\mu_x}(g'|\mc{A}^{S'}_{w})=\mb{E}_{\mu_x}(g'|\mc{A}^{S'}_{T_2})$, implying by Lemma \ref{lem:botsuff} that $T_1'~\bot_{\mu_x}T_2$.

Let $f,g,h_0$ be bounded functions $\Omega^S\to \mb{C}$ and suppose that $f$ is $\mc{A}^S_v$-measurable, that $g$ is $\mc{A}^S_{T_1'}$-measurable, and that  $h_0$ is $\mc{A}^S_{T_2}$-measurable with $\mb{E}_\mu(h_0|\mc{A}^S_w)=0$. Since $f,g$ are both $\mc{A}^S_{T_1}$-measurable, we have $\int f\,g\,h_0\ud\mu = \int f\, g\,\mb{E}_\mu(h_0|\mc{A}^S_{T_1}) \ud\mu$. Since $T_1~\bot_{\mu}~T_2$, we have by \eqref{eq:condindep} that $\mb{E}_\mu(h_0|\mc{A}^S_{T_1})= \mb{E}_\mu(h_0|\mc{A}^S_{w})$. Hence $\int_{\Omega^S} f\,g\,h_0\ud\mu = 0$. By the disintegration we have $\int_{\Omega^S} f\,g\,h_0\ud\mu=\int_\Omega f'(x)\big(\int_{\Omega^{S'}} g' h_0' \ud\mu_x\big) \ud\lambda$, where $f',g', h_0'$ are the functions given by Lemma \ref{lem:Doob} such that $f=f'\co p_v$, $g=g'\co p_{S'}$ and $h_0= h_0'\co p_{S'}$. Let $t$ denote the function $x\mapsto \int_{\Omega^{S'}} g'\,h_0' \ud\mu_x$. We have thus shown that for every  function $f\in L^\infty(\mc{A}^S_v)$ we have $\int_\Omega f'(x) t(x) \ud\lambda=0$. In particular, choosing $f=\overline{t}\co p_v$, we deduce that $\int_\Omega |t(x)|^2 \ud\lambda=0$, so $t$ vanishes $\lambda$-almost surely. Applying this fact to each term of a  sequence of bounded functions $(g_i)_{i\in\mb{N}}$ that is dense in $L^2(\mc{A}_{T_1'}^S)$, we deduce that for $\lambda$-almost-every $x\in \Omega$, we have $\int_{\Omega^{S'}} g'\, h_0' \ud\mu_x =0$ for every function $g\in L^\infty(\mc{A}_{T_1'}^S)$. Now we let $(h_j)_{j\in\mb{N}}$ be a sequence of bounded functions dense in the closed subspace of $L^2(\mc{A}^S_{T_2})$ consisting of functions $h_0$ with $\mb{E}_\mu(h_0|\mc{A}^S_w)=0$, and we apply the last sentence to each term $h_j$. We thus deduce that for some set $E\subset\Omega$ with $\lambda(E)=0$, for every $x\in \Omega\setminus E$, we have $\int_{\Omega^{S'}} g'\, h_0' \ud\mu_x =0$ for every bounded $\mc{A}_{T_1'}^S$-measurable function $g:\Omega^S\to\mb{C}$ and every bounded $\mc{A}^S_{T_2}$-measurable function $h_0:\Omega^S\to\mb{C}$ with $\mb{E}_\mu(h_0|\mc{A}^S_w)=0$.

Now fix any $x\in \Omega\setminus E$, any bounded $\mc{A}^{S'}_{T_1'}$-measurable $g':\Omega^{S'}\to\mb{C}$ and any bounded $\mc{A}^{S'}_{T_2}$-measurable $h':\Omega^{S'}\to\mb{C}$. Let $r$ be a version of $\mb{E}_\mu(h'\co p_{S'}|\mc{A}_w^S)$. By Lemma \ref{lem:Doob} there is a function $r'\in L^\infty(\mc{A}_w^{S'})$ such that $r=r'\co p_{S'}$.  Applying the last sentence from the previous paragraph to the functions $g=g'\co p_{S'}$ and $h_0 = (h'- r')\co p_{S'} $, we obtain \vspace{-0.2cm}
\begin{equation}\label{eq:fibreindep}
\int_{\Omega^{S'}} g'\; h' \ud\mu_x = \int_{\Omega^{S'}} g'\,r' \ud\mu_x.\vspace{-0.2cm}
\end{equation}
Since $r'$ is already $\mc{A}^{S'}_w$-measurable, the last integral above equals $\int_{\Omega^{S'}} \mb{E}_{\mu_x}(g'|\mc{A}^{S'}_{w}) \, \, r' \ud\mu_x$. By \eqref{eq:fibreindep} applied (in the opposite direction) with $\mb{E}_{\mu_x}(g'|\mc{A}^{S'}_w)$ instead of $g'$, we obtain that the last integral equals $\int_{\Omega^{S'}} \mb{E}_{\mu_x}(g'|\mc{A}^{S'}_w)\; h' \ud\mu_x$. This proves \eqref{eq:fibrebotindep}.

Now we prove the case $v=w$. Let $f\in L^\infty(\mc{A}_v^S)$, $g\in L^\infty(\mc{A}_{T_1\setminus\{v\}}^S)$, $h\in L^\infty(\mc{A}_{T_2\setminus\{v\}}^S)$. We have $\int_{\Omega^S} f\,g\,h \ud\mu = \int_{\Omega^S} f\,\mb{E}(g|\mc{A}_{T_2}^S) h \ud\mu$. This last integral equals $\int_{\Omega^S} f\,\mb{E}(g|\mc{A}_v^S)h \ud\mu = \int_{\Omega^S} f\,\mb{E}(g|\mc{A}_v^S)\,\mb{E}(h|\mc{A}_v^S) \ud\mu$, since  $T_1\bot_\mu T_2$. By \cite[Proposition 10.4.18]{Boga2}, for $\lambda$-almost every $x$ we have $\mb{E}(g|\mc{A}_v^S)(x)=\int_{\Omega^{S'}}g\ud\mu_x$ (and similarly for $h$), so $
\int_{\Omega} f(x) \Big(\int_{\Omega^{S'}} (g\,h) \ud\mu_x\Big)\ud\lambda = \int_{\Omega^S} f\,g\,h \ud\mu = \int_{\Omega^S} f\,\mb{E}(g|\mc{A}_v^S)\,\mb{E}(h|\mc{A}_v^S) \ud\mu = \int_{\Omega} f(x) \Big(\int_{\Omega^{S'}} g \ud\mu_x  \; \int_{\Omega^{S'}} h \ud\mu_x\Big)\ud\lambda$.
Since this holds for any such $f$, we deduce that $\int_{\Omega^{S'}} (g\,h) \ud\mu_x=\int_{\Omega^{S'}} g \ud\mu_x  \; \int_{\Omega^{S'}} h \ud\mu_x$ for almost every $x$, i.e.\ the desired independence. By an argument similar to the previous case, using $L^2$-dense sequences of functions $g,h$, we conclude that for almost every $x$ the last equality holds for all $g,h$. Hence the result holds in this case as well.
\end{proof}
The following lemma is almost trivial but we shall need it for the next result.
\begin{lemma}\label{lem:localize2} Let $\mu\in \coup(\varOmega,S)$, let $T\subset S$ be local, and let $x\mapsto \mu_x$ be \textup{(}a version of\textup{)} $f_{\mu,T}$. Let $A\in \mc{A}^S_{S\setminus T}$, let $f:=\mb{E}(1_A\co p_{S\setminus T}~|~\mc{A}^S_T)$, and let $f'$ be $\mc{A}^T$-measurable such that $f=_\mu f'\co p_T$. Then $f'(x)= \mu_x(A)$ for $\mu_T$-almost every $x\in\Omega^T$.
\end{lemma}
\begin{proof} Let $g$ denote the function $x\mapsto \mu_x(A)$ on $\varOmega^T$. By the essential uniqueness of conditional expectation, it suffices to prove that for every set $B\in \mc{A}^{T}$ we have $\int_B g(x) \ud \mu_T(x) = \mu(A\cap p_T^{-1}(B))$. But this holds by definition of the disintegration $(\mu_x)_{x\in \Omega^T}$ of $\mu$.
\end{proof}

\begin{defn}[Conditional coupling]\label{def:condcoup} Let $\mu\in \coup(\varOmega,S)$ and let $T\subset S$ be local. Let $M\in\mc{A}^T$ satisfy $\mu_T(M)>0$. Let $\mu'$ be the probability measure defined on $\mathcal{A}^{S\setminus T}$ by $\mu'(N):=\mu(M\times N)/\mu_T(M)$. Then, letting $f$ be the $T$-localization of $\mu$, by Lemma \ref{lem:localize2} we have $\mu'(N)=\mu_T(M)^{-1}\int_{x\in M}f(x)(N)\ud\mu_T(x)$. In particular, by convexity of $\coup(\varOmega,S\setminus T)$ we have $\mu'\in\coup(\varOmega,S\setminus T)$. We call $\mu'$ the \emph{conditional coupling} of $\mu$ relative to $M$.
\end{defn}
\noindent The gist of the following result is similar to that of Lemma \ref{lem:fibrebotindep}, but here the property that is inherited by the fibre measures is the locality of some index set.

\begin{lemma}\label{lem:embloc}
Let $\mu\in \coup(\varOmega,S)$, let $T\subset R\subset S$, and suppose that $R,T$ are local in $\mu$. Let $\nu\in \coup(\varOmega,S\setminus T)$ be a coupling in $\Supp(\mu_T\co f_{\mu,T}^{-1})$. Then $R\setminus T$ is local in $\nu$.
\end{lemma}

\begin{proof} 
We prove that for every $w\in S\setminus R$, any events $G\in\mc{A}^{S\setminus T}_{R\setminus T}$, $H\in\mc{A}^{S\setminus T}_w$ satisfy the equation $\nu(G\cap H)=\nu(G)\,\nu(H)$. First note that it suffices to prove this assuming that $G$ is a measurable product-set, i.e.\ of the form $(\prod_{v\in R\setminus T}G_v)\times \Omega^{S\setminus R}$ where $G_v\in \mc{A}$ for each $v$. Indeed, this clearly implies that the equation holds also for $G$ being any pairwise disjoint union of finitely many such product sets, and then this in turn implies the equation in full generality, by approximating any $G\in \mc{A}^{S\setminus T}_{R\setminus T}$ by such a disjoint union (see Lemma \ref{lem:pisysapproxApp} for more details on such approximations). So we may assume that $G=(\prod_{v\in R\setminus T}G_v)\times \Omega^{S\setminus R}$. Let $F=(f_v)_{v\in S\setminus T}$ with $f_v=1_{G_v}$ for $v\in R\setminus T$, $f_w=1_{H'}$ for $H'\in \mc{A}$ such that $H=p_w^{-1}(H')$, and $f_v=1$ otherwise. Let $f_{\mu,T}:x\mapsto\mu_x$ be the $T$-localization of $\mu$. Let $d$ be a metric generating the topology on $\coup(\varOmega,S\setminus T)$. For every $\epsilon>0$, let $A_\epsilon=\{ x\in\Omega^T: \,d(\mu_x,\nu)\leq \epsilon\}=f_{\mu,T}^{-1}(B_\epsilon(\nu))$, where $B_\epsilon(\nu)$ is the ball of radius $\epsilon$ with center $\nu$. Since $\nu\in \Supp(\mu_T\co f_{\mu,T}^{-1})$, we have $\mu_T(A_\epsilon)>0$. 

By the disintegration of $\mu$ in Lemma \ref{lem:localize}, and the  continuity of $\xi(\cdot,F)$, we have $\nu(G\cap H)=\lim_{\epsilon\to 0}\; \mu(A_\epsilon\times\Omega^{S\setminus T})^{-1}\;\mu\Big((A_\epsilon\times\Omega^{S\setminus T})\cap (G\times\Omega^T)\cap (H\times \Omega^T)\Big)$. Similarly  $\nu(G)=\lim_{\epsilon\to 0}\; \mu(A_\epsilon\times\Omega^{S\setminus T})^{-1}\; \mu\Big((A_\epsilon\times\Omega^{S\setminus T})\cap (G\times\Omega^T)\Big)$. Since $A_\epsilon\times\Omega^{S\setminus T},\, G\times\Omega^T\in \mc{A}^S_R$, we have $
\mu\Big((A_\epsilon\times\Omega^{S\setminus T})\cap (G\times\Omega^T)\cap (H\times \Omega^T)\Big)=\mu\Big((A_\epsilon\times\Omega^{S\setminus T})\cap (G\times\Omega^T)\Big)\; \mu(H\times\Omega^T)$, by the locality of $R$ in $\mu$. Combining the last three equations, we deduce that $\nu(G\cap H)=\nu(G)\,\mu(H\times \Omega^T)$. Finally, we have $\mu(H\times \Omega^T)=\lambda(H')=\nu(H)$.
\end{proof}

\subsection{Conditional independence in set lattices}\label{subsec:condlattice} \hfill \vspace{-0.5cm}\\
\begin{defn}
Let $S$ be a set. A \emph{set lattice} in $S$ is a family of subsets of $S$ closed under intersection and union. If $\mc{F}\subset 2^S$ is a family closed under intersection, then the family $\varLambda$ of all unions of sets in $\mc{F}$ is a set lattice, and we say that $\mc{F}$ \emph{generates} $\varLambda$.
\end{defn}
\begin{defn}\label{def:CIS}
Let $S$ be a finite set, let $\mu \in \coup(\varOmega,S)$, and let $\varLambda$ be a set lattice in $S$. A set $T\in\varLambda$ has the \emph{conditional independence of subsets} (\textsc{cis}) property  in $\varLambda$ if for every $T_1,T_2\in\varLambda$ with $T_1,T_2\subset T$ we have $T_1~\bot_\mu~T_2$.
\end{defn}
\noindent If the ambient coupling $\mu$ needs to be specified, we say that $T$ has the \textsc{cis} property in $\varLambda$ and $\mu$. The main result of this subsection is the following fact concerning the \textsc{cis} property.

\begin{proposition}\label{prop:latticeind} If $T_1,T_2$ in $\varLambda$ both have the \textsc{cis} property and $T_1~\bot ~T_2$, then $T_1\cup T_2$ has the \textsc{cis} property in $\Lambda$. 
\end{proposition}

\noindent To prove this we shall use the following result about the relation $\bot$, which will also be useful in later sections.

\begin{lemma}\label{lem:3set1}
Let $\mu\in\coup(\varOmega,S)$, and let $A,B,C$ be subsets of $S$ satisfying the following conditions: $\quad A~\bot~B, \quad (A\cap B)~\bot~C,  \quad (A\cup B)~\bot~C,  \quad A\supseteq (B\cap C)$. Then we have $B~\bot~C$, \quad $A~\bot~(B\cup C)$,\quad and \quad $(A\cup C)~\bot~B$.
\end{lemma}
\begin{proof}
We first prove that $B~\bot~C$. Let $f$ be bounded $\mc{A}^S_C$-measurable, and let $f'=\mb{E}(f|\mc{A}^S_{A\cup B})$. Since $(A\cup B)~\bot~C$, we have by Lemma \ref{lem:botsuff} that $f'=\mb{E}(f|\mc{A}^S_{(A\cup B)\cap C})$. This in turn equals $\mb{E}(f|\mc{A}^S_{A\cap C})$, since $(A\cup B)\cap C=A\cap C$ (using that $ A \supseteq B\cap C$).
Thus
\begin{equation}\label{eq:3set1}
f'=\mb{E}(f|\mc{A}^S_{A\cup B})=\mb{E}(f|\mc{A}^S_{A\cap C}).
\end{equation} 
Since $\mc{A}^S_{A\cup B} \supset \mc{A}^S_B$, we have $\mb{E}(f|\mc{A}^S_B)=\mb{E}(f'|\mc{A}^S_B)$. We also have $\mb{E}(f'|\mc{A}^S_B)=\mb{E}(f'|\mc{A}^S_{A\cap B})$, since $f'$ is $\mc{A}^S_A$-measurable and $A~\bot~B$. Now using that $C~\bot~(A\cap B)$ and the fact that $f'$ is $\mc{A}^S_C$-measurable (by \eqref{eq:3set1}), we obtain by \eqref{eq:condindep} that  $\mb{E}(f'|\mc{A}^S_{A\cap B})=\mb{E}(f'|\mc{A}^S_{A\cap B\cap C})$. The last three equalities imply that $\mb{E}(f|\mc{A}^S_B)$ is $\mc{A}^S_{A\cap B\cap C}$-measurable. Since $\mc{A}^S_{A\cap B\cap C} \subset \mc{A}^S_{B\cap C}$, we deduce that $\mb{E}(f|\mc{A}^S_B)$ is $\mc{A}^S_{B\cap C}$-measurable. This proves that $B~\bot~C$ (by Lemma \ref{lem:botsuff}).

To show that $A~\bot~(B\cup C)$, we use the fact that every function in  $L^2(\mc{A}^S_{B\cup C})$ is a limit in $L^2$ of finite sums of functions of the form $f g$ where $f\in L^\infty(\mc{A}^S_C)$ and $g\in L^\infty(\mc{A}^S_B)$ (this can be seen using that by Lemma \ref{lem:pisysapprox} any function in $L^2(\mc{A}^S_{B\cup C})$ is an $L^2$-limit of finite sums of rank-1 functions $\prod_{v\in B\cup C} f_v\co p_v$ with each $f_v$ being $\mc{A}$-measurable; then each of these can be written as $fg$ with $f=\prod_{v\in B} f_v\co p_v$ and $g=\prod_{v\in C\setminus B} f_v\co p_v$). For every such $f,g$, let $f'$ be defined as in \eqref{eq:3set1}. Then $\mb{E}(f\,g\,|\,\mc{A}^S_A)=\mb{E}(f'\,g\,|\,\mc{A}^S_A) = f'\,\mb{E}(g\,|\,\mc{A}^S_A)=f'\,\mb{E}(g\,|\,\mc{A}^S_{A\cap B})$, where the second equality uses that $f'$ is measurable relative to $\mc{A}^S_{A\cap C} \subset \mc{A}^S_A$ (by \eqref{eq:3set1}), and the third equality uses that $A~\bot~B$. Hence $\mb{E}(f\,g\,|\,\mc{A}^S_A)$ is a product of a function in $L^\infty(\mc{A}^S_{A\cap B})$ (i.e.\ $\mb{E}(g\,|\,\mc{A}^S_{A\cap B})$) with a function in $L^\infty(\mc{A}^S_{A\cap C})$ (i.e.\ $f'$), so $\mb{E}(f\,g\,|\,\mc{A}^S_A)$ is $\mc{A}^S_{A\cap(B\cup C)}$-measurable. Since this holds for every such function $fg$, it holds more generally for every function in $L^2(\mc{A}^S_{B\cup C})$. By Lemma \ref{lem:botsuff} we therefore have indeed $A~\bot~(B\cup C)$.

Finally we claim that $(C\cup A)~\bot~B$. To prove this, as in the previous paragraph, it suffices to show that if $f\in L^\infty(\mc{A}_A^S)$ and $g\in L^\infty(\mc{A}_C^S)$ then $\mb{E}(fg|\mc{A}_B^S)$ is in $L^\infty(\mc{A}_{(C\cup A)\cap B}^S)$. To see this, note first that $\mb{E}(fg|\mc{A}_B^S)=\mb{E}(\mb{E}(fg|\mc{A}_{A\cup B}^S)|\mc{A}_B^S)=\mb{E}( f\mb{E}(g|\mc{A}_{A\cup B}^S)|\mc{A}_B^S)$. Using that $C\bot (A\cup B)$ we have $\mb{E}(g|\mc{A}_{A\cup B}^S)=\mb{E}(g|\mc{A}^S_{(A\cup B)\cap C})$, and since $(A\cup B)\cap C=A\cap C$, we have $\mb{E}(g|\mc{A}_{A\cup B}^S)$ is $\mc{A}_A^S$-measurable. It follows that $\mb{E}(fg|\mc{A}_{A\cup B}^S)$ is $\mc{A}_A^S$-measurable, and since $A~\bot~B$, we have that $\mb{E}(fg|\mc{A}_B^S)=\mb{E}(\mb{E}(fg|\mc{A}_{A\cup B}^S)|\mc{A}_B^S)$ is $\mc{A}_{A\cap B}^S$-measurable. Since $A\cap B = (A\cup C)\cap B$, the claim follows.
\end{proof}

\begin{proof}[Proof of Proposition \ref{prop:latticeind}] We argue by induction on $|T_1\triangle T_2|$. If $|T_1\triangle T_2|=0$ then there is nothing to prove. We can assume that $|T_1\setminus T_2|\geq 1$. Let $F\in\varLambda$ be minimal with the property that $F\subseteq T_1$ and $|F\setminus T_2|\geq 1$. We then have the following fact:
\begin{equation}\label{eq:CIS-obs1}
\forall\, G \in \varLambda\textrm{ with }G\subseteq T_2\cup F\textrm{ and }|G\cap (F\setminus T_2)|\geq 1, \textrm{ we have } F\subseteq G.
\end{equation}
Indeed, otherwise $G\cap F\in\varLambda$ would contradict the minimality of $F$.

By the \textsc{cis} property of $T_1$ we have $(T_1\cap T_2)~\bot~ F$. Since $F\subset T_1$, we also have $(T_1\cup T_2)~\bot~ F$ (see the sentence after Lemma \ref{lem:botsuff}). This together with our assumption that $T_1~\bot~T_2$ implies, by Lemma \ref{lem:3set1} (applied with $A=T_1$, $B=T_2$, $C=F$), that
\begin{equation}\label{eq:CIS-obs2}
T_2~\bot~F\textrm{ and }T_1~\bot~ (T_2\cup F).
\end{equation}
The following observation will also be useful:
\begin{equation}\label{eq:CIS-obs3}
\forall\, Q\in \varLambda\textrm{ with } Q\subseteq T_2,\textrm{  we have }F~\bot~Q.
\end{equation}
To see this, note that if $Q\subseteq T_2$ is in $\varLambda$ then, since $T_2$ has the \textsc{cis} property, we have $(T_2\cap F)~\bot~Q$, and $F~\bot~Q$ then follows by Lemma \ref{lem:3set1} with $A=T_2$, $B=F$, $C=Q$.

Now we prove the \textsc{cis} property for $F\cup T_2$. Let $U,V\in\varLambda$ be subsets of $F \cup T_2$. If both $U$ and $V$ are contained in $T_2$ then the \textsc{cis} property of $T_2$ implies that $U~\bot~V$. If $|U\cap (F\setminus T_2)|\geq 1$ and $|V\cap (F\setminus T_2)|\geq 1$, then by \eqref{eq:CIS-obs1} we have $F\subseteq U$, $F\subseteq V$, and so we have $U=F\cup U'$, $V=F\cup V'$, for the sets $U':=U\cap T_2\supset F\cap T_2$, $V':=V\cap T_2\supset F\cap T_2$. Note that $U',V'$ are both in $\varLambda$ and contained in $T_2$. By the \textsc{cis} property of $T_2$ we have $U'~\bot~V'$. By \eqref{eq:CIS-obs3} applied with $Q=U'\cap V'$ and $Q=U'\cup V'$, we have $(U'\cup V')~\bot~F$ and $(U'\cap V')~\bot~F$. Now we apply Lemma \ref{lem:3set1} with $A=U'$, $B=V'$, $C=F$ (noting that $U'\supset V'\cap F$ since $V'\cap F=T_2\cap F$), obtaining $U'~\bot~V$, and we apply it similarly with $A=V'$, $B=U'$, $C=F$, obtaining $U'~\bot~F$. Thus, now we have $V~\bot~ F$ (because $V\supset F$), we have $(V\cap F)~\bot~U'$ (because $V\cap F=F$ and $F~\bot~U'$), we have $(V\cup F)~\bot~U'$, and $V\supseteq  (F\cap U')$. Hence, by one more application of Lemma \ref{lem:3set1}, with $A=V$, $B=F$, $C=U'$, we obtain $V~\bot~U$, as required. The final case is when (without loss of generality) $|U\cap (F\setminus T_2)|\geq 1$ (which as above implies $F\subseteq U$) and $V\subseteq T_2$. Then we can again write $U=U'\cup F$ for $U':=U\cap T_2\in\varLambda$. We then have $V~\bot~U'$ by the \textsc{cis} property of $T_2$, by \eqref{eq:CIS-obs3} we have $F~\bot~(U'\cup V)$ and $F~\bot~(U'\cap V)$, and we also have $U'\supset V\cap F$, whence by Lemma \ref{lem:3set1} applied with $A=U'$, $B=V$, $C=F$, we obtain again $V~\bot~U$.

We have obtained that $(F\cup T_2)~\bot~T_1$ (by \eqref{eq:CIS-obs2}) and that both $F\cup T_2$ and $T_1$ have the \textsc{cis} property. Since $|(F\cup T_2)\triangle T_1|<|T_2\triangle T_1|$, we have by induction that $(F\cup T_2)\cup T_1$ has the \textsc{cis} property, and since $(F\cup T_2)\cup T_1=T_2\cup T_1$, the proof is complete.
\end{proof}

\subsection{Idempotent couplings} \hfill \medskip\\
In this section we introduce and study the following special class of couplings.

\begin{defn}[Idempotent coupling]\label{def:idemp}
We say that a coupling $\mu\in \coup(\varOmega,\{a,b\})$ is \emph{idempotent} if the following holds. Let $\mu'\in\coup(\varOmega,\{a',b'\})$ be such that $\sigma:\{a,b\} \to \{a',b'\}$, $a \mapsto a'$, $b \mapsto b'$ is an isomorphism of $\mu$ and $\mu'$, and let $\nu\in\coup(\varOmega,\{a,a',b\})$ be the conditionally independent coupling of $\mu$ and $\mu'$ obtained by identifying $b$ and $b'$; then the bijection $\sigma':\{a,b\} \to \{a,a'\}$, $b \mapsto a'$ is an isomorphism of $\mu$ and $\nu_{\{a,a'\}}$.
\end{defn}
\noindent This notion leads in a natural way to the following more general notion of idempotence that will be crucial in the next section.
\begin{defn}[Idempotence along an isomorphism]
Let $\mu \in \coup(\varOmega,S)$, let $a,b$ be subsets of $S$ forming a  partition $S=a \sqcup b$, and let $\beta:a\to b$ be a bijection. We then say that $\mu$ is \emph{idempotent along} $\beta$ if we have $\mu_a\cong_\beta \mu_b$ (as per Definition \ref{def:coupiso}) and, letting $\varOmega'=(\Omega^a, \mc{A}^a, \mu_a)$, we have that $\mu$ is idempotent as a coupling in $\coup(\varOmega',\{a,b\})$.
\end{defn}
\begin{example}\label{ex:U2idemp}
Consider the coupling from Example \ref{ex:U2coup}, thus $\varOmega$ consists of a compact abelian group $\ab$ with Haar measure $\lambda$, and $\mu\in \coup(\varOmega,\db{2})$ is given by the Haar measure on $\{(x_{00},x_{10},x_{01},x_{11}): x_{00}-x_{10}=x_{01}-x_{11}\}\leq \ab^{\db{2}}$. Let $a$, $b$ be the faces $\{00,10\}$, $\{01,11\}$ respectively, and let $\beta: a\to b$ be the bijection that switches the second component from 0 to 1. Then $\mu$ is idempotent along $\beta$. Indeed, the coupling $\nu$ from Definition \ref{def:idemp} here is the Haar measure on $\{(x_{00},x_{10},x_{00}',x_{10}',x_{01},x_{11}): x_{00}-x_{10}=x'_{00}-x'_{10}=x_{01}-x_{11}\}\leq \ab^6$, and projection to the first 4 components here yields a coupling isomorphic to $\mu$.
\end{example}
\begin{remark}\label{rem:idembot}
Note that the coupling $\nu \in\coup(\varOmega',\{a,a',b\})$ in Definition \ref{def:idemp} is also a coupling in $\coup(\varOmega, a\sqcup a' \sqcup b)$. The construction of $\nu$ as a conditionally independent coupling then implies that $(a\sqcup b)\bot_\nu (a'\sqcup b)$, by Definition \ref{def:condindcoup}.
\end{remark}

\noindent The main result of this subsection is a characterization of idempotent couplings, stating that every such coupling is a product of the original probability space with itself \emph{relative} to some factor of the space, in the sense of the notion of \emph{relative product} of measure spaces from \cite[Definition 5.7]{Furst}. We only use the special case of relative products where the two measure spaces are the same, which we recall as follows.

\begin{defn}[Square of a probability measure relative to a factor]\label{def:relsquare}
Let $\varOmega=(\Omega,\mc{A},\lambda)$ be a probability space and let $\mc{B}$ be a sub-$\sigma$-algebra of $\mc{A}$. The \emph{square of $\lambda$ relative to} $\mc{B}$ is the coupling $\mu\in \coup(\varOmega, \{a,b\})$ defined by the following property:
\begin{equation}\label{eq:relsquare1}
\forall\,A,B\in \mc{A},\;\;\mu(A\times B)=\int_{\Omega}\mb{E}(1_A|\mc{B})\,\mb{E}(1_B|\mc{B}) \ud\lambda.
\end{equation}
\end{defn}
\noindent Note that formula \eqref{eq:relsquare1} indeed defines uniquely the coupling $\mu$,  by Lemma \ref{lem:coupcorresp}, since the formula determines the multilinear map $F\mapsto \xi(\mu,F)$ in that lemma.

For every coupling $\mu\in\coup(\varOmega,\{a,b\})$, we define a Hermitian form $\langle \cdot,\cdot\rangle_{\mu}$ on $L^\infty(\varOmega)$ by
\[
\langle f,g\rangle_{\mu} :=\int_{\Omega^2} f\co p_a\; \overline{g\co p_b} \,\ud\mu = \int_{\Omega^2} f(x_a)\; \overline{g(x_b)} \,\ud\mu(x_a,x_b).
\]
Property \eqref{eq:relsquare1} is equivalent (via $L^2$-approximations by simple functions) to the following:
\begin{equation}\label{eq:relsquare2}
\forall\,f,g\in L^\infty(\mc{A}),\;\;\langle f,g\rangle_{\mu}=\int_{\Omega}\mb{E}(f|\mc{B})\,\overline{\mb{E}(g|\mc{B})} \ud\lambda.
\end{equation}
The notion of relative square should be carefully distinguished from Definition \ref{def:condindcoup}. In the latter definition, the index sets of the two couplings may have some parts glued together, whereas if we take the relative square of a coupling $\mu\in\coup(\varOmega,S)$ then the result can be viewed as a coupling of $\varOmega$ with index set being the \emph{disjoint} union of two copies of $S$, and here we are focusing on a sub-$\sigma$-algebra $\mc{B}$ of $\mc{A}$.

\begin{lemma}\label{lem:IdempCondProp} Let $\mu$ be as in Definition \ref{def:relsquare}. Then we have the following properties.
\begin{enumerate} 
\item If $A\in\mc{B}$ then $\mu(p_a^{-1}(A)\Delta p_b^{-1}(A))=0$. In particular $p_a^{-1}(\mc{B})=_\mu p_b^{-1}(\mc{B})$.
\item $\mc{A}^{\{a,b\}}_a$ and $\mc{A}^{\{a,b\}}_b$ are conditionally independent relative to $\mc{G}:=p_a^{-1}(\mc{B})=_\mu p_b^{-1}(\mc{B})$.
\item $\mc{G}=\mc{A}^{\{a,b\}}_a\wedge \mc{A}^{\{a,b\}}_b$. In particular we have $\mc{A}^{\{a,b\}}_a\upmod_\mu \mc{A}^{\{a,b\}}_b$.
\item $\mu$ is an idempotent coupling.
\end{enumerate}
\end{lemma}
\begin{proof}
For $(i)$, note that $p_a^{-1}(A)=A\times\Omega$ and $p_b^{-1}(A)=\Omega\times A$, so we have to show that $\mu\big( (A\times\Omega)\setminus(\Omega\times A)\big)=0 = \mu\big((\Omega\times A)\setminus(A\times \Omega)\big)$. We prove the first equality (the second follows similarly). Note that $(A\times\Omega)\setminus(\Omega\times A)=A\times\overline{A}$. Then $\mu(A\times\overline{A})=\int_{\Omega}\mb{E}(1_A|\mc{B})\,\mb{E}(1_{\overline{A}}|\mc{B}) \ud\lambda=\int_{\Omega}1_A1_{\overline{A}}\ud\lambda=0$. For $(ii)$, note that by \eqref{eq:exprel} we have $\int_{\Omega^2} \mb{E}(f\co p_a|\mathcal{G})\; \overline{g\co p_b} \,\ud\mu =\int_{\Omega^2}\mb{E}(f|\mathcal{B})\co p_a\; \overline{g\co p_b} \,\ud\mu$. By \eqref{eq:relsquare2} the latter integral is $
\int_{\Omega}\mb{E}(f|\mc{B})\,\overline{\mb{E}(g|\mc{B})} \ud\lambda=\int_{\Omega^2} f\co p_a\; \overline{g\co p_b} \,\ud\mu$. Since $\mc{G}\subset_\mu\mc{A}^{\{a,b\}}_a\wedge \mc{A}^{\{a,b\}}_b$, these equalities imply that $\mb{E}(f\circ p_a|\mc{A}^{\{a,b\}}_b)=\mb{E}(f\co p_a|\mc{G})$, which implies conditional independence over $\mathcal{G}$. Now $(iii)$ follows from $(ii)$ and the fact that $\mc{G}\subset_\mu \mc{A}^{\{a,b\}}_a\wedge \mc{A}^{\{a,b\}}_b$. Finally, to see $(iv)$ consider the coupling $\nu \in\coup(\varOmega,\{a,a',b\})$ from Definition \ref{def:idemp}, and note that for every $A,A'\in \mc{A}$, we have $\nu(A\times \Omega\times A')=\int_\Omega \mb{E}(1_A\co p_a|\mc{A}^{\{a,b\}}_b)\mb{E}(1_{A'}\co p_{a'}|\mc{A}^{\{a',b\}}_{b})\ud\lambda$, by \eqref{eq:cid} applied with $\mu''=\nu$ and $B_2=\Omega$. This is $\int_\Omega \mb{E}(1_A|\mc{B})\mb{E}(1_{A'}|\mc{B})\ud\lambda=\mu(A\times A')$, by part (i) of Lemma \ref{lem:IdempCondProp}. It follows that the subcoupling of $\nu$ along $\{a,a'\}$ is $\mu$.
\end{proof}
\noindent We thus know that a square relative to a factor is always an idempotent coupling. We will see in this subsection that in fact every idempotent coupling is of this form. This enables us, in particular, to define a generalization of uniformity seminorms in the sequel.

To achieve these goals we begin by showing that for an idempotent coupling $\mu$ the form $\langle\cdot,\cdot\rangle_\mu$ has the following useful reduced expression.

\begin{lemma}\label{lem:idemp0}
Let $\mu\in\coup(\varOmega,\{a,b\})$ be an idempotent coupling. Then for every $f,g\in L^\infty(\varOmega)$ we have
\begin{equation}\label{eq:IdempotInnerprod}
\langle f,g\rangle_{\mu} = \int_{\Omega^2} \mb{E}(f\co p_a| \mc{A}^{\{a,b\}}_b) \; \overline{\mb{E}(g\co p_{a}|\mc{A}^{\{a,b\}}_b) }\; \ud\mu.
\end{equation}
\end{lemma}

\begin{proof}
Let $S=\{a,a',b\}$ and let $\nu$ be the coupling associated with $\mu$ as in Definition \ref{def:idemp}, with index set $S$. We then have the following sequence of equalities, explained below:
\begin{eqnarray*}
\langle f,g\rangle_{\mu}& = &\int_{\Omega^2} f\co p_a \; \overline{g\co p_{a'}}\; \ud\nu_{\{a,a'\}} = \int_{\Omega^3} f\co p_a \; \overline{g\co p_{a'} }\; \ud\nu\\
& = & \int_{\Omega^3} \mb{E}(f\co p_a|  \mc{A}_{\{a',b\}}^S) \; \overline{g\co p_{a'} }\; \ud\nu =
 \int_{\Omega^3} \mb{E}(f\co p_a| \mc{A}^S_b) \; \overline{g\co p_{a'} }\; \ud\nu \\
& = & \int_{\Omega^3} \mb{E}(f\co p_a| \mc{A}^S_b) \; \mb{E}(\overline{g\co p_{a'}}|\mc{A}^S_b) \; \ud\nu=\int_{\Omega^2} \mb{E}(f\co p_a| \mc{A}^{\{a,b\}}_b) \; \overline{\mb{E}(g\co p_a| \mc{A}^{\{a,b\}}_b) }\; \ud\mu.
\end{eqnarray*}
The first equality uses that $\mu$ is isomorphic to $\nu_{\{a,a'\}}$ (since $\mu$ is idempotent). The second equality uses that $\nu_{\{a,a'\}}$ is a sub-coupling of $\nu$. The third equality uses that $\overline{g\co p_{a'} }$ is $\mc{A}^S_{\{a',b\}}$-measurable. The fourth equality follows from $\{a,b\} \bot_\nu \{a',b\}$ and $f\co p_a$ being $\mc{A}^S_{\{a,b\}}$-measurable. The fifth equality is clear. The sixth equality uses firstly the fact that in the left side the term $\mb{E}(g\co p_{a'}|\mc{A}^S_b)$ can be replaced with $\mb{E}(g\co p_a| \mc{A}^S_b)$ (this follows from the definition of these conditional expectations, upon checking that for every bounded $\mc{A}_b^S$-measurable function $h$ we have $\int_{\Omega^3} h(g\co p_{a'})\ud\nu=\int_{\Omega^3} h(g\co p_a)\ud\nu$, this equality following from $\mu$, $\mu'$ being isomorphic subcouplings of $\nu$), and secondly uses that $\nu_{\{a,b\}}$ and $\mu$ are isomorphic.
\end{proof}
\begin{example}
Let us illustrate \eqref{eq:IdempotInnerprod} with the coupling from Example \ref{ex:U2idemp}, i.e.\ the Haar measure $\mu$  on $G=\{x=(x_{00},x_{10},x_{01},x_{11}): x_{00}-x_{10}= x_{01}-x_{11}\}\leq \ab^{\db{2}}$. Note that $\mu$ can be viewed as an idempotent coupling by viewing $\db{2}$ as $\{F_0,F_1\}$ with $F_0=\{00,10\}$, $F_1=\{01,11\}$, thus $\mu\in \coup(\varOmega,S)$ where $S=\{F_0,F_1\}$ and $\varOmega=(\ab^2,\lambda\times\lambda)$. Then, for all bounded Borel functions $f,g:\ab^2\to \mb{C}$, we have $\langle f,g \rangle_{\mu}=\int_G f(x_{00},x_{10})\overline{ g(x_{01},x_{11})} \ud\mu$. Using \eqref{eq:IdempotInnerprod} we obtain $\langle f,g \rangle_\mu=\int_{\Omega^2} \mb{E}(f\co p_{F_0}|\mc{A}_{F_1}^S)\;\overline{\mb{E}(g\co p_{F_0}|\mc{A}_{F_1}^S)} \ud\mu(x_{F_0},x_{F_1})$ as an alternative formula. Reasoning as in Example \ref{ex:U2coup}, one can check that $\mb{E}(f\co p_{F_0}|\mc{A}_{F_1}^S)$ is represented by the function $x\mapsto \int_{\ab} f\co p_{F_0} (x_{00}+h,x_{10}+h,x_{01},x_{11})\ud\lambda(h)$, that is the integral of $f$ over the set $\{(u,v)\in \ab^2:u-v=x_{01} - x_{11}\}$. 
\end{example}
\noindent Part of the usefulness of the alternative formula \eqref{eq:IdempotInnerprod} for $\langle f,g \rangle_\mu$ is that it reveals clearly that $\langle f,f \rangle_{\mu}$ is always non-negative, which is unclear in the first formula for $\langle f,g \rangle_\mu$ above. 
\begin{corollary}\label{cor:idempos}
Let $\mu\in\coup(\varOmega,\{a,b\})$ be idempotent. Then for every $f\in L^\infty(\varOmega)$ we have $\langle f,f\rangle_{\mu} \geq \big|\int_\Omega f \ud\lambda \big|^2$.
\end{corollary}

\begin{proof} By Lemma \ref{lem:idemp0} and the Cauchy-Schwarz inequality, we have 
\[
\langle f,f\rangle_{\mu}
=\int_{\Omega^2}  \big|\mb{E}\big(f\co p_a|\mc{A}^{\{a,b\}}_b\big)\big|^2 \,\ud\mu \geq \Big|\int_{\Omega^2} \mb{E}\big(f\co p_a|\mc{A}^{\{a,b\}}_b\big)\,  \ud\mu \Big|^2 = \Big|\int_\Omega f \ud \lambda\Big|^2. \qedhere
\]
\end{proof}
We can now prove the main result of this subsection.
\begin{proposition}\label{prop:idemp} Let $\varOmega=(\Omega,\mc{A},\lambda)$ be a probability space. A coupling $\mu\in\coup(\varOmega,\{a,b\})$ is idempotent if and only if there is a $\sigma$-algebra $\mc{B}\subseteq\mc{A}$ such that $\mu$ is the square of $\lambda$ relative to $\mc{B}$.
\end{proposition}

\begin{proof}
The backward implication follows from $(iv)$ in Lemma \ref{lem:IdempCondProp}. For the forward implication, let $S=\{a,a',b\}$ and let $\nu\in \coup(\varOmega,S)$ be the coupling in Definition \ref{def:idemp}.

We first prove that
\begin{equation}\label{eq:idempstep1}
\mc{A}_a^{\{a,b\}} \upmod_{\mu} \mc{A}_b^{\{a,b\}}.
\end{equation}
Let $\pi:\Omega^S\to  \Omega^S$ be the map that interchanges the $a$ and $a'$ components. Then $\pi$ preserves the measure $\nu$. Indeed, since $\nu_{\{a,b\}}\cong \nu_{\{a',b\}}$, it is checked from the definition of conditional expectation that $\mb{E}(g\co p_{a'}|\mc{A}^S_b)=\mb{E}(g\co p_a|\mc{A}^S_b)$ for every $g\in L^\infty(\mc{A})$. Using this together with $\{a,b\}\bot_\nu \{a',b\}$ (Cf. Remark \ref{rem:idembot}), for every product set in $\mc{A}^S$ we have $\nu(A\times A'\times B)$ $=\int \mb{E}(1_A\co p_{a}|\mc{A}^S_b)\mb{E}(1_{A'}\co p_{a'}|\mc{A}^S_b) 1_B\co p_{b}\ud\nu =\int \mb{E}(1_A\co p_{a'}|\mc{A}^S_b)\mb{E}(1_{A'}\co p_{a}|\mc{A}^S_b) 1_B\co p_b\ud\nu$ $= \nu(A'\times A\times B)$, so $\pi$ is $\nu$-preserving on product sets and hence on all $\mc{A}^S$ by standard results. Now, to prove \eqref{eq:idempstep1}, by Proposition \ref{prop:condindep} (and Lemma \ref{lem:Doob}) it suffices to show that for every function $f\in L^\infty(\mc{A})$ we have that $\mb{E}_\mu(f\co p_a| \mc{A}_b^{\{a,b\}})$ is $\mc{A}_a^{\{a,b\}}$-measurable. We prove this by showing that $\|\mb{E}(\mb{E}(f\co p_a| \mc{A}_b^{\{a,b\}})|\mc{A}_a^{\{a,b\}})\|_{L^2(\mu)}= \|\mb{E}(f\co p_a| \mc{A}_b^{\{a,b\}})\|_{L^2(\mu)}$. To this end, note first that $\|\mb{E}(f\co p_a| \mc{A}_b^{\{a,b\}})\|_{L^2(\mu)}^2=\int_{\Omega^S} |\mb{E}(f\co p_a| \mc{A}_{a'}^S)|^2\ud\nu$ (since $\mu\cong \nu_{\{a,a'\}}$ by idempotence). Since $\mc{A}_{a'}^S\subset_\nu\mc{A}_{a',b}^S$, this is $\int_{\Omega^S} |\mb{E}(\mb{E}(f\co p_a| \mc{A}_{a',b}^S)|\mc{A}_{a'}^S)|^2\ud\nu$, which equals $\int_{\Omega^S} |\mb{E}(\mb{E}(f\co p_a| \mc{A}_b^S)|\mc{A}_{a'}^S)|^2\ud\nu$ (since $\{a,b\}\bot_\nu \{a',b\}$). Since $\pi$ preserves $\nu$, this is $\int_{\Omega^S} |\mb{E}(\mb{E}(f\co p_{a'}| \mc{A}_b^S)|\mc{A}_a^S)|^2\ud\nu$. By the remark involving $g$ above, this in turn equals $\int_{\Omega^S} |\mb{E}(\mb{E}(f\co p_a| \mc{A}_b^S)|\mc{A}_a^S)|^2\ud\nu =\|\mb{E}(\mb{E}(f\co p_a| \mc{A}_b^{\{a,b\}})|\mc{A}_a^{\{a,b\}}\|_{L^2(\mu)}^2$. This proves \eqref{eq:idempstep1}.

Let $\mc{D}=\mc{A}_a^{\{a,b\}} \wedge_\mu \mc{A}_b^{\{a,b\}}$ and let $\mc{B}\subseteq \mc{A}$ be the image of $\mc{D}$ on $\varOmega$, i.e.\ the $\sigma$-algebra of sets $B\subset\Omega$ such that $p_a^{-1} B$ (equivalently $p_b^{-1} B$) is in $\mc{D}$. It then follows from the definitions that $p_a^{-1}\mc{B}=_\mu\mc{D}=_\mu p_b^{-1}\mc{B}$. This together with \eqref{eq:idempstep1}, \eqref{eq:condindep} and \eqref{eq:exprel}, implies that
\begin{equation}\label{eq:pfclaim}
\forall\,f\in L^\infty(\mc{A}), \quad \mb{E}(f\co p_a| \mc{A}_b^{\{a,b\}})=\mb{E}(f\co p_a|\mc{D})=\mb{E}(f|\mc{B})\co p_a=\mb{E}(f|\mc{B})\co p_b.
\end{equation}
We can now deduce that \eqref{eq:relsquare2} holds, by combining  \eqref{eq:IdempotInnerprod} with \eqref{eq:pfclaim}.
\end{proof}
We now deduce several consequences of Proposition \ref{prop:idemp}.
\begin{lemma}\label{lem:idemp3}
Let $\mu\in \coup(\varOmega,\{a,b\})$ be an idempotent coupling, and let $\nu$ be the associated coupling from Definition \ref{def:idemp}. Then $\nu$ is symmetric under every permutation of $\{a,a',b\}$.
\end{lemma}

\begin{proof}
By Proposition \ref{prop:idemp} there is a $\sigma$-algebra $\mc{B}\subseteq\mc{A}$ such that $\mu$ is the square of $\varOmega$ relative to $\mc{B}$. Letting $S=\{a,a',b\}$ as before, for arbitrary sets $A,A',B\in\mc{A}$ we have $\nu(A\times A'\times B) = \int_{\Omega^S} (1_A\co p_a)\, (1_{A'}\co p_{a'}) \, (1_B\co p_b) \ud\nu$, which equals 
\[
\int_{\Omega^S} \mb{E}\big( (1_A\co p_a)(1_{A'}\co p_{a'}) (1_B\co p_b) |\mc{A}^S_{\{a',b\}}\big)\ud\nu 
 = \int_{\Omega^S} \mb{E}(1_A\co p_a\,|\, \mc{A}^S_{\{a',b\}})\;\; 1_{A'}\co p_{a'} \;\; 1_B\co p_b\,\ud \nu.
\]
Since $\{a,b\}~\bot_\nu~\{a',b\}$, by Lemma \ref{lem:botsuff} we have $\mb{E}(1_A\co p_a\,|\, \mc{A}^S_{\{a',b\}}) =  \mb{E}(1_A\co p_a\,|\, \mc{A}^S_b)$. By \eqref{eq:exprel} and \eqref{eq:pfclaim} we have $\mb{E}(1_A\co p_a\,|\, \mc{A}^S_b)= \mb{E}(1_A\co p_a\,|\, \mc{A}^{\{a,b\}}_b)\co p_{\{a,b\}}=\mb{E}(1_A|\mc{B})\co p_b$. Therefore $\nu(A\times A'\times B)  =  \int_{\Omega^S} \mb{E}(1_A|\mc{B})\co p_b\; 1_{A'}\co p_{a'} \; 1_B\co p_b \ud\nu =\int_{\Omega^{\{a',b\}}} 1_{A'}\co p_{a'} \; \big(\mb{E}(1_A|\mc{B}) \, 1_B) \co p_b  \ud\mu$, where the last equality uses that $\nu_{\{a',b\}}\cong\mu$. Finally, using formula \eqref{eq:relsquare2} for $\mu$ we obtain
\begin{equation}
\nu(A\times A'\times B) = \int_\Omega \mb{E}(1_A|\mc{B})\;\mb{E}(1_{A'}|\mc{B})\;\mb{E}(1_B|\mc{B})\ud\lambda.
\end{equation}
The symmetry follows readily from this formula.
\end{proof}

\begin{lemma}\label{lem:idemp4}
Let $T\subseteq S$, and suppose that $\mu\in\coup(\varOmega, S\times\{0,1\})$  is idempotent when viewed as a coupling in $\coup(\varOmega^S, \{0 ,1\})$ \textup{(}identifying $(\Omega^S)^{\{0,1\}}$ with $\Omega^{S\times\{0,1\}}$\textup{)}, and similarly that $\mu_{T\times\{0,1\}}$ is idempotent as a coupling in $\coup(\varOmega^T, \{0, 1\})$. Then $(T\times\{0,1\}) ~\bot_\mu~ (S\times\{0\})$.
\end{lemma}

\begin{proof}
Let $S'=S\times \{0,1\}$, $T'=T\times \{0,1\}$. It suffices to show that for every pair of functions $f,g\in L^\infty(\mc{A}^T)$ we have that $\mb{E}(f\co p_{T\times\{1\}}\;g\co p_{T\times\{0\}}\,|\,\mc{A}_{S\times\{0\}}^{S'})$ is $\mc{A}_{T\times\{0\}}^{S'}$-measurable. Indeed, if this holds then using Lemma \ref{lem:pisysapprox} we can deduce that every function $h\in L^1(\mc{A}^{T'})$ has $\mb{E}(h\,|\,\mc{A}_{S\times\{0\}}^{S'}) \in L^1(\mc{A}_{T\times\{0\}}^{S'})$, and the result then follows from Lemma \ref{lem:botsuff}. Note moreover that for any $f,g$ as above, we have $\mb{E}(f\co p_{T\times\{1\}}\;g\co p_{T\times\{0\}}\,|\,\mc{A}_{S\times\{0\}}^{S'})=g\co p_{T\times\{0\}}\;\mb{E}(f\co p_{T\times\{1\}}\,|\,\mc{A}_{S\times\{0\}}^{S'})$. Hence it suffices to show that for every $f\in L^\infty(\mc{A}^T)$ we have that $\mb{E}(f\co p_{T\times\{1\}}\,|\,\mc{A}_{S\times\{0\}}^{S'})$ is $\mc{A}^{S'}_{T\times\{0\}}$-measurable.

Let $\mathcal{B}_1\subset \mathcal{A}^S$ and $\mathcal{B}_2\subset \mathcal{A}^T$ be the sub-$\sigma$-algebras given by Proposition \ref{prop:idemp}, thus $p_{S\times \{1\}}^{-1}(\mc{B}_1)=_\mu \mc{A}^{S'}_{S\times\{0\}} \wedge_{\mu} \mc{A}^{S'}_{S\times\{1\}}=_\mu p_{S\times \{0\}}^{-1}(\mc{B}_1)$, and $p_{T\times\{1\}}^{-1}(\mc{B}_2)=_{\mu_{T'}} \mc{A}^{T'}_{T\times\{0\}}\wedge_{\mu} \mc{A}^{T'}_{T\times\{1\}}=_{\mu_{T'}} p_{T\times\{0\}}^{-1}(\mc{B}_2)$. Let $\mu_S$ denote the subcoupling $\mu_{S\times\{0\}}$ of $\mu_{S'}$ viewed as a measure on $\Omega^S$ and similarly $\mu_T$ denote the subcoupling $\mu_{T\times\{0\}}$ of $\mu_{T'}$ viewed as a measure on $\Omega^T$. By  \eqref{eq:relsquare2} applied to $\mc{B}_1$ we have $\int_{\Omega^{S'}} f\co p_{T\times\{1\}}\overline{f\co p_{T\times\{0\}}} \ud\mu = 
\int_{\Omega^S} \big|\mb{E}(f\co p_T|\mc{B}_1)\big|^2\ud\mu_S$. By \eqref{eq:relsquare2} applied to $\mc{B}_2$ we have $\int_{\Omega^{T'}} f\co p_{T\times\{1\}}\;\overline{f\co p_{T\times\{0\}}} \ud\mu_{T'} =\int_{\Omega^T} \big|\mb{E}(f|\mc{B}_2)\big|^2\ud\mu_T$. 
Since $\mu_{T'}$ is the image of $\mu$ under $p_{T'}$, the left sides of the last two equalities are equal. Moreover, since $\mu_T$ is the image of $\mu_S$ under $p_T$, and $\mb{E}(f|\mc{B}_2)\co p_T=\mb{E}(f\co p_T|\mc{B}_2')$ for $\mc{B}_2'=p_T^{-1}\mc{B}_2\subset \mc{A}^S$, we have $\int_{\Omega^T} \big|\mb{E}(f|\mc{B}_2)\big|^2\ud\mu_T=\int_{\Omega^S} \big|\mb{E}(f\co p_T|\mc{B}_2')\big|^2\ud\mu_S$. We deduce that $\|\mb{E}(f\co p_T|\mc{B}_1)\|_2=\|\mb{E}(f\co p_T|\mc{B}_2')\|_2$. Since $\mc{B}_2'\subseteq\mc{B}_1$, the last two expectations are equal, whence $\mb{E}(f\co p_{T\times\{1\}}|p_{S\times \{1\}}^{-1}(\mc{B}_1))=\mb{E}(f\co p_{T\times\{1\}}|p_{T\times \{1\}}^{-1}(\mc{B}_2))$. Lemma \ref{lem:IdempCondProp} $(iii)$ and Proposition \ref{prop:condindep} imply that $\mb{E}(f\co p_{T\times\{1\}}|\mc{A}_{S\times\{0\}}^{S'})=\mb{E}(f\co p_{T\times\{1\}}|p_{S\times \{1\}}^{-1}(\mc{B}_1))$. The last two equalities imply that $\mb{E}(f\co p_{T\times\{1\}}|\mc{A}_{S\times\{0\}}^{S'})$ is $\mc{A}^{S'}_{T\times\{0\}}$-measurable, as required. 
\end{proof}

\begin{lemma}\label{lem:idemp5} Let $\mu\in\coup(\varOmega,S\times \{0,1\})$. Let $S_1,S_2,S_3\subset S$ be such that $S_1\cap S_2=S_3$ and $S_1\cup S_2=S$. Suppose that for $i=1,2,3$ we have that $\mu_{S_i\times\{0,1\}}$ is idempotent as a coupling in $\coup(\Omega^{S_i},\{0,1\})$. Suppose that $S_1\times\{0,1\}~\bot_\mu~S_2\times\{0,1\}$.  Then $\mu$ is idempotent as a coupling in $\coup(\Omega^S,\{0,1\})$.
\end{lemma}

\begin{proof} The result is clear if one of $S_1,S_2$ is $S$,  so we suppose they are both proper subsets, and so none of them includes the other, since $S_1 \cup S_2 =S$. Let $\mu'\in\coup(\varOmega,S\times\{0,1,2\})$ denote the coupling constructed in Definition \ref{def:idemp}. 
Let $T_1=\{0,1\}$, $T_2=\{1,2\}$, $T_3=\{1\}$. Let $\mc{F}=\{S_i\times T_j : 1\leq i,j\leq 3\}$. It is clear that $\mc{F}$ is closed under intersections. Let $\varLambda$ be the set lattice in $S\times\{0,1,2\}$ generated by $\mc{F}$. We claim that every set in $\mc{F}$ has the \textsc{cis} property in $\varLambda$. To see this, note first that for all sets in $\mc{F}$ other than the four sets $S_i\times T_j$, $1\leq i,j\leq 2$, the \textsc{cis} property holds just by inclusion. Furthermore, for each of those four sets, the relation $\bot_{\mu'}$ is equivalent to $\bot_\mu$ (since for such a set $\mu'$ restricts to $\mu$ by construction). For all pairs of such sets except one pair, the relation $\bot_\mu$ holds because of inclusion, and for the remaining pair the relation holds by Lemma \ref{lem:idemp4} (e.g.\ for $S_1\times \{0,1\}$ the proper subsets in $\varLambda$ are $S_1\times \{1\}$, $S_3\times \{0,1\}$, $S_3\times \{1\}$, and $(S_1\times \{1\}) \cup (S_3\times \{0,1\})$, and the pair in question is $S_1\times \{1\}$, $S_3\times \{0,1\}$); this proves our claim. By our assumptions and the symmetry of $\mu'$, we have $S_1\times T_1~\bot_{\mu'}~S_2\times T_1$ and $S_1\times T_2~\bot_{\mu'}~S_2\times T_2$. Proposition \ref{prop:latticeind} gives that $S\times T_1$ and $S\times T_2$ both have the \textsc{cis} property. By construction of $\mu'$ we have $S\times T_1~\bot_{\mu'}~S\times T_2$. Hence $S\times\{0,1,2\}$ has the \textsc{cis} property, by Proposition \ref{prop:latticeind}. It follows that $S_1\times T_1~\bot_{\mu'}~S_1\times T_2$, $S_2\times T_1~\bot_{\mu'}~S_2\times T_2$ and $S\times T_1~\bot_{\mu'}~S\times T_2$. The idempotence of $\mu'_{S_i\times T_1}$ implies by Lemma \ref{lem:idemp3} that $\mu'_{S_i\times\{0,1,2\}}$ is invariant under the permutations of $\{0,1,2\}$ for $i=1,2$. Since $S_1\times\{0,1,2\} ~\bot ~S_2\times\{0,1,2\}$ and $S=S_1\cup S_2$, we have that $\mu'$ is the (unique) conditionally independent coupling of $\mu'_{S_1\times\{0,1,2\}}$ and $\mu'_{S_2\times\{0,1,2\}}$. By \eqref{eq:cid} $\mu'$ is also invariant under these permutations. Therefore $\mu'_{S\times\{0,2\}}\cong\mu'_{S\times\{0,1\}}$, which completes the proof that $\mu'$ is idempotent. 
\end{proof}
\noindent The following result will be used in the next section to establish the nilspace composition axiom for non-injective morphisms on a cubic coupling; see the proof of Lemma \ref{lem:axioms1-2}.
\begin{lemma}\label{lem:idemp2} Let $\varOmega=(\Omega,\mc{A},\lambda)$ be a probability space such that $\Omega$ is a compact space and $\lambda$ is a Borel measure with $\Supp(\lambda)=\Omega$. Let $\mu\in\coup(\varOmega,\{a,b\})$ be an idempotent coupling. Then the support of $\mu$ on $\Omega\times \Omega$ includes the diagonal $\{(x,x) :x\in \Omega\}$.
\end{lemma}
\begin{proof} Let $A\subseteq \Omega$ be an arbitrary non-empty open set. Then $\mu(A\times A)=\langle 1_A,1_A\rangle_{\mu}$. By Corollary \ref{cor:idempos}, this is at least $\lambda(A)^2$. Since $\lambda$ is supported on $\Omega$, we have $\lambda(A)>0$, so $\mu(A\times A)>0$. Let $(x,x)\in \Omega^2$ and let $U$ be an arbitrary open set containing $(x,x)$. By definition of the product topology on $\Omega^2$,  there is an open set $A$ such that $(x,x)\in A\times A\subseteq U$. From the above argument it then follows that $\mu(U)>0$. Since this holds for $U$ an arbitrary such open set, we have that $(x,x)$ is in the support of $\mu$.
\end{proof}

\section{Cubic couplings}\label{sec:cc}
\noindent In this section we begin to study the main objects in this paper, namely cubic couplings.
\begin{defn}\label{def:cc}
A \emph{cubic coupling} on a probability space $\varOmega=(\Omega,\mc{A},\lambda)$ is a  sequence $\big(\mu^{\db{n}}\in \coup(\varOmega,\db{n})\big)_{n\geq 0}$ satisfying the following axioms for all $m,n\geq 0$:
\begin{enumerate}[leftmargin=0.7cm]
\item[1.] (Consistency)\; If $\phi:\db{m}\to\db{n}$ is an injective cube morphism then $\mu^{\db{n}}_\phi=\mu^{\db{m}}$.
\item[2.] (Ergodicity)\; The measure $\mu^{\db{1}}$ is the independent coupling $\lambda \times \lambda$.
\item[3.] (Conditional independence)\; We have $(\{0\}\times \db{n-1})~ \bot ~ (\db{n-1}\times \{0\})$ in $\mu^{\db{n}}$.
\end{enumerate}
\end{defn}
\noindent The notation $\mu^{\db{n}}$ is inspired by the notation $\mu^{[n]}$ used by Host and Kra in \cite[\S 3]{HK}.\footnote{Host and Kra use the notation $[n]$ instead of $\db{n}$, but we already use the former for the set $\{1,2,\ldots,n\}$.} In particular, just as in \cite{HK}, the superscripts $\db{n}$ in our notation are used only to label these measures, and do not have the meaning of the power notation for maps that was defined in the introduction (just before Theorem \ref{thm:MeasInvThmGenIntro}).
\begin{remark}\label{rem:altax3}
Applying axiom 1 with automorphisms\footnote{These are the bijective morphisms from $\db{n}$ to itself; see \cite[Definition 1.1.1]{Cand:Notes1}.} $\phi\in \aut(\db{n})$, and combining this with axiom 3, we deduce that axiom 3 can be stated equivalently as follows: for every pair of $(n-1)$-faces $F_0$, $F_1$ in $\db{n}$ that are \emph{adjacent} (i.e.\ with $F_0\cap F_1\neq \emptyset$), we have $F_0\,\bot_{\mu^{\db{n}}}\,F_1$. In particular, if axiom 1 holds for every $\phi\in \aut(\db{n})$, then to verify axiom 3 it suffices to check that $F_0~\bot_{\mu^{\db{n}}}~F_1$ holds for \emph{some} such pair of faces in $\db{n}$, for each $n$.
\end{remark}
\noindent Note that we must have $\mu^{\db{0}}=\lambda$ (indeed, axiom 1 implies that $\mu^{\db{0}}=\mu^{\db{n}}_{p_{0^n}}$, and since $\mu^{\db{n}}$ is a coupling of $\lambda$ we have $\mu^{\db{n}}_{p_{0^n}}=\lambda$). We can also define a cubic coupling \emph{on a measurable space} $(\Omega,\mc{A})$, as a sequence of measures as above, but without prescribing $\lambda$ as $\mu^{\db{0}}$.

Given a face $F_0=\{v\in \db{n}: v\sbr{i}=0 \}$ of codimension 1 in $\db{n}$, and letting $F_1$ denote the opposite face $F_1=\{v:v\sbr{i}=1\}$, from now on we denote by $\beta$ the bijection that maps an element $v\in F_0$ to $F_1$ by switching $v\sbr{i}$ to 1. Recall that a discrete cube morphism $\phi:\db{m}\to \db{n}$ is a \emph{face map} if its image is an $m$-face in $\db{n}$ (\cite[Definition 1.1.4]{Cand:Notes1}).

We shall establish that Definition \ref{def:cc} is equivalent to the following one, in which the consistency axiom is weakened and conditional independence is replaced by an axiom involving idempotent couplings.
\begin{defn}\label{def:cc-idemp}
A \emph{cubic coupling} on a probability space $\varOmega=(\Omega,\mc{A},\lambda)$ is a sequence $\big(\mu^{\db{n}}\in \coup(\varOmega,\db{n})\big)_{n\geq 0}$ satisfying the following axioms for all $m,n\geq 0$:
\begin{enumerate}[leftmargin=0.7cm]
\item[1.] (Face consistency)\; If $\phi:\db{m}\to\db{n}$ is a face map then $\mu^{\db{n}}_\phi=\mu^{\db{m}}$.
\item[2.] (Ergodicity)\; The measure $\mu^{\db{1}}$ is the independent coupling $\lambda \times \lambda$.
\item[3.] (Idempotence)\; For every pair of opposite $(n-1)$-faces $F_0$, $F_1$ in $\db{n}$, the coupling $\mu^{\db{n}}\in \coup(\varOmega^{\db{n-1}}, \{F_0,F_1\})$ is idempotent along $\beta$.
\end{enumerate}
\end{defn}
\noindent Note that in the idempotence axiom, the possibility to view $\mu^{\db{n}}$ as a self-coupling of $(\Omega^{\db{n-1}},\mu^{\db{n-1}})$ with index set $\{F_0,F_1\}$ follows from the face consistency axiom, since the latter axiom implies that the images of $\mu^{\db{n}}$ under the projections $p_{F_0}$, $p_{F_1}$ are both isomorphic (as couplings) to $\mu^{\db{n-1}}$. This alternative definition of cubic couplings is useful for applications in ergodic theory. In particular, this alternative axiom system is simpler to verify for the measures constructed by Host and Kra in \cite{HK}; this observation leads to our applications in Section \ref{sec:ergapps}.

To prove the equivalence of Definitions \ref{def:cc} and \ref{def:cc-idemp}, we begin with the following result.
\begin{lemma}\label{lem:deduc1}
Suppose that $\big(\varOmega,(\mu^{\db{n}})_{n\geq 0}\big)$ satisfies the three axioms in Definition \ref{def:cc}. Then it satisfies the three axioms in Definition \ref{def:cc-idemp}.
\end{lemma}
\noindent As we shall see, the implication stated in this lemma has a somewhat simpler proof than the converse. In this sense, the axiom system in Definition \ref{def:cc-idemp} may be viewed as more basic than the one in Definition \ref{def:cc}. From the viewpoint of the notions involved, however, the conditional independence axiom can be deemed simpler than the idempotence axiom.
\begin{proof}
Clearly $\big(\varOmega,(\mu^{\db{n}})_{n\geq 0}\big)$ satisfies the face consistency and ergodicity axioms. Concerning the idempotence axiom, note that the coupling $\nu$ on $\{F_0,F_1,F_0'\}$ from Definition \ref{def:idemp} can be realized as a subcoupling of $\mu^{\db{n+1}}$, by  taking two adjacent $n$-faces $V_0,V_1$ in $\db{n+1}$ and identifying $V_0\cap V_1$ with $F_1$, identifying $V_0\setminus F_1$ with $F_0$, and $V_1\setminus F_1$ with $F_0'$. The subcoupling $\mu^{\db{n+1}}_{F_0 \cup F_1 \cup F_0'}$ can then indeed be viewed as the coupling $\nu$ from the previous sentence, by the consistency axiom and the fact that $V_0\,\bot\, V_1$ (by the conditional independence axiom). The idempotence now follows,  since the subcoupling of $\nu$ on $\{F_0, F_0'\}$ equals $\mu^{\db{n}}$, by the consistency axiom applied with a bijective morphism $\phi:\db{n}\to F_0\cup F_0'$.
\end{proof}
\noindent We now prove the converse of Lemma \ref{lem:deduc1}.
\begin{lemma}\label{lem:deduc2}
Suppose that $\big(\varOmega,(\mu^{\db{n}})_{n\geq 0}\big)$ satisfies the three axioms in Definition \ref{def:cc-idemp}. Then it satisfies the three axioms in Definition \ref{def:cc}.
\end{lemma}
\begin{proof}
We begin by proving that the conditional independence axiom holds. Let $F_1$ be the $n$-face $\db{n}\times \{0\}$ in $\db{n+1}$, and let $F_2$ be the $n$-face $\db{n-1}\times\{0\}\times \{0,1\}$ in $\db{n+1}$. By Lemma \ref{lem:idemp4} applied with $S=\db{n}$, $T=\db{n-1}\times \{0\}\subset S$, we have $F_1\bot_{\mu^{\db{n+1}}} F_2$. (To apply this lemma we use the idempotence axiom for $\mu^{\db{n+1}}$ and also for $\mu^{\db{n+1}}_{F_2}\cong\mu^{\db{n}}$.) The conditional independence axiom follows (using Remark \ref{rem:altax3}). 
 
To prove the consistency axiom, let us first prove the following special case. Let $\phi:\db{n}\to\db{n+1}$ be the morphism $(v\sbr{1},v\sbr{2},\ldots,v\sbr{n})\mapsto (v\sbr{1},v\sbr{2},\ldots,v\sbr{n},v\sbr{n})$. Let $F_1$, $F_2$  be the $n$-faces defined in the previous paragraph, so in particular  $F_1\bot_{\mu^{\db{n+1}}} F_2$. Let $a=\db{n-1}\times \{(1,0)\}$, $b=\db{n-1}\times \{(0,0)\}$, $a'=\db{n-1}\times \{(0,1)\}$. By the face consistency axiom, the couplings with index sets $a\sqcup b$ and $a'\sqcup b$ are isomorphic, and since $F_1\bot_{\mu^{\db{n+1}}} F_2$, we have that $\mu^{\db{n+1}}$ restricted to $a \sqcup a' \sqcup b$ is the conditionally independent coupling of $\mu^{\db{n+1}}_{F_1}$ and $\mu^{\db{n+1}}_{F_2}$ (as per Definition \ref{def:condindcoup}) along the bijection $F_1\to F_2$ that permutes the coordinates $v\sbr{n},v\sbr{n+1}$. But then the idempotence axiom implies that $\mu^{\db{n+1}}_{\phi}$ is isomorphic to $\mu^{\db{n}}$, as required. Combining this special case with the face consistency axiom applied with automorphisms, we deduce that the consistency axiom holds for every morphism $\phi:\db{n}\to\db{n+1}$ that duplicates a coordinate, i.e.\ of the form $\phi(v)= (v\sbr{1},\ldots,v\sbr{i-1},v\sbr{i}, v\sbr{i},v\sbr{i+1},\ldots,v\sbr{n})$ for some $i$. Then, composing these maps we deduce consistency for every morphism $\phi':\db{m}\to\db{n}$ that replicates $k_i$-times each $v_i$, with $k_1+\cdots+k_m=n$. (For instance $\phi:(v_1,v_2)\mapsto (v_1,v_1,v_2,v_2,v_2)$ is the composition $\phi_3\co\phi_2\co\phi_1$ where $\phi_1:(v_1,v_2)\mapsto (v_1,v_1,v_2)$, $\phi_2:(v_1,v_2,v_3)\mapsto (v_1,v_2,v_3,v_3)$, $\phi_3:(v_1,v_2,v_3,v_4)\mapsto (v_1,v_2,v_3,v_4,v_4)$, whence $\mu^{[5]}_\phi= (\mu^{[5]}_{\phi_3})_{\phi_2\co\phi_1}= (\mu^{[4]}_{\phi_2})_{\phi_1} = \mu^{[3]}_{\phi_1}=\mu^{[2]}$.) Every injective morphism is a map of the form $\theta\co\varphi\co \phi'$, where $\phi'$ is as above, where $\varphi$ is a face map that simply adds some coordinates equal to 0 or 1, and where $\theta$ is an automorphism. The result follows.
\end{proof}
\noindent Before we continue the study of cubic couplings in general, let us pause to look at examples of such objects. The following result establishes that compact nilspaces with the Haar measures on cube sets are examples of cubic couplings. This provides a large supply of examples, including compact abelian groups and filtered nilmanifolds. In fact, the examples provided by compact nilspaces are in some sense exhaustive. Indeed, this is the content of our main result in Section \ref{sec:structhm}, namely Theorem \ref{thm:MeasInvThmGen}.

We say that a cubic coupling $\big(\varOmega,(\mu^{\db{n}})_{n\geq 0}\big)$ is a \emph{Borel cubic coupling} if $\varOmega$ is a Borel probability space and each $\mu^{\db{n}}$ is a Borel measure on the standard Borel space $(\Omega^{\db{n}},\mc{A}^{\db{n}})$.

\begin{proposition}\label{prop:nilspace-cc}
Let $\ns$ be a $k$-step compact nilspace, and for each $n\geq 0$ let $\mu^{\db{n}}$ denote the Haar measure on the cube set $\cu^n(\ns)$. Then $\big(\ns,(\mu^{\db{n}}\big)_{n\geq 0}\big)$ is a Borel cubic coupling.
\end{proposition}

\begin{proof}
By basic nilspace theory we have that each space $\cu^n(\ns)$ is a compact Polish space and that the Haar measure $\mu^{\db{n}}$ is a Borel probability measure on $\cu^n(\ns)$ (see \cite{CamSzeg, Cand:Notes2}, in particular \cite[Proposition 2.2.5]{Cand:Notes2}) so each space $(\cu^n(\ns),\mu^{\db{n}})$ is a Borel probability space. We now check that the three axioms from Definition \ref{def:cc} are satisfied.

To see that the consistency axiom holds, let $\phi:\db{m}\to\db{n}$ be an injective morphism (in particular $m\leq n$) and consider the set $\phi(\db{m})\subset \db{n}$ equipped with the cubespace structure induced from that on $\db{n}$. If this cubespace $\phi(\db{m})$ has the extension property in $\db{n}$ in the sense of \cite[\S 2.2.3]{Cand:Notes2}, then applying \cite[Lemma 2.2.14]{Cand:Notes2} with $P_1=\emptyset$ and $P_2=\phi(\db{m})$ we have that the restriction from $\db{n}$ to $P_2$ preserves the Haar measures, and the axiom follows. To check the extension property, let $g:\phi(\db{m})\to\ns$ be any nilspace morphism to a non-empty nilspace $\ns$. Then from the definitions we deduce that $\q:= g\co\phi$ is an $m$-cube on $\ns$, and our aim is to show that $g$ can be extended to an $n$-cube on $\ns$. Let $\psi:\db{n}\to\db{m}$ be a morphism such that for every $w\in \db{m}$ we have $\psi\co\phi(w)=w$. We can construct $\psi$ as follows: let $J\subset [n]$ be a set of cardinality $m$ with the property that for each $i\in [m]$ there is a unique $j\in J$ such that for all $w\in \db{m}$ we have $\phi(w)\sbr{j}=w\sbr{i}$ or $1-w\sbr{i}$ ($J$ exists by the injectivity of $\phi$; see \cite[(1.2)]{Cand:Notes1}); the map $w\mapsto \phi(w)|_J$ has an inverse $\phi':\{0,1\}^J\to \db{m}$, and then we can see that $\psi(v):=\phi'(v|_J)$ is a morphism of the desired form. Now note that since $\phi\co\psi$ is a morphism $\db{n}\to \phi(\db{m})$, we have $\q':=g\co \phi\co\psi\in\cu^n(\ns)$, and by construction $\q'$ agrees with $g$ on $\phi(\db{m})$.

To check the ergodicity axiom we can argue by induction on $k$. Note first that the axiom clearly holds for every 1-step compact nilspace $\ns$, since this is a (principal homogeneous space of a) compact abelian group $\ab$ (\cite[Lemma 2.1.4]{Cand:Notes2}) with Haar probability measure $\lambda$, and the Haar measure on $\ns\times\ns$ is then $\lambda\times\lambda$ as required. For $k>1$, we have by \cite[Lemma 2.1.10]{Cand:Notes2} that $\cu^1(\ns)=\ns\times \ns$ is a compact abelian bundle with base $\cu^1(\ns_{k-1})=\ns_{k-1}\times \ns_{k-1}$ and structure group $\ab_k\oplus\ab_k$. By induction the Haar measure on $\cu^1(\ns_{k-1})$ is the product measure $\mu_{k-1}\times \mu_{k-1}$ where $\mu_{k-1}$ is the Haar measure on the $(k-1)$-step nilspace factor $\ns_{k-1}$ of $\ns$. Then it follows from \cite[Lemma 2.2.4]{Cand:Notes2} that for any Borel sets $E_1, E_2\subset \ns$ we have $\mu^{\db{1}}(E_1\times E_2) = \int_{\ns_{k-1}^2}\mu_s\big((E_1\times E_2)\cap \pi^{-1}(s)\big)\ud(\mu_{k-1}\times\mu_{k-1})$, where $\mu_s$ is the Haar measure on the fibre $\pi^{-1}(s)$. By nilspace theory this fibre is $\cu^1(\mc{D}_k(\ab_k))$, which is a principal homogeneous space of the group $\ab_k\oplus\ab_k$ (the definition of $\mc{D}_k(\ab_k)$ may be recalled from \cite[(2.9)]{Cand:Notes1}). Hence $\mu_s$ is the image of the Haar measure on this group, which is $\mu_{\ab_k}\times \mu_{\ab_k}$ for $\mu_{\ab_k}$ the Haar measure on $\ab_k$. Hence $\mu^{\db{1}}(E_1\times E_2) = \prod_{i=1,2} \int_{\ns_{k-1}}\mu_{s_i}\big(E_i\cap \pi_{k-1}^{-1}(s_i)\big)\ud\mu_{k-1}(s_i)=\lambda(E_1)\lambda(E_2)$, as required.

Finally, we check the conditional independence axiom, arguing again by induction on $k$. For $k=1$ this can be checked directly as follows. We may assume as above that $\ns$ is a compact abelian group $\ab$ with Haar probability $\lambda$ and then $\mu^{\db{n}}$ is the Haar measure on the group $\cu^n(\ab)=\big\{\big(x+v\sbr{1}h_1+v\sbr{2}h_2+\cdots+v\sbr{n}h_n)\big)_{v\in \db{n}}: x,h_i\in\ab\}$. For every $i\in [n]$ and $j\in \{0,1\}$, let $V_{i,j}$ denote the face $\{v\in \db{n}:v\sbr{i}=j\}\subset \db{n}$. For any measurable function $f$ on $\ab^{\db{n}}$, we have for every $\q\in\cu^n(\ab)$ that $\mb{E}(f|\mc{B}^{\db{n}}_{V_{i,0}})(\q)=\int_{\ab} f(\q+h_i^{V_{i,1}}) \ud\lambda(h_i)$, where for $V\subset \db{n}$ and $h\in \ab$ we define $h^V(v)$ to equal $h$ if $v\in V$ and $0_{\ab}$ otherwise (see \cite[Definition 2.2.2]{Cand:Notes1}), and $\mc{B}$ denotes the Borel $\sigma$-algebra on $\ns$. In particular we have
\[
\mb{E}(\mb{E}(f|\mc{B}^{\db{n}}_{V_{1,0}})|\mc{B}^{\db{n}}_{V_{n,0}})(\q)=\int_{\ab^2} f(\q+h_1^{V_{1,1}}+h_n^{V_{n,1}}) \ud\lambda^2(h_1,h_n)=\mb{E}(f|\mc{B}^{\db{n}}_{V_{1,0}\cap V_{n,0}}),
\]
so by Lemma \ref{lem:botsuff} we have $V_{1,0}\,\bot\, V_{n,0}$, as claimed in the axiom. For $k>1$, let $V_i$ denote the face $\{v\sbr{i}=0\}$ in $\db{n}$ for each $i\in [n]$, and suppose that $f:\cu^n(\ns)\to\mb{C}$ is $\mc{B}^{\db{n}}_{V_n}$-measurable. Let $H\sbr{V_1}$ denote the abelian group $\hom_{V_1\to 0}(\db{n},\mc{D}_k(\ab_k))$, that is the group of degree-$k$ cubes on $\ab_k$ that send each vertex in $V_1$ to $0_{\ab_k}$ (recall \cite[\S 2.2.4]{Cand:Notes1} for the notion of degree-$k$ cube on an abelian group). Let $g:\cu^n(\ns_{k-1})\to\mb{C}$, $\q_0 \mapsto \int_{H\sbr{V_1}} f(\q_0'+\q)\ud\nu_H(\q)$, where $\q_0'\in \cu^n(\ns)$ is any cube with $\pi_{k-1}\co\q_0'=\q_0$. For almost every $\q\in \cu^n(\ns)$, we have that $\mb{E}(f|\mc{B}^{\db{n}}_{V_1})(\q)$ is the integral of $f$ over the fibre $p_{V_1}^{-1}(p_{V_1}( \q))$. By \cite[Lemma 2.1.10]{Cand:Notes2}, this fibre is a compact abelian bundle with structure group $H\sbr{V_1}$ and base the set of morphisms $M(\q)=\hom_{(\pi_{k-1}\co\q)|_{V_1}}(\db{n},\ns_{k-1})\subset \cu^n(\ns_{k-1})$. Letting $\nu$ be the Haar measure on $M(\q)$ and $\mc{B}_{k-1}$ be the Borel $\sigma$-algebra on the nilspace $\ns_{k-1}$, we then have
\[
\mb{E}(f|\mc{B}^{\db{n}}_{V_1})(\q) = \int_{M(\q)} g(\q_0) \ud\nu_0(\q_0)
 = \mb{E}(g|(\mc{B}_{k-1})^{\db{n}}_{V_1})(\pi_{k-1}\co \q).
\]
Now observe that the $\mc{B}^{\db{n}}_{V_n}$-measurability of $f$ implies that   $g$ is $(\mc{B}_{k-1})^{\db{n}}_{V_n}$-measurable. Since by induction we have $V_1\,\bot\, V_n$ in $\mu_{k-1}^{\db{n}}$, it follows that $\mb{E}(g|(\mc{B}_{k-1})^{\db{n}}_{V_1})=\mb{E}(g|(\mc{B}_{k-1})^{\db{n}}_{V_1\cap V_n})$. Moreover, the $\mc{B}^{\db{n}}_{V_n}$-measurability of $f$ also implies that  $g(\q_0)=\int_{H\sbr{V_1\cap V_n}} f(\q_0'+z)\ud\nu(z)$, where $H\sbr{V_1\cap V_n}=\hom_{V_1\cap V_n\to 0}(\db{n},\mc{D}_k(\ab_k))$. We thus deduce, using again \cite[Lemma 2.1.10]{Cand:Notes2}, that $\mb{E}(g|(\mc{B}_{k-1})^{\db{n}}_{V_1\cap V_n})(\pi_{k-1}\co \q)$ is an integral of $f$ over the fibre $p_{V_1\cap V_n}^{-1}(p_{V_1\cap V_n}(\q))$, and is therefore $\mb{E}(f|\mc{B}^{\db{n}}_{V_1\cap V_n})(\q)$. Lemma \ref{lem:botsuff} now implies that $V_1\,\bot_{\mu^{\db{n}}} V_n$, as required.
\end{proof}

\subsection{Conditional independence of simplicial sets}\hfill \medskip \\
In this subsection we establish a property of cubic couplings that will play a crucial role in the sequel, namely Theorem \ref{thm:sip} below, which tells us that the conditional independence in Definition \ref{def:cc} holds in a more general sense.

For $v\in \db{n}$ we denote by $|v|$ the cardinality of $\supp(v)$, i.e.\ the number of elements $i\in [n]$ with $v\sbr{i}=1$. Given another element $u\in\db{n}$, we  write $u\leq v$ (resp.\ $u< v$) if $\supp(u)\subset \supp(v)$ (resp.\ $\supp(u)\subsetneq \supp(v)$). 
\begin{defn}
A set $H\subseteq\db{n}$ is \emph{simplicial}\footnote{The term \emph{downset} is also used, especially in combinatorics.} at if for every $v\in H$ and $w\leq v$ we have $w\in H$. We denote by $\mc{S}_n$ the set of all simplicial sets in $\db{n}$.
\end{defn}
\noindent Thus every simplicial subset of $\db{n}$ encodes a family of subsets of $[n]$ that is closed under the operation of taking a subset. Note that $\mc{S}_n$ is closed under intersections and unions, and is therefore a set lattice in $\db{n}$. Our goal in this subsection is to prove that this lattice has the \textsc{cis} property defined in the previous section (recall Definition \ref{def:CIS}).

\begin{theorem}\label{thm:sip}
Let $\big(\varOmega, (\mu^{\db{n}})_{n\geq 0}\big)$ satisfy the three axioms in Definition \ref{def:cc-idemp}, and let $m$ be a positive integer. Then for every $H_1,H_2\in \mc{S}_m$ we have $H_1~\bot~ H_2$ in $\mu^{\db{m}}$.
\end{theorem}
\begin{remark}\label{rem:strongaxiom}
Theorem \ref{thm:sip} implies that the faces in $\db{n}$ containing $0^n$ form a conditionally independent system of sets (as per Definition \ref{def:condindepsys}). This implication follows from the fact that $H\subseteq\db{n}$ is simplicial if and only if it is the union of some family of faces containing $0^n$, indeed $H=\bigcup_{i\in [m]} F_i$ where for each $i$ we have $F_i=\{v\in\db{n}:v \leq w_i\}$ for some maximal element $w_i \in H$ in the partial order $\leq$.
\end{remark}
\noindent As mentioned above, the idempotence axiom from Defnition \ref{def:cc-idemp} is useful for applications in ergodic theory. On the other hand, the properties of cubic couplings given in Theorem \ref{thm:sip} and Remark \ref{rem:strongaxiom} are in a sense closer to the corner-completion axiom from nilspace theory (see \cite[Definition 1.1.1]{Cand:Notes1}), and are useful for the proofs of the main results in Section \ref{sec:structhm}, which characterize cubic couplings in terms of compact nilspaces.

Before we turn to the proof of Theorem \ref{thm:sip}, let us record the following consequence.
\begin{corollary}\label{cor:facelocality}
Let $\big(\varOmega, (\mu^{\db{n}})_{n\geq 0}\big)$ be a cubic coupling, and let $m\in\mb{N}$. Then every face in $\db{m}$ is local in $\mu^{\db{m}}$.
\end{corollary}

\begin{proof}
Let $F$ be an arbitrary face in $\db{m}$. We need to show that for every  $w\not\in F$ the $\sigma$-algebras $\mc{A}_F^{\db{m}}$ and $\mc{A}_w^{\db{m}}$ are independent. First we claim that there is a face $B\ni w$ such that $|F\cap B|=1$. To see this, recall that by definition of faces there is some index set $J\subset [m]$ and some $v_0\in \{0,1\}^J$ such that $F=\{v\in \db{m}: v|_J=v_0\}$. Let $v_1=w|_{[m]\setminus J}$. Then the face $B=\{v\in \db{m}: v|_{[m]\setminus J}=v_1\}$ contains $w$ and intersects $F$ only at the point $v$ with $v|_{[m]\setminus J}=v_1$ and $v|_J=v_0$, which proves our claim. Without loss of generality,  we can assume that $F\cap B=\{0^m\}$, using the consistency axiom if necessary (for some $\phi\in\aut(\db{m})$). Let $f$ be bounded $\mc{A}_w^{\db{m}}$-measurable and $g$ be bounded $\mc{A}_F^{\db{m}}$-measurable. Then $\int_{\Omega^{\db{m}}} f\, g \ud\mu^{\db{m}}=\int_{\Omega^{\db{m}}} f\,\mb{E}( g|\mc{A}_B^{\db{m}}) \ud\mu^{\db{m}}$, and since by Theorem \ref{thm:sip} we have $F~\bot~ B$ (using Remark \ref{rem:strongaxiom}), it follows that $\mb{E}( g|\mc{A}_B^{\db{m}})=\mb{E}( g|\mc{A}_{0^m}^{\db{m}})$. Since the latter expectation is $\mc{A}_{0^m}^{\db{m}}$-measurable, there exists a bounded measurable $h:\Omega\to \mb{C}$ such that $\mb{E}( g|\mc{A}_{0^m}^{\db{m}}) = h\co p_{0^m}$, and similarly $f = f'\co p_w$.  We thus obtain that $\int_{\Omega^{\db{m}}} f\, g \ud\mu^{\db{m}}=\int_{\Omega^{\db{m}}} (f'\co p_w) (h\co p_{0^m})\ud\mu^{\db{m}}$. By the consistency axiom this last integral equals $\int_{\Omega^{\{0,1\}}} (f'\co p_1) (h\co p_0) \ud\mu^{\db{1}}$. By ergodicity, this is $(\int_\Omega f'\ud\lambda)(\int_\Omega h\ud\lambda)$, which equals $(\int_{\Omega^{\db{m}}} f \ud\mu^{\db{m}})(\int_{\Omega^{\db{m}}} g \ud\mu^{\db{m}})$ as required.
\end{proof}

\noindent Observe that if $H\in\mc{S}_{n-1}$ then we have $H\times\{0,1\}\in\mc{S}_n$. To prove Theorem \ref{thm:sip} we shall use the following result. 

\begin{lemma}\label{lem:siplem} Let $H\in\mc{S}_{n-1}$ such that $H\times\{0,1\}$ has the \textsc{cis} property in $\mc{S}_n$ and $\mu^{\db{n}}$. Then the coupling $\mu^{\db{n}}_{H\times\{0,1\}}$ is idempotent along $H\times\{0\} \to H\times\{1\}$, $(h,0)\mapsto (h,1)$.
\end{lemma} 

\begin{proof} We argue by induction on $|H|$. If $|H|=1$ then $H=\{0^{n-1}\}$, so by the face-consistency axiom we have $\mu^{\db{n}}_{H\times\{0,1\}}=\mu^{\db{1}}$, and this is clearly idempotent by the ergodicity axiom. If $|H|>1$ then we have two cases, according to whether $H$ is a face or not. If $H$ is a face then $H\times\{0,1\}$ is also a face, so the result follows from the idempotence axiom. If $H$ is not a face then there exist $H_1,H_2\in\mc{S}_{n-1}$ with $|H_1|,|H_2|<|H|$ such that $H=H_1\cup H_2$. The \textsc{cis} property of $H\times\{0,1\}$ implies that $H_1\times\{0,1\},H_2\times\{0,1\},(H_1\cap H_2)\times\{0,1\}$ have the \textsc{cis} property, so by induction these sets satisfy the conclusion of Lemma \ref{lem:siplem}. From the \textsc{cis} property of $H\times\{0,1\}$ it also follows that $H_1\times\{0,1\}~\bot~H_2\times\{0,1\}$. The result now follows from Lemma \ref{lem:idemp5}.
\end{proof}

We can now establish the main result.

\begin{proof}[Proof of Theorem \ref{thm:sip}]
The result is equivalent to the statement that $\db{n}$ satisfies the \textsc{cis} property in $\mc{S}_n$ and $\mu^{\db{n}}$. We prove by induction on $|H|$ that if $H\subset \mc{S}_n$ then $H$ satisfies the \textsc{cis} property. We distinguish two cases. 

Case 1: $H$ is a face. Let $T_1,T_2\subseteq H$ be such that $T_1,T_2\in\mc{S}_n$. If $T_1\cup T_2$ is strictly smaller than $H$ then we can use our induction hypothesis for $T_1\cup T_2\in\mc{S}_n$ to conclude that $T_1~\bot~T_2$. If $T_1\cup T_2=H$ then either $T_1=H$ or $T_2=H$ (indeed since $H$ is a face containing $0^n$ there is $v\in H$ such that $H=\{w\in \db{n}: w\leq v\}$, and if $v\in T_1$, say, then $H=T_1$). If for example $T_1=H$, then $T_2\subseteq T_1$, whence $T_1~\bot~T_2$ (see Lemma \ref{lem:botsuff} and the sentence thereafter).

Case 2: $H$ is not a face. Since $H$ is simplicial, we must have for each $i\in [n]$ that $|H\cap\{v:v\sbr{i}=0\}|\geq |H\cap\{v:v\sbr{i}=1\}|$, and we claim that this inequality is strict for some $i$. Indeed, let $v\in H$ have maximal $|v|$, note that some coordinate $v\sbr{i}$ must be 0 (otherwise $H=\db{n}$), and that the point $v'$ obtained by switching $v\sbr{i}$ to 1 is not in $H$ (by maximality of $|v|$). On the other hand, for every $w'\in H\cap \{w:w\sbr{i}=1\}$, the element obtained by switching $w'\sbr{i}$ to 0 is in $H$. Hence the above inequality is indeed strict for this $i$. By transposing the coordinates $i$ and $n$ of all elements of $\db{n}$, we can assume that $|H\cap(\db{n-1}\times\{0\})|>|H\cap(\db{n-1}\times\{1\})|$. Let $H_0,H_1\subset \db{n-1}$ be such that $H\cap(\db{n-1}\times\{i\})=H_i\times\{i\}$ for $i=0,1$. It is clear that $H_1\subset H_0$ and $H_0,H_1\in\mc{S}_{n-1}$. We also have that $H=(H_0\times\{0\})\cup(H_1\times\{0,1\})$. Since $H_1\times\{0,1\}$ is a proper subset of $H$, by induction it has the \textsc{cis} property. It follows from Lemma \ref{lem:siplem} that $\mu^{\db{n}}_{H_1\times\{0,1\}}$ is idempotent. By the idempotence axiom we have that $\mu^{\db{n}}$ is also idempotent along the bijection $\beta:\db{n-1}\times\{0\}\to \db{n-1}\times\{1\}$. It then follows by Lemma \ref{lem:idemp4} that $H_1\times\{0,1\}~\bot~\db{n-1}\times\{0\}$. Now Lemma \ref{lem:3set1} applied with  $A=\db{n-1}\times\{0\}$, $B=H_1\times\{0,1\}$, and  $C=H_0\times\{0\}$, shows that $H_1\times\{0,1\}~\bot~H_0\times\{0\}$. By our induction hypothesis, both $H_1\times\{0,1\}$ and $H_0\times\{0\}$ have the \textsc{cis} property. It follows from Proposition  \ref{prop:latticeind} that $H$ also has the \textsc{cis} property.
\end{proof}

\subsection{Tricubes}\label{subsec:tricubes} \hfill \medskip\\
We recall from \cite[Definition 3.1.13]{Cand:Notes1} that the \emph{tricube} of dimension $n$ can be defined as the set $T_n=\{-1,0,1\}^n$ equipped with a certain cubespace structure (that we shall not recall here). Another useful way to view this cubespace is as a subset of $\db{2n}$, obtained as the image of $\{-1,0,1\}^n$ under the injection $q_n$ defined as follows. First we define $q_1: \{-1,0,1\}\to  \{0,1\}^2$, $-1\mapsto \binom{0}{1}$, $0\mapsto \binom{0}{0}$, $1\mapsto \binom{1}{0}$. Then we define
\[
q_n: \{-1,0,1\}^n \to \db{2n}, \; t	\mapsto \begin{psmallmatrix} v\sbr{1} & v\sbr{2} & \cdots & v\sbr{n}\\[0.1em] v\sbr{n+1} & v\sbr{n+2} & \cdots & v\sbr{2n} \end{psmallmatrix}, \textrm{ where } \begin{psmallmatrix} v\sbr{i} \\[0.1em] v\sbr{n+i} \end{psmallmatrix}=q_1(t_i), \, \forall\, i.
\]
We shall often denote by $\wt{T}_n$ this alternative version of the tricube, that is $\wt{T}_n=q_n(T_n)$. Note that $\wt{T}_n=\{v\in \db{2n}: \forall\, i\in [n],\, v_i v_{i+n}=0\}$, which makes it clear that $\wt{T}_n$ is simplicial in $\db{2n}$. Note also that $q_n^{-1}:\wt{T}_n \to T_n$ is defined by $q_n^{-1}(v)=t$ with $t_i=v_i-v_{i+n}$.
 
The direct power $S_3^n$ of the symmetric group $S_3$ acts in a clear way on $T_n$ by permuting the coordinate of $t$. Via the map $q_n$, this group $S_3^n$ acts on $\wt{T}_n$. Given a cubic coupling on $(\varOmega,(\mu^{\db{n}})_n)$, we can use this action of $S_3^n$ on $\wt{T}_n$ to define a coordinatewise action of $S_3^n$ on $\Omega^{\wt{T}_n}$. The main purpose of this subsection is to record the following very useful fact concerning this action.

\begin{lemma}\label{lem:trisym} Let $(\varOmega,(\mu^{\db{n}})_{n\geq 0})$ be a cubic coupling. Then for each $n$ the subcoupling $\mu^{\db{2n}}_{\wt{T}_n}$ of $\mu^{\db{2n}}$ is preserved by the coordinatewise action of $S_3^n$ on $\Omega^{\wt{T}_n}$.
\end{lemma}
\begin{proof}
Fix $n$ and consider the following subsets of $T_n$: $A=\{-1,0,1\}^{n-1}\times\{-1,0\}$, $B=\{-1,0,1\}^{n-1}\times\{1,0\}$, \, $C=\{-1,0,1\}^{n-1}\times\{0\}$. Let $\wt{A}=q_n(A)$, $\wt{B}=q_n(B)$, $\wt{C}=q_n(C)$, and note that $\wt{A}= \{v\in \wt{T}_n: v\sbr{n}=0\}$, that $\wt{B}=\{v\in \wt{T}_n: v\sbr{2n}=0\}$, and that $\wt{C}= \wt{A}\cap \wt{B}$, so these sets are all simplicial in $\db{2n}$. By Theorem \ref{thm:sip} we have $\wt{A}~\bot_{\mu^{\db{2n}}}~\wt{B}$. By Lemma \ref{lem:siplem} applied with $H=\wt{T}_{n-1}$, we have that $\mu^{\db{2n}}_{\wt{A}}$ is an idempotent coupling of two copies of $\mu_{\wt{T}_{n-1}}^{\db{2(n-1)}}$ indexed by $v\sbr{2n}$. Moreover $\mu^{\db{2n}}_{\wt{A}}$ and $\mu^{\db{2n}}_{\wt{B}}$ are isomorphic couplings, and $\wt{T}_n=\wt{A}\cup \wt{B}$, and thus $\mu^{\db{2n}}_{\wt{T}_n}$ is the coupling $\nu$ obtained by applying Definition \ref{def:idemp} to $\mu^{\db{2n}}_{\wt{A}}$. Hence Lemma \ref{lem:idemp3} implies that $\mu^{\db{2n}}_{\wt{T}_n}$ is symmetric under the action of $S_3$ on $\{-1,0,1\}$ applied to the last coordinate of $q_n^{-1}(v)$. Using that $\mu^{\db{2n}}_{\wt{T}_n}$ is also symmetric with respect to the permutation of the coordinates in $q_n^{-1}(v)$, the result follows. 
\end{proof}
\noindent Note that the above lemma together with the $S_n$ invariance implies that $\mu^{[2n]}_{\wt{T}_n}$ is also invariant under the action of the wreath product of $S_n$ and $S_3$. However we are only going to use the $S_3^n$ symmetries.

Recall from \cite{Cand:Notes1} that we denoted by $\omega_n$ the \emph{outer point map} of $T_n$, that is the map $\db{n} \to T_n$, $v\mapsto (2v\sbr{1}-1,\ldots,2v\sbr{n}-1)$.
\begin{corollary}[Outer point coupling]\label{cor:opcoup}
The subcoupling of $\mu^{\db{2n}}_{\wt{T}_n}$ along the map $q_n \co\omega_n$ is equal to $\mu^{\db{n}}$. 
\end{corollary}
\noindent This fact is an analogue for cubic couplings of the tricube composition lemma for nilspaces  (see \cite[Lemma 3.1.16]{Cand:Notes1}), and it is used in particular in Subsection \ref{subsec:Fk} below to prove a key property of higher-order Fourier $\sigma$-algebras (see Lemma \ref{lem:keybot}).
\begin{proof}
By the $S_3^n$ invariance, we can see that the subcoupling of $\mu^{\db{2n}}_{\wt{T}_n}$ along the index set $q_n \co \omega_n(\db{n})$ is isomorphic to the subcoupling of $\mu^{\db{2n}}$ on an $n$-dimensional face of $\db{2n}$ (using an element of $S_3^n$ that maps $\{-1,1\}^n$ to $\db{n}$).
\end{proof}

\subsection{$U^d$-convolutions and $U^d$-seminorms associated with a cubic coupling}
\hfill \medskip\\
We begin with the definition of a generalization of convolution that can be defined on a cubic coupling using the measures $\mu^{\db{n}}$. To that end, let us denote by $\mc{C}$ the conjugation operator on $L^1(\Omega)$, defined by $\mc{C}f(y)=\overline{f(y)}$. We denote by $K_d$ the set $\db{d}\setminus\{0^d\}$.
\begin{defn}[$U^d$-convolution]\label{def:Udconv}
Let $\big(\varOmega, (\mu^{\db{n}})_{n\geq 0}\big)$ be a cubic coupling. For each $d\geq 1$, and any system $F=(f_v)_{v\in K_d}$ of functions $f_v\in L^\infty(\varOmega)$, we define the \emph{$U^d$-convolution} of $F$, denoted by $[F]_{U^d}$, to be a function in $L^\infty(\varOmega)$ such that we have $\mu^{\db{d}}$-almost everywhere
\begin{equation}\label{eq:Udconvdef}
\mb{E}\big(\prod_{v\in K_d} \mc{C}^{|v|+1} f_v\co p_v ~|~ \mc{A}_{0^d}^{\db{d}}\big) = [F]_{U^d}\co p_{0^d}.
\end{equation}
\end{defn}
\noindent This defines the function $[F]_{U^d}$ up to a $\lambda$-null set (using Lemma \ref{lem:Doob}). When $\varOmega$ is a Borel probability space, an alternative  equivalent definition of $[F]_{U^d}$ can be given as follows. Letting $(\mu^{\db{d}}_x)_{x\in \Omega}$ be the disintegration of $\mu^{\db{d}}$ given by Lemma \ref{lem:localize} (thus $\mu^{\db{d}}_x\in \coup(\varOmega,K_d)$ for all $x$),  we can define $[F]_{U^d}(x)$ as the integral $\int_{\Omega^{K_d}} \prod_{v\in K_d} \mc{C}^{|v|+1} f_v\co p_v\, \ud\mu^{\db{d}}_x$.

\begin{defn}[$U^d$-product]
Let $\big(\varOmega, (\mu^{\db{n}})_{n\geq 0}\big)$ be a cubic coupling. For every $d\geq 1$, and every system $F=(f_v)_{v\in \db{d}}$ of functions $f_v\in L^\infty(\varOmega)$, we define the \emph{$U^d$-product} of these functions by the formula
\begin{equation}
\langle F\rangle_{U^d}= \langle (f_v)_{v\in \db{d}}\rangle_{U^d} = \int_{\Omega^{\db{d}}} \prod_{v\in \db{d}} \mc{C}^{|v|}\, f_v\co p_v \; \ud\mu^{\db{d}}.
\end{equation}
If all $f_v$ are equal to the same function $f$, we denote $\langle (f_v)_{v\in \db{d}}\rangle_{U^d}^{1/2^d}$ by $\|f\|_{U^d}$.
\end{defn}
\noindent Note that for every $f\in L^\infty(\varOmega)$, by the idempotence of the coupling $\mu^{\db{d}}\in \coup(\varOmega',\{F_0,F_1\})$ and Corollary \ref{cor:idempos}, we have when $f_v=f$ for all $v\in \db{d}$ that $\langle (f_v)_{v\in \db{d}}\rangle_{U^d}$ is a non-negative real number; hence $\|f\|_{U^d}$, as the $2^d$-th root of this number, is a well-defined non-negative real number.

We denote by $[F]_{U^d}^\times$ the rank-1 function $\prod_{v\in K_d} \mc{C}^{|v|+1} f_v\co p_v: \Omega^{\db{d}}\to\mb{C}$. 

The above definitions are related by the following observation: given a system $F=(f_v)_{v\in \db{d}}$, letting $F'=(f_v)_{v\in K_d}$ we have
\begin{equation}\label{eq:coincon}
\langle F\rangle_{U^d} = \int_\Omega f_{0^d}\; \overline{[F']}_{U^d}\,\ud\lambda = \int_{\Omega^{\db{d}}} f_{0^d}\co p_{0^d}\; \overline{[F']^\times_{U^d}}\, \ud\mu^{\db{d}}.
\end{equation}
\noindent We now prove a generalization, for this $U^d$-product, of the Gowers-Cauchy-Schwarz inequality, using the idempotence axiom for the couplings $\mu^{\db{n}}$.

\begin{lemma}\label{lem:GCSineq}
Let $\big(\varOmega, (\mu^{\db{n}})_{n\geq 0}\big)$ be a cubic coupling and let $d\in \mb{N}$. Then for every system $(f_v)_{v\in \db{d}}$ of bounded measurable functions on $\Omega$, we have 
\begin{equation}\label{eq:GCS}
|\langle\, (f_v)_{v\in \db{d}}\, \rangle_{U^d} | \leq \prod_{v\in \db{d}} \|f_v\|_{U^d}.
\end{equation}
\end{lemma}
\noindent The idea of the proof is that the idempotence axiom makes it possible to apply a standard argument, originating in \cite{GSz}, that uses the Cauchy-Schwarz inequality repeatedly.
\begin{proof}
For $i=0,1$ let $F_i=\{v\in \db{d}: v\sbr{d}=i\}$, thus $F_0,F_1$ are two opposite faces of codimension 1 in $\db{d}$. Letting $g_i:\Omega^{\db{d-1}}\to \mb{C}$, $y\mapsto \prod_{v\in \db{d-1}} \mc{C}^{|v|}\, f_{(v,i)}\co p_v(y)$, we have 
\[
\langle (f_v)_{v\in \db{d}}\rangle_{U^d} = \int_{\Omega^{\db{d}}} g_0\co p_{F_0} \;\overline{g_1\co p_{F_1}} \; \ud\mu^{\db{d}}.
\]
The idempotence axiom tells us that the coupling $\mu^{\db{d}}\in \coup(\varOmega',\{F_0,F_1\})$ is idempotent along the bijection $F_0\to F_1$ that switches the coordinate $v\sbr{d}$ in $v\in F_0$ from 0 to 1, where $\varOmega'=\varOmega^{\db{d-1}}$. Letting $\mc{B}$ denote the sub-$\sigma$-algebra of $\mc{A}^{\db{d-1}}$ given by Proposition \ref{prop:idemp}, we then have by \eqref{eq:relsquare2} that $\langle (f_v)_{v\in \db{d}}\rangle_{U^d}  = \int_{\Omega^{\db{d-1}}} \mb{E}\big(g_0|\mc{B}\big)\,\overline{\mb{E}\big(g_1|\mc{B}\big)}\; \ud\mu^{\db{d-1}}$. By the Cauchy-Schwarz inequality, this integral is at most a product of two factors, namely $\big(\int_{\Omega^{\db{d-1}}} \mb{E}\big(g_i|\mc{B}\big)\,\overline{\mb{E}\big(g_i|\mc{B}\big)}\; \ud\mu^{\db{d-1}}\big)^{1/2}$ for $i=0,1$. By \eqref{eq:relsquare2} again, these factors equal $\langle g_i,g_i\rangle_{\mu^{[d]}}^{1/2}$, $i=0,1$, and these in turn can be seen to equal respectively $\langle (f_v')_{v\in \db{d}}\rangle_{U^d}^{1/2}$ and $\langle (f_v'')_{v\in \db{d}}\rangle_{U^d}^{1/2}$, where $f'_v=f_{(v_1,\ldots,v_{d-1},0)}$ and $f''_v=f_{(v_1,\ldots,v_{d-1},1)}$ for all $v\in \db{d}$. Thus we have obtained
$|\langle (f_v)_{v\in \db{d}}\rangle_{U^d}| \leq \langle (f_v')_{v\in \db{d}}\rangle_{U^d}^{1/2}\; \langle (f_v'')_{v\in \db{d}}\rangle_{U^d}^{1/2}$. Repeating this argument for each of these two factors, and so on inductively, we obtain \eqref{eq:GCS} after $d$ steps.
\end{proof}

\begin{corollary}\label{cor:udsn}
Let $\big(\varOmega, (\mu^{\db{n}})_{n\geq 0}\big)$ be a cubic coupling and let $d\in \mb{N}$. Then $\|\cdot\|_{U^d}$ is a seminorm on $L^\infty(\varOmega)$. We call $\|\cdot\|_{U^d}$ the \emph{uniformity seminorm of order} $d$, or \emph{$U^d$-seminorm}, on this cubic coupling. We also have
\begin{equation}\label{eq:nestedUd}
\|f\|_{U^d}\leq \|f\|_{U^{d+1}}, \textrm{ for every }d\geq 1\textrm{ and every } f\in L^\infty(\varOmega).
\end{equation}
\end{corollary}
\begin{proof}
Given \eqref{eq:GCS}, the triangle inequality for $\|\cdot\|_{U^d}$ follows by the same argument that proves it for the Gowers norms (see \cite[Lemma 3.9]{GSz}). To see \eqref{eq:nestedUd}, note that the consistency axiom implies $\|f\|_{U^d}^{2^d}=\langle F\rangle_{U^{d+1}}$, for $F=(f_v)_{v\in \db{d+1}}$ the system with $f_v=f$ for $v\sbr{d+1}=0$ and $f_v=1$ otherwise. Then $\langle F\rangle_{U^{d+1}} \leq \|f\|_{U^{d+1}}^{2^d}$, by \eqref{eq:GCS}, and \eqref{eq:nestedUd} follows.
\end{proof}

\subsection{Fourier $\sigma$-algebras}\label{subsec:Fk}\hfill \medskip\\
In this section we study the following special sub-$\sigma$-algebras of the ambient $\sigma$-algebra $\mc{A}$ in a cubic coupling, which play a crucial role in Section \ref{sec:structhm}.

\begin{defn}\label{def:FourierAlg}
Let $\big(\varOmega=(\Omega,\mc{A},\lambda),\; (\mu^{\db{n}})_{n\geq 0}\big)$ be a cubic coupling. For each $d\in \mb{N}$, the $d$-th \emph{Fourier $\sigma$-algebra} on $\Omega$, denoted by $\mc{F}_d$, is the sub-$\sigma$-algebra of $\mc{A}$ generated by all $U^{d+1}$-convolutions of bounded $\mc{A}$-measurable functions.
\end{defn}
\noindent A first observation about these $\sigma$-algebras is that
\begin{equation}\label{eq:Fd1stobs}
\mc{F}_0 \textrm{ is the trivial $\sigma$-algebra, and $\mc{F}_{d-1}\subset \mc{F}_d$ for every $d\in \mb{N}$}.
\end{equation} 
Indeed the inclusion $\mc{F}_{d-1}\subset \mc{F}_d$ follows from the fact that every convolution of order $d$ can be viewed as a convolution of order $d+1$. More precisely, given any system $F=(f_v)_{v\in K_d}$, note that if we extend this to a system $F'=(f_v')_{v\in K_{d+1}}$ by embedding $K_d$ in some $d$-face $S \subset \db{d+1}$ containing $0^{d+1}$ and letting $f'_v$ be the constant 1 function for every $v\not\in S$, then $[F]_{U^d}=[F']_{U^{d+1}}$.

Most of the properties of the $\sigma$-algebras $\mc{F}_d$ given in this subsection are consequences of the following key fact about cubic couplings (which was illustrated in Example \ref{ex:U2coup}).
\begin{lemma}\label{lem:keybot}
Let $\big(\varOmega, (\mu^{\db{n}})_{n\geq 0}\big)$ be a cubic coupling and let $d\in \mb{N}$. Then  $\mc{A}_{0^d}^{\db{d}}\, \upmod_{\mu^{\db{d}}}\, \mc{A}^{\db{d}}_{K_d}$.
\end{lemma}
\begin{proof}
We use the tricube coupling. Consider the following subsets of $T_d=\{-1,0,1\}^d$:
\begin{eqnarray*}
&& V_1=\{-1,1 \}^d\setminus \{-1^d\}, \quad V_2=T_d\setminus \{-1^d\},\\
&& V_3=\{ v\in T_d: \forall\,i,\;v\sbr{i}\in \{-1,0\}\}, \quad V_4=\{-1^d\}, \quad V_5= \{-1,1 \}^d.
\end{eqnarray*}
For each $i$ let $\mc{V}_i=\mc{A}^{\wt{T}_d}_{\wt{V}_i}$, where $\wt{V}_i,\wt{T}_d$ are the corresponding subsets of $\db{2d}$ under the bijection $q_d$ from Subsection \ref{subsec:tricubes}. Let $g$ be any bounded $\mc{A}^{\db{d}}_{K_d}$-measurable function on $\Omega^{\db{d}}$. It suffices to prove that $\mb{E}(g| \mc{A}_{0^d}^{\db{d}})$ is still $\mc{A}^{\db{d}}_{K_d}$-measurable. Let $\wt{g}:\Omega^{\wt{T}_d}\to\mb{C}$, $x\mapsto g(\pi(x))$ where $\pi:\Omega^{\wt{T}_d}\to \Omega^{\db{d}}$ is the projection to the outer-point coordinate-set $q_d(\{-1,1\}^d)$ (composed with the bijection $\Omega^{q_d(\{-1,1\}^d)}\to \Omega^{\db{d}}$ induced by $\omega_d^{-1}\co q_d^{-1}$). By Corollary \ref{cor:opcoup}, we have that $\wt{g}$ is $\mc{V}_1$-measurable. By the face consistency axiom, the subcoupling of $\mu^{\db{2d}}_{\wt{T}_d}$ along $\wt{V}_3$ is isomorphic to $\mu^{\db{d}}$ (since $\wt{V}_3$ is a face in $\db{2d}$). We therefore have (using \eqref{eq:exprel} to relate conditional expectations of $g$ and $\wt{g}$) that it suffices to show that $\mb{E}(\wt{g}|\mc{V}_4)$ is $\mc{A}_{\wt{V}_3\setminus \wt{V}_4}^{\wt{T}_d}$-measurable. We first claim that $\wt{V}_2~\bot~\wt{V}_3$ in $\mu^{\db{2d}}_{\wt{T}_d}$. To see this note first that $\wt{V}_2~\bot_{\mu^{\db{2d}}}~\wt{V}_3$, by Theorem \ref{thm:sip} applied to $\mu^{\db{2d}}$, using the fact that $\wt{V}_2$ is a union of faces sharing the ``central point" of $\wt{T}_d$ (i.e.\ the point $0^{2d}$, which corresponds to the central point $0^d$ in $T_d$) and that $\wt{V}_3$ is also such a face. The claim then follows by Remark \ref{rem:botinsubcoup}. Given this claim and the fact that $V_2\cap V_3=V_3\setminus V_4$, it now suffices to show that $\mb{E}(\wt{g}|\mc{V}_4)$ is $\mc{V}_2\wedge \mc{V}_3$-measurable. To this end, note first that $\mb{E}(\wt{g}|\mc{V}_4)=\mb{E}(\wt{g}|\mc{V}_3)$, since $\wt{V}_4=\wt{V}_3\cap \wt{V}_5$ and since we also have the fact that $\wt{V}_3~\bot~\wt{V}_5$ in $\mu^{\db{2d}}_{\wt{T}_d}$, a fact that can be seen using Lemma \ref{lem:trisym} and Theorem \ref{thm:sip} combined with Remark \ref{rem:botinsubcoup} again. More precisely, note that $\wt{V}_3$ and $\wt{V}_5$ are not both faces in $\db{2d}$ (so we cannot conclude the fact directly from Theorem \ref{thm:sip} as before) but, by Lemma \ref{lem:trisym} applied within $\mu^{\db{2d}}_{\wt{T}_d}$, if we  apply the transformation corresponding to the element of $S_3^n$ that transposes $-1$ and $0$ in each coordinate in $T_d$, then $\wt{V}_3$ remains globally invariant while $\wt{V}_5$ becomes now a face containing $0^{2d}$, so we can then conclude the conditional independence of these sets by Theorem \ref{thm:sip}, and then revert the transformation to conclude that indeed $\wt{V}_3~\bot~\wt{V}_5$. Having thus shown that $\mb{E}(\wt{g}|\mc{V}_4)=\mb{E}(\wt{g}|\mc{V}_3)$, it now suffices to show that $\mb{E}(\wt{g}|\mc{V}_3)$ is $\mc{V}_2\wedge \mc{V}_3$-measurable. But this follows from $\wt{g}$ being $\mc{A}_{\wt{V}_2}^{\wt{T}_d}$-measurable and the above fact that $\wt{V}_2~\bot~\wt{V}_3$.
\end{proof}
Let us record a useful immediate consequence of Lemma \ref{lem:keybot}.
\begin{corollary}\label{cor:Udconvmeas}
For every $U^d$-convolution $[F]_{U^d}$ we have that $[F]_{U^d}\co p_{0^d}$ is $\mc{A}_{0^d}^{\db{d}}\wedge \mc{A}_{K_d}^{\db{d}}$-measurable on $\Omega^{\db{d}}$.
\end{corollary}
\begin{proof}
We have $F=(f_v)_{v\in K_d}$ for some functions $f_v\in L^\infty(\varOmega)$. By definition we have $\mu^{\db{d}}$-almost everywhere $[F]_{U^d}\co p_{0^d}=\mb{E}(\prod_{v\in K_d} f_v| \mc{A}^{\db{d}}_{0^d})$, and since $\prod_{v\in K_d} f_v$ is $\mc{A}^{\db{d}}_{K_d}$-measurable, the result follows by Lemma \ref{lem:keybot}.
\end{proof}
The following theorem is the main result of this subsection.

\begin{theorem}[Properties of $\mc{F}_d$]\label{thm:Fkprops}
Let $\big(\varOmega=(\Omega,\mc{A},\lambda), (\mu^{\db{n}})_{n\geq 0}\big)$ be a cubic coupling. For every positive integer $d$, the following statements hold:
\begin{enumerate}
\item In $\mu^{\db{d}}$ we have $(\mc{F}_{d-1})_{0^d}^{\db{d}} = \mc{A}^{\db{d}}_{0^d} \wedge \mc{A}_{K_d}^{\db{d}}$.

\item For $f\in L^\infty(\varOmega)$ we have $\| f\|_{U^d}=0$ if and only if $\mb{E}(f| \mc{F}_{d-1})=0$.

\item In $\mu^{\db{d}}$ we have $(\mc{F}_{d-1})_{0^d}^{\db{d}} = \mc{A}^{\db{d}}_{0^d} \wedge (\mc{F}_{d-1})_{K_d}^{\db{d}}$.

\item $\|\cdot\|_{U^d}$ is a norm on $L^\infty(\mc{F}_{d-1})$.
\end{enumerate}
\end{theorem}
\begin{proof}
To see statement $(i)$, note first that the inclusion $(\mc{F}_{d-1})_{0^d}^{\db{d}} \subset \mc{A}^{\db{d}}_{0^d} \wedge \mc{A}_{K_d}^{\db{d}}$ follows immediately from Corollary \ref{cor:Udconvmeas}, since this tells us that, for every set $B$ that is the preimage of a Borel set by a $U^d$-convolution, we have $p_{0^d}^{-1}(B)\in \mc{A}^{\db{d}}_{0^d} \wedge \mc{A}_{K_d}^{\db{d}}$. To see the opposite inclusion, let $f$ be a bounded $\mc{A}^{\db{d}}_{0^d} \wedge \mc{A}_{K_d}^{\db{d}}$-measurable function and, since this is $\mc{A}^{\db{d}}_{0^d}$-measurable, let $f'$ be bounded $\mc{A}$-measurable such that $f=f'\co p_{0^d}$ almost everywhere (using Lemma \ref{lem:Doob}). We have by Lemma \ref{lem:pisysapprox} that for any fixed $\epsilon>0$ there is a finite sum $h=\sum_{i\in [m]} \prod_{v\in K_d} g_{i,v}\co p_v$ such that $\|f-h\|_{L^2}\leq \epsilon$. Now note that by linearity and \eqref{eq:Udconvdef} we have $\mb{E}(h|\mc{A}_{0^d}^{\db{d}})=(\sum_{i\in [m]} [F_i]_{U^d})\co p_{0^d}$, where $F_i=(g_{i,v})_{v\in \db{d}}$, and then $\|f'-\sum_{i\in [m]} [F_i]_{U^d}\|_{L^2(\lambda)}=\|f-\mb{E}(h|\mc{A}_{0^d}^{\db{d}})\|_{L^2(\mu^{\db{d}})}=\|\mb{E}(f-h|\mc{A}_{0^d}^{\db{d}})\|_{L^2(\mu^{\db{d}})}\leq \epsilon$, so $f'$ is an $L^2$-limit of $\mc{F}_{d-1}$-measurable functions and is therefore $\mc{F}_{d-1}$-measurable.

To see $(ii)$, let $R$ denote the set of rank-1 bounded $\mc{A}^{\db{d}}_{K_d}$-measurable functions on $\Omega^{\db{d}}$. We claim that $\mb{E}(f|\mc{F}_{d-1})=0$ holds if and only if $f\co p_{0^d}$ is orthogonal to every function in $R$. To see the forward implication, fix any $g\in R$ and let $G$ be the system of functions $g_v$, $v\in K_d$ such that $g=[G]_{U^d}^\times$. We have $[G]_{U^d}\in L^\infty(\mc{F}_{d-1})$, so 
$\mb{E}_{\mu^{\db{d}}}( f\co p_{0^d}\, \overline{[G]_{U^d}^\times}) = \mb{E}_\lambda ( f\, \overline{[G]_{U^d}}) =\mb{E}( \mb{E}(f\,\overline{[G]_{U^d}}|\mc{F}_{d-1}) )=\mb{E}( \overline{[G]}_{U^d}\,\mb{E}(f|\mc{F}_{d-1}))=0$.

To see the backward implication, note that if $f\co p_{0^d}$ is orthogonal to $R$ then by Lemma \ref{lem:pisysapprox} we have that $f\co p_{0^d}$ is orthogonal to every function in $L^2(\mc{A}^{\db{d}}_{K_d})$, hence, since $\mb{E}(f\co p_{0^d}| \mc{A}^{\db{d}}_{K_d})$ is the orthogonal projection of $f\co p_{0^d}$ to the subspace $L^2(\mc{A}^{\db{d}}_{K_d})$, it is 0. By statement $(i)$ we have $\mb{E}(f\co p_{0^d}| \mc{A}^{\db{d}}_{K_d})=\mb{E}(f\co p_{0^d}|(\mc{F}_{d-1})^{\db{d}}_{0^d})$ and since this is $\mb{E}(f|\mc{F}_{d-1})\co p_{0^d}$, we deduce indeed that $\mb{E}(f|\mc{F}_{d-1})=0$. Having proved our claim, note now that by inequality \eqref{eq:GCS} we have that $\|f\|_{U^d}=0$ if and only if $f\co p_{0^d}$ is orthogonal to $R$, and statement $(ii)$ follows.

To see $(iii)$, note that by statement $(i)$ we clearly have $(\mc{F}_{d-1})_{0^d}^{\db{d}}\supset \mc{A}_{0^d}^{\db{d}}\wedge (\mc{F}_{d-1})_{K_d}^{\db{d}}$, so we just need to prove the opposite inclusion. To do this, first note the following fact:
\begin{equation}\label{eq:moduleprop}
\textrm{for all $f\in L^\infty(\mc{A})$ with $\|f\|_{U^d}=0$ and  $g\in L^\infty(\mc{F}_{d-1})$, we have $\|f g \|_{U^d}=0$.}
\end{equation}
Indeed, by $(ii)$ applied to $f$ we have $\mb{E}(f g|\mc{F}_{d-1})= g\, \mb{E}(f|\mc{F}_{d-1})=0$, and then $\|f g \|_{U^d}=0$ follows by $(ii)$ applied now to $fg$. We claim that it follows from \eqref{eq:moduleprop} that for every system $F=(f_v)_{v\in \db{d}}$ of bounded $\mc{A}$-measurable functions, we have $\mu^{\db{d}}$-almost-everywhere
\begin{equation}\label{eq:factorFdproj}
\mb{E}\Big( \prod_v f_v\co p_v\,\big |\, (\mc{F}_{d-1})^{\db{d}}\Big)= \prod_v \mb{E}(f_v| \mc{F}_{d-1})\co p_v.
\end{equation}
To prove this, fix any such system $F=(f_v)_{v\in \db{d}}$, and note that it suffices to show that for every system $(g_v)_{v\in \db{d}}$ of  functions in $L^\infty(\mc{F}_{d-1})$ we have $\int_{\Omega^{\db{d}}} \prod_{v\in \db{d}} (f_v \, g_v)\co p_v \ud\mu^{\db{d}} =  \int_{\Omega^{\db{d}}} \prod_{v\in \db{d}} (\mb{E}(f_v|\mc{F}_{d-1}) g_v)\co p_v \ud\mu^{\db{d}}$. Using multilinearity, the difference between these integrals is seen to be a sum of finitely many integrals of the form $\int_{\Omega^{\db{d}}} \prod_{v\in \db{d}} (h_v g_v)\co p_v \ud\mu^{\db{d}}$ where for some $v$ we have $h_v= f_v-\mb{E}(f_v|\mc{F}_{d-1})$. By statement $(ii)$ this function $h_v$ has $U^d$-seminorm 0. Hence, since by \eqref{eq:moduleprop} we have $\|h_v g_v\|_{U^d}=0$, by Lemma \ref{lem:GCSineq} we conclude that each such integral is 0, which proves the above equality of integrals, and our claim follows. To finish proving $(iii)$, fix any system $F=(f_v)_{v\in \db{d}}$ as above, and note that for every $\epsilon>0$ we have 
$\| [F]_{U^d}\co p_{0^d}- \sum_{i=1}^{m_\epsilon} \prod_{v\in K_d} g_{i,v}\co p_v\|_{L^2}\leq \epsilon$ for some bounded $\mc{A}$-measurable function $g_{i,v}$ and some $m_\epsilon\in \mb{N}$. Since trivially $[F]_{U^d}\co p_{0^d} = \mb{E}([F]_{U^d}\co p_{0^d} |(\mc{F}_{d-1})^{\db{d}})$, we conclude that $
\Big\| [F]_{U^d}\co p_{0^d} -  \sum_{i=1}^{m_\epsilon} \mb{E}( \prod_{v\in K_d} g_{i,v}\co p_v |(\mc{F}_{d-1})^{\db{d}}) \Big\|_{L^2} \leq \epsilon$, and by \eqref{eq:factorFdproj} the sum here is $\sum_{i=1}^{m_\epsilon} \prod_{v\in K_d} \mb{E}( g_{i,v}|(\mc{F}_{d-1}) \co p_v$, which is $(\mc{F}_{d-1})^{\db{d}}_{K_d}$-measurable. This shows that $[F]_{U^d}\co p_{0^d}$ is an $L^2$-limit of $(\mc{F}_{d-1})^{\db{d}}_{K_d}$-measurable functions, and $(iii)$ follows. 

To see $(iv)$ note that the seminorm $\|\cdot\|_{U^d}$ is indeed non-degenerate on $L^\infty(\mc{F}_{d-1})$, for if $f\in L^\infty(\mc{F}_{d-1})$ then $f=\mb{E}(f|\mc{F}_{d-1})$, so if this is 0 then so is $\|f\|_{U^d}$ by statement $(ii)$.
\end{proof}
\noindent The following consequence of statement $(ii)$ above is useful and can be viewed as an alternative definition of Fourier $\sigma$-algebras (but note that it would be less clear from such a definition of $\mc{F}_d$ that this is indeed a $\sigma$-algebra).
\begin{corollary}\label{cor:noisebotalg}
Let $f\in L^\infty(\mc{A})$. We have that $f\in L^\infty(\mc{F}_{d-1})$ if and only if for every $g\in L^\infty(\mc{A})$ with $\|g\|_{U^d}=0$ we have $\mb{E}(f\,\overline{g})=0$.
\end{corollary}

\begin{proof}
For the forward implication note that if $f\in L^\infty(\mc{F}_{d-1})$ then for every such $g$ we have by statement $(ii)$ above that $\mb{E}(g|\mc{F}_{d-1})=0$ and so $\mb{E}(f\,\overline{g})=\mb{E}(f\,\overline{\mb{E}(g|\mc{F}_{d-1})})=0$. For the backward implication, note that if $f$ is orthogonal to every such $g$ then in particular for $g=f-\mb{E}(f|\mc{F}_{d-1})$, since by statement $(ii)$ above we have $\|g\|_{U^d}=0$, it follows that $\mb{E}(f\overline{g})=0$, and this implies that $ \mb{E}(|f|^2)=\mb{E}(|\mb{E}(f|\mc{F}_{d-1})|^2)$, so $f$ must be in $L^\infty(\mc{F}_{d-1})$.
\end{proof}

\subsection{Properties of $U^d$-convolutions}\hfill \medskip \\
Let us introduce the following notation:
\begin{equation}
\db{n}_{\leq d} = \{v\in \db{n}: |v|\leq d\}, \qquad K_{n,\leq d}= \{v\in K_n: |v|\leq d\}.
\end{equation}
Recall also that the \emph{height} of a simplicial set $S\subset \db{n}$ is $\max_{v\in S} |v|$.
\begin{lemma}\label{lem:prefdconvfd}
Let $(\varOmega,(\mu^{\db{n}})_{n\geq 0})$ be a cubic coupling, let $n\in\mb{N}$, and let $S$ be a simplicial subset of $\db{n}$ of height $d\geq 1$. Let $F=(f_v)_{v\in K_n}$ be a system of bounded measurable functions on $\Omega$ where $f_v=1$ for $v\in K_n\setminus S$. Then $[F]_{U^n}$ is $\mc{F}_{d-1}$-measurable.
\end{lemma}
\begin{proof}
We prove this for each fixed $n$ by induction on $|S|$. Since $d\geq 1$, we must have $|S|\geq 2$. If we have equality, then we must have $S=\{0^n,v\}$ with $|v|=1$, so $d=1$. Then by \eqref{eq:Udconvdef} and the consistency axiom we have $[F]_{U^n}\co p_{0^n}=\mb{E}_{\mu^{\db{1}}}\big(f_v\co p_1|\mc{A}_0^{\db{1}}\big)$, and by the ergodicity axiom $\mu^{\db{1}}=\lambda\times\lambda$, so this expectation equals $\lambda$-almost everywhere the constant $\int_\Omega f_v\ud\lambda$, and is therefore in $L^\infty(\mc{F}_0)$ as required.

For $|S|>2$, by \eqref{eq:Udconvdef} it suffices to show that $\mb{E}\big(\prod_{v\in K_n\cap S} f_v\co p_v|\mc{A}_{0^n}^{\db{n}}\big)$ is $(\mc{F}_{d-1})_{0^n}^{\db{n}}$-measurable. Fix any maximal element $w\in S$, i.e.\ an element with $|w|=d$. Let $\phi:\db{d}\to \db{n}$ be a face map satisfying $\phi(0^d)=w$ and $\phi(v)< w$ otherwise (in particular $\phi(1^d)=0^n$), let $V(w)$ denote the image of $\phi$ (i.e.\ the $d$-face with maximal element $w$) and let $K(w)$ denote the corner $\phi(K_d)$, which is included in $S$. Let $\mc{B}=\mc{A}_{S\setminus\{w\}}^{\db{n}}$. Since $\mc{B}\supset \mc{A}_{0^n}^{\db{n}}$ and every $f_v\co p_v$ with $v\not= w$ in the product $\prod_{v\in K_n\cap S} f_v\co p_v$ is $\mc{B}$-measurable, we have
\begin{equation}\label{eq:B-reduc}
\mb{E}\Big(\prod_{v\in K_n\cap S} f_v\co p_v|\mc{A}_{0^n}^{\db{n}}\Big) = \mb{E}\Big(\mb{E}(f_w\co p_w| \mc{B})\prod_{v\in K_n \cap S\setminus\{w\}} f_v\co p_v ~ |~\mc{A}_{0^n}^{\db{n}}\Big).
\end{equation}
Now note that $S\setminus\{w\}$ is still simplicial, so by Theorem \ref{thm:sip} we have $S\setminus\{w\}~ \bot~ V(w)$ in $\mu^{\db{n}}$. Since $f_w\co p_w$ is $\mc{A}_{V(w)}^{\db{n}}$-measurable, and since we have $(S\setminus\{w\})\cap V(w)=K(w)$, it follows that $\mb{E}(f_w\co p_w| \mc{B})=\mb{E}(f_w\co p_w| \mc{A}_{K(w)}^{\db{n}})$. Now Lemma \ref{lem:keybot} implies that $\mc{A}_{K(w)}^{\db{n}}\upmod \mc{A}_w^{\db{n}}$, so the last expectation is $\mb{E}(f_w\co p_w| \mc{A}_w^{\db{n}}\wedge \mc{A}_{K(w)}^{\db{n}})$. Statements $(i)$ and $(iii)$ in Theorem \ref{thm:Fkprops} imply that this expectation is in fact $\mb{E}(f_w\co p_w| \mc{A}_w^{\db{n}}\wedge (\mc{F}_{d-1})_{K(w)}^{\db{n}})$. This expectation can then be approximated in $L^2(\mu^{\db{n}})$ arbitrarily closely by $(\mc{F}_{d-1})_{K(w)}^{\db{n}}$-measurable rank-1 functions (by Lemma \ref{lem:pisysapprox}). Thus, fixing any $\epsilon>0$, substituting such an approximation of $\mb{E}(f_w\co p_w| \mc{B})$ into \eqref{eq:B-reduc} we obtain that
\[
\Big\| \mb{E}\big(\prod_{v\in K_n\cap S} f_v\co p_v|\mc{A}_{0^n}^{\db{n}}\big) ~ - ~ 
\sum_{i\in [m]} g_{i,0^n} \co p_{0^n}\; \mb{E}\big(\prod_{v\in K_n\cap S \setminus\{w\}} f'_{i,v}\co p_v ~|\mc{A}_{0^n}^{\db{n}}\big)\Big\|_{L^2(\mu^{\db{n}})} \leq \epsilon,
\]
where $f'_{i,v}=f_{i,v}$ for $v\not\in K(w)$ and $f'_{i,v}=f_v g_{i,v}$ for $v\in K(w) \setminus\{0^n\}$, and each $g_{i,v}$ is $\mc{F}_{d-1}$-measurable. The sum on the right side above is $\mc{F}_{d-1}$-measurable by induction. Since $\epsilon$ was arbitrary, we deduce that $\mb{E}(\prod_{v\in K_n\cap S} f_v\co p_v ~|\mc{A}^{\db{n}}_{0^n})$ is an $L^2$-limit of $\mc{F}_{d-1}$-measurable functions, and the result follows.
\end{proof}
We now use Lemma \ref{lem:prefdconvfd} to deduce the following result.

\begin{lemma}\label{lem:fdconvfd} Let $(\varOmega,(\mu^{\db{n}})_{n\geq 0})$ be a cubic coupling and let $n,d\in\mb{N}$. Let $F=(f_v)_{v\in K_n}$ be a system of functions in $L^\infty(\mc{F}_d)$. Then $[F]_{U^n}$ is $\mc{F}_d$-measurable.
\end{lemma}

\begin{proof}
If $n\leq d+1$ then by \eqref{eq:Fd1stobs} the result is clear. Assuming then that $n>d+1$, by \eqref{eq:Udconvdef} it suffices to show that $\mb{E}\big(\prod_{v\in K_n} f_v\co p_v|\mc{A}_{0^n}^{\db{n}}\big)$ is $(\mc{F}_d)_{0^n}^{\db{n}}$-measurable. Now $\prod_{v\in K_n} f_v\co p_v$ is $(\mc{F}_d)_{K_n}^{\db{n}}$-measurable, and we claim that it suffices to show that $(\mc{F}_d)_{K_n}^{\db{n}}\subset_{\mu^{\db{n}}} (\mc{F}_d)_{K_{n,\leq d+1}}^{\db{n}}$. Indeed, if this holds then $\prod_{v\in K_n} f_v\co p_v$ is in fact $(\mc{F}_d)_{K_{n,\leq d+1}}^{\db{n}}$-measurable, and then it is an $L^2$-limit of finite sums of rank-1 bounded $(\mc{F}_d)_{K_{n,\leq d+1}}^{\db{n}}$-measurable functions $h_i$; but for each such $h_i$ we have that $\mb{E}(h_i|\mc{A}_{0^n}^{\db{n}})$ is $\mc{F}_d$-measurable, by Lemma \ref{lem:prefdconvfd}, so $\mb{E}(\prod_{v\in K_n} f_v\co p_v|\mc{A}_{0^n}^{\db{n}})$ is an $L^2$-limit of $\mc{F}_d$-measurable functions and our claim follows.

Now, to show that $(\mc{F}_d)_{K_n}^{\db{n}}\subset_{\mu^{\db{n}}} (\mc{F}_d)_{K_{n,\leq d+1}}^{\db{n}}$, we can proceed as follows: fix any $w$ with $|w|=d+2$ and note that $(\mc{F}_d)_{K_{n,\leq d+1}\cup \{w\}}^{\db{n}}\subset_{\mu^{\db{n}}} (\mc{F}_d)_{K_{n,\leq d+1}}^{\db{n}}$ because for any rank-1 function $\prod_{v\in K_{n,\leq d+1}\cup \{w\}} f_v\co p_v$ measurable relative to the former $\sigma$-algebra, the function $f_w\co p_w$ is an $L^2$-limit of sums of rank-1 bounded $\mc{A}_{\{v:v<w\}}^{\db{n}}$-measurable functions. Applying this recursively for each $w$ with $|w|=d+2$, we deduce that $(\mc{F}_d)_{K_{n,\leq d+2}}^{\db{n}}\subset_{\mu^{\db{n}}} (\mc{F}_d)_{K_{n,\leq d+1}}^{\db{n}}$. Arguing similarly for the next height level $d+3$, we deduce that $(\mc{F}_d)_{K_{n,\leq d+3}}^{\db{n}}\subset_{\mu^{\db{n}}} (\mc{F}_d)_{K_{n,\leq d+2}}^{\db{n}}$. Continuing thus up to level $n$, the result follows.
\end{proof}
\noindent For the next result we introduce the following notation: given a simplicial set $S\subset \db{n}$, we define the \emph{degree} of an element $v\in S$, denoted by $d(v)$, to be the maximal value of $|w|$ over all $w\in S$ with $w\geq v$. In particular we always have $d(v)\geq |v|$, with equality if and only if $v$ is maximal in $S$.
\begin{lemma}\label{lem:simpzero}
Let $S$ be a simplicial subset of $\db{n}$, let $u\in S$ and let $d=d(u)\geq 1$. Let $F=(f_v)_{v\in \db{n}}$ be a system of functions in $L^\infty(\mc{A})$ with $\|f_u\|_{U^d}=0$ and $f_v=1$ for $v\not\in S$. Then $\langle F\rangle_{U^n}=0$.
\end{lemma}
\begin{proof}
We argue by induction on $|S|+d(u)-|u|$, starting with the case $|S|+d(u)-|u|=2$. In this case note that we must have $|S|=2$ (otherwise $S=\{0^n\}$ and $d=0$), so $S=\{0^n,v\}$ for some $v$ of height 1, and the result then follows from the ergodicity axiom and the assumption that $\|f_u\|_{U^1}=\big|\int_\Omega f_u\ud\lambda\,\big| = 0$.

For $|S|+ d(u)-|u|> 2$, we distinguish two cases: either there is some $w\in S$ with $w>u$, or $u$ is maximal in $S$.

In the first case, take $w>u$ to have $|w|=\max\{|v|:v\in S,\;v > u\}$. Firstly, we can reduce $f_w$ to a function that is $\mc{F}_{d-1}$-measurable. Indeed $f_w=g_w+h_w$ where $g_w:=\mb{E}(f_w|\mc{F}_{d-1})$ is $\mc{F}_{d-1}$-measurable, and $h_w=f_w-g_w$ has zero $U^d$-seminorm by statement $(ii)$ in Theorem \ref{thm:Fkprops}; then the $U^n$-convolution with $h_w$ is 0 by induction because $|S|+d(w)-|w|<|S|+d(u)-|u|$  (since $|w|>|u|$ and $d(w)=d(u)$). Hence we can assume that $f_w$ is $\mc{F}_{d-1}$-measurable. Using Theorem \ref{thm:Fkprops} and Lemma \ref{lem:pisysapprox} as in the proof of Lemma \ref{lem:prefdconvfd} (since $|w|=d$), for any fixed $\epsilon>0$, letting $M=\max_{v\in S}\|f_v\|_{L^\infty}$ (which we can suppose to be positive, to avoid trivialities), we can approximate $f_w\co p_w$ within $\epsilon/M^{|S|}$ in $L^2(\mu^{\db{n}})$ by a sum of rank-1 functions $\prod_{v\in K(w)} g_{i,v}\co p_v$, $g_{i,v}\in L^\infty(\mc{F}_{d-1})$, $i\in [m_\epsilon]$ (where $K(w)=\{v:v< w\}$). Then we have $\big |\langle (f_v)_{v\in \db{n}}\rangle_{U^n} ~ - ~ \sum_{i\in [m_\epsilon]} \langle (f'_{i,v})_{v\in \db{n}}\rangle_{U^n}\big | \leq \epsilon$, 
where $f'_{i,v}=f_v$ for $v\in S\setminus K(w)$, $f'_{i,v}=f_v g_{i,v}$ for $v\in K(w)$, and $f_w'=1$. For each $i \in [m_\epsilon]$, note that in $(f'_{i,v})_{v\in \db{n}}$ the functions are non-trivial only for $v$ in the simplicial set $S\setminus\{w\}$, and note also that by \eqref{eq:moduleprop} we still have $\|f'_{i,u}\|_{U^d}=0$. Hence, by induction, we have $\langle (f'_{i,v})_{v\in \db{n}}\rangle_{U^n}=0$. Since this holds for each $i$ we deduce that $|\langle (f_v)_{v\in \db{n}}\rangle_{U^n}|\leq \epsilon$, and since $\epsilon$ was arbitrary, we are done in this case.

In the second case we have $|u|=d$. Let $V(u)$ denote the $d$-dimensional face $\{v\in \db{n}: v\leq u\}$, and let $S'$ denote the simplicial set $S\setminus \{u\}$. Since $f_v\co p_v$ is $\mc{A}_{S'}^{\db{n}}$-measurable for every $v\in S'$, we have $\langle F\rangle_{U^n}= \mb{E}\big(\mb{E}(f_u\co p_u| \mc{A}_{S'}^{\db{n}})\, \prod_{v\in S'} f_v\co p_v \big)$. By Theorem \ref{thm:sip} we have $S'~\bot_{\mu^{\db{n}}} V(u)$. Since $S' \cap V(u)=V(u)\setminus\{u\}$, we have $\mb{E}(f_u\co p_u| \mc{A}_{S'}^{\db{n}})=\mb{E}(f_u\co p_u| \mc{A}_{V(u)\setminus\{u\}}^{\db{n}})$. By Lemma \ref{lem:keybot}, this expectation is $\mb{E}(f_u\co p_u| \mc{A}_{V(u)\setminus\{u\}}^{\db{n}}\wedge\mc{A}_{u}^{\db{n}})$, and by statement $(i)$ in Theorem \ref{thm:Fkprops} this in turn is $\mb{E}(f_u\co p_u| (\mc{F}_{d-1})_u^{\db{n}})=\mb{E}(f_u| \mc{F}_{d-1})\co p_u$. By statement $(ii)$ in Theorem \ref{thm:Fkprops} the last expectation is 0, since $\|f_u\|_{U^d}=0$. Thus we obtain that $\langle F\rangle_{U^n}=0$ in this case as well.
\end{proof}
\noindent For the following result we consider the notion of a $U^n$-convolution of a system $F=(f_v)_{v\in \db{n}}$ taken at a vertex $r\neq 0^n$. By this we mean a function $g\in L^\infty(\mc{A})$ satisfying $\mu^{\db{n}}$-almost everywhere $g\co p_r = \mb{E}(\prod_{v\in \db{n}\setminus\{r\}}f_v \co p_v| \mc{A}_r^{\db{n}})$. We denote such a convolution by $[F]_{U^n,r}$, to distinguish it from the original one $[F]_{U^n}$ (which is taken at $0^n$).
\begin{corollary}\label{cor:simpzero}
Let $S$ be a simplicial subset of $\db{n}$ and let $u,r\in S$, $u\neq r$. Let $K=\db{n}\setminus\{r\}$ and let $d = d(u)$. Let $F=(f_v)_{v\in K}$ be a system of functions in $L^\infty(\mc{A})$ such that $\|f_u\|_{U^d}=0$ and $f_v=1$ for $v\in K\setminus S$. Then $[F]_{U^n,r}$ is zero $\lambda$-almost everywhere. 
\end{corollary}

\begin{proof}
We claim that the result is equivalent to Lemma \ref{lem:simpzero}. To see this, let $F$ be the system supposed in the corollary, let $f$ be any function in $L^\infty(\mc{A})$, and let $F'$ denote the system $(f_v')_{v\in \db{n}}$ with $f_v'=f_v$ for $v\neq r$ and $f_r'=f$. Applying Lemma \ref{lem:simpzero} to $F'$ we obtain $0=\langle F'\rangle_{U^n}=\mb{E}_\lambda(f \,\overline{[F']_{U^n,r}})$. In particular letting $f=[F']_{U^n,r}$ we deduce that $\|[F']_{U^n,r}\|_{L^2}=0$, which implies the conclusion in the corollary. The opposite implication is also clear using that $\langle F'\rangle_{U^n}=\mb{E}_\lambda (f\, \overline{[F']_{U^n,r}})$.
\end{proof}
\noindent Another consequence of Lemma \ref{lem:simpzero} is the following useful fact about $U^d$-products.
\begin{corollary}\label{cor:simpzerocor}
Let $S$ be a simplicial subset of $\db{n}$ of height at most $d$, and let $(f_v)_{v\in\db{n}}$ be a system of functions in $L^\infty(\mc{A})$ such that $f_v=1$ for $v\in\db{n}\setminus S$. Let $G=(g_v)_{v\in\db{n}}$ with $g_v=\mb{E}(f_v|\mc{F}_{d-1})$ for each $v$. Then $\langle F\rangle_{U^n}=\langle G\rangle_{U^n}$. 
\end{corollary}
\begin{proof}
We decompose each $f_v$ for $v\in S$ as $g_v+(f_v-g_v)$ where $\|f_v-g_v\|_{U^d}=0$. By the multilinearity of $\langle F\rangle_{U^n}$, this $U^n$-product expands as a sum of $\langle G\rangle_{U^n}$ plus finitely many $U^n$-products, each involving a function with $U^d$-seminorm 0 at some $v\in S$. By Lemma \ref{lem:simpzero}, every such $U^n$-product is 0, and the result follows. 
\end{proof}
\noindent We say that two vertices $w_1,w_2\in \db{n}$ are \emph{neighbours} if they are neighbours in the graph of 1-faces on $\db{n}$. The following third consequence of Lemma \ref{lem:simpzero} gives a sufficient condition for a $U^n$-product to vanish. The condition can be more useful than asking for some function in the product to have zero $U^n$-seminorm.

\begin{lemma}\label{lem:neighverts}
Let $F=(f_v)_{v\in \db{d+1}}$ be a function system in $L^\infty(\mc{A})$ such that for some neighbours $w_1,w_2\in \db{d+1}$ we have $f_{w_1}\in L^\infty(\mc{F}_{d-1})$ and $\|f_{w_2}\|_{U^d}=0$. Then $\langle F\rangle_{U^{d+1}}=0$.
\end{lemma}

\begin{proof}
By the consistency axiom we can assume that $w_1=1^{d+1}$. We have that $\langle F\rangle_{U^{d+1}}= \mb{E}_{\mu^{\db{d+1}}}\big(\prod_v f_v\co p_v\big)$. Using part $(iii)$ of Theorem \ref{thm:Fkprops} as in previous proofs, we have that this expectation is the limit of similar expectations but with $f_{1^{d+1}}\co p_{1^{d+1}}$ replaced by a finite sum of bounded $(\mc{F}_{d-1})^{\db{d+1}}_K$-measurable rank-1 functions, where $K$ is a $d$-corner of the form $K=\{v: v_0\leq v < 1^{d+1}\}$ for some (any) $v_0$ of height 1. It therefore suffices to show that the expectation for each such rank-1 function is 0. For each such function $\prod_{v\in K} g_v\co p_v$, the corresponding expectation is of the form $\mb{E}_{\mu^{\db{d+1}}}\big(\prod_{v\in \db{d+1}\setminus \{1^{d+1}\}} f_v'\co p_v\big)$, where $f_v'=f_vg_v$ for $v\in K$ and $f'_v=f_v$ otherwise, and where   $\|f_{w_2}'\|_{U^d}=0$ by \eqref{eq:moduleprop}. Applying Lemma \ref{lem:simpzero} with $u=w_2$ and $S=\db{d+1}\setminus \{1^{d+1}\}$, we obtain that the last expectation is 0, and the result follows.
\end{proof}

\begin{lemma}\label{lem:vetites1}
Let $F=(f_v)_{v\in K_{d+1}}$ be a system of functions in $L^\infty(\mc{A})$, for each $v$ let $g_v=\mb{E}(f_v|\mc{F}_{d-1})$, and let $G=(g_v)_{v\in K_{d+1}}$. Then $[G]_{U^{d+1}}\,=_\lambda\, \mb{E}([F]_{U^{d+1}}|\mc{F}_{d-1})$.
\end{lemma}

\begin{proof}
By Lemma \ref{lem:fdconvfd}, the function $[G]_{U^{d+1}}$ is $\mc{F}_{d-1}$-measurable. Therefore it suffices to prove that for every $h\in L^\infty(\mc{F}_{d-1})$ we have $\mb{E}_\lambda([G]_{U^{d+1}} h)=\mb{E}_\lambda([F]_{U^{d+1}} h)$. Let us decompose $f_v$ into $f_v=g_v+r_v$, for each $v\in K_{d+1}$, where $\|r_v\|_{U^{d-1}}=0$. Let $F'=(f'_v)_{v\in\db{d+1}}$ be the function system with $f'_v=f_v$ if $v\neq 0^{d+1}$ and $f'_{0^{d+1}}=\overline{h}$. By \eqref{eq:coincon} we have $\mb{E}([F]_{U^{d+1}} h)=\langle F'\rangle_{U^{d+1}}$. By multilinearity of the $U^{d+1}$-product, we can expand $\langle F'\rangle_{U^{d+1}}$ into a sum of $2^{2^{d+1}-1}$ different $U^{d+1}$-products, each of which involves one of $g_v$, $r_v$ for each $v\neq 0^{d+1}$ and $h$ at $v= 0^{d+1}$. By Lemma \ref{lem:neighverts} we have that of all these $U^{d+1}$-products the only one that can be non-zero is the one involving $g_v$ for every $v\neq 0^{d+1}$. Indeed if there are both $r_v$ and $g_v$ factors, or if all factors are $r_v$, then we can find two neighbour vertices $w_1,w_2$, such that the function corresponding to $w_1$ is $g_{w_1}$ or $h$, and the function corresponding to $w_2$ is $r_{w_2}$, whence Lemma \ref{lem:neighverts} implies that this $U^{d+1}$-product is 0. The only remaining term is equal to $\mb{E}([G]_{U^{d+1}} h)$, and the result follows.
\end{proof}
\noindent We close this subsection with a result that is natural and is also useful in what follows.
\begin{proposition}\label{prop:cubicfactor}
Let $\big((\Omega,\mc{A},\lambda),(\mu^{\db{n}})_{n\geq 0}\big)$ be a cubic coupling. Then for each $d$ we have that $\big((\Omega,\mc{F}_d,\lambda),(\mu^{\db{n}})_{n\geq 0}\big)$ is also a cubic coupling.
\end{proposition}

\begin{proof}
The first two axioms in Definition \ref{def:cc} clearly hold for the restriction of the measures $\mu^{\db{n}}$ to $\mc{F}_d$, so it suffices to prove the conditional independence axiom, that is, to show that for the faces $A=\db{n}\times \{0\}$, $B=\{0\} \times\db{n}$ in $\db{n+1}$ we have $A~\bot~ B$ in $_{\mc{F}_d|}\mu^{\db{n+1}}$. Letting $F=A\cap B$, by Lemma \ref{lem:botsuff} it suffices to prove that for every function $f\in L^\infty\big((\mc{F}_d)^{\db{n+1}}_A\big)$ we have $\mb{E}(f| (\mc{F}_d)^{\db{n+1}}_B)=\mb{E}(f| (\mc{F}_d)^{\db{n+1}}_F)$. 
Let $Q=\db{n+1}_{\leq d}$. By an iterated application of statement (iii) in Theorem \ref{thm:Fkprops}, we have $(\mc{F}_d)^{\db{n+1}}_A=_{\mu^{\db{n+1}}}(\mc{F}_d)^{\db{n+1}}_{A\cap Q}$ (indeed we can use the statement to generate each copy of $\mc{F}_d$ at a vertex $v\in A\setminus Q$ by copies at vertices $w\leq v$ forming a copy of $K_d$, and thus we can eliminate the copy of $\mc{F}_d$ at $v$). Similarly we have $(\mc{F}_d)^{\db{n+1}}_B =_{\mu^{\db{n+1}}} (\mc{F}_d)^{\db{n+1}}_{B\cap Q}$ and $(\mc{F}_d)^{\db{n+1}}_F=_{\mu^{\db{n+1}}}(\mc{F}_d)^{\db{n+1}}_{F\cap Q}$. Hence $f$ is $ (\mc{F}_d)^{\db{n+1}}_{A\cap Q}$-measurable and it suffices to show that $\mb{E}(f| (\mc{F}_d)^{\db{n+1}}_{B\cap Q})=\mb{E}(f| (\mc{F}_d)^{\db{n+1}}_{F\cap Q})$. We have $\mb{E}\big(f| (\mc{F}_d)^{\db{n+1}}_{B\cap Q}\big) = \mb{E}\big(\mb{E}(f|\mc{A}^{\db{n+1}}_{B\cap Q}) \,| (\mc{F}_d)^{\db{n+1}}_{B\cap Q}\big) = \mb{E}\big(\mb{E}(f|\mc{A}^{\db{n+1}}_{F\cap Q})\, | (\mc{F}_d)^{\db{n+1}}_{B\cap Q}\big)$, where the last equality uses that $(A\cap Q)~\bot~(B\cap Q)$ in the original coupling $\mu^{\db{n}}$ (by Theorem \ref{thm:sip}, noting that $A\cap Q$, $B\cap Q$ are simplicial sets). Now $\mb{E}(f|\mc{A}^{\db{n+1}}_{F\cap Q})$ is an $L^2$ limit of finite sums $\sum_{i\in [M]} g_i$ of bounded $\mc{A}^{\db{n+1}}_{F\cap Q}$-measurable rank-1 functions $g_i=\prod_{v\in F\cap Q} g_{i,v}\co p_v$. Each factor $g_{i,v}$ is the sum of the $\mc{F}_d$-measurable function $\mb{E}(g_{i,v}|\mc{F}_d)$ and the function with zero $U^{d+1}$-seminorm $g_{i,v}-\mb{E}(g_{i,v}|\mc{F}_d)$. Expanding the product, we write $g_i$ as a sum of $\mc{A}_{F\cap Q}^{\db{n+1}}$-measurable rank-1 functions $h_{i,0}, h_{i,1}, \ldots, h_{i,m}$, where $h_{i,0}$ has all factors $\mc{F}_d$-measurable, and $h_{i,j}$ has at least one factor of zero $U^{d+1}$-seminorm, for every $j>0$. We claim that for every $i$ and $j>0$ we have $\mb{E}(h_{i,j}| (\mc{F}_d)^{\db{n+1}}_{B\cap Q})=0$. This will follow if we show that each $h_{i,j}$ is orthogonal to every bounded $(\mc{F}_d)^{\db{n+1}}_{B\cap Q}$-measurable rank-1 function, since the latter functions are dense in $L^2((\mc{F}_d)^{\db{n+1}}_{B\cap Q})$ (Lemma \ref{lem:pisysapprox}). To see the orthogonality, note that the inner product of $h_{i,j}$ with any of these rank-1 functions is a $U^{n+1}$-product in which, for some $v$ in the simplicial set $S=B\cap Q$ of height at most $d$, we have $\|f\|_{U^{d+1}}=0$ (by \eqref{eq:moduleprop}), so this $U^{n+1}$-product is zero by Lemma \ref{lem:simpzero}.

We deduce that $\mb{E}(g_i| (\mc{F}_d)^{\db{n+1}}_{B\cap Q})=\mb{E}(h_{i,0}| (\mc{F}_d)^{\db{n+1}}_{B\cap Q})=h_{i,0}$, so this is indeed $(\mc{F}_d)^{\db{n+1}}_{F\cap Q}$-measurable, as required.
\end{proof}

\subsection{Topologization of cubic couplings}\label{ChTopMeas}\hfill \smallskip \\
Given a cubic coupling $(\varOmega,(\mu^{\db{n}})_{n\geq 0})$, the goal of this subsection is to define a compact topological space $\ns$ associated with $\varOmega$ that is fine enough to capture all the information about the couplings $\mu^{\db{n}}$ that is relevant for us. More precisely, the space $\ns$ will be such that, letting $\mc{B}$ denote the Borel $\sigma$-algebra on $\ns$, there is a measurable map $\gamma:\Omega\to\ns$, defined $\lambda$-almost everywhere, such that each coupling $\mu^{\db{n}}$, $n\geq 0$ is relatively independent over the factor corresponding to the $\sigma$-algebra $\mc{F}=\gamma^{-1}\mc{B}$ (recall Definitions \ref{def:factorcoup} and  \ref{def:indoverfact}). 

We will obtain $\ns$ by first introducing a sequence of measure space homomorphisms $\gamma_i:\Omega\to\ns_i$ for increasingly finer topological spaces $\ns_i$, $i\in\mb{N}$, and then letting $\ns$ be the inverse limit of the spaces $\ns_i$. By ``increasingly finer", we mean that for every pair of natural numbers $i\geq j$ there is a surjective continuous maps $\pi_{i,j}:\ns_i\rightarrow\ns_j$, such that for every $i\geq j\geq k$ we have $\pi_{j,k}\co\pi_{i,j}=\pi_{i,k}$ everywhere on $\ns_i$, and we have $\pi_{i,j}\co\gamma_i=\gamma_j$ $\lambda$-almost-surely on $\Omega$.

Recall from Definition \ref{def:loc} the notion of localization of a coupling. To use the tools related to that notion, in this subsection we assume that $\varOmega$ is a Borel probability space.

\begin{defn}[Topological factors of a Borel cubic coupling]\label{def:couptopofactor}
Let $\varOmega=(\Omega,\mc{A},\lambda)$ be a Borel probability space, and let $\big(\varOmega,(\mu^{\db{n}})_{n\geq 0}\big)$ be a cubic coupling. For each $k\geq 0$, let $\gamma_k:\Omega\to \coup(\varOmega,K_{k+1})$ be (a version of) the $\{0^{k+1}\}$-localization of $\mu^{\db{k+1}}$. We define the topological space $\ns_k=\Supp(\lambda\co\gamma_k^{-1})\subseteq \coup(\varOmega,K_{k+1})$. For each $n\geq 0$ we define the \emph{set of $n$-cubes} on $\ns_k$, denoted by $\cu^n(\ns_k)$, to be the set $\Supp\big(\mu^{\db{n}}\co(\gamma_k^{\db{n}})^{-1}\big)$.
\end{defn}
\noindent Thus $\ns_k$ is a closed subspace of $\coup(\varOmega,K_{k+1})$, so it follows from Proposition \ref{prop:coupspace} that $\ns_k$ is a compact Polish space. Let us add the following justification of the above definition, motivated by the fact that $\gamma_k$ is defined only up to a change on a $\lambda$-null set.
\begin{lemma}\label{lem:Xkdesc}
In Definition \ref{def:couptopofactor} the space $\ns_k$ is well-defined and we can assume that $\gamma_k(\Omega)\subset \ns_k$. Thus we also have $\cu^n(\ns_k)\subset \ns_k^{\db{n}}$.
\end{lemma}
\noindent Thus $\gamma_k(\Omega)$ is a subset of $\ns_k$ of $(\lambda\co\gamma_k^{-1})$-measure 1. Note that $\ns_k$ is the closure of $\gamma_k(\Omega)$.
\begin{proof}
To see that $\ns_k$ is well-defined note that, since $\gamma_k$ is Borel measurable, we can use the fact that if a Borel function $g:\Omega\to \coup(\varOmega,K_{k+1})$ satisfies $g =_\lambda \gamma_k$ then $\Supp(\lambda\co\gamma_k^{-1})=\Supp(\lambda\co g^{-1})$ (this fact follows from the definitions). Hence it suffices to show that $\gamma_k$ can be redefined on some $\lambda$-null set so that $\gamma_k(\Omega)\subset \ns_k$. By definition of the support, the complement of $\ns_k$ is a $\lambda\co\gamma_k^{-1}$-null set, which means that the complement of $\gamma_k^{-1}(\ns_k)$ is a $\lambda$-null set, so we can redefine $\gamma_k$ as desired simply by re-assigning the same single value in $\ns_k$ to every $\omega\in \Omega$ that was mapped outside $\ns_k$ by $\gamma_k$.
\end{proof}
\begin{remark}
As explained in Remark \ref{rem:dualfns}, dual functions from previous works in this area are generalized via the notion of localization of a cubic coupling. Dual functions have been used before to define topologies in related settings; see for instance the definition of intrinsic topologies on systems of order $k$ in \cite[Chapter 13, \S 3.1]{HK}. The latter definition is a posteriori, once structure theorems have been proved for these systems. In contrast to this, here the topologization  occurs at the start of the argument, and in Section \ref{sec:structhm} we then work with the topological space $\ns_k$ to prove that it yields a compact nilspace.
\end{remark}
\noindent A useful fact about the spaces $\ns_k$ is that certain properties holding almost-surely on $\Omega$ translate into properties holding everywhere on $\ns_k$. This works with the following lemma.

\begin{lemma}\label{lem:closedall}
Let $\mc{P}$ be a closed subset of $\coup(\varOmega, K_{n+1})$, and suppose that $\mc{P}\cap \ns_n$ has probability 1 in $\ns_n$ \textup{(}relative to the regular Borel measure $\lambda\co \gamma_n^{-1}$ on $\ns_n$\textup{)}. Then $\ns_n\subset \mc{P}$.
\end{lemma}
\begin{proof}
The set $\ns_n\setminus \mc{P}$ is open in the relative topology on $\ns_n$, and by assumption we have $\lambda\co \gamma_n^{-1}(\ns_n\setminus \mc{P})=0$. But $\lambda\co \gamma_n^{-1}$ is strictly positive, so $\ns_n\setminus \mc{P}$ must be empty.
\end{proof}
\begin{remark}\label{rem:cc-sym}
As a first use of Lemma \ref{lem:closedall}, let us show that every coupling $\nu\in \ns_k$ has the following symmetries: for every  automorphism $\theta\in\aut(\db{k+1})$ that fixes $0^{k+1}$, we have $\nu_\theta=\nu$. Indeed, for any such $\theta$ the set $\mc{P}=\{\nu\in \coup(\varOmega, K_{k+1}):\nu_\theta=\nu\}$ is closed (using that $\nu\mapsto \nu_\theta$ is continuous, and the closed graph theorem).  Moreover, the consistency axiom implies that $\mu^{\db{k+1}}_\theta=\mu^{\db{k+1}}$, and since the disintegration $(\mu^{\db{k+1}}_x)_x$ yielding the $0^{k+1}$-localization is unique up to a $\lambda$-null set, we have $(\mu^{\db{k+1}}_x)_\theta=\mu^{\db{k+1}}_x$ for $\lambda$-almost every $x$, so $\lambda\co\gamma_k^{-1}(\ns_k\cap\mc{P})=1$, so the symmetry follows by Lemma \ref{lem:closedall}.
\end{remark}
\noindent Let us now define the projection maps $\ns_i\to\ns_j$ for $i\geq j$ (this will use Definition \ref{def:subcoupembeds}).

\begin{defn}[Projections between topological factors]\label{def:factorprojs}
We define the \emph{projection} $\pi_{i,j}:\ns_i\to\ns_j$ as follows. For every coupling $\mu\in \ns_i$ we set $\pi_{i,j}(\mu)= \mu_\tau$, where $\tau: K_{j+1}\to K_{i+1}$, $v\mapsto w$ with $w\sbr{n}=v\sbr{n}$ for $n\in [j+1]$ and $w\sbr{n}=0$ otherwise.
\end{defn}
\begin{remark}\label{rem:corner-sym}
By the symmetries of $\mu$ pointed out in Remark \ref{rem:cc-sym}, the subcoupling $\mu_\tau$ in Definition \ref{def:factorprojs} is equal to $\mu_{\theta\co\tau}$ for every other $\theta\in\aut(\db{i+1})$ fixing $0^{i+1}$. In particular, for every other injective morphism $\tau':K_{j+1}\to K_{i+1}$ with image corner rooted at $0^{i+1}$, we have $\mu_{\tau'}=\mu_\tau$. 
\end{remark}
\begin{remark}
Since $\pi_{i,j}\co\gamma_i$ is a $\{0^{j+1}\}$-localization of $\mu^{j+1}$, arguing as in Remark \ref{rem:cc-sym} using the uniqueness of disintegration, we obtain that $\pi_{i,j}\co\gamma_i=\gamma_j$ holds almost everywhere on $\Omega$. Since there are countably many such equations, the set $E=\{\omega\in\Omega: \pi_{i,j}\co\gamma_i(\omega)\neq \gamma_j(\omega)\textrm{ for some }i,j\}$ is $\lambda$-null. Fixing some sequence $(x_i\in \ns_i)_{i\in \mb{N}}$ with $\pi_{i,j}(x_i)=x_j$ for all $i,j$, and changing for each $\omega\in E$, $i\in \mb{N}$ the value $\gamma_i(\omega)$ to $x_i$, we conclude that we can actually ensure also the following convenient property for the system of maps $\gamma_i,\pi_{i,j}$:
\begin{equation}\label{eq:gammapi}
\textrm{for all $i > j$ and every $\omega\in \Omega$, we have } \pi_{i,j}\co\gamma_i(\omega)=\gamma_j(\omega).
\end{equation}
\end{remark}

\begin{lemma}\label{lem:pijctsmorph}
Each map $\pi_{i,j}:\ns_i\to \ns_j$ is continuous, surjective, and preserves cubes.
\end{lemma}
\begin{proof}
To see that $\pi_{i,j}$ is continuous, let $(\mu_n)_{n\in \mb{N}}$ be a sequence converging to $\mu$ in $\ns_i$. Then for every system $F=(f_v)_{v\in K_{j+1}}$ of bounded measurable functions we can extend $F$ to a system $F'=(f_w')_{w\in K_{i+1}}$ (letting for instance $f_w'=f_v$ for $w=\tau(v)\in \tau(K_{j+1})$ and $f_v'=1$ otherwise, where $\tau$ is the map in Definition \ref{def:factorprojs}), and since $\xi(F',\mu_n)\to \xi(F',\mu)$, we have $\xi\big(F,\pi_{i,j}(\mu_n)\big)\to \xi\big(F,\pi_{i,j}(\mu)\big)$, whence continuity follows. 

To see that $\pi_{i,j}$ is surjective, note that the image of $\pi_{i,j}$ is closed (by continuity, and compactness of $\ns_i$), that this image includes $\gamma_j(\Omega)$ (by \eqref{eq:gammapi}), and that $\gamma_j(\Omega)$ is dense in $\ns_j$ (as noted after Lemma \ref{lem:Xkdesc}), so the image of $\pi_{i,j}$ is $\ns_j$.

To see that $\pi_{i,j}$ preserves cubes, we have to show that for every $\q\in\cu^n(\ns_i)$ we have $\pi_{i,j}\co \q\in \cu^n(\ns_j)$. Fix any  open set $U\ni \pi_{i,j}\co \q$. By continuity the preimage $(\pi_{i,j}^{\db{n}})^{-1}(U)$ is open, and it contains $\q$. Then by \eqref{eq:gammapi} we have $\mu^{\db{n}}\co (\gamma_i^{\db{n}})^{-1}\co(\pi_{i,j}^{\db{n}})^{-1}(U)=\mu^{\db{n}}\co (\gamma_j^{\db{n}})^{-1}(U)$, and the left side here is positive since $\q\in\cu^n(\ns_j)$. We have thus shown that every open neighbourhood of $\pi_{i,j}\co\q$ has positive measure $\mu^{\db{n}}\co (\gamma_j^{\db{n}})^{-1}$, so $\pi_{i,j}\co\q\in \cu^n(\ns_j)$ as required.
\end{proof}
\noindent In several cases it would suffice to have that property \eqref{eq:gammapi} holds just almost-surely, rather than pointwise. The following pointwise property, however, is crucial.
\begin{corollary}\label{cor:chainofprojs}
For all $i\geq j \geq k$ we have $\pi_{j,k}\co\pi_{i,j}(x)=\pi_{i,k}(x)$ for every $x\in\ns_i$.
\end{corollary}
\begin{proof}
The uniqueness almost-surely of disintegrations implies that $\pi_{j,k}\co\pi_{i,j}=\pi_{i,k}$ holds  $\lambda \co \gamma_i^{-1}$-almost-surely on $\ns_i$. The continuity of the projection maps implies that this equality holds on a closed set in $\ns_i$. The result then follows by Lemma \ref{lem:closedall}.
\end{proof}
\noindent We can now define the space $\ns$ announced above.
\begin{defn}\label{def:maintopspace}
We define $\ns$ to be the inverse limit of the inverse system of compact polish spaces $(\pi_{i,j}:\ns_i\to\ns_j)_{i\geq j}$. Thus, letting $\wt{\ns}$ denote the product space $\prod_{i=1}^\infty \ns_i$, we have $\ns=\{ (x_i)_{i\in \mb{N}}\in \tilde\ns: \pi_{i,j}(x_i)=x_j\textrm{ for all }i\geq j\}$. The $i$-th coordinate map $\ns\to \ns_i$ defines a surjective map, which we denote by $\pi_i$. Let $\gamma:\Omega\to\wt{\ns}$ denote the measurable map $\omega \mapsto (\gamma_i(\omega))_{i\in \mb{N}}$. By \eqref{eq:gammapi} we have $\gamma(\Omega)\subseteq \ns$. We equip $\ns$ with the Borel probability measure $\lambda\co\gamma^{-1}$. We define the \emph{set of $n$-cubes} on $\ns$, denoted by $\cu^n(\ns)$, by declaring that an element $\q\in\ns^{\db{n}}$ is in $\cu^n(\ns)$ if and only if for every $k\in\mb{N}$ we have $\pi_k\co \q \in \cu^n(\ns_k)$.
\end{defn} 
\noindent The space $\wt{\ns}$ is compact by Tychonoff's theorem, and $\ns$ is a closed subset of $\wt{\ns}$. It follows that $\ns$ is a compact Polish space. Note that for each $x\in \ns$ the element $\pi_i(x)\in \ns_i$ is a coupling in $\coup(\varOmega,K_{i+1})$. We usually refer to such couplings as \emph{corner couplings}.

We describe the $\sigma$-algebra generated by $\gamma_k$ in terms of the Fourier $\sigma$-algebra $\mc{F}_k$.
\begin{lemma}\label{lem:gammaFk}
Let $k\in\mb{N}$, and let $\mc{B}_k$ be the Borel $\sigma$-algebra on $\ns_k$. Then $\gamma_k^{-1}(\mc{B}_k)=_\lambda \mc{F}_k$.
\end{lemma}
\begin{proof}
To prove that $\gamma_k^{-1}(\mc{B}_k)\supset_\lambda \mc{F}_k$ it suffices to show that every $U^{k+1}$-convolution is in $L^\infty(\gamma_k^{-1}(\mc{B}_k))$. Let $F=(f_v)_{v\in K_{k+1}}$ be a system of functions in $L^\infty(\Omega)$. Recall from the paragraph after Definition \ref{def:Udconv} that, since $\varOmega$ is a Borel probability space, for $\lambda$-almost every $x$, letting $\nu$ denote the coupling $\gamma_k(x)$, we have $[F]_{U^{k+1}}(x)=\int_{\Omega^{K_{k+1}}}\prod_{v\in K_{k+1}} f_v\co p_v \ud\nu$. Thus, recalling the function $\xi(\cdot,F)$ on $\coup(\varOmega,K_{k+1})$ from Definition \ref{def:couptop}, we have
\begin{equation}\label{eq:Udconvfact}
[F]_{U^{k+1}}=_\lambda \xi(\cdot,F)\co \gamma_k.
\end{equation}
Since $\xi(\cdot,F)$ is continuous (by definition of $\coup(\varOmega,K_{k+1})$), we have $[F]_{U^{k+1}}\in L^\infty(\gamma_k^{-1}(\mc{B}_k))$ as required. To see the inclusion $\gamma_k^{-1}(\mc{B}_k)\subset_\lambda \mc{F}_k$, note that since every open set $V\subset \ns_k$ can be written as a countable union of finite intersections of sets of the form $\xi(\cdot,F)^{-1}(U)$ (the union can be countable since $\ns_k$ is a strongly Lindel\"of space), for $U$ open in $\mb{C}$, it follows by \eqref{eq:Udconvfact} that $\gamma_k^{-1}(V)$ is in $\mc{F}_k$ up to a $\lambda$-null set, and the inclusion follows.
\end{proof}
\noindent The following lemma explains why it suffices to study the factor $\gamma:\Omega\to\ns$ in order to describe the structure of the cubic coupling $\varOmega$.
\begin{lemma}
For every $n\in\mb{N}$, and every $m\geq n-1$, the coupling $\mu^{\db{n}}$ is relatively independent over the factor $\gamma_m^{-1}(\mc{B}_m)$. In particular $\mu^{\db{n}}$ is relatively independent over the factor generated by $\gamma:\Omega\to\ns$.
\end{lemma}

\noindent This lemma can be viewed as a measure-theoretic analogue of a fact concerning nilspaces, namely that for a nilspace $\ns$ and any factor $\ns_m$ with $m\geq n-1$, if $\q$ is an $n$-cube on $\ns_{n-1}$ then any lift of $\q$ to a map $\db{n}\to \ns_m$ is also a cube on $\ns_m$ (see \cite[Remark 3.2.12]{Cand:Notes1}).

\begin{proof}
By Definition \ref{def:indoverfact}, Lemma \ref{lem:gammaFk}, and the fact that the functions $F\mapsto \xi(\mu^{\db{n}},F)$ are $U^n$-products, it suffices to show that for every system $F=(f_v)_{v\in \db{n}}$ of functions in $L^\infty(\mc{A})$ we have $\langle F\rangle_{U^n}=\big\langle \big(\mb{E}(f_v|\mc{F}_m)\big)_{v\in \db{n}}\big\rangle_{U^n}$. This follows by first decomposing each $f_v$ as $f_v+ \big(\mb{E}(f_v|\mc{F}_m)-f_v\big)$ (note that $\|f_v-\mb{E}(f_v|\mc{F}_m)\|_{U^{m+1}}=0$), then expanding $\langle F\rangle_{U^n}$ into $\big\langle \big(\mb{E}(f_v|\mc{F}_m)\big)_{v\in \db{n}}\big\rangle_{U^n}$ plus other $U^n$-products in each of which some function has zero $U^{m+1}$-seminorm, and then using \eqref{eq:GCS}, \eqref{eq:nestedUd} to see that each such product vanishes. This proves the first sentence in the lemma; the second sentence follows by definition of $\gamma$.
\end{proof}

\subsection{Continuous $U^n$-convolutions}\hfill\medskip\\
Recall from Definition \ref{def:Udconv} that a $U^n$-convolution $[F]_{U^n}$ is a function in $L^\infty(\varOmega)$ that is defined up to a change on a  null set. We now introduce a ``perfected version" of $[F]_{U^n}$ which is a continuous function on $\ns$. 

\begin{defn}\label{def:Unstarconv}
Let $F=(f_v)_{v\in K_n}$ be a system of functions in $L^\infty(\varOmega)$. We denote by $[F]_{U^n}^*$ the function $\ns\to \mb{C}$, $x\mapsto \int_{\Omega^{K_n}} \prod_{v\in K_n} f_v\co p_v\, \ud\nu $, for the coupling $\nu=\pi_{n-1}(x)$.
\end{defn}
\noindent Recalling the function $\xi(\cdot,F)$ from Definition \ref{def:couptop}, we deduce the following result immediately from the definitions.
\begin{lemma}\label{lem:Unstarconv}
Let $F=(f_v)_{v\in K_n}$ be a system of functions in $L^\infty(\varOmega)$. Then we have $[F]_{U^n}^*= \xi(\cdot,F)\co \pi_{n-1}$. In particular $[F]_{U^n}^*$ is continuous.
\end{lemma}
The following lemma is also straightforward.
\begin{lemma}\label{lem:topcor1}
Let $F=(f_v)_{v\in K_n}$ with $f_v\in L^\infty(\varOmega)$ for all $v$. Then $[F]_{U^n}\,=_{\lambda}\, [F]^*_{U^n}\co\gamma$. 
\end{lemma}

\begin{proof}
We have $[F]_{U^n}\,=_{\lambda}\,\xi(\cdot,F)\co \gamma_{n-1}$, by definition of $\gamma_{n-1}$ and the paragraph after Definition \ref{def:Udconv}. By Definition \ref{def:maintopspace} we have $\gamma_{n-1}=\pi_{n-1}\co\gamma$. Hence we have $\lambda$-almost surely $[F]_{U^n}=_\lambda \xi(\cdot,F)\co \gamma_{n-1}=\xi(\cdot,F)\co \pi_{n-1}\co\gamma =[F]^*_{U^n}\co\gamma$.
\end{proof}

\begin{lemma}\label{lem:topcor2}
Let $F=(f_v)_{v\in K_n}$ be a system of functions in $L^\infty(\varOmega)$. If $[F]_{U^n}\,=_{\lambda}\, 0$, then $[F]^*_{U^n}(x)=0$ holds for every $x\in\ns$.
\end{lemma}

\begin{proof}
We use again that $[F]_{U^n}\,=_{\lambda}\,\xi(\cdot,F)\co \gamma_{n-1}$. The assumption then implies that $\ns_{n-1}\cap \supp\big(\xi(\cdot,F)\big)$ is a $\lambda\co \gamma_{n-1}^{-1}$-null open subset of $\ns_{n-1}$. By Lemma \ref{lem:closedall}, this set is empty. Hence, for every $x\in \ns$ we have $[F]_{U^n}^*(x)=\xi(\pi_{n-1}(x),F)=0$.
\end{proof}

\subsection{Topological nilspace factors of \texorpdfstring{$\ns$}{X}}\hfill\medskip\\
In this subsection $\ns$ denotes the compact Polish space associated with a given Borel cubic coupling, and $\pi_k:\ns\to \ns_k$, $k\in \mb{N}$,  denote the associated projections (see Definition \ref{def:maintopspace}). 

Recall from \cite{CamSzeg,Cand:Notes2} that every compact $k$-step nilspace $\nss_k$ can be viewed as a $k$-fold compact abelian bundle with structure groups $\ab_1,\ldots,\ab_k$ being compact abelian groups \cite[Definition 2.1.8 and Proposition 2.1.9]{Cand:Notes2}, with factors $\nss_1,\nss_2,\ldots,\nss_{k-1}$, and with continuous \emph{nilspace factor maps} $\nss_k\to\nss_i$ (these are the canonical projections the definition of which may be recalled from \cite[Lemma 3.2.10]{Cand:Notes1}). We can equip $\nss_k$ with a unique Borel probability measure satisfying certain natural invariance properties, which we call the Haar measure on $\nss_k$ (see \cite[Proposiiton 2.2.5]{Cand:Notes2}), and similarly every cube set $\cu^n(\nss_k)$ and every \emph{rooted} cube set $\cu^n_x(\nss_k)$ can be  equipped with a Haar measure, since these sets are also equipped with compact-abelian-bundle structures (see \cite[Lemma 2.2.17]{Cand:Notes2}).

Our main goal in this subsection is to prove the following result.
\begin{theorem}\label{thm:rmp}
Let $x\in\ns$, and let $n,k\in\mb{N}$. Suppose that the cubespace $\ns_k$ from Definition \ref{def:couptopofactor} is a $k$-step compact nilspace, and that for every $j\in [k-1]$ the map $\pi_{k,j}:\ns_k\to \ns_j$ is equal to the nilspace factor map $\ns_k\to \ns_j$. Then the image of the coupling $\pi_{n-1}(x)\in\coup(\varOmega, K_n)$ under $\gamma_k^{K_n}$ is the Haar measure on $\cu_{\pi_k(x)}^n(\ns_k)$.
\end{theorem}
\noindent To motivate this result, let us record straightaway the following important consequence, which tells us that the continuous $U^n$-convolutions on $\ns$ are functions that factor not just through $\ns_{n-1}$ (this being  given immediately by Definition \ref{def:Unstarconv} and Lemma \ref{lem:Unstarconv}), but also through spaces $\ns_k$ with $k<n-1$. This will play a key role in the proof of the structure theorem in the next section (see for instance Lemma  \ref{lem:convfactor}).
\begin{corollary}\label{cor:factoring}
Suppose that $\ns_k$ together with the cube sets from Definition \ref{def:couptopofactor} is a compact $k$-step nilspace, and let  $F=(f_v)_{v\in K_n}$ be a system of functions in $L^\infty(\Omega,\mc{F}_k,\lambda)$. Then there is a continuous function $f:\ns_k\to \mb{C}$ such that $[F]_{U^n}^*=f\co \pi_k$.
\end{corollary}
\begin{proof}
For each $x\in \ns$ we have by definition $[F]_{U^n}^*(x)=\int_{\Omega^{K_n}}\prod_v f_v\co p_v\ud\nu$ where $\nu$ is the coupling $\pi_{n-1}(x)\in\coup(\varOmega,K_n)$. Since each $f_v$ is $\mc{F}_k$-measurable, by Lemmas  \ref{lem:Doob} and \ref{lem:gammaFk} it follows that there is a Borel measurable function $g_v:\ns_k\to \mb{C}$ such that $f_v\,=_{\lambda}\, g_v\co \gamma_k$. Let $f:\ns_k\to\mb{C}$, $y\mapsto \int_{\cu^n_y(\ns_k)}\prod_v g_v\co p_v\ud\nu_y$ where $\nu_y$ is the Haar measure on the rooted cube set $\cu^n_y(\ns_k)$. By Theorem \ref{thm:rmp} we then have $[F]_{U^n}^*(x)=f(\pi_k(x))$. To see the continuity of $f$, note that by combining Lusin's theorem applied to each $g_v$ with the multilinearity of $(g_v)_{v\in K_n}\mapsto \int_{\cu^n_y(\ns_k)}\prod_{v\in K_n} g_v\co p_v\ud\nu_y$, we obtain that $f$ can be approximated arbitrarily closely in the supremum norm by functions of the form $f':y\mapsto \int_{\cu^n_y(\ns_k)}\prod_v g_v'\co p_v\ud\nu_y$ where each $g_v'$ is continuous. Each such function $f'$ is continuous by \cite[Lemma 2.2.17]{Cand:Notes2}, and the continuity of $f$ follows.
\end{proof}
\noindent For the proof of Theorem \ref{thm:rmp} we use the following concepts from nilspace theory. 
\begin{defn}
Let $\nss$ be a $k$-step compact nilspace, let $\ab_k$ be the $k$-th structure group of $\nss$, and let $\chi$ be a character in $\wh{\ab_k}$. We denote by $W(\chi,\nss)$ the Hilbert space of functions $f\in L^2(\nss)$ that satisfy $f(x+z)=f(x)\chi(z)$ for every $x\in \nss$ and $z\in \ab_k$.
\end{defn}
\noindent These Hilbert spaces were already used in \cite[Definition 2.8]{Szegedy:HFA}. Similar concepts are also used concerning nilmanifolds (e.g.\ the concept of a  \emph{nilcharacter} from \cite[Definition 6.1]{GTZ}).
\begin{lemma}\label{lem:nilspchar}
Let $\nss$ be a $k$-step compact nilspace and let $\chi\in\wh{\ab_k}$. Then there exists $\phi\in W(\chi,\nss)$ such that $|\phi(x)|=1$ for all $x\in \nss$. Furthermore, for every $f\in W(\chi,\nss)$ we have $f(x)=\phi(x)\, h\co \pi_{k-1}(x)$ where $h$ is the function in $L^2(\nss_{k-1})$ such that $h\co \pi_{k-1}= f\overline{\phi}$.
\end{lemma}
\begin{proof}
By \cite[Lemma 2.4.5]{Cand:Notes2} there exists a Borel measurable map $\cs:\nss_{k-1}\to \nss$ satisfying $\pi_{k-1}\co \cs(y)=y$ for every $y\in \nss_{k-1}$. The function $\phi(x)=\chi(x-\cs\co \pi_{k-1}(x))$ has modulus 1 everywhere and is in $W(\chi,\nss)$. Furthermore, for each $f\in W(\chi,\nss)$ we have $f \overline{\phi} (x+z)=f \overline{\phi} (x)$ for every $x\in \nss$, $z\in \ab_k$, so we have indeed $f \overline{\phi} = h\co \pi_{k-1}$ for some $h\in L^2(\nss_{k-1})$.
\end{proof}
\noindent Recall from \cite[(2.9)]{Cand:Notes1} that the degree-$k$ nilspace  structure on an abelian group $\ab$ is denoted by $\mc{D}_k(\ab)$ and defined by declaring its cube sets to be as follows:
\[
\cu^n(\mc{D}_k(\ab))= \{\q:\db{n}\to \ab ~|~ \textrm{for every face map }\phi:\db{k+1}\to \db{n},\; \sigma_{k+1}(\q\co\phi)=0\},
\]
where $\sigma_{k+1}(\q\co\phi)=\sum_{v\in\db{k+1}}(-1)^{|v|} \q\co\phi(v)$. For the purpose of the following result, it is convenient to define $\mc{D}_k(\ab)$ for $k<0$ to be $\{0_{\ab}\}$ with $\cu^n(\mc{D}_k(\ab))=\{0_{\ab}\}^{\db{n}}$.
\begin{lemma}\label{lem:dualker1}
Let $n$ be a non-negative integer, let $k\in \mb{Z}$, let $\ab$ be a compact abelian group and let $\eta:\db{n}\to \wh{\ab}$, $v\mapsto \eta_v$. Then the character $\prod_{v\in\db{n}} \mc{C}^{|v|} \eta_v\co p_v$ on $\ab^{\db{n}}$ annihilates the subgroup $\cu^n(\mc{D}_k(\ab))$ if and only if $\eta \in \cu^n(\mc{D}_{n-k-1}(\wh{\ab}))$.
\end{lemma}
\begin{proof}
First note that for each fixed $n$ the equivalence holds clearly if $k\geq n$ (for then $\cu^n(\mc{D}_k(\ab))=\ab^{\db{n}}$ so $\eta$ must indeed be the $1$-map), and it also holds trivially if $k<0$. Hence we can suppose that $k$ and $n-k-1$ are both non-negative.
 
We first prove the backward implication, arguing by induction on $n$. For $n=0$ the statement is trivial. For $n>0$, we suppose that $\eta \in \cu^n(\mc{D}_{n-k-1}(\wh{\ab}))$, and we have to show that for every $\q\in \cu^n(\mc{D}_k(\ab))$ we have $\prod_{v\in\db{n}} \mc{C}^{|v|} \eta_v(\q(v))=1$. Let $F$ denote the face $\{(v,0):v\in\db{n-1}\}$ and let $w=(0,\dots,0,1)\in\db{n}$. Then the last product equals $\prod_{v\in F}\mc{C}^{|v|}\big(\eta_v\overline{\eta_{v+w}}(\q(v))\big)\; \prod_{v\in F}\mc{C}^{|v|}\big(\eta_{v+w}(\q(v)-\q(v+w))\big)$. The function $v\mapsto \eta_v\overline{\eta_{v+w}}$ is in $\cu^{n-1}(\mc{D}_{n-k-2}(\widehat{\ab}))$, so by induction the product on the left above equals 1 (the assumption in the lemma is satisfied for this product with indices $n-1,k$). The function $v\mapsto \eta_{v+w}$ is in $\cu^{n-1}(\mc{D}_{n-k-1}(\wh{\ab}))$, and the function $v\mapsto \q(v)-\q(v+w)$ is in $\cu^{n-1}(\mc{D}_{k-1}(\ab))$, so the product on the right above equals 1 as well (the conditions hold with indices $n-1,k-1$).

To see the forward implication, suppose that $\eta \not\in \cu^n(\mc{D}_{n-k-1}(\wh{\ab}))$, so there is some $(n-k)$-face $F$ such that $\prod_{v\in F} \mc{C}^{|v|} \eta_v$ is not the principal character in $\wh{\ab}$, and so there exists $z\in \ab$ such that $\prod_{v\in F} \mc{C}^{|v|} \eta_v(z)\neq 1$. Let $\q$ be the map $\db{n}\to \ab$ defined by $\q(v)=z$ for $v\in F$ and $\q(v)=0$ otherwise. Every $(k+1)$-face in $\db{n}$ has intersection with $F$ of dimension at least 1, whence $\q\in \cu^n(\mc{D}_k(\ab))$. By construction we have $\prod_{v\in F} \mc{C}^{|v|} \eta_v(\q(v))\neq 1$, so the character $\prod_{v\in\db{n}} \mc{C}^{|v|} \eta_v\co p_v$ does not annihilate $\cu^n(\mc{D}_k(\ab))$, as required.
\end{proof}
\noindent We record a consequence of Lemma \ref{lem:dualker1} concerning annihilators of \emph{rooted} cube sets. We denote by $\cu^n_0(\mc{D}_k(\ab))$ the cube set $\{\q\in \cu^n(\mc{D}_k(\ab)): \q(0^n)=0_{\ab}\}$.
\begin{corollary}\label{cor:dualker2}
Let $n$ be a non-negative integer, let $k\in \mb{Z}$, let $\ab$ be a compact abelian group and let $\eta:K_n\to \wh{\ab}$, $v\mapsto \eta_v$. Then the character $\prod_{v\in K_n} \mc{C}^{|v|} \eta_v\co p_v$ on $\ab^{\db{n}}$ annihilates $\cu^n_0(\mc{D}_k(\ab))$ if and only if $\eta \in \hom(K_n,\mc{D}_{n-k-1}(\wh{\ab}))$.
\end{corollary}
\noindent Here $\hom(K_n,\mc{D}_{n-k-1}(\wh{\ab}))$ denotes the space of cubespace morphisms $K_n\to \mc{D}_{n-k-1}(\wh{\ab})$ (where $K_n$ is equipped with the cubespace structure induced from $\db{n}\supset K_n$; see \cite[Definition 3.1.1 and \S 3.3.2]{Cand:Notes1} for a discussion of cubespaces and morphisms).
\begin{proof} We extend $\eta$ to a map $\tilde{\eta}$ on $\db{n}$ by setting $\tilde{\eta}_{0^n}=\prod_{v\in K_n}\mc{C}^{|v|}\eta_v\in \wh{\ab}$.

For the backward implication, by Lemma \ref{lem:dualker1} it suffices to prove that the assumption $\eta \in \hom(K_n,\mc{D}_{n-k-1}(\wh{\ab}))$ implies that $\tilde{\eta}\in\cu^n(\mc{D}_{n-k-1}(\wh{\ab}))$. If $k<0$ or $k\geq n$ then this is clear. Suppose then that $0\leq k\leq n-1$ and let $F$ be an $(n-k)$-face in $\db{n}$. If $F\subset K_n$ then we have by our assumption that $\prod_{v\in F} \mc{C}^{|v|}\eta_v$ is the principal character $1\in \wh{\ab}$. If $0^n\in F$ then note that $\prod_{v\in \db{n}\setminus F}\mc{C}^{|v|}\eta_v=1$ since $\db{n}\setminus F$ is a disjoint union of $(n-k)$-faces not containing $0^n$, and then we have $\prod_{v\in F} \mc{C}^{|v|}\tilde \eta_v=\prod_{v\in \db{n}}\mc{C}^{|v|}\tilde\eta_v \prod_{v\in \db{n}\setminus F}\mc{C}^{|v|+1}\eta_v=1$. This proves that $\tilde{\eta}\in\cu^n(\mc{D}_{n-k-1}(\wh{\ab}))$ as required.

To see the forward implication, note that if $\eta \not\in \hom(K_n,\mc{D}_{n-k-1}(\wh{\ab}))$ then $\tilde{\eta}\not\in\cu^n(\mc{D}_{n-k-1}(\wh{\ab}))$ so by Lemma \ref{lem:dualker1} there is $\q\in \cu^n(\mc{D}_k(\ab))$ such that $\prod_{v\in \db{n}}\mc{C}^{|v|}\tilde\eta_v(\q(v))\neq 1$. Let $z=\q(0^n)$, let $\q'$ be the cube in $\cu^n(\mc{D}_k(\ab))$ with constant value $z$, and note that $\prod_{v\in \db{n}}\mc{C}^{|v|}\tilde\eta_v(\q'(v))=1$. Then $\q'':=\q-\q'\in \cu^n_0(\mc{D}_k(\ab))$ and $\prod_{v\in K_n}\mc{C}^{|v|}\eta_v(\q''(v))=\prod_{v\in \db{n}}\mc{C}^{|v|}\tilde\eta_v(\q(v))\neq 1$, so $\prod_{v\in K_n} \mc{C}^{|v|} \eta_v\co p_v$ does not annihilate $\cu^n_0(\mc{D}_k(\ab))$. 
\end{proof}
\noindent From this corollary we deduce the following two facts about functions in modules $W(\chi,\ns_k)$, which we shall use in the proof of Theorem \ref{thm:rmp}.
\begin{lemma}\label{lem:etainhom}
Let $\nss$ be a $k$-step compact nilspace. Let $\eta:K_n\to \wh{\ab_k}$, $v\mapsto \eta_v$ and for each $v\in K_n$ let $\phi_v$ be a function of modulus 1 in $W(\eta_v,\nss)$. If $\eta\in \hom\big(K_n,\mc{D}_{n-k-1}(\wh{\ab_k})\big)$, then on every rooted cube set $\cu^n_{x}(\nss)$ the function $\prod_{v\in K_n} \mc{C}^{|v|} \phi_v\co p_v$ factors through $\pi_{k-1}^{\db{n}}$ .
\end{lemma}
\noindent Here $\pi_{k-1}$ denotes the nilspace factor map $\nss\to\nss_{k-1}$ \cite[Lemma 3.2.10]{Cand:Notes1}.
\begin{proof}
By nilspace theory $\cu^n_{x}(\nss)$ is a compact abelian bundle over $\cu^n_{\pi_{k-1}(x)}(\nss_{k-1})$ with fiber a principal homogeneous space of the compact abelian group $\cu_0^n(\mc{D}_k(\ab_k))$ (see \cite[Lemma 2.1.10]{Cand:Notes2}). The result then follows clearly from Corollary \ref{cor:dualker2}.
\end{proof}
\begin{lemma}\label{lem:char-rep}
Let $\nss$ be a $k$-step nilspace, let $\chi\in \wh{\ab_k}$, and let $\phi$ be an element of modulus 1 in $W(\chi,\nss)$. Then the function $\mc{C}^{k+1}(\phi \co p_{1^{k+1}}): \cu^{k+1}(\nss)\to \mb{C}$ is equal to a function of the form $g\cdot \prod_{v\in \db{k+1}\setminus\{1^{k+1}\}} \mc{C}^{|v|+1}(\phi\co p_v)$, where $g$ factors through $\pi_{k-1}^{\db{k+1}}$.
\end{lemma}
\begin{proof}
By nilspace theory $\cu^{k+1}(\nss)$ is a compact abelian bundle over $\cu^{k+1}(\nss_{k-1})$ with fiber a principal homogeneous space of $\cu^{k+1}(\mc{D}_k(\ab_k))$. It then follows by Lemma \ref{lem:dualker1} that the function $g=\prod_{v\in \db{k+1}} \mc{C}^{|v|} \phi\co p_v$ on $\cu^{k+1}(\ns_k)$ factors through $\pi_{k-1}^{\db{k+1}}$. Since $|\phi|$ is equal to 1 everywhere we clearly have $\mc{C}^{k+1}(\phi \co p_{1^{k+1}})= g\cdot \prod_{v\neq 1^{k+1}} \mc{C}^{|v|+1}(\phi\co p_v)$.
\end{proof} 
The following lemma is purely about the $(k-1)$-step setting.
\begin{lemma}\label{lem:Fk-1-pd}
Let $n>k$, suppose that the space $\big(\ns_{k-1},(\cu^n(\ns_{k-1}))_{n\geq 0}\big)$ from Definition \ref{def:couptopofactor} is a $(k-1)$-step compact nilspace and that for every $x\in \ns$ the image of the measure $\pi_{n-1}(x)$ under $\gamma_{k-1}^{K_n}$ is the Haar measure on $\cu^n_{\pi_{k-1}(x)}(\ns_{k-1})$. Then for every $x\in \ns$,
\begin{equation}\label{eq:Fk-1-pd}
(\mc{F}_{k-1})^{K_n}\; =_{\pi_{n-1}(x)} \; (\mc{F}_{k-1})^{K_n}_{K_{n,\leq k}}.
\end{equation}
\end{lemma}
\begin{proof}
It suffices to check the statement for nilspaces, since the assumptions imply that everything factors through $\gamma_{k-1}$. Thus we just have to prove that if $\nu$ is the Haar measure on the rooted cube set $\cu^n_{\pi_{k-1}(x)}(\ns_{k-1})$ then $(\mc{B}_{k-1})^{K_n} \subset_{\nu} (\mc{B}_{k-1})^{K_n}_{K_{n,\leq k}}$ (the opposite inclusion is clear). 
On the $(k-1)$-step nilspace $\ns_{k-1}$ we have unique completion for $k$-corners, and this implies the result. To see this implication, recall from nilspace theory that the unique completion function is continuous, hence Borel measurable. This implies the fact that for every $k$-face $F\subset K_n$, letting $v$ be the highest vertex in $F$, we have that $(\mc{B}_{k-1})^{K_n}_v$ is generated by $(\mc{B}_{k-1})^{K_n}_{w\in F, w<v}$. Indeed, letting $f$ denote the corner-completion function, this fact is a consequence of the following general result. Let $S\subset A\times B$ be the graph of a continuous function $f:A\to B$ containing the support of $\mu$. Let $Y\subset B$ be a Borel set. We claim that $\big((A\times Y) \Delta (f^{-1}(Y) \times B)\big)\cap S =\emptyset$. Indeed, if a point $(a,b)$ is in $S$ and in $A\times Y$, then it must be of the form $(a,f(a))$, so $a\in f^{-1}(Y)$, so $(a,b)\in f^{-1}(Y)\times B$. Similarly if $(a,b)$ is in $S$ and in $f^{-1}(Y)\times B$, then $b=f(a)\in Y$, so $(a,b)\in A\times Y$. Having proved our claim, it follows that $\mu \big(\big((A\times Y) \Delta (f^{-1}(Y) \times B)\big)\cap S \big)=0$. But this proves that $A\times Y$ is (up to a $\mu$-null set) in the $\sigma$-algebra generated by Borel sets of the form $f^{-1}(Y) \times B$. To obtain the desired fact we apply this general result with $A$ the space of $k$-corners on $\ns_{k-1}$ (with vertices identified with the $w\in F$ with $w < v$), and $B=\ns_{k-1}$. 

Now, applying the above fact iteratively for a sequence of vertices $v$ of decreasing height, we complete the proof of the desired inclusion $(\mc{B}_{k-1})^{K_n} \subset_{\nu} (\mc{B}_{k-1})^{K_n}_{K_{n,\leq k}}$.
\end{proof}

We can now prove the main result of this section.
\begin{proof}[Proof of Theorem \ref{thm:rmp}]
By induction we assume the result for $k-1$ (it holds trivially for  $k=0$). Let $Q_k=\cu^n_{\pi_k(x)}(\ns_k)$, let $\nu_k=\nu_{k,x}$ be the Haar measure on $Q_k$, and let $\mu=\pi_{n-1}(x)$. Thus $\mu\co (\gamma_k^{K_n})^{-1}$ and $\nu_k$ are Borel measures on the compact Polish space $Q_k$, and $\nu_k(Q_k)=1$. We claim that we can assume also $\mu\co (\gamma_k^{K_n})^{-1}(Q_k)=1$. Indeed, consider $\mc{P}=\{\mu=\pi_{n-1}(x)\in \ns_{n-1}: \mu\co (\gamma_k^{K_n})^{-1}=\nu_{k,x}\}$, i.e.\ the subset of $\ns_{n-1}$ where the conclusion of the theorem holds. It is not hard to show that $\mc{P}$ is closed (given a sequence $(\mu_\ell=\pi_{n-1}(x_\ell))_{\ell\in\mb{N}}$ in $\mc{P}$ with $\mu_\ell \to \mu=\pi_{n-1}(x)$, then also $\pi_k(x_\ell)\to \pi_k(x)$, and then $\mu\co (\gamma_k^{K_n})^{-1}=\nu_{k,x}$ is deduced using in particular that Haar measures on sets $\cu^n_y(\ns_k)$ form a continuous system of measures \cite[Lemma 2.2.17]{Cand:Notes2}). Note also that the assumption $\mu\co (\gamma_k^{K_n})^{-1}(Q_k)=1$ holds $(\lambda\co\gamma_{n-1}^{-1})$-almost surely, by $\mu^{\db{n}}\co(\gamma_k^{\db{n}})^{-1}(\cu^n(\ns_k))=1$ and the disintegration of $\mu^{\db{n}}$ given by $\gamma_{n-1}$. Thus, if we prove the conclusion of the theorem under this assumption, then the full theorem follows by Lemma \ref{lem:closedall}. This proves our claim. 

Now to check that $\mu\co (\gamma_k^{K_n})^{-1}=\nu_k$, by the Riesz representation theorem (see \cite[Corollary 7.10.5]{Boga2}) it suffices to prove that for every continuous function $g:Q_k\to \mb{C}$ we have $\int_{Q_k}  g \ud\nu_k = \int_{\Omega^{K_n}} g \co \gamma_k^{K_n} \ud\mu$. Note that $\ns_k$ can be assumed to have \emph{finite rank}. Indeed, by \cite[Theorem 2.7.3]{Cand:Notes2} the original nilspace $\ns_k$ is an inverse limit of a system $\{\varphi_{ij}:\ns_{k,i}\to \ns_{k,j}\}_{i\geq j}$ of compact finite-rank nilspaces $\ns_{k,i}$, where the maps $\varphi_{ij}$ are continuous fibre-surjective nilspace morphisms (in particular they preserve the Haar measures \cite[Corollary 2.2.7]{Cand:Notes2}). Letting $\varphi_i:\ns_k\to \ns_{k,i}$, $i\in \mb{N}$ be the projections, it follows by the Stone--Weierstrass theorem that any given function $g$ as above can be approximated arbitrarily closely in the supremum norm by functions of the form $g'\co \varphi_i^{K_n}:Q_k\to \mb{C}$ for some $i\in \mb{N}$ and some continuous $g':\cu^n_{\varphi_i\co\pi_k(x)}(\ns_{k,i})\to\mb{C}$. Hence it suffices to prove that $\int_{Q_{k,i}}  g' \ud\nu_{k,i} = \int_{\Omega^{K_n}} g'\co(\varphi_i\co \gamma_k)^{K_n} \ud\mu$ for any such function $g'$, where $Q_{k,i}=\cu^n_{\varphi_i\co\pi_k(x)}(\ns_{k,i})$ and $\nu_{k,i}$ is the Haar measure on $Q_{k,i}$. Note that for the Borel $\sigma$-algebra $\mc{B}_{k,i}$ on $\ns_{k,i}$ we have $(\varphi_i\co \gamma_k)^{-1}(\mc{B}_{k,i})\subset_\lambda \mc{F}_k$ by Lemma \ref{lem:gammaFk}. Thus, we assume that $\ns_k$ has finite rank.

Let $C=\{f:\ns_k\to \mb{C}\textrm{ continuous}~|~\exists\,\chi \in \wh{\ab_k},\,f\in W(\chi,\ns_k)\}$. Let $R$ denote the set of rank-1 functions on $Q_k$ of the form $\prod_{v\in K_n} \mc{C}^{|v|} f_v\co p_v$ where $f_v\in C$ for each $v\in K_n$. Note that $R$ is closed under pointwise multiplication, vanishes nowhere, and separates the points of $Q_k$. The separation property can be seen as follows: if $x,y\in \ns_k$ are distinct, then either $\pi_{k-1}(x)\neq \pi_{k-1}(y)$, in which case the separation property can be deduced by induction on $k$ from that on $\ns_{k-1}$; or $\pi_{k-1}(x) = \pi_{k-1}(y)$, in which case there is a character $\chi\in \wh{\ab}_k$ such that $\chi(x-y)\neq 1$, and we can then obtain a continuous function $f\in W(\chi,\ns_k)$ separating $x,y$, using that there is a neighbourhood of $x$ homeomorphic to $U\times \ab_k$ for some open neighbourhood $U$ of $\pi_{k-1}(x)$ (as $\ns_k$ has finite rank, it is a locally trivial $\ab_k$-bundle over $\ns_{k-1}$ \cite[Lemma 2.5.3]{Cand:Notes2}). Given the above properties of $R$, by the Stone--Weierstrass theorem every function $g$ as above is a uniform limit of a sequence of finite sums of functions in $R$. Hence, it suffices to prove the  desired equality of integrals for functions $g\in R$. Applying Lemma \ref{lem:nilspchar} to each factor of such $g$, we deduce that it suffices to prove the following fact: for every map $\eta:K_n\to \wh{\ab_k}$, $v\mapsto \eta_v$ and every choice of functions $\phi_v$ of modulus 1 in $W(\eta_v,\ns_k)$ and $\alpha_v\in L^\infty(\pi_{k,k-1}^{-1}\mc{B}_{k-1})$, $v\in K_n$, we have \vspace{-0.2cm}\\
\begin{equation}\label{eq:rk1charint}
\int_{Q_k} \prod_{v\in K_n} \mc{C}^{|v|} (\alpha_v \phi_v) \co p_v\; \ud\nu_k = \int_{\Omega^{K_n}} \Big(\prod_{v\in K_n}\mc{C}^{|v|}  (\alpha_v \phi_v) \co p_v\Big)\co \gamma_k^{K_n} \; \ud\mu.
\end{equation}
We distinguish two cases, according to whether $\eta$ is in $\hom(K_n,\mc{D}_{n-k-1}(\wh{\ab}_k))$ or not.

If $\eta\not\in\hom(K_n,\mc{D}_{n-k-1}(\wh{\ab}_k))$, then we have by nilspace theory that the left side of \eqref{eq:rk1charint} is 0. Indeed, it follows from the construction of the Haar measure $\nu_k$ that this integral can be evaluated by first integrating over the $k$-th fibre of the bundle $Q_k$ (see \cite[(2.3)]{Cand:Notes2}), which amounts to integrating the character $\prod_{v\in K_n} \mc{C}^{|v|} \eta_v\co p_v$ over $\cu_0^n(\mc{D}_k(\ab_k))$. By Corollary \ref{cor:dualker2}, this character does not annihilate this group, so the integral vanishes. Therefore, we just have to show that the right side of \eqref{eq:rk1charint} vanishes as well. To this end we first note that, by applying Lemma \ref{lem:char-rep} recursively, we may assume that in $\prod_{v\in K_n} \mc{C}^{|v|} (\alpha_v \phi_v)\co p_v$ we have $\phi_v=1$ for $|v|>k$. Indeed, if there is $v$ with $|v|>k$ and $\phi_v\neq 1$, then we take $|v|$ maximal with this property, and then by Lemma \ref{lem:char-rep} (using $\mu\co (\gamma_k^{K_n})^{-1}(Q_k)=1$) we replace $\mc{C}^{k+1}\phi_v\co p_v$ in the product by an appropriate alternating product of maps $\phi_v\co p_w$ over vertices $w$ forming a $(k+1)$-corner under $v$. Note that if $|v|=k+1$ then this produces factors corresponding to the vertex $0^n$, but these cause no problem for the argument since they come out of the integral on the right side of \eqref{eq:rk1charint} as constant factors (products of constants of the form $\phi_v(\pi_k(x))$). Note also that every such step of the process may modify the function $\eta$, but it conserves the property that $\eta\not\in\hom(K_n,\mc{D}_{n-k-1}(\wh{\ab}_k))$. Now the other part of the product, namely $\Big(\prod_{v\in K_n}\mc{C}^{|v|}  \alpha_v \co p_v\Big)\co \gamma_k^{K_n}$, is $(\mc{F}_{k-1})^{K_n}$-measurable, so by Lemma \ref{lem:Fk-1-pd} it is in fact $(\mc{F}_{k-1})^{K_n}_{K_{n,\leq k}}$-measurable. This product is therefore a limit in $L^1(\mu)$ of finite sums of rank-1 functions of the form $ \Big(\prod_{v\in K_{n,\leq k}}\mc{C}^{|v|}\alpha_v' \co p_v\Big)\co \gamma_k^{K_n}$ where $\alpha_v'=1$ for $|v|>k$. Consequently, it suffices to prove that the right side of \eqref{eq:rk1charint} is 0 under the additional assumption that $\alpha_v \phi_v=1$ for $|v|> k$. Since  $\eta\not\in \hom(K_n,\mc{D}_{n-k-1}(\wh{\ab}_k))$, there is $v$ with $|v|\leq k$ such that $\eta_v\neq 1$. Since $\alpha_v\co\gamma_k\in L^{\infty}(\mc{F}_{k-1})$,  we have $\mb{E}\big( (\alpha_v \phi_v)\co\gamma_k|\mc{F}_{k-1})=(\alpha_v\co\gamma_k) \mb{E}(\phi_v\co\gamma_k|\mc{F}_{k-1})=(\alpha_v\co\gamma_k) \mb{E}(\phi_v|\mc{B}_{k-1})\co\gamma_k$ by \eqref{eq:exprel}, and by nilspace theory this is 0 (indeed it follows from \cite[(2.3)]{Cand:Notes2} that $\mb{E}(\phi_v|\mc{B}_{k-1})$ can be evaluated at $x\in \ns_k$ as the integral of $\phi_v$ over the $\ab_k$-orbit containing $x$, and this integral is 0 since it amounts to integrating $\eta_v$ over $\ab_k$). Hence by statement $(ii)$ of Theorem \ref{thm:Fkprops} we have $\|\alpha_v \phi_v\|_{U^k}=0$. By Corollary \ref{cor:simpzero} and Lemma \ref{lem:topcor2}, we conclude that the right side of \eqref{eq:rk1charint} is indeed 0.

If $\eta\in\hom(K_n,\mc{D}_{n-k-1}(\wh{\ab_k}))$, then by Lemma \ref{lem:etainhom} the  function $\prod_{v\in K_n} \mc{C}^{|v|} (\alpha_v \phi_v) \co p_v $ factors through $\pi_{k,k-1}^{K_n}$, so there is $g:\ns_{k-1}^{K_n}\to \mb{C}$ such that 
$\prod_{v\in K_n} \mc{C}^{|v|} (\alpha_v \phi_v)\co p_v$ $=g\co\pi_{k,k-1}^{K_n}$. The desired equality \eqref{eq:rk1charint} becomes $\int_{Q_k}  g\co \pi_{k,k-1}^{K_n} \ud\nu_k = \int_{\Omega^{K_n}} g\co \pi_{k,k-1}^{K_n} \co \gamma_k^{K_n} \ud\mu$, which is equivalent to $\int_{Q_{k-1}}  g \ud\nu_{k-1} = \int_{\Omega^{K_n}} g\co \gamma_{k-1}^{K_n} \ud\mu$, where $\nu_{k-1}$ is the Haar measure on $Q_{k-1}=\pi_{k,k-1}^{K_n}(Q_k)$. The latter equality of integrals holds by induction on $k$.
\end{proof}
\noindent Another consequence of Theorem \ref{thm:rmp} worth recording is the following, which tells us that the map $\gamma_k$ carries each measure $\mu^{\db{n}}$ to the Haar measure on $\cu^n(\ns_k)$. This measure-preserving property is important for the structure theorem in Section \ref{sec:structhm}.
\begin{corollary}\label{cor:unrooted}
Suppose that the space $\big(\ns_k,(\cu^n(\ns_k))_{n\geq 0}\big)$ from Definition \ref{def:couptopofactor} is a compact $k$-step nilspace, and that for every $j\in [k-1]$ the map $\pi_{k,j}:\ns_k\to \ns_j$ is equal to the nilspace factor map $\ns_k\to \ns_j$. Then for each $n\in \mb{N}$ the measure $\mu^{\db{n}} \co (\gamma_k^{\db{n}})^{-1}$ is the Haar measure on $\cu^n(\ns_k)$.
\end{corollary}
\noindent To prove this we use the following lemma, which is also used several times in the next section. Recall from Definition \ref{def:subcoupembeds} the notation $\nu_\phi$ for subcouplings along injective maps.
\begin{lemma}\label{lem:subfacecoup}
For every corner coupling $\pi_n(x)\in \coup(\varOmega, K_{n+1})$ and every $m$-face map $\phi:\db{m}\to \db{n+1}$ with image included in $K_{n+1}$, the subcoupling $\pi_n(x)_\phi$ is $\mu^{\db{m}}$.
\end{lemma}
\begin{proof}
Let $\nu=\pi_n(x)$. First we claim that the conclusion holds for almost every coupling in the disintegration of $\mu^{\db{n+1}}$ relative to $p_{0^n}$. To see this, recall from Corollary \ref{cor:facelocality} that the face $F=\phi(\db{m})$ is a local set in $\mu^{\db{n}}$, which implies (recall Definition \ref{def:loc}) that the $\sigma$-algebras $\mc{A}^{\db{n+1}}_F$ and $\mc{A}^{\db{n+1}}_{0^n}$ are independent in $\mu^{\db{n}}$. Our claim now follows from Lemma \ref{lem:fibmeaspres}, since this gives us that for $\lambda$-almost every $\omega\in \Omega$, the subcoupling of $\mu^{\db{n+1}}_\omega$ along $\phi$ is equal to $\mu^{\db{n+1}}_\phi$, and this in turn is $\mu^{\db{m}}$ by the face consistency axiom in Definition \ref{def:cc-idemp}.

Finally, let us deduce from our claim that the result holds for every $\pi_n(x)$. By Lemma \ref{lem:botclosed} the set of couplings in $\coup(\varOmega, K_{n+1})$ satisfying the conclusion of the lemma is closed, and then by Lemma \ref{lem:closedall} it follows that all couplings $\pi_n(x)$ satisfy this conclusion.
\end{proof}

\begin{proof}[Proof of Corollary \ref{cor:unrooted}]
Let $F$ be an $n$-face of $\db{n+1}$ with $F\subset K_{n+1}$, let $\mu\in \coup(\Omega,K_{n+1})$ be an element of $\ns_n$, consider the image measure $\nu=\mu\co (\gamma_k^{K_{n+1}})^{-1}$ on $\ns_k^{K_{n+1}}$, and let $\phi_F:\db{n}\to\db{n+1}$ be a morphism with image $F$. 

Note that the measure $\mu^{\db{n}} \co (\gamma_k^{\db{n}})^{-1}$ in the lemma is equal to the subcoupling $\nu_{\phi_F}$. Indeed $\nu_ {\phi_F}$ is equal to $\mu_{\phi_F}\co(\gamma_k^{\db{n}})^{-1}$, and by Lemma \ref{lem:subfacecoup} we have $\mu_{\phi_F}=\mu^{\db{n}}$. By Theorem \ref{thm:rmp}, the measure $\nu$ is the Haar measure on $\cu^{n+1}_{\pi_k(x)}(\ns_k)$. Now one can use  nilspace theory to show that $\nu_{\phi_F}$ is the Haar measure on $\cu^n(\ns_k)$ as required. Indeed, this follows from \cite[Lemma 2.2.14]{Cand:Notes2} applied with $P=\db{n+1}$, $P_1=\{0^{n+1}\}$, $P_2=F$, $f:0^{n+1}\mapsto \mu$, using the fact that $P_1,P_2$ form a \emph{good pair} (as per the terminology from \cite{CamSzeg} or \cite[Definition 2.2.13]{Cand:Notes2}). To see this fact, note that if $\q:F\to \mc{D}_k(\ab)$ is a cube (identifying $F$ with $\db{n-1}$) and $w$ is the unique vertex in $F$ with $|w|=1$, then the map $\q'$ defined by $\q'(v)=\q(v)$ for $v\in F$ and $\q'(v)=\q(v+w)-\q(w)$ otherwise, satisfies $\q'(0^{n+1})=0$, and we also have $\q'\in \cu^{n+1}(\mc{D}_k(\ab))$ by \cite[Lemma 3.3.37]{Cand:Notes1}.
\end{proof}

\section{The structure theorem for cubic couplings}\label{sec:structhm}

\noindent In this section we establish the main result of this paper, Theorem \ref{thm:MeasInvThmGenIntro}, which we restate here in slightly more precise form.

\begin{theorem}\label{thm:MeasInvThmGen}
Let $\big(\varOmega,(\mu^{\db{n}})_{n\geq 0}\big)$ be a cubic coupling. Then there is a compact nilspace $\ns$ and a measure-preserving map $\gamma: \Omega\to \ns$ such that for each $n\in \mb{N}$ the map $\gamma^{\db{n}}$ is measure-preserving $\big(\Omega^{\db{n}}, \mu^{\db{n}}\big) \to \big(\ns^{\db{n}}, \mu_{\cu^n(\ns)}\big)$. Furthermore, for each $n$ the coupling $\mu^{\db{n}}$ is relatively independent over the factor generated by $\gamma^{\db{n}}$.
\end{theorem}

\noindent Here the factor generated by $\gamma^{\db{n}}$ is just $\mu^{\db{n}}\co (\gamma^{\db{n}})^{-1}$ on $\ns^{\db{n}}$. As we shall see, one can take the space $\ns$ and the map $\gamma$ to be those given in Definition \ref{def:maintopspace}. More precisely, what we shall prove in this section is the following result.

\begin{theorem}\label{thm:MeasInvThm} Let $\big(\varOmega,(\mu^{\db{n}})_{n\geq 0}\big)$ be a cubic coupling. For each $k\in \mb{N}$, let $\gamma_k$, $\ns_k$ and $\cu^n(\ns_k)$ be as given in Definition \ref{def:couptopofactor}. Then $\ns_k$ is a $k$-step compact nilspace for every $k$, and for each $n\geq 0$ the image of $\mu^{\db{n}}$ under $\gamma_k^{\db{n}}$ is the Haar measure on $\cu^n(\ns_k)$. Moreover $\mu^{\db{k+1}}$ is relatively independent over its factor induced by $\gamma_k^{\db{k+1}}$.
\end{theorem} 
\noindent The last sentence here can be rephrased as saying that $\ns_k$ is the characteristic factor for $\|\cdot\|_{U^{k+1}}$ (in the usual sense, namely that the factor satisfies \cite[Lemma 4.3]{HK}). This rephrasing follows by Lemma \ref{lem:gammaFk} and Theorem \ref{thm:Fkprops} $(ii)$. Thus, Theorem \ref{thm:MeasInvThm} tells us that the characteristic factor for the $U^{k+1}$-seminorm on $\varOmega$ is a compact $k$-step nilspace.
\begin{remark}
Given Theorem \ref{thm:MeasInvThm}, we can deduce that Theorem \ref{thm:MeasInvThmGen} holds with the space $\ns$ and map $\gamma:\Omega\to\ns$ from Definition \ref{def:maintopspace}.
\end{remark}
\noindent In Section \ref{sec:ergapps} we use Theorem \ref{thm:MeasInvThm} to study measure-preserving actions of nilpotent filtered groups. To explain this, let us recall the definition of a filtered group.
\begin{defn}\label{def:fg}
A \emph{filtration} on a group $G$ is a sequence $G_\bullet= (G_i)_{i=0}^\infty$ of subgroups $G = G_0 = G_1 \geq G_2 \geq \cdots$ such that\footnote{By $[G_i,G_j]$ we mean the subgroup of $G$ generated by the commutators $[g,h]\!=\!g^{-1}h^{-1} g h$, $g \in G_i, h\in G_j$.} $[G_i,G_j] \subset G_{i+j}$ for all $i, j \geq 0$. We then call $(G,G_\bullet)$ a \emph{filtered group}. If $G_i=\{\id_G\}$ for some $i$, then the \emph{degree} of the filtration, denoted by $\deg(G_\bullet)$, is the least integer $k$ such that $G_{k+1}=\{\id_G\}$. We then say that $(G,G_\bullet)$ is a filtered group \emph{of degree} $k$. When the condition $G = G_0 = G_1 \geq G_2 \geq \cdots$ is weakened to $G \geq G_0 \geq G_1  \geq \cdots$, we say that $G_\bullet$ is a \emph{prefiltration} (see \cite[\S 6, Remarks]{GTorb}).
\end{defn}
\noindent Recall that for each $n\geq 0$ a filtered group $(G,G_\bullet)$ can be equipped with the group of $n$-cubes, or Host--Kra cubes of dimension $n$, which we denote by $\cu^n(G_\bullet)$ (we recall this definition in more detail in Section \ref{sec:ergapps}; see also \cite[\S 2.2.1]{Cand:Notes1}).

The above-mentioned use of Theorem \ref{thm:MeasInvThm} in Section \ref{sec:ergapps} goes via the following result. This result describes the structure of a cubic coupling when it is equipped with a measure-preserving action by a filtered group. Recall from \cite[\S 2.9]{Cand:Notes2} the notion of the group of (continuous) \emph{translations} on $\ns_k$, group denoted by $\tran(\ns_k)$, which is naturally equipped with a filtration of subgroups denoted by $\tran_i(\ns_k)$, $i\geq 0$. 
\begin{theorem}\label{thm:k-level-erg}
Let $\big(\varOmega,(\mu^{\db{n}})_{n\geq 0}\big)$ be a cubic coupling. Let $(G,G_\bullet)$ be a filtered group such that $\cu^n(G_\bullet)$ acts on $(\Omega^{\db{n}},\mu^{\db{n}})$ by measure-preserving transformations for each $n\geq 0$. Let $\ns_k$, $\gamma_k$ be the nilspace and map from Theorem \ref{thm:MeasInvThm}. Then $\gamma_k$ induces a filtered-group homomorphism $\wh{\gamma_k}:G\to \tran(\ns_k)$ such that for every $g\in G$ we have $\gamma_k\co g\,=_\lambda\,\wh{\gamma_k}(g)\co \gamma_k$.
\end{theorem}
\noindent This result tells us that, given such a group action on the cubic coupling, the map $\gamma_k$ from Theorem \ref{thm:MeasInvThm} is also a factor map in the sense of ergodic theory (see for example \cite[Definition 1.7]{BTZ}).
\begin{proof}
First we define how $\wh{\gamma_k}(g)$ acts on $\coup(\varOmega, K_{k+1})$ for each $g\in G$: for every $x\in \coup(\varOmega, K_{k+1})$, we define $\wh{\gamma_k}(g)(x)\in \coup(\varOmega, K_{k+1})$ as the image of the measure $x$ under $g^{K_{k+1}}$, that is $\wh{\gamma_k}(g): x\mapsto g^{K_{k+1}}_*(x)$. It is clear from the measure-preserving property of $g$ that $\wh{\gamma_k}(g)(x)\in \coup(\varOmega, K_{k+1})$. Moreover, the map $\wh{\gamma_k}(g)$ is continuous, by Lemma \ref{lem:mpcoupact}.

Now we show that the commutativity claimed in the theorem holds, namely that for $\lambda$-almost every $\omega\in \Omega$ we have $g^{K_{k+1}}_*\big(\gamma_k(\omega)\big)=\gamma_k(g\cdot \omega)$.  Recall that by definition $\gamma_k$ is a disintegration of $\mu^{\db{k+1}}$ relative to $p_{0^{k+1}}$. The map $\omega \mapsto g^{K_{k+1}}_*(\gamma_k(g^{-1}\omega))$ is also a disintegration of  $\mu^{\db{k+1}}$ relative to $p_{0^{k+1}}$, indeed it is clearly a disintegration of $g^{\db{k+1}}_*\mu^{\db{k+1}}$, and this measure is equal to $\mu^{\db{k+1}}$. Then, by uniqueness of disintegrations, we have $g^{K_{k+1}}_*\big(\gamma_k(g^{-1}\cdot \omega)\big)=\gamma_k(\omega)$ for $\lambda$-almost every $\omega$, and the commutativity follows.

We now show that if $x\in \ns_k$ then $\wh{\gamma_k}(g)(x)\in\ns_k$. By Lemma \ref{lem:Xkdesc} we know that $\ns_k$ is the closure of $\gamma_k(\Omega)$. It follows from this, and the almost-sure commutativity above, that if $x\in \ns_k$ then there is a sequence $(\omega_i)_{i\in \mb{N}}$ in $\Omega$ such that $\gamma_k(\omega_i)\to x$ in $\coup(\varOmega,K_{k+1})$ as $i\to\infty$ and for each $\omega_i$ we have $\wh{\gamma_k}(g)\big(\gamma_k(\omega_i)\big)=\gamma_k(g\cdot \omega_i)\in \ns_k$. Now, by continuity of $\wh{\gamma_k}(g)$ and the closure of $\ns_k$, we have $\wh{\gamma_k}(g)(x)=\lim_{i\to \infty} \wh{\gamma_k}(g)\big(\gamma_k(\omega_i)\big)\in \ns_k$.

Finally, we show that $\wh{\gamma_k}$ is a filtered-group homomorphism $(G,G_\bullet)\to \tran(\ns_k)$, i.e.\ that $\wh{\gamma_k}$ is a homomorphism $G_i\to \tran_i(\ns_k)$ for each $i$. From the definition and continuity of each map $\wh{\gamma_k}(g)$, we see that $\wh{\gamma_k}$ is a homomorphism from $G$ into the group of homeomorphisms of $\ns_k$. We prove that $\wh{\gamma_k}(g)\in \tran_i(\ns_k)$ for every $g\in G_i$. By definition of $\tran_i(\ns_k)$ (recall \cite[Definition 3.2.27]{Cand:Notes2}), this means proving that for every $i$-codimensional face $F\subset \db{n}$, the map $\wh{\gamma_k}(g)^F$ preserves $\cu^n(\ns_k)$. By assumption $g^F$ preserves the measure $\mu^{\db{n}}$, so by the above commutativity $\wh{\gamma_k}(g)^F$ preserves $\mu^{\db{n}}\co (\gamma_k^{\db{n}})^{-1}$. Thus $\wh{\gamma_k}(g)^F$ is a homeomorphism preserving $\mu^{\db{n}}\co (\gamma_k^{\db{n}})^{-1}$, so it maps the set $\Supp(\mu^{\db{n}}\co (\gamma_k^{\db{n}})^{-1})=\cu^n(\ns_k)$ to itself.
\end{proof}

To prove that $\ns$ with its cubic structure is a nilspace we need the following result.

\begin{lemma}\label{lem:internsproj}
If $i\geq j$ are natural numbers then $\ns_j$ is the $j$-step factor of $\ns_i$ and $\pi_{i,j}:\ns_i\to \ns_j$ is the corresponding nilspace factor map. In particular $\pi_{i,j}$ is a cubespace morphism.
\end{lemma}
\noindent This lemma will be obtained as a consequence of the inductive argument proving Theorem \ref{thm:MeasInvThm}. From now on in this section we assume that both Theorem \ref{thm:MeasInvThm} and Lemma \ref{lem:internsproj} are true for the factors $\ns_i$ for every $0\leq i\leq k-1$. Our goal is to show that Theorem \ref{thm:MeasInvThm} holds with $\gamma_k$ the map from Definition \ref{def:couptopofactor}, and that Lemma \ref{lem:internsproj} holds for $i=k$.

\subsection{Verifying the ergodicity and composition axioms} \hfill \medskip \\
We now check the first two nilspace axioms for $\ns_k$.

\begin{lemma}\label{lem:axioms1-2}
The space $\ns_k$ together with the cube sets from Definition \ref{def:couptopofactor} satisfy the composition and ergodicity axioms.
\end{lemma}

\begin{proof}
To check the composition axiom, let $\phi:\db{m}\to\db{n}$ be a morphism. We have to show that for every $\q\in \Supp\big(\mu^{\db{n}}\co (\gamma_k^{\db{n}})^{-1}\big)$ we have $\q\co\phi \in \Supp\big(\mu^{\db{m}}\co (\gamma_k^{\db{m}})^{-1}\big)$. 

Suppose first that $\phi$ is injective. Let $V=\phi(\db{m})$, let $\psi:\Omega^{\db{m}}\to \Omega^V$ be the bijection that relabels each coordinate $\omega_v$ to $\omega_{\phi(v)}$, and let $\xi:\ns_k^{\db{m}}\to \ns_k^V$ be the similar bijection. By the consistency axiom in Definition \ref{def:cc}, we have $\mu^{\db{m}}=\mu^{\db{n}}_\phi:=\mu^{\db{n}}_V\co\psi$, and $\psi\co(\gamma_k^{\db{m}})^{-1}=(\gamma_k^V)^{-1}\co \xi$, whence $\Supp\big(\mu^{\db{m}}\co (\gamma_k^{\db{m}})^{-1}\big)=\xi^{-1}\big( \Supp\big(\mu^{\db{n}}_V\co(\gamma_k^V)^{-1}\big)\big)$. Hence it suffices to show that $p_V(\q) \in \Supp\big(\mu^{\db{n}}_V\co(\gamma_k^V)^{-1}\big)$. Fix any open set $U\ni p_V(\q)$. Then, since $\mu^{\db{n}}_V\co(\gamma_k^V)^{-1}$ is the image of $\mu^{\db{n}}\co (\gamma_k^{\db{n}})^{-1}$ under $p_V$, and the latter map is continuous (so that $p_V^{-1}(U)$ is an open set containing $\q$), we have $\mu^{\db{n}}_V\co(\gamma_k^V)^{-1}(U)>0$, so $p_V(\q)$ is indeed in $\Supp\big(\mu^{\db{n}}_V\co(\gamma_k^V)^{-1}\big)$.

For the case where $\phi$ is not injective, we first claim that if $\q\in \cu^{n-1}(\ns_k)$ then the map obtained by copying $\q$ on two opposite faces of $\db{n}$ is also a cube. More precisely, letting $\phi:\db{n}\mapsto \db{n-1}$ be the morphism $v\mapsto v|_{[n-1]}$, we have $\q\co\phi \in \ns_k^{\db{n}}$. This claim follows from Lemma \ref{lem:idemp2} applied with $\varOmega=\cu^{n-1}(\ns_k)$, provided the fact that the measure $\mu^{\db{n}}\co(\gamma_k^{\db{n}})^{-1}$ on $\cu^n(\ns_k)$ is an idempotent coupling of two copies of the measure $\mu^{\db{n-1}}\co(\gamma_k^{\db{n-1}})^{-1}$ on $\cu^{n-1}(\ns_k)$. To see this fact, note that by Proposition \ref{prop:cubicfactor} the coupling $\mu^{\db{n}}$ restricted to $\mc{F}_k$ is idempotent, and then Lemma \ref{lem:gammaFk} implies that $\mu^{\db{n}}\co(\gamma_k^{\db{n}})^{-1}$ inherits this idempotence (this is checked in a straightforward way from Definition \ref{def:idemp}, using \eqref{eq:exprel} to relate the case of this definition for $_{\mc{F}_k|}\mu^{\db{n}}$ to the case for $\mu^{\db{n}}\co(\gamma_k^{\db{n}})^{-1}$). This proves our claim. This claim combined with the consistency axiom for automorphisms implies that the composition axiom holds whenever $\phi:\db{n}\to\db{n-1}$ is a projection along a single coordinate. The composition axiom for general morphisms now follows by noting that any such map is a composition of coordinate projections with an injective morphism.

The ergodicity axiom follows readily from the fact that the support of the product measure $(\lambda\co\gamma_k^{-1})\times (\lambda\co\gamma_k^{-1})$ is the Cartesian square of $\Supp(\lambda\co\gamma_k^{-1})=\ns_k$.
\end{proof}	

\subsection{Complete dependence of corner couplings}
\begin{defn}\label{def:DepCoups}
Let $\varOmega=(\Omega,\mc{A},\lambda)$ be a probability space. A coupling $\mu\in \coup\big(\varOmega,S\big)$ is said to be \emph{completely dependent} if for every $v\in S$ we have $\mc{A}^S_v\subset_\mu \mc{A}^S_{S\setminus \{v\}}$.
\end{defn}
\noindent In this subsection we prove the following result towards Theorem \ref{thm:MeasInvThmGen}.  (Recall Definition \ref{def:factorcoup} for the notion of a factor coupling.)
\begin{proposition}\label{prop:meas-thy-compl}
Let $x\in \ns$ and $k\in \mb{N}$. The factor coupling of $\pi_k(x)\in \coup(\varOmega,K_{k+1})$ corresponding to $\mc{F}_k$ is completely dependent.
\end{proposition}
\noindent This proposition can be viewed as a measure-theoretic analogue of  the following property of $k$-step nilspaces, which is a consequence of the uniqueness of completion of $(k+1)$-corners: let $\nss$ be a $k$-step nilspace and for any fixed $y\in\nss$ consider the rooted cube set $\cu^{k+1}_y(\nss)$; then for every cube $\q$ in this set, for every $u\in K_{k+1}$ the value $\q(u)$ is determined by the other values $\q(v)$, $v\in K_{k+1}\setminus\{u\}$ (since the remaining value $\q(0^{k+1})=y$ is fixed). In fact, using this property, one can see that the Haar measure on $\cu^{k+1}_y(\nss)$ is a completely dependent coupling in $\coup(\nss,K_{k+1})$; this is precisely what Proposition \ref{prop:meas-thy-compl} says in this case.

Recall that the height of a simplicial set $S\subset \db{n}$ is $\max_{u\in S} |u|$, and that the degree $d(u)$ of an element $u\in S$ is the greatest height of an element $v\in S$ with $v\geq u$. An element $u\in S$ is \emph{maximal} in $S$ if $|u| = d(u)$.

To prove Proposition \ref{prop:meas-thy-compl}, from now on in this subsection we assume that the complete dependence in question holds for the factor $_{\mc{F}_d|}\pi_d(x)$ for each $d\leq k-1$, and we establish the case $d=k$ by induction. To this end we use the following result.
\begin{lemma}\label{lem:claim1}
Let $x\in\ns$, let $n\in \mb{N}$, and let $\mu$ be the coupling $\pi_{n-1}(x)\in \coup(\varOmega,K_n)$. Let $S_1,S_2\subset \db{n}$ be simplicial sets such that $S_1$ has height at most $k$, let $w\in S_2\setminus S_1$, and let us define the following $\sigma$-algebras on $\Omega^{K_n}$: \; 
$\mc{D}_1=(\mc{F}_k)^{K_n}_{(S_1\cup S_2)\setminus \{0^n,w\}}$, \; 
$\mc{D}_2= (\mc{F}_k)^{K_n}_{S_2\setminus \{0^n,w\}}$. If $(\mc{F}_k)^{K_n}_w\subset_\mu \mc{D}_1$, then $(\mc{F}_k)^{K_n}_w\subset_\mu \mc{D}_2$.
\end{lemma}

\noindent Thus, if we want to cover $(\mc{F}_k)^{K_n}_w$ by a join of $\sigma$-algebras $(\mc{F}_k)^{K_n}_v$, $v\in K_n\setminus \{w\}$, then a simplicial set $S_1$ of height at most $k$ not containing $w$ is always superfluous. 

\begin{proof}
We argue by induction on $|S_1\setminus S_2|$, and we assume that $S_1\not\subset S_2$ (otherwise the result is trivial). Let $u\in S_1\setminus S_2$ be a maximal element, and let $d=d(u)=|u|\leq k$. Note that $u$ is also maximal in $S_1$. Indeed, for every $v\in S_1$ dominating $u$ we have $v\in S_1\setminus S_2$, since $S_2$ is simplicial, so $v=u$ by maximality in $S_1\setminus S_2$.

Let $S_1'=S_1\setminus \{u\}$ and define the $\sigma$-algebra $\mc{D}_1'= (\mc{F}_k)^{K_n}_{(S_1'\cup S_2) \setminus \{0^n,w\}} = (\mc{F}_k)^{K_n}_{(S_1\cup S_2) \setminus \{0^n,u,w\}}$. Since $S_1'$ is still a simplicial set of height at most $k$, it suffices to prove that\vspace{-0.1cm}
\begin{equation}\label{eq:maininclu}
(\mc{F}_k)^{K_n}_w \subset_\mu \mc{D}_1',\vspace{-0.1cm}
\end{equation}
for then by induction on $|S_1\setminus S_2|$ we would have $(\mc{F}_k)^{K_n}_w\subset_\mu \mc{D}_2$, as required. 

To prove \eqref{eq:maininclu}, we shall use the $\sigma$-algebra $
\mc{D}_3=(\mc{F}_k)^{K_n}_{(S_1'\cup S_2)\setminus\{0^n\} } = (\mc{F}_k)^{K_n}_{(S_1\cup S_2)\setminus\{0^n,u\} }$. Note that $\mc{D}_1'\subset \mc{D}_3$. The key fact that we shall use is that $(\mc{F}_k)^{K_n}_u \upmod_{\mu} \mc{D}_3$. By Lemma \ref{lem:condindepequiv}, this fact follows if we show that for every bounded $(\mc{F}_k)^{K_n}_u$-measurable function $f\co p_u$ such that $\mb{E}_\mu(f\co p_u|(\mc{F}_k)^{K_n}_u\wedge_{\mu} \mc{D}_3)=0$, we also have $\mb{E}_\mu(f\co p_u|\mc{D}_3)=0$. To show this, it suffices to prove that $\|f\|_{U^d}=0$. Indeed, if the latter equation holds then, for any bounded $\mc{D}_3$-measurable rank-1 function $h=\prod_{v\in (S_1'\cup S_2)\setminus\{0^n\} } g_v\co p_v$, applying Corollary \ref{cor:simpzero} with $S=S_1\cup S_2$, $r=0^n$, and $F$ the system with $f_v=1$ for $v\in K_n\setminus S$, with $f_v=g_v$ for $v\in S\setminus\{0^n,u\}$, and $f_u=f$, we have that the convolution $[F]_{U^n}$ vanishes $\lambda$-almost-surely, so by Lemma \ref{lem:topcor2} we have $[F]_{U^n}^*(x)=0$, which means that $\mb{E}_\mu\big((f\co p_u)\, h\big)=0$, so $f\co p_u$ is orthogonal in $\mu$ to every such rank-1 function $h$, whence indeed $\mb{E}_\mu(f\co p_u|\mc{D}_3)=0$. To show that $\|f\|_{U^d}=0$, we use the inductive hypothesis stated in the paragraph just before Lemma \ref{lem:claim1}. From this hypothesis we deduce that $(\mc{F}_{d-1})^{K_n}_u\subset_\mu (\mc{F}_{d-1})^{K_n}_{0^n\neq v \leq u}$ (the deduction uses the fact that, by Definition \ref{def:factorprojs} and Remark \ref{rem:corner-sym}, the subcoupling $\mu_{\{v: 0^n\neq v < u\}}$ is isomorphic to $\pi_{d-1}(x)$). Hence $(\mc{F}_{d-1})^{K_n}_u\; \subset \; \Big((\mc{F}_{d-1})^{K_n}_u\wedge\, (\mc{F}_{d-1})^{K_n}_{0^n\neq v < u}\Big) \;\subset \; (\mc{F}_k)^{K_n}_u\wedge \mc{D}_3$. Hence our assumption above that $\mb{E}(f\co p_u|(\mc{F}_k)^{K_n}_u\wedge_{\mu} \mc{D}_3)=0$ implies that $\mb{E}_\lambda(f|\mc{F}_{d-1})=\mb{E}_\mu\big(f\co p_u| (\mc{F}_{d-1})^{K_n}_u\big)=0$, whence by property $(ii)$ in Theorem  \ref{thm:Fkprops} we have $\|f\|_{U^d}=0$ as required.

Having proved that $(\mc{F}_k)_u^{K_n} \upmod_\mu \mc{D}_3$, let us now prove \eqref{eq:maininclu}. It suffices to show that \vspace{-0.1cm}
\begin{equation}\label{containment2}
(\mc{F}_k)_w^{K_n}\;\; \subset_\mu \;\; \mc{D}_1\wedge \mc{D}_3,\; \textrm{ and }\;(\mc{F}_k)_u^{K_n} \wedge \mc{D}_3\;\; \subset_\mu \;\;  \mc{D}_1'.
\end{equation}
Indeed $\mc{D}_1\wedge \mc{D}_3 = \big((\mc{F}_k)_u^{K_n} \vee \mc{D}_1'\big) \wedge \mc{D}_3 $, and by Lemma \ref{lem:modlaw} this equals $\big((\mc{F}_k)_u^{K_n} \wedge \mc{D}_3\big)\vee \mc{D}_1'$ (since $(\mc{F}_k)_u^{K_n} \upmod \mc{D}_3$), so we have indeed that \eqref{eq:maininclu} follows from  \eqref{containment2}.

To see the first inclusion in \eqref{containment2}, note that clearly $(\mc{F}_k)_w^{K_n}\subset \mc{D}_3$ and therefore, since $(\mc{F}_k)_w^{K_n}\subset_\mu \mc{D}_1$ by assumption, the inclusion in question is clear.

To prove the second inclusion in \eqref{containment2}, we show that in fact
\begin{equation}\label{containment4}
(\mc{F}_k)_u^{K_n}\wedge \mc{D}_3\;\; =_\mu \;\; (\mc{F}_{d-1})_u^{K_n} \;\; \subset \;\; \mc{D}_1'.
\end{equation}
To prove the equality in \eqref{containment4} it suffices to show that $(\mc{F}_k)_u^{K_n}\wedge \mc{D}_3\subset_\mu (\mc{F}_{d-1})_u^{K_n}$ (the opposite inclusion was proved above). Suppose for a contradiction that there exists $f\in L^\infty((\mc{F}_k)_u^{K_n}\wedge \mc{D}_3)$ that is not $(\mc{F}_{d-1})_u^{K_n}$-measurable. Then $g:=f-\mb{E}\big(f|(\mc{F}_{d-1})_u^{K_n}\big)$ is non-zero, and is still $(\mc{F}_k)_u^{K_n}\wedge \mc{D}_3$-measurable (by the opposite inclusion). In particular $g$ is $(\mc{F}_k)_u^{K_n}$-measurable, so $g=g'\co p_u$ almost surely, for some $\mc{F}_k$-measurable $g'$. Then since $\mb{E}_\mu\big(g|(\mc{F}_{d-1})_u^{K_n}\big)=0$, we have $\mb{E}_\lambda(g'|\mc{F}_{d-1})=0$, so by statement $(ii)$ in Theorem \ref{thm:Fkprops} we have $\|g'\|_{U^d}=0$. This implies, by the argument above using Corollary \ref{cor:simpzero}, that $\mb{E}(g|\mc{D}_3)=0$, and so (since $g$ is also $\mc{D}_3$-measurable) we must have $g=0$, a contradiction. 

To see the inclusion in \eqref{containment4}, we use our induction hypothesis for $d-1$, as we did above, to obtain that $(\mc{F}_{d-1})_u^{K_n} \subset_\mu  (\mc{F}_{d-1})^{K_n}_{0^n\neq v < u}$, and then note that the latter $\sigma$-algebra is included in $\mc{D}_1'$, since $\mc{F}_{d-1}\subset \mc{F}_k$ and $\{v: 0\neq v < u\}\subset (S_1\cup S_2)\setminus \{0,u,w\}$.
\end{proof}
\noindent We can now move on to the proof of Proposition \ref{prop:meas-thy-compl}. Our goal is to show that for every $x\in \ns$ and $u\in K_{k+1}$ we have \vspace{-0.2cm}
\begin{equation}\label{eq:measthyglue}
(\mc{F}_k)_u^{K_{k+1}} \subset_{\pi_k(x)} (\mc{F}_k)_{K_{k+1}\setminus\{u\}}^{K_{k+1}}.\vspace{-0.1cm}
\end{equation}
\noindent To prove this we shall use the notation $q_{k+1}$, recalled in Subsection \ref{subsec:tricubes}, for the embedding of the tricube $T_{k+1}$ as the simplicial subcubespace $\wt{T}_{k+1}$ of $\db{2k+2}$.

Consider the following simplicial sets in $\db{2k+2}$, which are subsets of $\wt{T}_{k+1}$:\vspace{-0.1cm}
\begin{eqnarray*}
&& V_1= \{ v\in \db{2k+2}: v\sbr{i}=0\textrm{ for }i\in [k+1]\},\\
&& V_2=\{v\in \db{2k+2}: v\sbr{i+k+1}=0\textrm{ for }i\in [k+1]\},\\
&& V_3=\{ v\in \wt{T}_{k+1}: v\sbr{2k+2}=0,\;|v|\leq k \}, \\
&& V_4=\{ v\in \wt{T}_{k+1}: v\sbr{2k+2}=0\}. \vspace{-0.1cm}
\end{eqnarray*}
The set $V_1$ is the image under $q_{k+1}$ of the subcube of $T_{k+1}$ with all coordinates non-positive, and $V_2$ corresponds similarly to the subcube with all coordinates non-negative. The set $V_3$ corresponds to the set of $v\in T_{k+1}$ having at most $k$ non-zero entries and having last entry either 0 or 1 (but not $-1$). Finally $V_4$ corresponds to the set of $v\in T_{k+1}$ having last entry either 0 or 1.

Note that every $v\in \wt{T}_{k+1}$ has $|v|\leq k+1$. Note also that $V_2\cup V_3\subset V_4$, and that in $V_4$ there may be elements $v$ with $|v|=k+1$, but these elements must then have $v\sbr{k+1}=1$. It is also clear that all sets $V_i$, $i\in [4]$, are simplicial (they are all defined by monotone decreasing properties).

Now since $u$ is assumed to lie in $K_{k+1}$ and therefore $u\neq 0^{k+1}$, we can suppose without loss of generality that $u\sbr{k+1}=1$. Let $w\in V_1$ be the element  such that $w\sbr{i}=0$ and $w\sbr{i+k+1}=u\sbr{i}$ for $i\in [k+1]$. Let us now define some auxiliary $\sigma$-algebras. Here $0$ will stand for the element $0^{2k+2}\in \db{2k+2}$.
\begin{defn}
We define the following $\sigma$-algebras: let $\mc{G}= (\mc{F}_k)^{K_{2k+2}}_{V_1\setminus\{w,0\}}$, and for $i\in [4]$ let $\mc{G}_i=(\mc{F}_k)^{K_{2k+2}}_{V_i\setminus\{0\}}$.
\end{defn}
\begin{proof}[Proof of Proposition \ref{prop:meas-thy-compl}]
Let $\nu$ denote the coupling $\pi_{2k+1}(x)\in \coup(\varOmega, K_{2k+2})$. We begin by noting that to obtain \eqref{eq:measthyglue} it suffices to prove that $(\mc{F}_k)_w^{K_{2k+2}} \subset_\nu \mc{G}$ (this follows from Definition \ref{def:factorprojs} and Corollary \ref{cor:chainofprojs}). To prove this we first show that 
\begin{equation}\label{eq:Gcontainement1}
 (\mc{F}_k)^{K_{2k+2}}_w\subset_\nu \mc{G} \vee \mc{G}_4.
\end{equation}
Consider the injective morphism $\phi: \db{k+1}\to \db{2k+2},\;\;v\mapsto \begin{psmallmatrix} 0 & \cdots & 0 & 1-v\sbr{k+1}\\[0.1em] v\sbr{1} & \cdots & v\sbr{k} & v\sbr{k+1}\end{psmallmatrix}$. Note that $\phi(u)=w$ (since $u\sbr{k+1}=1$). Let $K=\db{k+1}\setminus\{u\}$. 

For every $v\in K$ we have $\phi(v)\in V_1\cup V_4$. Indeed, \emph{either} $\phi(v)$ has coordinates $\phi(v)\sbr{i}=0$ for all $i\in [k+1]$, in which case $\phi(v)\in V_1$, \emph{or} $\phi(v)\sbr{k+1}=1$ and then $\phi(v)\sbr{2k+2}=0$ whence $\phi(v)\in V_4$. Since we also have that $\phi(\db{k+1})\subset K_{2k+2}$, it follows that for every $v\in K$ we have $(\mc{F}_k)^{K_{2k+2}}_{\phi(v)}\subset \mc{G}\vee\mc{G}_4$. To prove \eqref{eq:Gcontainement1} it now suffices to prove the inclusion $(\mc{F}_k)^{K_{2k+2}}_{\phi(u)}\subset_\nu (\mc{F}_k)^{K_{2k+2}}_{\phi(K)}$. By Lemma \ref{lem:subfacecoup} the subcoupling of $\nu$ along $\phi$ is equal to $\mu^{\db{k+1}}$, and the desired inclusion then follows by statement $(iii)$ in Theorem \ref{thm:Fkprops}.

Having proved \eqref{eq:Gcontainement1}, the next step is to prove that 
\begin{equation}\label{eq:Gcontainement2}
\mc{G}_4= \mc{G}_2\vee\mc{G}_3.
\end{equation}
As noted above, we have $V_2\cup V_3\subset V_4$, whence $\mc{G}_4 \supset \mc{G}_2\vee \mc{G}_3$. For each $v\in \db{2k+2}$ let us define $h^*(v)=\sum_{i\in [k+1]}v\sbr{i+k+1}$. We prove by induction on $h^*(v)$ that if $v\in V_4$ then $(\mc{F}_k)^{K_{2k+2}}_v\subset \mc{G}_2\vee\mc{G}_3$. If $h^*(v)=0$ then $v\in V_2$ and the statement is trivial. Suppose then that for $n\in [k+1]$ the statement holds for every $v\in V_4$ with $h^*(v)\leq n-1$, and suppose that $b\in V_4$ satisfies $h^*(b)=n$. If $|b|\leq k$ then $b\in V_3$ and the statement is trivial, so we can assume that $|b|=k+1$. This means that $b\sbr{i}+b\sbr{i+k+1}=1$ for every $i\in [k+1]$. Consider the injective morphism $\psi:\db{k+1}\to V_4$ defined by
\[
\psi(v)\sbr{i}=v\sbr{i}\, b\sbr{i}+(1-v\sbr{i})\,b\sbr{i+k+1}~,~\psi(v)\sbr{i+k+1}=v\sbr{i}\, b\sbr{i+k+1}.
\]
Note that $\psi(1^{k+1})=b$. Moreover, since $n>0$, we have $b\sbr{i+k+1}=1$ for some $i\in [k+1]$, so $\binom{\psi(v)\sbr{i}}{\psi(v)\sbr{i+k+1}}\neq \binom{0}{0}$ for all $v$, whence $\psi(\db{k+1})\subset K_{2k+2}$.

Let $K'=\db{k+1}\setminus\{1^{k+1}\}$. We claim that for every $v\in K'$, either $h^*(\psi(v))<n$ or $\psi(v)\in V_3$. Indeed, if $\psi(v)\notin V_3$ then since $\psi(v)\in V_4$ we must have $|\psi(v)|=k+1=|b|$; moreover since $\psi$ is injective we have $\binom{\psi(v)\sbr{j}}{\psi(v)\sbr{j+k+1}}\neq \binom{b\sbr{j}}{b\sbr{j+k+1}}$ for some $j$, and since by definition of $\psi$ we have $\psi(v)\sbr{j+k+1}\leq b\sbr{j+k+1}$, the only possibility is $\binom{\psi(v)\sbr{j}}{\psi(v)\sbr{j+k+1}}=\binom{1}{0}$, $\binom{b\sbr{j}}{b\sbr{j+k+1}}=\binom{0}{1}$, whence $h^*(\psi(v))<h^*(b)=n$, as claimed. Now, by this claim and the induction hypothesis, for every $v\in K'$ we have $(\mc{F}_k)^{K_{2k+2}}_{\psi(v)}\subset \mc{G}_2\vee\mc{G}_3$. By Lemma  \ref{lem:subfacecoup} again we have $\nu_\psi=\mu^{\db{k+1}}$, so by property $(iii)$ of Theorem \ref{thm:Fkprops} and the fact that $b=\psi(1^{k+1})$ we have $(\mc{F}_k)_b^{K_{2k+2}}\subset (\mc{F}_k)_{\psi(K')}^{K_{2k+2}}$. Since $(\mc{F}_k)_{\psi(K')}^{K_{2k+2}}\subset \mc{G}_2\vee\mc{G}_3$, we deduce \eqref{eq:Gcontainement2}.

By \eqref{eq:Gcontainement1} and \eqref{eq:Gcontainement2} we have $(\mc{F}_k)^{K_{2k+2}}_w\subset\mc{G}\vee\mc{G}_2\vee\mc{G}_3$. By Lemma \ref{lem:claim1}, we can omit $\mc{G}_3$, so $(\mc{F}_k)^{K_{2k+2}}_w\subset\mc{G}\vee\mc{G}_2$. Now $\mc{G}_1=(\mc{F}_k)_w^{K_{2k+2}}\vee \mc{G}$, and $(\mc{F}_k)_w^{K_{2k+2}}\subset (\mc{G}_2\vee\mc{G})\wedge \mc{G}_1$. Moreover, since $V_1~\bot_{\mu^{\db{2k+2}}}~ V_2$ (by Theorem \ref{thm:sip}), it follows from combining Lemma \ref{lem:fibrebotindep} with Lemma \ref{lem:indepclosed} and  Lemma \ref{lem:closedall} that $V_1\setminus\{0\}~ \bot_\nu ~ V_2\setminus\{0\}$. In particular $\mc{G}_1$, $\mc{G}_2$ are independent (since $V_1\setminus\{0\}$, $V_2\setminus\{0\}$ are disjoint). By Lemma \ref{lem:modlaw} we have $(\mc{G}_2\vee \mc{G})\wedge \mc{G}_1=(\mc{G}_2\wedge\mc{G}_1)\vee\mc{G}$. The independence of $\mc{G}_1$, $\mc{G}_2$ implies that $\mc{G}_2\wedge\mc{G}_1$ is the trivial $\sigma$-algebra (up to null sets, as usual). We thus obtain $(\mc{F}_k)_w^{K_{2k+2}}\subset_\nu (\mc{G}_2\vee\mc{G})\wedge \mc{G}_1=(\mc{G}_2\wedge\mc{G}_1)\vee\mc{G}=\mc{G}$, as desired.
\end{proof}

\subsection{Convolution neighbourhoods}\hfill \medskip\\
In the proof of the corner-completion axiom in Subsection \ref{subsec:corcomp}, a crucial role will be played by certain special open sets in $\ns_k$. Let $F=\{f_v:\Omega\to \{0,1\}\}_{v\in K_{k+1}}$ be a system of measurable indicator functions. Recall that the function $\xi(\cdot,F):\coup(\varOmega,K_{k+1})\to [0,1]$ is continuous (by Definitions \ref{def:multilin} and \ref{def:couptop}). In particular, the set $U=\ns_k\cap\supp\big(\xi(\cdot,F)\big)$ is open in $\ns_k$, and since $[F]_{U^{k+1}}^*=\xi(\cdot,F)\co \pi_k$ we have $\supp([F]_{U^{k+1}}^*)=\pi_k^{-1}(U)$.
\begin{defn}
We say that an open set $U\subset \ns_k$ is a \emph{convolutional set} if there is a system $F$ of indicator functions of measurable subsets of $\Omega$ such that $U=\ns_k\cap\supp\big(\xi(\cdot,F)\big)=\pi_k\big(\supp([F]_{U^{k+1}}^*)\big)$. We then say that $F$ is a system \emph{generating} $U$. Given a point $x\in \ns_k$, we call a convolutional set $U$ containing $x$ a \emph{convolution neighbourhood} of $x$.
\end{defn}
\noindent Note that if $F$ is a system generating $U$ then $x\in U$ if and only if $\xi(x,F)>0$. 

The main result in this subsection, Proposition \ref{prop:opcontconv}, tells us that convolution neighbourhoods form a basis for the topology on $\ns_k$. This fact will be crucial for the proof of the corner-completion axiom (specifically, in the proof of Lemma \ref{lem:nkcorncomp}). To obtain this fact we first prove the following result.
\begin{proposition}\label{prop:uclos}
Let $\varrho:K_{k+1}\to \ns_k$ be some function. Then there is at most one element $z\in \ns_k$ with the following property: for every system of open sets $U(v)\ni\varrho(v)$ in $\ns_k$,  $v\in K_{k+1}$, we have $z \in \pi_k\big(\overline{\supp([F]_{U^{k+1}}^*)}\big)$, where $F=(1_{U(v)}\co \gamma_k)_{v\in K_{k+1}}$.
\end{proposition}
\noindent Here $\overline{\supp([F]_{U^{k+1}}^*)}$ denotes the closure of $\supp([F]_{U^{k+1}}^*)$. Proposition \ref{prop:uclos} has the following important consequence.
\begin{corollary}\label{cor:uniquecomp}
Let $\q'$ be a $(k+1)$-corner on $\ns_k$. Then there is at most one $(k+1)$-cube on $\ns_k$ completing $\q'$.
\end{corollary}
\begin{proof}
Without loss of generality (using a discrete-cube automorphism sending $1^{k+1}$ to $0^{k+1}$) we can suppose that $\q'$ is a map $K_{k+1}\to \ns_k$. Let $x\in \ns_k$ be an element yielding a completion $\q$ of $\q'$, that is, such that the map $\q:\db{k+1}\to\ns_k$ with $\q|_{K_{k+1}}=\q'$ and $\q(0^{k+1})=x$ is in $\cu^{k+1}(\ns_k)$. By Proposition \ref{prop:uclos}, it suffices to prove that for every system $(U\sbr{v})_{v\in K_{k+1}}$ of open sets $U\sbr{v}\ni \q'(v)$, for $F=(1_{U\sbr{v}}\co\gamma_k)_{v\in K_{k+1}}$ we have $x\subset \pi_k\big(\overline{\supp([F]_{U^{k+1}}^*)}\big)$. Let $V$ be the complement of $\pi_k\big(\overline{\supp([F]_{U^{k+1}}^*)}\big)$ (the latter set is closed by the continuity of $\pi_k$ and the closed map lemma, so $V$ is open). Suppose for a contradiction that $x\in V$. We have $\int_{\ns}[F]_{U^{k+1}}^* \, (1_V\co \pi_k) \,\ud(\lambda\co\gamma^{-1})=0$ by definition of $V$. On the other hand, this integral equals the $\mu^{\db{k+1}}\co (\gamma_k^{\db{k+1}})^{-1}$-measure of the product set $\prod_{v\in \db{k+1}} U\sbr{v}$ where $U\sbr{0^{k+1}}=V$. This product is an open set containing $\q\in \Supp\big(\mu^{\db{k+1}}\co (\gamma_k^{\db{k+1}})^{-1}\big)$, so its measure must be positive, a contradiction.
\end{proof}

\noindent To prove Proposition \ref{prop:uclos}, suppose that there exists such an element $z=\pi_k(x)\in \ns_k$ and consider the factor-coupling of $\pi_k(x)$ corresponding to $\mc{F}_k$. Then it suffices to show that this coupling is uniquely determined by the factor couplings $_{\mc{F}_k|}\varrho(v)$, $v\in K_{k+1}$. Indeed, note that every $(k+1)$-corner coupling $\nu$ (in particular the element $\pi_k(x)$) is uniquely determined by its factor $_{\mc{F}_k|}\nu$. This follows from the fact that for every system $F=(f_v)_{v\in K_{k+1}}$ of bounded measurable functions we have $\xi(F,\nu)=\xi\big((\mb{E}(f_v|\mc{F}_k))_{v\in K_{k+1}},\nu\big)$, as can be shown by applying Lemma \ref{lem:simpzero} and Lemma \ref{lem:topcor2}.

To prove that $_{\mc{F}_k|}\pi_k(x)$ is uniquely determined by the couplings $_{\mc{F}_k|}\varrho(v)$, $v\in K_{k+1}$, we shall use a tricube structure to  construct a large coupling $\Upsilon$ in which all the couplings $\varrho(v)$ are included as subcouplings in a useful interrelated manner.

\medskip

\noindent{\bf Construction of the coupling $\Upsilon$.}\\
For every $v\in K_{k+1}$ let us choose a decreasing sequence $(U_i(v))_{i\in\mb{N}}$ of open neighbourhoods of $\varrho(v)$ which forms a neighbourhood basis in $\ns_k$. Let $F_i=(1_{U_i(v)}\co \gamma_k)_{v\in K_{k+1}}$ and let $(x_i)_{i\in \mb{N}}$ be a sequence in $\ns$ such that $\lim_{i\to\infty}\pi_k(x_i)=z$ and $x_i\in \supp([F_i]_{U^{k+1}}^*)\subset \ns$. (This sequence exists by the assumptions on $z$.) Again we work with the embedded tricube $\wt{T}_{k+1}=q_{k+1}(T_{k+1})$ in $\db{2k+2}$. Let $V=\{v\in \wt{T}_{k+1}: v\sbr{i} = 0~{\rm for}~i\in [k+1]\}$ be the set $V_1$ from the previous subsection, i.e.\ the subset corresponding to $\{-1,0\}^{k+1}\subset T_{k+1}$. Let $t\in \db{2k+2}$ be the element with $t\sbr{i}=0$, $t\sbr{i+k+1}=1$ for $i\in [k+1]$, i.e.\ $t=q_{k+1}(-1^{k+1})$. Let $\tau:\db{k+1}\to \wt{T}_{k+1}$ be the injective morphism with image $V$ defined by $\tau(v)\sbr{i}=0$ and $\tau(v)\sbr{i+k+1}=1-v\sbr{i}$ for each $i\in [k+1]$. In particular $\tau(0^{k+1})=t$ and $\tau(1^{k+1})=0^{2k+2}$.

We now consider the corner coupling $\pi_{2k+1}(x_i)\in \coup(\varOmega, K_{2k+2})$, and for convenience we take it to be rooted at the vertex $t\in \db{2k+2}$ defined above, rather than at the usual vertex $0^{2k+2}$ (these two versions of the coupling are isomorphic, by the symmetry of $\mu^{\db{2k+2}}$ given by the consistency axiom). Let $M_i=\{\omega\in\Omega^{\db{2k+2}\setminus\{t\}}\,:\,\gamma_k(p_{\tau(v)}(\omega))\in U_i(v), \,\forall v\in K_{k+1}\}$. Observe that the measure of $M_i$ in the coupling $\pi_{2k+1}(x_i)$ is equal to $[F_i]_{U^{k+1}}^*(x_i)$, which is positive by our assumption on $x_i$. By Corollary \ref{cor:facelocality}, we have that $V$ is local in $\mu^{\db{2k+2}}$. By Lemma \ref{lem:embloc} applied to $\{t\}$ and $V$, we have that the set $V\setminus\{t\}$ is local in $\pi_{2k+1}(x_i)$.  Let $\Upsilon_i'$ be the coupling $\pi_{2k+1}(x_i)$ conditioned with respect to $M_i$, as per Definition \ref{def:condcoup}. As explained in that definition, we have that $\Upsilon_i'\in\coup(\varOmega,\db{2k+2}\setminus V)$. Let $\Upsilon_i$ be the factor coupling of $\Upsilon_i'$ corresponding to $\mc{F}_k\subset\mc{A}$. We can now define the coupling $\Upsilon$.
\begin{defn}
We define $\Upsilon$ as the limit of some convergent subsequence of $(\Upsilon_i)_{i\in \mb{N}}$ in the compact space $\coup\big((\Omega,\mc{F}_k,\lambda), \db{2k+2}\setminus V\big)$.
\end{defn}
\noindent Now that we have the coupling $\Upsilon$, the next step is to prove that it satisfies the properties that we announced, namely that it includes the corner couplings associated with $\varrho$ in a suitable interdependent way. To do so, for each $w\in \db{k+1}$ we shall use a discrete-cube morphism $\phi_w:\db{k+1}\to \db{2k+2}$ that sends $0^{k+1}$ to some point of $V$ and all other points of $\db{k+1}$ to $\wt{T}_{k+1}\setminus V$.

\begin{defn}\label{def:phiw}
For each $w\in\db{k+1}$ let $\phi_w:\db{k+1}\to \wt{T}_{k+1}$ be the injective map defined by $\phi_w(v)\sbr{i}=v\sbr{i}$ and $\phi_w(v)\sbr{i+k+1}=(1-w\sbr{i})(1-v\sbr{i})$, for $i\in [k+1]$.
\end{defn}
\noindent Thus for every $w$ the map $\phi_w$ is a cube morphism $\db{k+1}\to \db{2k+2}$. Note that the point $u:=\phi_w(0^{k+1})$ is equal to $q_{k+1}(w-1^{k+1})=\begin{psmallmatrix} 0 & \cdots & 0\\[0.1em] 1-w\sbr{1} & \cdots & 1-w\sbr{k+1} \end{psmallmatrix}$, in particular $u$ is indeed in $V=q_{k+1}(\{-1,0\}^{k+1})$. We also have for every $w$ that $\phi_w(1^{k+1})=\begin{psmallmatrix} 1 & \cdots & 1\\[0.1em] 0 & \cdots & 0 \end{psmallmatrix}$, which corresponds to $1^{k+1}\in T_{k+1}$ and is not in $V$. More generally, we have $\phi_w(v)\not\in V$ for all $v\neq 0^{k+1}$, since for any such $v$ there is $i\in [k+1]$ with $\phi_w(v)\sbr{i}=v\sbr{i}=1$, whereas for $\phi_w(v)$ to be in $V$ requires $\phi_w(v)\sbr{i}=0$ for all $i\in [k+1]$.

In what follows we shall argue by induction on the height $|v|$. To that end, for $j\in [k+1]$ we denote by $B_j$ the set $\db{2k+2}_{\leq j} := \{v\in \db{2k+2}: |v|\leq j\}$, and we define
\begin{equation}\label{eq:dj}
D_j=B_j\cap \wt{T}_{k+1}.
\end{equation}
One can think of $D_j$ in $\mb{R}^{k+1}$ as the intersection of the tricube $T_{k+1}=\{-1,0,1\}^{k+1}$ with the closed $\ell^1$-ball of radius $j$ centered at the origin.

Recall from Definition \ref{def:indoverfact} the notion of a relatively independent coupling.
\begin{lemma}\label{lem:Ups2props}
The coupling $\Upsilon$ has the following properties.
\begin{enumerate}[leftmargin=0.7cm]
\item The subcoupling of $\Upsilon$ along the bijection $\phi_w|_{K_{k+1}}:K_{k+1}\to \wt{T}_{k+1}\setminus V$ is equal to $_{\mc{F}_k|}\varrho(w)$ for every $w\in K_{k+1}$, and is equal to $_{\mc{F}_k|}\pi_k(x)$ for $w=0^{k+1}$.
\item Let $\Upsilon^j$ denote the subcoupling of $\Upsilon$ along the set  $B_j\setminus V\subset \db{2k+2}\setminus V$. Then $\Upsilon^j$ is relatively independent over its factor corresponding to $\mc{F}_{j-1}$, for all $j\in [k+1]$.
\end{enumerate}
\end{lemma}
\begin{proof}
Throughout this proof we denote by $\nu_i$ the coupling $\pi_{2k+1}(x_i)$ (rooted at $t$ as mentioned above). To prove property $(i)$ we will show that for every $\epsilon>0$, for all $i$ sufficiently large the coupling $\Upsilon_i$ satisfies the property up to $\epsilon$. More precisely, we  first metrize the Polish space $\coup\big((\Omega,\mc{F}_k,\lambda), \db{2k+2}\setminus V\big)$ so that balls in the metric are convex (see Proposition \ref{prop:coupspaceApp}), and then we use this to show that for every $\epsilon>0$, if $i$ is sufficiently large then for every $w\in K_{k+1}$ the subcoupling of $\Upsilon_i$ indexed by $\phi_w(K_{k+1})$ is $\epsilon$-close in this metric to $_{\mc{F}_k|}\varrho(w)$, and is equal to $_{\mc{F}_k|}\pi_k(x)$ for $w=0^{k+1}$. 

To prove this in the case $w\in K_{k+1}$, it suffices to prove that the subcoupling of $\Upsilon_i$ indexed by $\phi_w(K_{k+1})$ is a convex combination of couplings all of which are $\epsilon$-close to $_{\mc{F}_k|}\varrho(w)$, as this implies that $\Upsilon_i$ also has this property (since the $\epsilon$-ball centered on $_{\mc{F}_k|}\varrho(w)$ is convex in the chosen metric). We prove this by considering an integral of an arbitrary rank-1 function for this subcoupling of $\Upsilon_i$. Since these integrals characterize uniquely this subcoupling, considering just these integrals will suffice. So, for each $v\in \phi_w(K_{k+1})$, let $f_v$ be a function in $L^\infty(\mc{F}_k)$. Recall from the definition of $\Upsilon_i$ that we take $M_i\subset \Omega^{\db{2k+2}\setminus\{t\}}$ to be the cylinder-set $\bigcap_{v\in K_{k+1}}p_{\tau(v)}^{-1}\gamma_k^{-1}(U_i(v))$, and that  $\nu_i(M_i)>0$. The integral of the rank-1 function that we have just fixed is then
\begin{equation}\label{eq:upsint}
\int_{\Omega^{\phi_w(K_{k+1})}} \prod_{v\in \phi_w(K_{k+1})} f_v\co p_v\; \ud\Upsilon_i.
\end{equation}
By Definition \ref{def:condcoup}, this is (abusing the notation in $U_i(v)$ by identifying $V$ with $\db{k+1}$)
\[
 \tfrac{1}{\nu_i(M_i)} \int_{\Omega^{\db{2k+2}\setminus\{t\}}} 1_{U_i(w)}\co\gamma_k\co p_u \; \prod_{v\in V\setminus \{t,u\}} 1_{U_i(v)}\co\gamma_k\co p_v \;  \prod_{v\in \phi_w(K_{k+1})} f_v\co p_v\; \ud\nu_i,
\]
where we recall that $u=\phi_w(0^{k+1})$. Let $g=\mb{E}\Big(\prod_{v\in V\setminus \{t,u\}} 1_{U_i(v)}\co\gamma_k\co p_v\,|\,\mc{A}^{\db{2k+2}}_{\phi_w(\db{k+1})}\Big)$. The integral above is then equal to $\tfrac{1}{\nu_i(M_i)}\int_{\Omega^{\db{2k+2}\setminus\{t\}}} 1_{U_i(w)}\co\gamma_k\co p_u \; g \;  \prod_{v\in \phi_w(K_{k+1})} f_v\co p_v\; \ud\nu_i$.

We claim that $V\setminus \{t\} ~\bot_{\nu_i}~ \phi_w(\db{k+1})$. This can be seen by the following argument similar to one used in the proof of Lemma \ref{lem:keybot}. Letting $\mu$ denote the tricube coupling $\mu_{\wt{T}_{k+1}}^{\db{2k+2}}$, we have $V~\bot_\mu~ \phi_w(\db{k+1})$, because by the symmetry of this coupling under the action of $S_3^{k+1}$ (established in Lemma  \ref{lem:trisym}), we can apply an element $\sigma\in S_3^{k+1}$ that leaves $V$ globally invariant but turns $\phi_w(\db{k+1})$ into the set $q_{k+1}(\db{k+1})$, which is now a \emph{face} in $\db{2k+2}$; in $\mu$ the faces $V$ and $q_{k+1}(\db{k+1})$ are conditionally independent because they are so in $\mu^{\db{2k+2}}$ by Theorem \ref{thm:sip} (here we use Remark \ref{rem:botinsubcoup}). Since this conditional independence is not affected by this action of $S_3^n$, we conclude indeed that $V~\bot_\mu~ \phi_w(\db{k+1})$, and so $V~\bot_{\mu^{\db{2k+2}}}~ \phi_w(\db{k+1})$ (again by Remark \ref{rem:botinsubcoup}). Combining lemmas \ref{lem:fibrebotindep},  \ref{lem:botclosed} (for which we use Lemma \ref{lem:subfacecoup}), and  \ref{lem:closedall}, we obtain our claim. 

Noting that $(V\setminus \{t\}) \cap \phi_w(\db{k+1})=\{u\}$, we see that, by the above claim, the function $g$ is $\mc{A}_u^{\db{2k+2}\setminus\{t\}}$-measurable, and is therefore equal almost everywhere to $g'_u\co p_u$ for some $[0,1]$-valued measurable function $g'_u$ (by Lemma \ref{lem:Doob}). It follows that the last integral equals $\tfrac{1}{\nu_i(M_i)}
\int_{\Omega^{\db{2k+2}\setminus\{t\}}} (1_{\gamma_k^{-1} U_i(w)} \, g'_u)\co p_u \; \prod_{v\in \phi_w(K_{k+1})} f_v\co p_v\; \ud\nu_i$. But now this is an integral of a function depending only on components indexed by $\phi_w(\db{k+1})$ so, letting $\nu_{i,w}$ denote the subcoupling of $\nu_i$ indexed by $\phi_w(\db{k+1})$, this integral is equal to $\tfrac{1}{\nu_i(M_i)}
\int_{\Omega^{\phi_w(\db{k+1})}} (1_{\gamma_k^{-1} U_i(w)} \, g'_u)\co p_u \;   \prod_{v\in \phi_w(K_{k+1})} f_v\co p_v\; \ud\nu_{i,w}$. 
Now note that by Lemma \ref{lem:subfacecoup} (combined with the $S_3^n$ symmetry again), we have that $\nu_{i,w}\cong\mu^{\db{k+1}}$. We now disintegrate $\nu_{i,w}$ relative to the map $p_u:\Omega^{\phi_w(\db{k+1})} \to \Omega$, into measures $\nu_x$ for $x\in \Omega$ which are isomorphic to corner couplings in $\ns_k$ for almost every $x$. We thus conclude that the integral in \eqref{eq:upsint}  equals
\[
\int_{\Omega} \frac{(1_{\gamma_k^{-1} U_i(w)} \, g'_u) (x)}{\nu_i(M_i)} \;\Big( \int_{\Omega^{\phi_w(K_{k+1})}}  \prod_{v\in \phi_w(K_{k+1})} f_v\co p_v \ud\nu_x\Big)\; \ud\lambda(x).
\]
Note that this is indeed a convex combination of the kind we claimed. Indeed, the weight function $x\mapsto \frac{(1_{\pi_k^{-1} U_i(w)} \, g'_u) (x)}{\nu_i(M_i)}$ is a non-negative function having integral over $\Omega$ equal to 1. Now by the presence of the indicator function of $U_i(w)$ in the weight function, whenever $\nu_x$ is a corner coupling (which is the case for almost every $x$), we have that $\nu_x$ is $\epsilon$-close to $_{\mc{F}_k|}\varrho(w)$, by definition of $U_i(w)$. This completes what we needed to prove for $w\in K_{k+1}$.

The case $w=0^{k+1}$ is simpler. Here note first that by an argument similar to the one above, we have that $\mc{A}^{\wt{T}_{k+1}\setminus\{t\}}_{\phi_{0^{k+1}}(K_{k+1})}$ and $\mc{A}^{\wt{T}_{k+1}\setminus\{t\}}_{V\setminus \{t\}}$ are independent in $\nu_i$. This then implies the result, since the conditional coupling $\Upsilon_i$ is then equal to $_{\mc{F}_k|}\pi_k(x)$ (this is seen by applying the formula $\mu'(N):=\mu(M\times N)/\mu_T(M)$ in Definition \ref{def:condcoup} in this case, where independence yields the factorization $\mu(M\times N)=\mu_T(M)\pi_k(x)(N)$).

To see property $(ii)$, let $\Upsilon^j_i$ denote the subcoupling of $\Upsilon_i$ indexed by $B_j\setminus V$. If we show for every $i$ that $\Upsilon^j_i$ is relatively independent over its $\mc{F}_{j-1}$-factor, then  we obtain property $(ii)$ by passing to the limit as $i\to \infty$ and using that the desired relative independence is a closed property, by Lemma \ref{lem:relindclosed}. Let $F=(f_v)_{v\in B_j\setminus V}$ be a system of functions $f_v\in L^\infty(\mc{F}_k)$, and suppose that for some $v'\in B_j\setminus V$ we have $\mb{E}(f_{v'}|\mc{F}_{j-1})=0$, so that $\|f_{v'}\|_{U^j}=0$ by statement $(ii)$ of Theorem \ref{thm:Fkprops}. We have to show that $\xi(\Upsilon^j_i,F)=0$. Let $G=\{g_v\}_{v\in \db{2k+2}\setminus\{t\}}$ be the extended function system defined as follows. If $v\in B_j\setminus V$ then $g_v=f_v$, if $v\in \db{2k+2}\setminus (B_j\cup V)$ then $g_v=1$ and if $v\in V\setminus\{t\}$ then $g_v$ is the characteristic function of $U_i(\tau^{-1}(v))$. By definition we have that $\xi(\Upsilon^j_i,F)=[G]_{U^{2k+2}}^*(x_i)$ where the convolution here is rooted at $t$. Using that $V\cup B_j$ is simplicial and that $d(v')\leq j$, we obtain using Corollary \ref{cor:simpzero} and Lemma \ref{lem:topcor2} that $[G]_{U^{2k+2}}^*(x_i)=0$, and the result follows.
\end{proof}

\noindent The next lemma, which is the main one in this step of the proof, establishes that the coupling $\Upsilon$ is the only one that has the properties in Lemma \ref{lem:Ups2props}.

\begin{lemma}\label{lem:Upsunique}
There is at most one coupling $\theta\in \coup\big((\Omega,\mc{F}_k,\lambda), \wt{T}_{k+1}\setminus V\big)$ satisfying the following two properties:
\begin{enumerate}[leftmargin=0.7cm]
\item $\forall\,w\in K_{k+1}$, the subcoupling of $\theta$ along $\phi_w:K_{k+1}\to \wt{T}_{k+1}\setminus V$ is equal to $_{\mc{F}_k|}\varrho(w)$.
\item Let $\theta^j$ be the subcoupling of $\theta$ indexed by $D_j\setminus V\subset \wt{T}_{k+1}\setminus V$. Then $\theta^j$ is relatively independent over its factor corresponding to $\mc{F}_{j-1}$, for every $j\in [k+1]$.
\end{enumerate}
\end{lemma}
\begin{proof}
Suppose that $\theta$ is such a coupling. We prove by induction on $j$ that by properties $(i)$ and $(ii)$ in Lemma \ref{lem:Upsunique} the subcoupling $\theta^j$ is uniquely determined. For $j=1$, by property $(ii)$ the coupling $\theta^j$ is the independent coupling with index $D_1\setminus V$, which is indeed unique. Let $j>1$ and suppose that the uniqueness holds for $j-1$. Property $(ii)$ implies that it suffices to prove the uniqueness of the $\mc{F}_{j-1}$-factor coupling of $\theta^j$ (this suffices indeed since, as mentioned after Definition \ref{def:indoverfact}, by $(ii)$ the maps $\xi$ for this $\mc{F}_{j-1}$-factor uniquely determine those for $\theta^j$). We prove this by another inductive argument.

Recall from the previous subsection that $h^*(v)=\sum_{i\in [k+1]} v\sbr{i+k+1}$ for $v\in \db{2k+2}$, and recall from \eqref{eq:dj} that $D_j=B_j\cap \wt{T}_{k+1}$. Let $D_{j,n}=D_{j-1}\cup\{v\in D_j:h^*(v)\leq n\}$. We shall prove the following statement by induction on $n=h^*(v)$:
\begin{eqnarray}\label{eq:subinduct}
&&\textrm{The }\mc{F}_{j-1}\textrm{-factor of the subcoupling }\theta^j_{D_{j,n}}\textrm{ of }\theta^j\textrm{ is uniquely determined} \\
&&\textrm{by the }\mc{F}_{j-1}\textrm{-factor of the subcoupling }\theta^j_{D_{j,n-1}}.\nonumber
\end{eqnarray}
We start by establishing the base case $n=0$ by using the case $j-1$ of the global induction. Suppose that $v\in \wt{T}_{k+1}$ satisfies $v\in D_j\setminus D_{j-1}$ and $h^*(v)=0$. In particular $|v|=j$ and every $v'\in \db{2k+2}$ with $v'< v$ is in $D_{j-1}$. (Note that $\{v\in \wt{T}_{k+1}: h^*(v)=0\}$ is the $(k+1)$-face $V_2$ from the previous subsection.) The set $P=\{v'\in \db{2k+2} : 0\neq v'\leq v\}$ can be identified with $K_j$. From the assumed property $(i)$ for $\theta$, applied with $w=1^{k+1}$, and the composition rule Corollary \ref{cor:chainofprojs}, we have that the subcoupling of $\theta$ indexed by $P$ is isomorphic to $\pi_{j-1}(\varrho(1^{k+1}))$, and then the  factor $_{\mc{F}_{j-1}|}\pi_{j-1}(\varrho(1^{k+1}))$ is just the subcoupling of $\theta^{j-1}$ indexed by $P$. (Note that $\varrho(1^{k+1})$ corresponds to $0^{2k+2}$ in $\db{2k+2}$.) By Proposition \ref{prop:meas-thy-compl} the subcoupling $_{\mc{F}_{j-1}|}\pi_{j-1}(\varrho(1^{k+1}))$ is completely dependent. Hence $(\mc{F}_{j-1})_v^{\wt{T}_{k+1}}\subset _\theta (\mc{F}_{j-1})_{P\setminus\{v\}}^{\wt{T}_{k+1}}$, and this is included in $(\mc{F}_{j-1})^{\wt{T}_{k+1}}_{D_{j-1}}$ by the above remarks. Now, applying Lemma  \ref{lem:coupdeterm} with $T_1=D_{j-1}$ and $T_2=P$ we deduce that the subcoupling of $\theta^j$ indexed by $D_{j-1}\cup \{v\}$ is uniquely determined. Applying this argument recursively for each vertex  $v$ with $h^*(v)=0$ (with $T_1$ including each time all the vertices $v$ from  the previous steps in the recursion), the base case follows.

Now suppose that $n>0$, that $_{\mc{F}_{j-1}|}\theta^j_{D_{j,n-1}}$ is uniquely determined, and that $h^*(v)=n$ (and $v\in D_j\setminus D_{j-1}$ as above). Let $v',w\in\db{k+1}$ be defined by $v'\sbr{i}=v\sbr{i}+v\sbr{i+k+1}$ and $w\sbr{i}=1-v\sbr{i+k+1}$. Let $R = \big\{r\in\db{k+1} : 0^{k+1}\neq r\leq v'\big\}$, and consider the restriction to $R$ of the map $\phi_w$ from Definition \ref{def:phiw}. We claim that if $r\in R$ satisfies $\phi_w(r)\neq v$, then either $\phi_w(r) \in D_{j-1}$, or $h^*(\phi_w(r))\leq n-1$. Indeed, if $\phi_w(r) \not\in D_{j-1}$ then $|\phi_w(r)|\geq j$, and by definition of $R$
\[
\{i\in [k+1]: \phi_w(r)\sbr{i}+\phi_w(r)\sbr{i+k+1}=1\}\;\subset\; \{i\in [k+1]: v\sbr{i}+v\sbr{i+k+1}=1\}
\]
(indeed if $v\sbr{i}+v\sbr{i+k+1}=0$ then $\phi_w(r)\sbr{i+k+1}=v\sbr{i+k+1}(1-r\sbr{i})=0$ and $r\leq v'$ implies that $\phi_w(r)\sbr{i}=r\sbr{i}=0$). Considering the sizes of the sets in the above inclusion, we deduce that these sets are equal; then for $\phi_w(r)\neq v$ to hold there must exist $i$ in this set such that $\binom{\phi_w(r)\sbr{i}}{\phi_w(r)\sbr{i+k+1}}=\binom{1-v\sbr{i}}{1-v\sbr{i+k+1}}$. But if it were the case that $\phi_w(r)\sbr{i}=\phi_w(r)\sbr{i}=0\neq 1=v\sbr{i}$, then by definition of $\phi_w(r)$ we would have $\phi_w(r)\sbr{i+k+1}=v\sbr{i+k+1}$, which contradicts the equality $\phi_w(r)\sbr{i+k+1}=1-v\sbr{i+k+1}$ established in the previous sentence. Hence we must have $\phi_w(r)\sbr{i}=z\sbr{i}=1$ and then $\phi_w(r)\sbr{i+k+1}=0$, whence indeed $h^*(\phi_w(r))< h^*(v)=n$. This proves our claim. This claim implies that $\phi_w(R)\setminus\{v\}\subset D_{j,n-1}$. Now, by property $(i)$ for $\theta$ we have that $\theta_{\phi_w(R)}= \pi_{j-1}(\varrho(w))$, so by Proposition  \ref{prop:meas-thy-compl} we have $(\mc{F}_{j-1})_v^{\wt{T}_{k+1}}\subset _\theta (\mc{F}_{j-1})_{\phi_w(R)\setminus\{v\}}^{\wt{T}_{k+1}}$. Applying Lemma \ref{lem:coupdeterm} with $T_1=D_{j,n-1}$ and $T_2=\phi_w(R)$, we deduce that the subcoupling of $\theta^j$ indexed by $D_{j,n-1}\cup \{v\}$ is uniquely determined. Applying this argument again recursively for each such $v$ (similarly as in the base case), we deduce that the subcoupling of $_{\mc{F}_{j-1}|}\theta^j$ indexed by $D_{j,n}$ is uniquely determined. This completes the induction on $n$.
\end{proof}
We can now complete the proof of the main result in this subsection.
\begin{proof}[Proof of Proposition \ref{prop:uclos}]
Let $\wt{\Upsilon}$ be the subcoupling of $\Upsilon$ indexed by $\wt{T}_{k+1}\setminus V$. Lemmas \ref{lem:Ups2props}, \ref{lem:Upsunique} imply that $\wt{\Upsilon}$ is uniquely determined by the function $\varrho$. As explained after Corollary \ref{cor:uniquecomp}, the coupling $z=\pi_k(x)$ is uniquely determined by its factor $_{\mc{F}_k|}z$. Since the latter is a sub-coupling of $\wt{\Upsilon}$, we deduce that $z$ is uniquely determined, so the proof is complete.
\end{proof}

\noindent We can now use this to establish the fact announced at the beginning of this subsection, namely that convolutional sets form a basis for the topology on $\ns_k$.

\begin{proposition}\label{prop:opcontconv} Let $U$ be an open set in $\ns_k$ and let $x\in U$. Then there are open sets $U\sbr{v}\subset \ns_k$, $v\in K_{k+1}$ such that $x\in \pi_k\big(\supp([F]^*_{U^{k+1}})\big)\subset U$, where $F=(1_{U(v)}\co\gamma_k)_{v\in K_{k+1}}$.
\end{proposition}

\begin{proof}
Recall that $x$ is a measure $\nu\in \coup(\varOmega, K_{k+1})$. Let $\varrho:K_{k+1}\to \ns_k$ be an element of $\Supp(\nu\co (\gamma_k^{K_{k+1}})^{-1})$. By definition of this support, for every system $F=(1_{U(v)}\co\gamma_k)_{v\in K_{k+1}}$ with open sets $U(v)\ni\varrho(v)$ for each $v\in K_{k+1}$, we have $[F]_{U^{k+1}}^*(x)>0$. Therefore it suffices to prove that there is some such system $F$ satisfying also $\pi_k\big(\supp([F]^*_{U^{k+1}})\big)\subset U$. To show this, for every $i\in\mb{N}$ let $(U_i(v))_{v\in K_{k+1}}$ be a system of open sets such that $(U_i(v))_{i\in \mb{N}}$ is a nested decreasing neighbourhood basis for $\varrho(v)$ for each $v$. Let $F_i=(1_{U_i(v)}\co\gamma_k)_{v\in K_{k+1}}$. Suppose for a contradiction that $(\ns_k\setminus U)\cap \pi_k\big(\supp([F_i]_{U^{k+1}}^*)\big)\neq\emptyset$ for every $i$. The sets $B_i=(\ns_k\setminus U)\cap\pi_k\big(\overline{\supp([F_i]_{U^{k+1}}^*)}\big)$, $i\in \mb{N}$, form a decreasing nested sequence of closed sets in the compact space $\ns_k$, so there exists $y\in\bigcap_{i=1}^\infty B_i$. By Proposition \ref{prop:uclos}, this implies that $x=y$, which contradicts the fact that $x$ and $y$ are separated by the open set $U$. 
\end{proof}

\subsection{Verifying the corner-completion axiom}\label{subsec:corcomp} \hfill \medskip \\
For the proof of the completion axiom we shall use the following lemma concerning supports of conditional expectations.
\begin{lemma}\label{lem:supports}
Let $(\Omega,\mc{A},\lambda)$ be a probability space, and let $\mc{B}$ be a sub-$\sigma$-algebra of $\mc{A}$. Then for any non-negative function $f\in L^\infty(\mc{A})$, we have $\supp(f)\subset_\lambda \supp\big(\mb{E}(f|\mc{B})\big)$.
\end{lemma}
\begin{proof}
Let $B=\Omega\setminus \supp\big(\mb{E}(f|\mc{B})\big)$. The desired conclusion  is equivalent to $\int_\Omega f \cdot 1_B \ud\lambda=0$. Since $B$ is $\mc{B}$-measurable, we have $\int_\Omega f 1_B \ud\lambda = \int \mb{E}(f  |\mc{B}) 1_B  \ud\lambda$, and this last integral is zero by definition of $B$.
\end{proof}

\begin{lemma}\label{lem:intnest}
Let $\mu\in \coup(\varOmega,S)$, let $F=(f_v)_{v\in S}$ and $G=(g_v)_{v\in S}$ be systems of non-negative functions in $L^\infty(\varOmega)$, and suppose that $ \supp f_v\, \subset_\lambda\, \supp g_v$ for every $v\in S$. If $\int_{\Omega^S} \prod_{v\in S} f_v\co p_v \ud\mu >0$ then $\int_{\Omega^S} \prod_{v\in S} g_v\co p_v \ud\mu >0$.
\end{lemma}

\begin{proof}
Recall that for a bounded non-negative function $h$ on a probability space $\varOmega$ we have $\lambda(\supp h)>0$ if and only if $\int h \ud\lambda >0$. Using this,  since $\int \prod_v f_v\co p_v \ud\mu >0$ we have that $\mu\big(\supp(\prod_v f_v\co p_v)\big)>0$, and since $\supp(\prod_v f_v\co p_v)\subset \supp(\prod_v g_v\co p_v)$, we have also $\mu\big(\supp(\prod_v g_v\co p_v)\big)>0$, so (by the previous sentence again) $\int \prod_v g_v\co p_v \ud\mu$ must be positive.
\end{proof}
\noindent We now show that the functions supported on convolution neighbourhoods in $\ns_k$ factor through functions on $\ns_{k-1}$ in a useful way.
\begin{lemma}\label{lem:convfactor}
Let $x\in \ns_k$, let $F$ be a system of indicator functions generating a convolution neighbourhood of $x$. Then there exists a continuous function $g:\ns_{k-1}\to [0,1]$ such that $\mb{E}\big([F]_{U^{k+1}}\big|\mc{F}_{k-1}\big)\, =_\lambda \,g\co \gamma_{k-1}$ and $g\big(\pi_{k-1}(x)\big)>0$.
\end{lemma}
\begin{proof}
We have by assumption that $F=(f_v)_{v\in K_{k+1}}$ is a system of $\mc{A}$-measurable indicator functions on $\Omega$ such that $\xi(x,F)>0$. Let $G=(g_v)_{v\in K_{k+1}}$ be the system of functions $g_v=\mb{E}(f_v|\mc{F}_{k-1})$. Observe that, by Corollary \ref{cor:factoring} and our inductive assumptions on $\ns_{k-1}$, we have that $[G]^*_{U^{k+1}}=g\co \pi_{k-1}$ for some continuous function $g:\ns_{k-1}\to [0,1]$, which means that $g(\pi_{k-1}(x'))=\xi(\pi_k(x'),G)$ for every $x'\in \ns$. Note also that by Lemma \ref{lem:vetites1} we have $[G]_{U^{k+1}} =_\lambda \mb{E}([F]_{U^{k+1}}|\mc{F}_{k-1})$, so by Lemma \ref{lem:topcor1} we have  $g\co \gamma_{k-1}\,=_\lambda\,\mb{E}([F]_{U^{k+1}}|\mc{F}_{k-1})$. To see that $g\big(\pi_{k-1}(x)\big)>0$, note that $g\big(\pi_{k-1}(x)\big)=\xi(x,G)$, and $\supp(g_v)\supset \supp(f_v)$ for each $v$ by Lemma \ref{lem:supports}, so $\xi(x,G)>0$ by Lemma \ref{lem:intnest}.
\end{proof}
\noindent We can now proceed to the proof of the completion axiom. We shall first prove the existence of a completion for any morphism from $\db{n}_{\leq k}$ to $\ns_k$, that is, any map $\varrho:\db{n}_{\leq k}\to \ns_k$ such that for every cube morphism $\phi:\db{k}\to \db{n}$ with image included in $\db{n}_{\leq k}$ we have $\varrho\co\phi\in \cu^k(\ns_k)$.
\begin{lemma}\label{lem:nkcorncomp}
Let $\varrho:\db{n}_{\leq k}\to \ns_k$ be a morphism. Then there exists a cube $\q\in \cu^n(\ns_k)$ such that $\q|_{\db{n}_{\leq k}}=\varrho$.
\end{lemma}
\begin{proof}
It suffices to prove the following claim: let $H=(h_v)_{v\in \db{n}}$ be a  system of functions $h_v\in L^\infty(\mc{A})$ such that for $|v|>k$ we have $h_v=1$ and for $|v|\leq k$ we have $h_v=[H_v]_{U^{k+1}}$ where $H_v$ is a system generating a convolution neighbourhood $U\sbr{v}$ of $\varrho(v)$;  then $\langle H\rangle_{U^n}>0$. To see that this suffices, for each $v\in \db{n}_{\leq k}$ let $(U_i\sbr{v})_{i\in \mb{N}}$ be a decreasing open neighbourhood basis of $\varrho(v)$ in $\ns_k$. Then by Proposition \ref{prop:opcontconv} for each $i$ there is a system $H_i$ of such functions $h_{i,v}$ that generate a convolution neighbourhood of $\varrho(v)$ included in $U_i\sbr{v}$. The claim above implies that for every $i$ there is a cube $\q_i\in \cu^n(\ns_k)$ such that $\q_i(v)\in U_i\sbr{v}$ for each $v\in \db{n}_{\leq k}$. By compactness of $\cu^n(\ns_k)$ there is then a subsequence of $(\q_i)_i$ converging to some $\q\in \cu^n(\ns_k)$, and by construction we then have $\q|_{\db{n}_{\leq k}}=\varrho$.

To prove the claim above, for each $v$ with $|v|\leq k$ let $g_v$ be the function obtained by applying Lemma \ref{lem:convfactor} to $h_v$, and let $G=(g_v\co\gamma_{k-1})_{v\in \db{n}}$ where $g_v=1$ for $|v|>k$. By Corollary \ref{cor:simpzerocor} applied with $S=\db{n}_{\leq k}$ we have $\langle H\rangle_{U^n}=\langle G\rangle_{U^n}$. By our inductive assumptions and the fact that $\ns_{k-1}$ is a compact nilspace and $\pi_{k-1}$ is measure-preserving, we have $\langle G\rangle_{U^n}=\int_{\cu^n(\ns_{k-1})} \prod_v g_v \co p_v \ud\nu$ where $\nu$ is the Haar measure on $\cu^n(\ns_{k-1})$. It follows from Lemma \ref{lem:pijctsmorph} that $\pi_{k-1}\co \varrho$ is also a morphism, and then by nilspace theory (see \cite[Lemma 3.1.5]{Cand:Notes1}) there is a cube $\q\in \cu^n(\ns_{k-1})$ such that $\q|_{\db{n}_{\leq k}}=\pi_{k-1}\co \varrho$. By Lemma \ref{lem:convfactor}, we have $g_v\big(\pi_{k-1}(\varrho(v)\big)>0$ for each $v$, whence $\prod_{v\in \db{n}} g_v\big(\q(v)\big)>0$. This implies $\int_{\cu^n(\ns_{k-1})} \prod_v g_v\co p_v\ud\nu >0$, since the integrand is continuous and positive at $\q$, and $\nu$ is a strictly positive measure (this is seen as in the proof of \cite[Proposition 2.2.11]{Cand:Notes2}). 
\end{proof}
\begin{proposition}
The corner-completion axiom holds on $\ns_k$.
\end{proposition}
\begin{proof}
Let $\q':\db{n}\setminus\{1^n\} \to \ns_k$ be an $n$-corner. Let $\varrho=\q'|_{\db{n}_{\leq k}}$. By Lemma \ref{lem:nkcorncomp} there exists a cube $\q\in \cu^n(\ns_k)$ such that $\q|_{\db{n}_{\leq k}}=\varrho$. It suffices to show that $\q|_{\db{n}\setminus\{1^n\}}=\q'$. Note that for each $v\in \db{n}$ with $|v|=k+1$, the restrictions of $\q$ and $\q'$ to the $(k+1)$-face $\{w\in\db{n}: w \leq v\}$ are both completions of the restriction of $\q'$ to $\{w\in \db{n}: w < v\}$. By Corollary \ref{cor:uniquecomp} these restrictions are equal. Arguing similarly for each such $v$, it follows that actually $\q|_{\db{n}_{\leq k+1}}=\q'|_{\db{n}_{\leq k+1}}$. We can then argue similarly to deduce that $\q|_{\db{n}_{\leq k+j}}=\q'|_{\db{n}_{\leq k+j}}$ for $j=2,3,\ldots$, and thus we deduce that $\q|_{\db{n}\setminus\{1^n\}}=\q'$, as required.
\end{proof}
\section{On characteristic factors associated with nilpotent group actions}\label{sec:ergapps}
\noindent In \cite{HK}, Host and Kra carried out their groundbreaking analysis of ergodic $\mb{Z}$-actions by first defining a sequence of probability measures, denoted by $\mu^{[n]}$, $n\geq 0$, and then studying the characteristic factor associated with each such measure. In the language developed in this paper, these measures $\mu^{[n]}$ can be  checked to form a cubic coupling (see Definition \ref{def:HKcoup} below and the explanation thereafter). We refer to this cubic coupling as the \emph{Host--Kra coupling} associated with the given $\mb{Z}$-action. The Host--Kra seminorm associated with $\mu^{[n]}$ is then the corresponding $U^n$-seminorm in our language.

Our goal in this section is to generalize the Host--Kra couplings and related seminorms from \cite{HK}, and combine this with the main results from Section \ref{sec:structhm} to treat measure-preserving actions of  countable nilpotent groups. In particular, we obtain Theorem \ref{thm:HKgen} below, a generalization of the Host--Kra structure theorem \cite[Theorem 10.1]{HK}. 

Recall the notion of a filtered group $(G,G_\bullet)$ from Definition \ref{def:fg}. We may consider also the \emph{pre}filtration $G_\bullet^{+k}$, defined by setting its $i$-th term to be $G_{i+k}$. In this section we also assume that $G$ is countable and discrete.

As we saw in Theorem \ref{thm:k-level-erg}, the cubic coupling that we associate with an action of a nilpotent group $G$ depends on a choice of a filtration $G_\bullet$ on $G$. Indeed, the filtration yields the cube structure on $G$ consisting of the groups of $n$-cubes $\cu^n(G_\bullet)$, $n\geq 0$, and in Theorem \ref{thm:k-level-erg} the measure $\mu^{\db{n}}$ is supposed to be preserved by the action of $\cu^n(G_\bullet)$. Recall that this group is the subgroup of $G^{\db{n}}$ generated by elements of the form $g^F$, defined by $g^F(v)=g$ if $v\in F$ and $g^F(v)=\id_G$ otherwise, where $F$ is some face in $\db{n}$ of dimension $d$ and $g\in G_{n-d}$ (these cubes are detailed in \cite[\S 2.2.1]{Cand:Notes1} for instance). More generally, for each integer $k\geq 0$ we denote by $H_{n,k}$ the group $\cu^n(G_\bullet^{+k})$ (this is shown to be a group in \cite[\S 6]{GTorb}, for instance, where the notation $\textrm{HK}^n$ is used instead of $\cu^n$). Note that if no filtration is specified on $G$ then we can always let $G_\bullet$ be the lower central series.

We shall use the following basic result.
\begin{lemma}\label{lem:cubefilt}
Let $(G,G_\bullet)$ be a filtered group. Then $(H_{n,k})_{k\geq 0}$ is a filtration on $H_{n,0}$.
\end{lemma}
\begin{proof}
We aim to show that for every $\q_j\in H_{n,j}$, $\q_k\in H_{n,k}$ we have $[\q_j,\q_k]\in H_{n,j+k}$. Note first that this holds for all generators $g^F\in H_{n,j}$, $h^{F'}\in H_{n,k}$, since $[g^F,h^{F'}]=[g,h]^{F\cap F'}$, which is easily seen to be in $H_{n,j+k}$. Before we generalize from generators to arbitrary elements $\q_j\in H_{n,j}$, $\q_k\in H_{n,k}$, let us use this case of generators to prove that $H_{n,k}\lhd H_{n,0}$ for every $k\geq 0$ (this will be used for the general case). For every $\q'\in H_{n,k}$ and generator $g^F\in H_{n,0}$, since $\q'=h_1^{F_1'}\cdots h_r^{F_r'}$ for some generators $h_i^{F_i'}$, we have $(g^F)^{-1}\q' g^F = (g^F)^{-1}h_1^{F_1'} g^F\, (g^F)^{-1}h_2^{F_2'} g^F \cdots (g^F)^{-1}h_r^{F_r'} g^F$, and this is in $H_{n,k}$ since each factor $(g^F)^{-1}h_i^{F_i'} g^F$ is in $H_{n,k}$ by the case of generators. Now, given any $\q=g_1^{F_1}\cdots g_t^{F_t}\in H_{n,0}$ and $\q'\in H_{n,k}$, using the previous case we have $(g_1^{F_1})^{-1}\q' g_1^{F_1}=\q_1 \in H_{n,k}$, then we have $(g_2^{F_2})^{-1}\q_1 g_2^{F_2}=\q_2 \in H_{n,k}$, and so on iteratively until we conclude that $\q^{-1} \q' \q\in H_{n,k}$. We have thus proved that $H_{n,k}\lhd H_{n,0}$. Now, to show that $[\q_j,\q_k]\in H_{n,j+k}$ in general, we can argue by induction on $\ell(\q_j)+\ell(\q_k)$, where $\ell(\q_j)$ is a positive integer such that there is an expression of $\q_j$ as a product of generators $g_1^{F_1}\cdots g_\ell^{F_\ell}$ with $\ell\leq \ell(\q_j)$. For $\ell(\q_j)+\ell(\q_k)=2$ we are in the case of generators. For $\ell(\q_j)+\ell(\q_k)>2$, we show that the coset $[\q_j,\q_k]\,H_{n,j+k}$ is $H_{n,j+k}$. Letting $\q_j=g_1^{F_1}\cdots g_\ell^{F_\ell}$ and $\q_k=h_1^{F_1'}\cdots h_{\ell'}^{F_{\ell'}'}$, we have
\begin{equation}\label{eq:cubefilt}
[\q_j,\q_k]= (g_\ell^{-1})^{F_\ell}\cdots (g_1^{-1})^{F_1}\, (h_{\ell'}^{-1})^{F_{\ell'}'}\cdots (h_1^{-1})^{F_1'}\, g_1^{F_1}\cdots g_\ell^{F_\ell}\, h_1^{F_1'}\cdots h_{\ell'}^{F_{\ell'}'}.
\end{equation}
We then have $g_\ell^{F_\ell} h_1^{F_1'}= h_1^{F_1'} g_\ell^{F_\ell}[g_\ell^{F_\ell},h_1^{F_1'}]$, where $[g_\ell^{F_\ell},h_1^{F_1'}]\in H_{n,j+k}$. This together with $H_{n,j+k}\lhd H_{n,0}$ implies that the coset $[\q_j,\q_k]\,H_{n,j+k}$ is equal to the coset represented by the product in the right side of \eqref{eq:cubefilt} with $g_\ell^{F_\ell}$ and $h_1^{F_1'}$  swapped. Applying this repeatedly, the term $h_1^{F_1'}$ can be moved to the left in the product until it cancels $(h_1^{-1})^{F_1'}$. We can then conclude by induction that $[\q_j,\q_k]\,H_{n,j+k} = H_{n,j+k}$, so $[\q_j,\q_k]\in H_{n,j+k}$ as required.
\end{proof}
\noindent We identify $G^{\db{n}}\times G^{\db{n}}$ with $G^{\db{n+1}}$ by viewing an element $(g_0,g_1)$ of the former group as the element of $G^{\db{n+1}}$ whose value restricted to $\{0,1\}^n\times\{i\}$ is equal to $g_i$, for $i=0,1$. For a group $K$ and a normal subgroup $K_2\lhd K$, we define the following subgroup of $K\times K$:
\[
{\rm diag}(K,K_2):=\{(ga,gb):g\in K,\, a,b\in K_2\}.
\]
\begin{lemma}\label{lem:nilplem1}
Let $(G,G_\bullet)$ be a filtered group. For all $n,k\geq 0$, $H_{n+1,k}\!=\!{\rm diag}(H_{n,k},H_{n,k+1})$.
\end{lemma}

\begin{proof} 
The group ${\rm diag}(H_{n,k},H_{n,k+1})$ is clearly generated by the elements of the form $(g_1^{F_1},g_1^{F_1})$, $(g_2^{F_2},1)$ and $(1,g_2^{F_2})$, for some faces $F_1$, $F_2$ in $\db{n}$, where $g_1\in G_{n-\dim(F_1)+k}$ and $g_2\in G_{n-\dim(F_2)+k+1}$. Using that $(g_1^{F_1},g_1^{F_1})=g_1^{F_1\times\{0,1\}}$, it is checked in a straightforward way that ${\rm diag}(H_{n,k},H_{n,k+1})$ and $H_{n+1,k}$ have equal sets of generators.
\end{proof}
\noindent Recall from Definition \ref{def:relsquare} the notion of the square of a measure relative to a factor. If $G$ acts on $\varOmega=(\Omega,\mc{A},\lambda)$ by measure-preserving transformations, then a set $B\in \mc{A}$ is said to be \emph{$G$-invariant} if we have $\lambda\big((g\cdot B) \Delta B\big)=0$ for every $g\in G$. The $G$-invariant sets form a sub-$\sigma$-algebra of $\mc{A}$, and the action of $G$ is \emph{ergodic} if every $G$-invariant set $B$ has $\lambda(B)\in\{0,1\}$. Recall also that a \emph{factor} of the measure-preserving system $(\varOmega,G)$ is a $\sigma$-algebra $\mc{B}\subset \mc{A}$ such that for every $g\in G$ and $B\in \mc{B}$ we have $g\cdot B\in \mc{B}$.
\begin{lemma}\label{lem:nilplem2}
Let $K$ be a group and $K_2$ be a normal subgroup of $K$. Suppose that $K$ acts on a probability space $\varOmega_0=(\Omega_0,\mathcal{A}_0,\lambda_0)$ by measure-preserving transformations. Let $\lambda_1\in\coup(\varOmega_0,\{0,1\})$ be the square of $\lambda_0$ relative to the $\sigma$-algebra of $K_2$-invariant sets. Then the group ${\rm diag}(K,K_2)$ acts by measure-preserving transformations on $(\varOmega_0^{\{0,1\}},\mc{A}_0^{\{0,1\}},\lambda_1)$.  
\end{lemma}

\begin{proof} 
Let $\mc{B}$ denote the $\sigma$-algebra of $K_2$-invariant sets in $\mc{A}_0$. Since $K_2\lhd K$, we have\footnote{Indeed, for every $k\in K$, $k_2\in K_2$ and $B\in \mc{B}$, since $k^{-1}k_2k\in K_2$ we have $k_2\cdot k\cdot B = k\cdot(k^{-1}k_2k)\cdot B = k\cdot B$.} that $\mc{B}$ is a factor of $(\varOmega_0,K)$. In particular $\mb{E}(f^g|\mc{B})=\mb{E}(f|\mc{B})^g$ for every $g\in K$ and $f\in L^\infty(\varOmega_0)$ (where $f^g(\omega):=f(g( \omega))$). By Lemma \ref{lem:pisysapprox}, it suffices to prove that for every function $h$ on $\Omega_0^{\{0,1\}}$ of the form $(f_0\co p_0)(f_1\co p_1)$ with $f_0,f_1\in L^\infty(\varOmega_0)$, for all $t\in {\rm diag}(K,K_2)$ we have $\int h^t \ud\lambda_1=\int h \ud\lambda_1$. Let $g\in K$, $a,b\in K_2$ satisfy $t =(ga,gb)$. By \eqref{eq:relsquare2} we have $\int h^t \ud\lambda_1 = \int \mb{E}(f_0^{ga}\,|\mc{B})\, \mb{E}(f_1^{gb}\,|\mc{B}) \ud\lambda_0 = \int \mb{E}(f_0^{g}\,|\mc{B})^a\, \mb{E}(f_1^{g}\,|\mc{B})^b \ud\lambda_0$ $= \int \mb{E}(f_0^{g}\,|\mc{B}) \, \mb{E}(f_1^{g}\,|\mc{B}) \ud\lambda_0$, and this equals $\int \big(\mb{E}(f_0\,|\mc{B}) \, \mb{E}(f_1\,|\mc{B})\big)^g \ud\lambda_0=\int \mb{E}(f_0\,|\mc{B})\, \mb{E}(f_1\,|\mc{B}) \ud\lambda_0=\int h \ud\lambda_1$.
\end{proof}
\noindent In what follows we shall often say that a filtered group $(G,G_\bullet)$ ``acts on a probability space" just to mean that $G$ acts on the space by measure-preserving transformations. We can now generalize the couplings introduced by Host and Kra in \cite{HK}.
\begin{defn}[Host--Kra couplings for filtered groups]\label{def:HKcoup}
Let $\varOmega=(\Omega,\mc{A},\lambda)$ be a probability space, and let $(G,G_\bullet)$ be a filtered group acting on $\varOmega$ by measure-preserving transformations. For each $n\in \mb{N}$ we define an $H_{n,0}$-invariant measure $\mu^{\db{n}}\in\coup(\varOmega,\db{n})$ recursively as follows. We set $\mu^{\db{0}}:=\lambda$. Having defined $\mu^{\db{n}}$, let $I_n$ be the $\sigma$-algebra of $H_{n,1}$-invariant sets. Then we define $\mu^{\db{n+1}}$ to be the square of $\mu^{\db{n}}$ relative to $I_n$.
\end{defn}
\noindent The fact that $H_{n,1}$ is normal in $H_{n,0}$ implies that $I_n$ is a factor of the measure-preserving system $\big((\Omega^{\db{n}},\mu^{\db{n}}),H_{n,0}\big)$, so we can apply Lemma  \ref{lem:nilplem2} and then Lemma \ref{lem:nilplem1} to deduce that $\mu^{\db{n+1}}$ is indeed $H_{n+1,0}$-invariant, and the recursion can thus proceed.

The construction in Definition \ref{def:HKcoup} generalizes the construction of the measures $\mu^{[n]}$ in \cite[\S 3.1]{HK}. Indeed, the latter construction concerns the $\mb{Z}$-action generated by a single transformation $T$, and if we let $\mb{Z}_\bullet$ be the lower central series on $\mb{Z}$ then $H_{n,1}=\langle T^{\db{n}} \rangle$, so the $\sigma$-algebra of $T^{\db{n}}$-invariant sets used in \cite{HK} is precisely the $\sigma$-algebra $I_n$ used above. 

Given a power $\varOmega^S$ of a probability space $\varOmega$, and given a bijection $\theta:S\to S$ (more generally, a group $G$ of such bijections), recall that the \emph{coordinatewise action} of $\theta$ (or $G$) on $\varOmega^S$ is the measure-preserving action defined by $\theta\cdot ((\omega_v)_{v\in S})=(\omega_{\theta(v)})_{v\in S}$ (for each $\theta\in G$). To study the symmetries of Host--Kra couplings we use the following result.

\begin{lemma}\label{lem:nilplem3}
Let $K$ be a nilpotent group and $K_2$ be a subgroup of $K$ with $[K,K]\leq K_2$. Suppose that $K$ acts on a probability space $\varOmega_0=(\Omega_0,\mc{A}_0,\lambda_0)$ by measure-preserving transformations. Let $\lambda_1$ be the square of $\lambda_0$ relative to the $\sigma$-algebra of $K$-invariant sets. Let $\lambda_2$ be the square of $\lambda_1$ relative to the $\sigma$-algebra of ${\rm diag}(K,K_2)$-invariant sets. Then the measure $\lambda_2$ is invariant under the coordinatewise action of $\aut(\db{2})$.
\end{lemma}
\noindent Here $\aut(\db{2})$ denotes the group of automorphisms of the cube $\db{2}=\{0,1\}^2$, that is, the group of bijections $\{0,1\}^2\to \{0,1\}^2$ that extend to affine homomorphisms $\mb{Z}^2\to\mb{Z}^2$.
\begin{proof}
We first claim that $\lambda_1$ is invariant under the action of ${\rm diag}(K,K_2)$. To prove this, we first note that by Lemma \ref{lem:nilplem2} applied with the pair $K\leq K$, we have that $\lambda_1$ is invariant under  the action of $K\times K$ (note that $K\times K={\rm diag}(K,K)$). Now using Lemma \ref{lem:nilplem2} again for the pair ${\rm diag}(K,K_2)\leq {\rm diag}(K,K_2)$ we obtain that $\lambda_2$ is invariant under the action of ${\rm diag}(K,K_2)\times{\rm diag}(K,K_2)$. In particular $\lambda_2$ is invariant under the action of $K_2^{\db{2}}$. We will use this to show that $\lambda_2$ is relatively independent over its factor $_{\mc{B}_2|}\lambda_2$, where $\mc{B}_2$ is the $\sigma$-algebra of $K_2$-invariant sets. To this end it suffices to prove that for every system $(f_v)_{v\in \db{2}}$ of functions $f_v\in L^\infty(\mc{A}_0)$ we have
\begin{equation}\label{eq:meanergapp}
\int \prod_{v\in \db{2}}f_v\co p_v \,\ud\lambda_2=\int \prod_{v\in \db{2}}\mathbb{E}(f_v|\mathcal{B}_2)\co p_v \,\ud\lambda_2.
\end{equation}
By the mean ergodic theorem for amenable groups, for every function $f\in L^\infty(\mc{A}_0)$ the projection $\mb{E}(f|\mc{B}_2)$ is the limit in $L^2(\mc{A}_0)$ of averages of the form $|F_n|^{-1}\sum_{g\in F_n} f^g$, where $(F_n)_{n\in \mb{N}}$ is a F\o lner sequence in $K_2$ (see \cite[Theorem 2.1]{Weiss}). Replacing the conditional expectations by such averages and using the $K_2^{\db{2}}$ invariance, we deduce \eqref{eq:meanergapp}, and the claimed relative independence follows. 
Given this relative independence, to show that $\lambda_2$ is invariant under the action of $\aut(\db{2})$ it suffices to prove it for the factor coupling $_{\mc{B}_2|}\lambda_2$. Since $K_2$ acts trivially on $\mc{B}_2$, we have that $_{\mc{B}_2|}\lambda_2$ is in fact equal to the Host--Kra coupling $\mu^{\db{2}}$ for the action $K/K_2$ on $(\Omega_0,\mc{B}_2,\lambda_0)$. But $K/K_2$ is an abelian group, so now the desired invariance follows from the original argument of  Host and Kra, which was extended for actions of arbitrary countable abelian groups in \cite[Appendix A]{BTZ} (see \cite[Lemma A.14]{BTZ} and \cite[Proposition 3.7]{HK}).
\end{proof}
\noindent The following theorem will enable us to apply our main results from Section \ref{sec:structhm}.

\begin{theorem}\label{thm:nilaction}
Let $(G,G_\bullet)$ be a filtered group acting ergodically on a probability space $\varOmega=(\Omega,\mathcal{A},\lambda)$, and let $(\mu^{\db{n}})_{n\geq 0}$ be the associated sequence of Host--Kra couplings. Then $\big(\varOmega,(\mu^{\db{n}})_{n\geq 0}\big)$ is a cubic coupling. Moreover, the group $\cu^n(G_\bullet)$ acts on $\Omega^{\db{n}}$ by transformations preserving the measure $\mu^{\db{n}}$.
\end{theorem}

\begin{proof}
We check that the three axioms from Definition \ref{def:cc-idemp} are satisfied.

For the ergodicity axiom, note that since $H_{0,1}=G$ and the action of $G$ is ergodic, we have that $I_0$ is the trivial $\sigma$-algebra, so $\mu^{\db{1}}$ is the product measure $\lambda\times\lambda$ on $(\Omega,\mc{A})^2$.

To check the other two axioms, first we prove the fact that $\mu^{\db{n}}$ is invariant under the coordinatewise action of $\aut(\db{n})$. We argue by induction on $n$, noting first that this fact is clear for $\mu^{\db{1}}$, since $\aut(\db{1})$ consists only of the reflection $v\sbr{1}\mapsto 1-v\sbr{1}$, which indeed leaves $\lambda\times\lambda$ invariant. 
Suppose by induction that $\mu^{\db{n-1}}$ is invariant under the action of $\aut(\db{n-1})$.
It follows that $\mu^{\db{n}}$ is invariant under all automorphisms $\phi$ of the form $\phi(v)|_{[n-1]}=\phi'(v\sbr{1},\ldots,v\sbr{n-1})$, $\phi(v)\sbr{n}=v\sbr{n}$, for some $\phi'\in \aut(\db{n-1})$. To prove that $\mu^{\db{n}}$ is invariant under all of $\aut(\db{n})$, we apply Lemma \ref{lem:nilplem3} with $\lambda_0=\mu^{\db{n-2}}$,  $K=H_{n-2,1}$, and $K_2=H_{n-2,2}$. This gives us that the measure $\mu^{\db{n}}$, when viewed as $\lambda_2$ in that lemma (i.e.\ as a self-coupling of $\mu^{\db{n-2}}$ indexed by $\db{2}$)  is invariant under the action of $\aut(\db{2})$. In particular, it is invariant under swapping the two coordinates of elements of $\db{2}$, and this implies by induction that $\mu^{\db{n}}$ is invariant under every element of $\aut(\db{n})$ that just permutes coordinates of $v$. Moreover $\mu^{\db{n}}$ is also invariant under the reflection $\sigma_n$ that sends $v\sbr{n}$ to $1-v\sbr{n}$. Indeed we have $\sigma_n=\theta^{-1} \co  \sigma_{n-1}\co \theta$ where $\theta$ permutes $v\sbr{n},v\sbr{n-1}$ and $\sigma_{n-1}$ is the reflection sending $v\sbr{n-1}$ to $1-v\sbr{n-1}$, and we already know that $\mu^{\db{n}}$ is invariant under $\theta$ and $\sigma_{n-1}$. It follows by induction that $\mu^{\db{n}}$ is invariant under all reflections $\sigma_j$, $j\in [n]$. The claimed invariance of $\mu^{\db{n}}$ follows.

To check the face consistency axiom, let $\phi:\db{m}\to\db{n}$ be a face map, assuming without loss of generality that $m<n$. To show that $\mu^{\db{n}}_\phi=\mu^{\db{m}}$, we can suppose, by composing $\phi$ with an element of $\aut(\db{n})$ and using the last paragraph, that $\phi(\db{m})=\db{m}\times \{0^{n-m}\}$. But then the desired equality $\mu^{\db{n}}_\phi=\mu^{\db{m}}$ follows clearly from the construction of $\mu^{\db{n}}$. 

The idempotence axiom holds, in the case of faces $F_i=\{v\in \db{n}: v\sbr{n}=i\}$, $i=0,1$, by construction of $\mu^{\db{n}}$ as a relative square of $\mu^{\db{n-1}}$ and by Lemma \ref{lem:IdempCondProp} $(iv)$. This together with the invariance under $\aut(\db{n})$ implies the idempotence axiom in full generality.

Since $\cu^n(G_\bullet)=H_{n,0}$, by Definition \ref{def:HKcoup} the action of $\cu^n(G_\bullet)$ preserves $\mu^{\db{n}}$.
\end{proof}
\begin{defn}[$U^k$-seminorms for ergodic filtered-group actions]\label{def:ergsn}
Let $(G,G_\bullet)$ be a filtered group acting ergodically on a probability space $\varOmega$. The \emph{$k$-th uniformity seminorm} on $\big(\varOmega,(G,G_\bullet)\big)$ is the seminorm $\|\cdot \|_{U^k}$ associated with $\mu^{\db{k}}$ as per Corollary \ref{cor:udsn}.
\end{defn}
\noindent We shall use Theorem \ref{thm:k-level-erg} to describe the characteristic factors corresponding to these seminorms. First let us define these factors, which requires the following result.
\begin{lemma}\label{lem:HKfactor}
Let $(G,G_\bullet)$ be a filtered group acting ergodically on a probability space $\varOmega=(\Omega,\mc{A},\lambda)$. For every $k\in \mb{N}$, there is a $\sigma$-algebra $\mc{H}_k\subset \mc{A}$ such that 
\begin{equation}\label{eq:ergFk}
L^\infty(\mc{H}_k)=\{ f\in L^\infty(\mc{A}):\, \forall g\in L^\infty(\mc{A})\textrm{ with }\|g\|_{U^{k+1}}=0,\textrm{ we have }\mb{E}(f\overline{g})=0\}. 
\end{equation}
Moreover $\mc{H}_k$ is a factor of the system $(\varOmega,G)$ and is unique up to $\lambda$-null sets.
\end{lemma}

\begin{proof}
We know by Theorem \ref{thm:nilaction} that $\Omega$ together with the Host--Kra couplings $\mu^{\db{n}}$ is a cubic coupling. We then let $\mc{H}_k$ be the Fourier $\sigma$-algebra $\mc{F}_k$ corresponding to this cubic coupling. Corollary \ref{cor:noisebotalg} then gives us \eqref{eq:ergFk}. To show that $\mc{H}_k$ is factor of $(\varOmega,G)$, we first note the fact that for each $n$ the diagonal action\footnote{That is, the action defined by $g\cdot (\omega_v)_{v\in\db{n}}=(g\cdot \omega_v)_{v\in\db{n}}$, for any $g\in G$.} of $G$ on $\Omega^{\db{n}}$ preserves $\mu^{\db{n}}$. This follows from Definition \ref{def:HKcoup}, since this is a sub-action of the action of $H_{n,0}$, and we know that the latter action preserves $\mu^{\db{n}}$. Given this, we can show that $\mc{F}_k$ is preserved by the action of $G$ as follows. Let $F$ be a system of functions $f_v\in L^\infty(\mc{A})$, $v\in K_{k+1}$, and for any $g\in G$ let $F^g$ denote the system $(f_v^g)_{v\in K_{k+1}}$. Then for every function $h\in L^\infty(\mc{A})$, the invariance of $\lambda$ under $g$ implies that $\langle [F]_{U^{k+1}},h\rangle = \langle [F]^g_{U^{k+1}},h^g\rangle$, and the invariance of $\mu^{\db{k+1}}$ under the diagonal element $(g)_{v\in \db{k+1}}$ implies that $\langle [F]_{U^{k+1}},h\rangle = \langle [F^g]_{U^{k+1}},h^g\rangle$. Since $h$ was arbitrary it follows that $[F^g]_{U^{k+1}}=[F]_{U^{k+1}}^g$. Hence,  shifts of convolutions $[F]_{U^{k+1}}$ by elements $g$ are again such convolutions. Since these convolutions generate $\mc{F}_k$, the  invariance of $\mc{F}_k$ follows. To see the uniqueness, note that $\mc{H}_k$ is defined by describing $L^\infty(\mc{H}_k)$ in \eqref{eq:ergFk}, so any other $\sigma$-algebra $\mc{B}$ satisfying \eqref{eq:ergFk} must satisfy $\mc{B}=_\lambda \mc{H}_k$.
\end{proof}

\begin{defn}\label{def:HKfactor}
Let $(G,G_\bullet)$ be a filtered group acting ergodically on a probability space $\varOmega=(\Omega,\mc{A},\lambda)$. We call the $\sigma$-algebra $\mc{H}_k$ from Lemma \ref{lem:HKfactor} the $k$-th \emph{Host--Kra factor} of the system $(\varOmega, (G,G_\bullet))$. We say that $(\varOmega,(G,G_\bullet))$ is a \emph{system of order} $k$ if $\mc{A}=_\lambda \mc{H}_k$. When we do not specify a particular filtration on $G$, and speak only of the \emph{Host--Kra factors on} $(\varOmega,G)$, we always take implicitly $G_\bullet$ to be the lower central series on $G$.
\end{defn}
\noindent This notion of system of order $k$ extends the one introduced by Host and Kra in \cite[Definition 4.10]{HK}. Every factor as defined above, with the induced action of $G$, is itself a measure-preserving system, and we can now characterize these systems using our results from Section \ref{sec:structhm}. We formulate this characterization in terms of a class of measure-preserving systems which we define next. As recalled in the previous section, a compact nilspace $\ns$ is naturally equipped with a filtered group $\tran(\ns)$ of \emph{translations} on $\ns$; see \cite{CamSzeg} or \cite[\S 2.9]{Cand:Notes2}. These translations are a special kind of homeomorphisms from $\ns$ to itself that preserve the cube structure and also the Haar measure on $\ns$.

\begin{defn}[Nilspace systems]
A \emph{nilspace system} is a triple $(\ns,G,\phi)$ where $\ns$ is a compact nilspace, where $G$ is a group, and $\phi:G\to\tran(\ns)$ is a group homomorphism. If $G_\bullet$ is a filtration on $G$,  and $\phi$ is a \emph{filtered-group} homomorphism, then we call $(\ns, (G,G_\bullet), \phi)$ a \emph{filtered nilspace system}. We say that $(\ns, G, \phi)$ (or $(\ns, (G,G_\bullet), \phi)$) is  \emph{$k$-step} if $\ns$ is $k$-step.
\end{defn}
\noindent Thus the action of $G$ on $\ns$ is defined by $g\cdot x:= \phi(g)\, (x)$. Note that $(\ns,G,\phi)$ can be viewed as a measure-preserving system by equipping $\ns$ with its Haar probability measure $\mu_{\ns}$, which is invariant under any translation. We say that the nilspace system is \emph{ergodic} if $G$ acts ergodically relative to $\mu_X$. We can now obtain the main result of this section.
\begin{theorem}\label{thm:ergthyapp}
Let $(G,G_\bullet)$ be a finite-degree filtered group acting ergodically on a Borel probability space $\varOmega$. Then for each $k$, the $k$-th Host--Kra factor of $(\varOmega, (G,G_\bullet))$ is isomorphic to the ergodic $k$-step filtered nilspace system $\big(\ns_k, (G,G_\bullet), \wh{\gamma_k}\big)$, with $\ns_k$, $\wh{\gamma_k}$ as given by Theorem \ref{thm:k-level-erg}.
\end{theorem}
\begin{proof}
By Theorem \ref{thm:nilaction} the Host--Kra couplings associated with the given system form a cubic coupling, and the action of each group $\cu^n(G_\bullet)$ on $\Omega^{\db{n}}$ preserves $\mu^{\db{n}}$. Applying Theorem \ref{thm:k-level-erg} we obtain that the $k$-th Host--Kra factor of $(\varOmega, (G,G_\bullet))$ is isomorphic to $\big(\ns_k, (G,G_\bullet), \wh{\gamma_k}\big)$. The ergodicity of this system follows from that of $(\varOmega, (G,G_\bullet))$.
\end{proof}
\noindent Recall from \cite[Chapter 11, \S 1.1]{HKbook} that if $L$ is a $k$-step nilpotent Lie group, with a lattice $\Gamma$, with Haar measure $\mu$ on $L/\Gamma$, and $T:L/\Gamma\to L/\Gamma$ is a transformation $x\mapsto \tau\cdot x$ for some $\tau\in L$, then $(L/\Gamma,\mu,T)$ is a (measure theoretic) \emph{$k$-step nilsystem}. There is a natural generalization to multiple transformations: for a discrete group $G$, we say that $(L/\Gamma, \mu, G)$ is a nilsystem if $G$ acts on $L/\Gamma$ via a group homomorphism $\phi : G\to L$, i.e.\ $(g, x)\mapsto \phi(g)\cdot x$. Turning $L/\Gamma$ into a nilspace using the natural cube structure (see \cite[Proposition 1.1.2]{Cand:Notes2}), and noting that $x\mapsto \tau\cdot x$ is then a translation in $\tran(L/\Gamma)$, we see that nilsystems  are examples of nilspace systems. It turns out that the latter systems can often be usefully expressed in terms of the former. For example, from the existing theory of compact nilspaces it follows that every ergodic nilspace system $(\ns,G,\phi)$ with finitely generated group $G$ is an inverse limit of nilsystems. This is proved in \cite[Theorem 5.1]{CGS}, and can also be derived from \cite[Theorem 1.29]{GMV3}. Thus, Theorem \ref{thm:ergthyapp} yields the following generalization of the Host--Kra structure theorem \cite[Theorem 10.1]{HK}.

\begin{theorem}\label{thm:HKgen}
Let $G$ be a finitely generated nilpotent group acting ergodically on a Borel probability space $\varOmega$, and let $G_\bullet$ be a filtration on $G$. Then for each positive integer $k$ the $k$-th Host--Kra factor of $(\varOmega, (G,G_\bullet))$ is isomorphic to an inverse limit of $k$-step nilsystems.
\end{theorem}
\noindent In particular, if $(\varOmega, (G,G_\bullet))$ is of order $k$ then it is an inverse limit of $k$-step nilsystems.
\begin{proof}
By Theorem \ref{thm:ergthyapp} the $k$-th Host--Kra factor of $(\varOmega, (G,G_\bullet))$ is isomorphic to an ergodic $k$-step filtered nilspace system. Since $G$ is finitely generated, this nilspace system is an inverse limit of $k$-step nilsystems, by \cite[Theorem 5.2]{CGS}.
\end{proof}

\begin{remark}\label{rem:ergavs}
Using Theorem \ref{thm:ergthyapp} (resp.\ Theorem \ref{thm:HKgen}), the analysis of the asymptotic behaviour of a multiple ergodic average for a nilpotent group action can be reduced to the analysis of the corresponding average on a nilspace system (resp.\ nilsystem), provided that the average can be controlled by one of the seminorms from Definition \ref{def:ergsn} (in the usual sense of ``control" used in this area; see \cite[\S 2.3]{Fra} and the estimates in \cite[(7)]{Fra}). While the family of averages controllable this way clearly includes the ones treated in \cite{HK}, there are also averages of interest in the area which it does not include (for instance, the averages in \cite{Host} are treated with seminorms whose construction differs from ours). Determining exactly which averages are controlled by each of the seminorms in Definition \ref{def:ergsn} is an interesting and potentially vast project which we do not pursue in this paper.
\end{remark}

\section{On cubic exchangeability}\label{sec:exchange} 

\noindent In this section we denote by $S$ a countable set. We denote by $\db{S}$ the set of elements $v\in \{0,1\}^S$ with only finitely many coordinates $v\sbr{i}$ equal to $1$, that is $\db{S}=\{0,1\}^S\cap \bigoplus_{i\in S}\mb{Z}$.

\begin{defn}
A map $\phi:\db{S_1}\to \db{S_2}$ is a \emph{cube morphism} if it extends to an affine homomorphism from $\bigoplus_{i\in S_1}\mb{Z}$ to $\bigoplus_{i\in S_2}\mb{Z}$.
\end{defn}
\begin{remark}
This generalizes the notion of a morphism between discrete cubes of finite dimension, introduced in \cite{CamSzeg} (see also \cite[\S 1.1]{Cand:Notes1}). It can be checked that $\phi:\db{S_1}\to \db{S_2}$ is a morphism if and only if for every $j\in S_2$, the function $v\mapsto \phi(v)\sbr{j}$ is either constant, or for some $i\in S_1$ it is $v\mapsto v\sbr{i}$ or $v\mapsto 1-v\sbr{i}$, and the following properties hold:
\begin{enumerate}
\item There are only finitely many $j\in S_2$ such that $v\mapsto \phi(v)\sbr{j}$ is either the constant 1 or is $v\mapsto 1-v\sbr{i}$ for some $i\in S_1$.
\item For each $i\in S_1$ there are only finitely many $j$ such that $v\mapsto \phi(v)\sbr{j}$ is $v\mapsto v\sbr{i}$.
\end{enumerate}
\end{remark}
\noindent If for each $i\in S_1$ there is exactly one $j\in S_2$ such that $v\mapsto \phi(v)\sbr{j}$ a non-constant function of $v\sbr{i}$, then we call $\phi$ a \emph{face map}. Note that, for $k\in \mb{N}$, face maps from $\db{k}$ to $\db{k}$ are bijective, but this is not necessarily true for face maps from $\db{\mb{N}}$ to $\db{\mb{N}}$. 

A set $F\subset \db{S}$ is a \emph{face} if it is of the form $\db{S'}\times z$ where $S'\subset S$ and $z\in\db{S\setminus S'}$. We say that $S'$ is the set of \emph{free coordinates} of $F$. We say that two faces are \emph{independent} if they have trivial intersection and their sets of free coordinates are disjoint. Note that face maps take faces to faces, but this is not true for cube morphisms in general. 

Let $\Bo$ be a standard Borel space, with Borel $\sigma$-algebra $\mc{B}$. Since $\db{\mb{N}}$ is countable, the product set $\Bo^{\db{\mb{N}}}$ with the $\sigma$-algebra $\bigotimes_{v\in\db{\mb{N}}} \mc{B}$ is also a standard Borel space. Our goal is to characterize Borel probabilities on $\Bo^{\db{\mb{N}}}$ that have the following properties.

\begin{defn}\label{def:cubexch}
A Borel probability measure $\mu$ on $\Bo^{\db{\mb{N}}}$ is \emph{cubic exchangeable} if it has the following \emph{consistency property}: for every $k\geq 0$ and every pair of injective morphisms $\phi_1,\phi_2:\db{k}\to\db{\mb{N}}$, we have $\mu_{\phi_1}\!=\!\mu_{\phi_2}$. We say that $\mu$ has the \emph{independence property} if for all finite independent faces $F_1,F_2\subset \db{\mb{N}}$, the $\sigma$-algebras $\mc{B}^{\db{\mb{N}}}_{F_1}, \mc{B}^{\db{\mb{N}}}_{F_2}$ are independent in $\mu$. 
\end{defn}

\noindent Note that the consistency property above is the consistency axiom from Definition \ref{def:cc} formulated for a measure on a product space of the form $\Bo^{\db{\mb{N}}}$.

\begin{remark}\label{rem:exchmappres} 
Let $\mu$ be a cubic exchangeable measure on $\Bo^{\db{\mb{N}}}$ and let $\phi:{\db{\mb{N}}}\to{\db{\mb{N}}}$ be an injective morphism. Since $\mu$ and the subcoupling $\mu_\phi$ are determined by their marginals on finite subsets of $\db{\mb{N}}$, and $\db{\mb{N}}$ is the union of finite-dimensional faces, the consistency property of $\mu$ implies that $\mu_\phi=\mu$. Let $\wh{\phi}$ be the map $\Bo^{\db{\mb{N}}}\to \Bo^{\db{\mb{N}}}$ obtained by first projecting to the coordinates in $\phi(\db{\mb{N}})$ and then relabeling the coordinates using $\phi^{-1}$. Then the property $\mu_\phi=\mu$ implies that $\wh{\phi}$ preserves the measure $\mu$.
\end{remark}
\begin{remark} A useful probabilistic viewpoint concerning measures on $\Bo^{\db{\mb{N}}}$ is to consider them as joint distributions of $\Bo$-valued random variables $Y_v$ indexed by the elements $v\in \db{\mb{N}}$. In this language, the independence property of $\mu$ means that $\{Y_v\}_{v\in F_1}$ is independent from $\{Y_v\}_{v\in F_2}$ whenever $F_1$ and $F_2$ are independent faces.
\end{remark}
\begin{remark}\label{rem:exrels}
As mentioned in the introduction, in \cite[\S 16]{Aldous2} Aldous considered a property related to cubic exchangeability. A measure $\mu$ on $\Bo^{\db{\mb{N}}}$ has this property if it is invariant under all transformations of $\Bo^{\db{\mb{N}}}$ induced by automorphisms of $\db{\mb{N}}$. These automorphisms form the  group that we denote by $\aut(\db{\mb{N}})$, which is isomorphic to $S^\infty\ltimes \mb{Z}_2^\infty$, where $S^\infty$ denotes the group of finitely-supported permutations of $\mb{N}$, and $\mb{Z}_2^\infty=\bigoplus_{i\in \mb{N}}\mb{Z}_2$. Note that cubic exchangeability as per Definition \ref{def:cubexch} implies this $\aut(\db{\mb{N}})$-exchangeability property of Aldous. Indeed, given $\theta\in \aut(\db{\mb{N}})$, for $m$ sufficiently large, the set $\db{m}\times \{0^{\mb{N}\setminus[m]}\}$ is  globally invariant under $\theta$. We can thus view $\theta$ as an injective morphism $\db{m}\to \db{\mb{N}}$, and deduce from cubic exchangeability of $\mu$ that $\mu$ is $\theta$-invariant. Note also that cubic exchangeability is strictly stronger than $\aut(\db{\mb{N}})$-exchangeability, because not all injective morphisms can be viewed as automorphisms this way (automorphisms take faces to faces in $\db{\mb{N}}$, whereas injective morphisms can take faces to subcubes that are not faces). In \cite[\S 5.3]{Austin2}, Austin observed that it is also natural to consider a stronger variant of $\aut(\db{\mb{N}})$-exchangeability, in which $\mu$ is required to be invariant not just under the action of $\aut(\db{\mb{N}})$, but rather under the action of the full \emph{affine} automorphism group of $\db{\mb{N}}$, denoted by $\textrm{Aff}(\mb{F}_2^\infty)$ (identifying $\db{\mb{N}}$ and $\mb{F}_2^\infty$ as sets), which is isomorphic to $\textrm{GL}(\mb{F}_2^\infty)\ltimes \mb{Z}_2^\infty$. It can be checked that this stronger property implies cubic exchangeability, using the fact that for every injective morphism $\phi:\db{n}\to\db{m}$, viewing $\phi$ as a map from $\db{n}\times\{0^{m-n}\}\subset \db{m}$ to $\db{m}$ the obvious way, there is a matrix $M\in \textrm{GL}(\mb{F}_2^m)$ and $w\in \mb{F}_2^m$ such that the affine linear map $\mb{F}_2^m\to \mb{F}_2^m$, $v\mapsto M(v)+w$ agrees with $\phi$ on $\db{n}\times\{0^{m-n}\}$.
\end{remark}

\smallskip

\noindent Recall that a compact nilspace $\ns$ is equipped with cube sets $\cu^n(\ns)$ for each $n\geq 0$, on each of which we can define a Haar probability measure $\mu_{\cu^n(\ns)}$. We can then define morphisms from $\db{\mb{N}}$ to $\ns$ by declaring that $\phi:\db{\mb{N}}\to\ns$ is a morphism if for every integer $n\geq 0$, for every cube morphism $\psi:\db{n}\to \db{\mb{N}}$ we have $\phi\co \psi\in\cu^n(\ns)$. The measures $\mu_{\cu^n(\ns)}$ can be put together to determine a well-defined probability on $\ns^{\db{\mb{N}}}$ (see Remark \ref{rem:infinitecc} below), which enables us in particular to define a random morphism $\phi:\db{\mb{N}}\to\ns$.

To formulate the main theorem of this section, the following construction is crucial.

\medskip
\noindent{\it The nilspace construction of cubic exchangeable measures}:~let $\mc{P}(\Bo)$ denote the standard Borel space consisting of the set of Borel probability measures on $\Bo$  equipped with the $\sigma$-algebra generated by the maps $A\mapsto \mu(A)$, $A\in \mc{B}$ (see  \cite[p.\ 113]{Ke}). Let $\ns$ be a compact nilspace and let $m:\ns\to\mc{P}(\Bo)$ be a Borel  map. Let $\phi:\db{\mb{N}}\to \ns$ be a random morphism. Then $\big(\phi(v)\big)_{v\in\db{\mb{N}}}$ is a sequence of $\ns$-valued random variables. Now we introduce a second randomization in which for every $v$ independently we choose an element $Y_v\in \Bo$ with distribution $m(\phi(v))$. We denote by $\zeta_{\ns,m}$ the resulting Borel probability measure on $\Bo^{\db{\mb{N}}}$, thus $\zeta_{\ns,m}$ is the joint distribution of the sequence of random variables $\big(Y_v\big)_{v\in\db{\mb{N}}}$.

\smallskip

The main result of this section can now be stated.

\begin{theorem}\label{thm:mainexch1}
Let $\mu$ be a Borel probability on $\Bo^{\db{\mb{N}}}$. Then the following statements hold:
\begin{enumerate}[leftmargin=0.7cm]
\item The measure $\mu$ is cubic exchangeable with the independence property if and only if $\mu=\zeta_{\ns,m}$ where $\ns$ is a compact nilspace and $m:\ns\to\mc{P}(\Bo)$ is Borel measurable. 
\item The measure $\mu$ is cubic exchangeable if and only if it is the convex combination of cubic exchangeable measures that have the independence property.
\end{enumerate}
\end{theorem}
\noindent The rest of the section is devoted to proving Theorem \ref{thm:mainexch1}. We use the following notion.

\begin{defn}
A \emph{weak cubic coupling} on a measurable space $(\Omega,\mc{A})$ is a sequence $\big(\mu^{\db{n}}\big)_{n\geq 0}$ of measures $\mu^{\db{n}}$ on $(\Omega^{\db{n}},\mc{A}^{\db{n}})$ satisfying the consistency axiom and the conditional independence axiom from Definition \ref{def:cc}.
\end{defn}

\begin{remark}\label{rem:infinitecc}
Every weak cubic coupling on $\Omega$ can be viewed as a single cubic exchangeable measure $\mu$ on $\Omega^{\db{\mb{N}}}$ such that each measure $\mu^{\db{n}}$ in the definition is the marginal of $\mu$ corresponding to the index set $\db{n}\times \{0^{\mb{N}\setminus [n]}\}$ (where in general $0^S$ is the element of $\db{S}$ with all entries $0$). Note that the consistency axiom implies at once that each measure $\mu^{\db{n}}$ is a self-coupling of $\mu^{\db{0}}$ indexed by $\db{n}$, that these measures $\mu^{\db{n}}$ can be put together to determine a well-defined probability $\mu$ on $\Omega^{\db{\mb{N}}}$, and that $\mu$ is cubic exchangeable.
\end{remark}

\noindent As a key step towards the proof of Theorem \ref{thm:mainexch1}, we obtain  the following result.

\begin{proposition}\label{prop:mainexch}
A probability measure on $\Bo^{\db{\mb{N}}}$ is cubic exchangeable if and only if it is a factor coupling of some weak cubic coupling. 
\end{proposition}

The proof relies mainly on the following lemma. 

\begin{lemma}\label{lem:infface}
Let $F_1,F_2$ be faces of $\db{\mb{N}}$ such that the face $F_1\cap F_2$ has infinite dimension, and let $T\subset F_1$ be a finite subset. Then there is a face map $\tau:\db{\mb{N}}\to\db{\mb{N}}$ such that $\tau(F_2)=F_1\cap F_2$ and $\tau(t)=t$ holds for every $t\in T$.
\end{lemma}

\begin{proof} Let $z\in F_1\cap F_2$ and let $\phi:\db{\mb{N}}\to\db{\mb{N}}$ be the bijective face map such that for every $i$ in the finite set $\supp(z)\subset \mb{N}$ we have $\phi(v)\sbr{i}=1-v\sbr{i}$ and for all $i\in \mb{N}\setminus \supp(z)$ we have $\phi(v)\sbr{i}=v\sbr{i}$. We have $\phi(z)=0^{\mb{N}}$. Note that it suffices to find the map $\tau$ for the faces $\phi(F_1),\phi(F_2)$ and the set $\phi(T)$, since then, by conjugating with $\phi$ we obtain a map satisfying the conclusion of the lemma for $F_1,F_2$ and $T$. Hence without loss of generality we can assume that $0^\mb{N}\in F_1\cap F_2$. In this case we have $F_i=\db{S_i}\times \{0^{\mb{N}\setminus S_i}\}$ for some sets $S_1,S_2\subset\mb{N}$. Let $S_3\subset S_1$ be a finite set such that $T\subseteq \{0,1\}^{S_3}\times \{0^{\mb{N}\setminus S_3}\}$. Let $\rho':S_2\setminus S_3\to (S_1\cap S_2)\setminus S_3$ be a bijection, and let $\rho:\mb{N}\to\mb{N}$ be the injection equal to the identity on $\mb{N}\setminus(S_2\setminus S_3)$ and equal to $\rho'$ on $S_2\setminus S_3$. Let $\tau$ be defined by $\tau(v)\sbr{i}= v\sbr{\rho^{-1}(i)}$ if $i\in\rho(\mb{N})$ and $\tau(v)\sbr{i}=0$ otherwise. The map $\tau$ satisfies the required conclusion.
\end{proof}

\begin{lemma}\label{lem:infbot}
Let $\nu$ be a cubic exchangeable probability measure on $\Bo^{\db{\mb{N}}}$. Let $F_1$, $F_2$ be faces of $\db{\mb{N}}$ such that $F_1\cap F_2$ has infinite dimension. Then $F_1\,\bot_\nu\, F_2$.
\end{lemma}

\begin{proof}
We need to show that if $f\in L^\infty(\mathcal{B}_{F_1}^{\db{\mb{N}}})$ then $\mb{E}(f|\mc{B}_{F_2}^{\db{\mb{N}}})$ is $\mc{B}_{F_1\cap F_2}^{\db{\mb{N}}}$-measurable. It is enough to show this for functions that depend on a finite set of coordinates, since every other bounded measurable function can be approximated in $L^2$ with arbitrary precision using such functions (Lemma \ref{lem:pisysapprox}). Suppose that $f\in L^\infty(\mc{B}_{T}^{\db{\mb{N}}})$ for some finite set $T\subset F_1$. By Lemma \ref{lem:infface} there is a face map $\tau: \db{\mb{N}}\to \db{\mb{N}}$ fixing $T$ pointwise and with $\tau(F_2)=F_1\cap F_2$. Let $\wh{\tau}:\Bo^{\db{\mb{N}}}\to \Bo^{\db{\mb{N}}}$ denote the map that first projects to the coordinates in $\tau(\db{\mb{N}})$ and then renames the coordinates using $\tau^{-1}$. By Remark \ref{rem:exchmappres} the map $\wh{\tau}$ preserves $\nu$. By \eqref{eq:exprel} we then have $\mb{E}(f|\mc{B}_{F_2}^{\db{\mb{N}}})\co\wh{\tau}=\mb{E}(f\co\wh{\tau}|\mc{B}_{\tau(F_2)}^{\db{\mb{N}}})=\mb{E}(f|\mc{B}_{F_1\cap F_2}^{\db{\mb{N}}})$, whence $\|\mb{E}(f|\mc{B}_{F_2}^{\db{\mb{N}}})\|_{L^2}=\|\mb{E}(f|\mc{B}_{F_1\cap F_2}^{\db{\mb{N}}})\|_{L^2}$. Since the latter expectation is a projection of the former, this equality of their $L^2$-norms implies that $\mb{E}(f|\mc{B}_{F_2}^{\db{\mb{N}}})=\mb{E}(f|\mc{B}_{F_1\cap F_2}^{\db{\mb{N}}})$, and the result follows.
\end{proof}

\begin{proof}[Proof of Proposition \ref{prop:mainexch}]
The backward implication is clear. For the converse, let $\nu$ be a cubic exchangeable measure on $\Bo^{\db{\mb{N}}}$. Let $\mb{N}=E\sqcup O$ where $E$ and $O$ denote the set of even and odd numbers respectively. We can write $\Bo^{\db{\mb{N}}}=(\Bo^{\db{O}})^{\db{E}}$. Let $V=\Bo^{\db{O}}$ and let $q:V\to \Bo$ denote the projection $p_v$ where $v=0^{\db{O}}$. Since $\Bo^{\db{\mb{N}}}=V^{\db{E}}$, we can view the measure $\nu$ as a coupling on $V^{\db{E}}$. To avoid confusion we denote this coupling by $\nu'$. It is easy to see from the cubic exchangeability property of $\nu$ that $\nu'$ is also cubic exchangeable. Lemma \ref{lem:infbot} applied to $\nu$ implies that $\nu'$ satisfies the conditional independence axiom. Hence $\nu'$ is a weak cubic coupling. Let $\phi: \db{E}\to\db{\mb{N}}$, $(v_i)_{i\in E}\mapsto (v_{2i})_{i\in \mb{N}}$. Equipping $\Bo^{\db{E}}$ with the measure $\nu_\phi$, it is clear that $q^{\db{E}}$ is a measure preserving map from $V^{\db{E}}$ to $\Bo^{\db{E}}$. In this construction the coordinates in the cubes are all indexed by even numbers. Renaming these coordinates by halving their indices, we obtain a weak cubic coupling $\mu$ on $V^{\db{\mb{N}}}$ such that $q^{\db{\mb{N}}}:V^{\db{\mb{N}}}\to \Bo^{\db{\mb{N}}}$ satisfies $\nu=\mu\co (q^{\db{\mb{N}}})^{-1}$.
\end{proof}

\noindent Proposition \ref{prop:mainexch} will be combined with the following result, which tells us that every weak cubic coupling on a Borel probability space is a convex combination of cubic couplings.

\begin{proposition}\label{ergoddecomp1}
Let $(\eta^{\db{n}})_{n\geq 0}$ be a weak cubic coupling on a standard Borel space $(\Omega,\mc{A})$, and let $\eta$ denote the corresponding cubic exchangeable measure on $\Omega^{\db{\mb{N}}}$. Then there is a probability measure $\kappa$ on $\mathcal{P}(\Omega^{\db{\mb{N}}})$, supported on the set of cubic couplings \textup{(}viewed as measures $\mu$  on $\Omega^{\db{\mb{N}}}$\textup{)}, such that $\eta=\int_{\mathcal{P}(\Omega^{\db{\mb{N}}})} \mu \ud\kappa$.
\end{proposition}

\begin{proof}
We show that there is a $\sigma$-algebra $\mc{G}\subset \mc{A}^{\db{\mb{N}}}$ such that in the disintegration of $\eta$ relative to $\mc{G}$, almost every measure is a cubic coupling.

For every 1-dimensional face $\{v,w\}$ in $\db{\mb{N}}$ let $\mc{G}_{v,w}=\mc{A}_v^{\db{\mb{N}}}\wedge_\eta\mc{A}_w^{\db{\mb{N}}}$. Fix some (any) $\{v,w\}$, and let $\mc{G}$ be a countably generated sub-$\sigma$-algebra of $\mc{G}_{v,w}$ such that $\mc{G}=_\eta \mc{G}_{v,w}$. (We can obtain $\mc{G}$ by  taking a countable subset $T$ of $\mc{G}_{v,w}$ that is dense in the $\eta$-metric, and letting $\mc{G}=\sigma(T)$.) For any other face $\{w,z\}$ intersecting $\{v,w\}$, by the weak cubic coupling axioms we have that $\{v,w\}\,\bot_\eta\,\{w,z\}$ and that the marginal distributions on $\{v,w\},\{w,z\}$ and $\{v,z\}$ are all equal. It follows that the coupling $\eta_{\{v,w\}}$ is idempotent. By Lemma \ref{lem:IdempCondProp} we have $\mc{G}_{v,w}=_\eta \mc{G}_{w,z}$ and $\mc{A}_v^{\db{\mb{N}}}\upmod_\eta \mc{A}_w^{\db{\mb{N}}}$. By iterating this for other such faces, we deduce that $\mc{G}=_\eta\mc{G}_{v',w'}$ for every face $\{v',w'\}\subset \db{\mb{N}}$. By Proposition \ref{prop:idemp} and Lemma \ref{lem:IdempCondProp} we also have that there is a $\sigma$-algebra $\mc{H}$ such that $\mc{G}=_\eta p_v^{-1}(\mc{H})$ holds for every $v\in\db{\mb{N}}$, and for every set $H\in\mc{H}$ and pair $v,w\in\db{\mb{N}}$ we have $p_v^{-1}(H)=_\eta p_w^{-1}(H)$. By \cite[(17.35) i)]{Ke} (applied with $Y$ the quotient standard Borel space $\Omega^{\db{\mb{N}}}/\mc{G}$ and $f:\Omega^{\db{\mb{N}}}\to Y$ the canonical quotient map) we obtain a Borel map $t: y \mapsto \mu_y$ from $Y$ into $\mc{P}(\Omega^{\db{\mb{N}}})$ such that, letting $\kappa=\eta\co f^{-1}\co t^{-1}$, we have $\eta=\int_{\mathcal{P}(\Omega^{\db{\mb{N}}})} \mu \ud\kappa$. It now suffices to show that, for almost every $y$, the images of $\mu_y$ on faces $\db{n}\subset \db{\mb{N}}$, for increasing $n$, form a sequence satisfy the axioms in Definition \ref{def:cc}.

Firstly, suppose that the consistency axiom failed for every $\mu_y$ in some set of positive $\eta\co f^{-1}$-measure. Then there would exist a set $X\in \mc{G}$ with $\eta(X)>0$, such that the measure $\eta_X$ obtained from $\eta$ by conditioning on $X$ (as per Definition \ref{def:condcoup}) does not satisfy the consistency axiom. We shall obtain a contradiction by showing that $\eta_X$ must in fact satisfy this axiom. Let $\phi_1,\phi_2:\db{k}\to\db{\mb{N}}$ be two injective cube morphisms. Let $\wh{\phi_i}:\Bo^{\db{\mb{N}}}\to \Bo^{\db{k}}$ be the map that projects to the coordinates in $\phi_i(\db{k})$ and then renames the coordinates using $\phi_i^{-1}$. Our goal is to show that for every Borel set $Q\subset \Bo^{\db{k}}$ we have $\eta_X\co\wh{\phi_1}^{-1}(Q)=\eta_X\co \wh{\phi_2}^{-1}(Q)$. Let $v_i=\phi_i(0^k)$ and recall that there is a set $Y\in \mc{H}$ such that $p_{v_1}^{-1}(Y)=_\eta p_{v_2}^{-1}(Y)=_\eta X$. Then $\eta_X\co \wh{\phi_i}^{-1}(Q)=\frac{1}{\eta(X)}\eta\big(\wh{\phi_i}^{-1}(Q)\cap p_{v_i}^{-1}(Y)\big)=\frac{1}{\eta(X)}\eta\co\wh{\phi_i}^{-1}\big(Q\cap p_{0^k}^{-1}(Y)\big)=\frac{1}{\eta(X)}\eta_{\phi_i}\big(Q\cap p_{0^k}^{-1}(Y)\big)$. By the consistency axiom for $\eta$ we have $\eta_{\phi_1}=\eta_{\phi_2}$, and the consistency axiom for $\eta_X$ follows.

By the consistency axiom, almost every $\mu_y$ has all its marginals $\mu_y\co p_v^{-1}$, $v\in \db{\mb{N}}$ equal to a single measure $\mu_y^{\db{0}}$. Moreover, for any fixed face $\{v,w\}$, since we disintegrate $\eta$ relative to $\mc{G}=_\eta \mc{A}_v^{\db{\mb{N}}}\wedge \mc{A}_w^{\db{\mb{N}}}$, basic facts from probability imply that for almost every $y$ the image of $\mu_y$ on this face is $\mu_y^{\db{0}}\times\mu_y^{\db{0}}$. Hence $\mu_y$ satisfies the ergodicity axiom.

For the conditional independence axiom, note that if $F_1$, $F_2$ are $n$-dimensional faces with $(n-1)$-dimensional intersection, then by construction we have $\mc{G}\subset_\eta \mc{A}_{F_1\cap F_2}^{\db{\mb{N}}}$. Then $F_1\,\bot_\eta\, F_2$ implies, by similar basic facts as above, that $F_1\,\bot_{\mu_y}\, F_2$ for almost every $y$.
\end{proof}
\noindent Propositions \ref{prop:mainexch} and \ref{ergoddecomp1} together imply directly the following result, which tells us that cubic exchangeable measures are convex combinations of factors of cubic couplings.
\begin{theorem}\label{thm:mainexch3}
Let $\nu$ be a cubic exchangeable probability measure on $\Bo^{\db{\mb{N}}}$. Then there is a standard Borel space $\Omega$, a Borel measurable map $q:\Omega\to \Bo$, and a Borel probability measure $\kappa$ on $\mc{P}(\Omega^{\db{\mb{N}}})$ supported on the set of cubic couplings \textup{(}viewed as measures $\mu$ on $\Omega^{\db{\mb{N}}}$\textup{)}, such that $\nu=\int_{\mc{P}(\Omega^{\db{\mb{N}}})} \mu\co(q^{\db{\mb{N}}})^{-1}\, \ud\kappa(\mu)$.
\end{theorem}
\noindent  Now we turn to the independence property. We need two lemmas.
\begin{lemma}\label{lem:cubindep}
Every cubic coupling has the independence property.
\end{lemma}

\begin{proof}
Let $F_1$, $F_2$ be independent finite-dimensional faces in $\db{\mb{N}}$. Note that there is a finite dimensional face $F_3\supset F_2$ with $|F_1\cap F_3|=1$. Let $v = F_1\cap F_3$ and let $f\in L^\infty(\mc{A}^{\db{\mb{N}}}_{F_2})$. By Remark  \ref{rem:strongaxiom} we have $F_3\,\bot_\mu \, F_1$, and it follows that $\mb{E}(f|\mc{A}^{\db{\mb{N}}}_{F_1})=\mb{E}(f|\mc{A}^{\db{\mb{N}}}_v)$. Now it suffices to show that $\mc{A}^{\db{\mb{N}}}_{F_2}$, $\mc{A}^{\db{\mb{N}}}_v$ are independent. This follows from Corollary \ref{cor:facelocality}.
\end{proof}

\begin{lemma}\label{indepcomb}
Let $Q$ be a standard Borel space. Let $\kappa$ be a probability measure on $\mc{P}(Q)$ such that, for some $\mu\in \mc{P}(Q)$, if $\mu'$ is taken with distribution $\kappa$ then the average of $\mu'\times\mu'$ is equal to $\mu\times \mu$. Then $\kappa$ is the Dirac measure $\delta_{\mu}$.
\end{lemma}

\begin{proof}
Fix any measurable set $B\subset Q$. Since $\mu'$ is a random measure we have that $\mu'(B)$ is a random variable. We have $\mb{E}(\mu'(B)^2)=\mb{E}((\mu'\times\mu')(B\times B))=(\mu\times\mu)(B\times B)=\mu(B)^2$, and $\mb{E}(\mu'(B)) = \mb{E}((\mu'\times\mu')(B\times Q)=(\mu\times\mu)(B\times Q)=\mu(B)$ (by our assumptions). These equations and the linearity of expectation imply that $\mb{E}((\mu(B)-\mu'(B))^2)=0$. It follows that $\mu'(B)=\mu(B)$ holds almost surely. Now using this argument for a countable generating set $\mc{S}$ of the Borel $\sigma$-algebra of $Q$ we get that almost surely we have $\mu'(B)=\mu(B)$ for every $B\in\mc{S}$ simultaneously. This completes the proof.
\end{proof}

\begin{theorem}\label{thm:mainexch4}
A probability measure $\nu$ on $\Bo^{\db{\mb{N}}}$ is cubic exchangeable with the independence property if and only if for some Borel probability space $\varOmega$ there is a cubic coupling $\mu$ on $\Omega^{\db{\mb{N}}}$ and a Borel measurable map $q:\Omega\to \Bo$ such that $\nu=\mu\co (q^{\db{\mb{N}}})^{-1}$.
\end{theorem}

\begin{proof}
Let $\kappa$ and $q:\Omega\to \Bo$ be as given by Theorem \ref{thm:mainexch3}, so that $\nu=\mb{E}_\kappa(\mu\co(q^{\db{\mb{N}}})^{-1})$. Let $\nu'$ denote the random measure $\mu\co(q^{\db{\mb{N}}})^{-1}$ on $\Bo^{\mb{N}}$. The independence property is preserved under composition with $(q^{\db{\mb{N}}})^{-1}$, so Lemma \ref{lem:cubindep} implies that $\nu'$ has the independence property almost surely. Let $n\geq 0$ and let $F_1$, $F_2$ be independent faces of dimension $n$ in $\db{\mb{N}}$. Now, almost surely, the marginals of $\nu'$ on $F_1$ and $F_2$ are equal; we denote this marginal by $\mu'_1$. Moreover $\mc{B}^{\db{\mb{N}}}_{F_1}$, $\mc{B}^{\db{\mb{N}}}_{F_2}$ are independent in $\nu'$ almost surely, so $\nu'$ is of the form $\mu'_1\times\mu'_1$ on $\Bo^{F_1}\times \Bo^{F_2}$. The same independence holds in $\nu$ (by assumption), so $\nu$ is similarly of the form $\mu_1\times \mu_1$ on $\Bo^{F_1}\times \Bo^{F_2}$. Applying Lemma \ref{indepcomb} with $Q=\Bo^{F_1}$ (identified with $\Bo^{F_2}$), we obtain that almost surely $\mu'_1=\mu_1$. Since this holds for every $n$, we must have $\nu=\nu'$ almost surely, so $\nu$ is indeed the image under $q^{\db{\mb{N}}}$ of a cubic coupling.
\end{proof}
\noindent To complete the proof of our main result, the last step is to express a factor of a cubic coupling as a measure of the form $\zeta_{\ns,m}$ as in the nilspace construction.

\begin{lemma}\label{lem:exch4}
Let $\mu$ be a cubic coupling on $\Omega^{\db{\mb{N}}}$ and let $\gamma:\Omega\to \ns$ be as in Theorem \ref{thm:MeasInvThmGen}. Let $\Bo$ be a standard Borel space, let $q:\Omega\to \Bo$ be a Borel map, let $q':\Omega\to\mc{P}(\Bo)$, $\omega\mapsto \delta_{q(\omega)}$, and let $m:\ns\to \mc{P}(\Bo)$ be such that $\mb{E}(q'|\gamma)=m\co\gamma$. Then $\mu\co(q^{\db{\mb{N}}})^{-1}=\zeta_{\ns,m}$. 
\end{lemma}

\begin{proof}
Let $V\subset\db{\mb{N}}$ be a finite set and $\{A_v\}_{v\in V}$ a collection of Borel subsets of $\Bo$. Let $A\subseteq\Omega^{\db{\mb{N}}}$ be the preimage of $\times_{v\in V} A_v$ under $p_V:\Bo^{\db{\mb{N}}}\to \Bo^V$. It suffices to show that $\mu\co(q^{\db{\mb{N}}})^{-1} (A)= \zeta_{\ns,m}(A)$ for every such set $A$. By Theorem \ref{thm:MeasInvThmGen} the coupling $\mu$ is independent relative to the factor $\gamma$, so $\mu\co(q^{\db{\mb{N}}})^{-1}(A)=\int_{\Omega^{\db{\mb{N}}}} \prod_{v\in S}\mb{E}(1_{A_v}\co q|\gamma)\co p_v\,\ud \mu$. Now for each $v$ and $\lambda$-almost every $\omega\in \Omega$ (where $\lambda$ is the original measure on $\Omega$), note that $\mb{E}(1_{A_v}\co q|\gamma)(\omega)$ is the measure $\mb{E}(q'|\gamma)(\omega)$ evaluated at $A_v$. Plugging this into the right side of the last equality gives us $\mu\co(q^{\db{\mb{N}}})^{-1}(A)=\int_{\Omega^{\db{\mb{N}}}} \prod_{v\in S}\big(m\co\gamma(\omega_v)\big)(A_v)\,\ud \mu(\omega)$, and letting $\mu'=\mu\co(\gamma^{\db{\mb{N}}})^{-1}$, the last integral is $\int_{\ns^{\db{\mb{N}}}} \prod_{v\in S}\big(m(x_v)\big)(A_v)\,\ud \mu'(x)$. Now note that this integral is precisely $\zeta_{\ns,m}(A)$.
\end{proof}

\begin{proof}[Proof of Theorem \ref{thm:mainexch1}]
Statement $(i)$ follows from Theorem \ref{thm:mainexch4} and Lemma \ref{lem:exch4}. For the second statement we use Theorem \ref{thm:mainexch3} and Lemma \ref{lem:cubindep}.
\end{proof} 

\section{Limits of functions on compact nilspaces}\label{sec:limits}

\noindent Let $\Bo\subset\mb{C}$ be a compact set, let $\ns$ be a compact nilspace, and let $f:\ns\to \Bo$ be a Borel measurable function. Nilspace theory enables us to define densities in $f$ of configurations given by systems of linear forms. In this section we focus on the following  configurations.
\begin{defn}\label{def:cp}
A \emph{cubic pattern} is determined by two multisets $S_1$ and $S_2$ in a cube $\db{k}$. The density of such a pattern in $f:\ns\to \Bo$, denoted by $t(S_1,S_2,f)$, is defined by
\begin{equation}\label{eq:limcubconf}
t(S_1,S_2,f)=\int_{\cu^k(\ns)} \Big(\prod_{v\in S_1} f(\q(v))\Big)\Big(\prod_{v\in S_2}\overline{f(\q(v))}\Big)\ud\mu^{\db{k}}(\q),
\end{equation}
where $\mu^{\db{k}}$ is the Haar probability measure on the cube set $\cu^k(\ns)$.
\end{defn}
\noindent For example if $S_1$ is the set of vertices with even coordinate sum and $S_2$ is the set of vertices with odd coordinate sum in $\db{k}$ then we have $t(S_1,S_2,f)=\|f\|_{U^k}^{2^k}$.
We say that a sequence of functions $(f_i:\ns_i\to \Bo)_{i\in \mb{N}}$ is \emph{cubic convergent} if $\lim_{i\to\infty}t(S_1,S_2,f_i)$ exists for every cubic configuration $(S_1,S_2)$. It can be seen in a straightforward way from the definitions that cubic convergence is equivalent to the convergence of the measures $\zeta_{\ns_i,p\co f_i}$ in the weak topology on $\mc{P}(\Bo^{\db{\mb{N}}})$, where $p:\Bo\to\mc{P}(\Bo)$ is the function that maps $z\in \Bo$ to the Dirac measure $\delta_z$. Note that for a Borel function $f:\ns\to\mc{P}(\Bo)$ we can also define $t(S_1,S_2,f)$ using \eqref{eq:limcubconf} for the $\db{k}$-marginals of the probability measure $\zeta_{\ns,f}$.
The following result provides limit objects for cubic convergent sequences of functions.

\begin{theorem}\label{thm:limob}
Let $\Bo$ be a compact subset of $\mb{C}$, and let $(f_i:\ns_i\to \Bo)_{i\in \mb{N}}$ be a cubic convergent sequence of functions on compact nilspaces. Then there is a compact nilspace $\ns$ and a measurable function $f:\ns\to\mc{P}(\Bo)$ such that $\lim_{i\to\infty} t(S_1,S_2,f_i)=t(S_1,S_2,f)$ holds for every cubic pattern $(S_1,S_2)$.
\end{theorem}

\begin{proof}
Let $\nu$ be the weak limit of the measures $\zeta_{\ns_i,p\co f_i}$. Since each $\zeta_{\ns,p\co f_i}$ is cubic exchangeable with the independence property, and independence is preserved under weak limits, we have that $\nu$ is cubic exchangeable with the independence property. The result now follows from Theorem \ref{thm:mainexch1}.
\end{proof}
 
\begin{remark}
It can indeed happen that limits of functions that take values in $\Bo$ cannot be represented by functions with values in $\Bo$ and the more general $\mc{P}(\Bo)$-valued functions have to be used. A simple example is when $f_i$ is a random function on the cyclic group $\mb{Z}_i$ with independent values in $\{1,-1\}$ with probability $1/2$ each. Then with probability 1 the limit of the sequence $f_i$ is the function which goes from the one point nilspace to the probability distribution $(\delta_1)/2+(\delta_{-1})/2$.
\end{remark}

\begin{remark}
The convergence of the densities of cubic patterns provides a rich enough limit concept to study various interesting phenomena in arithmetic combinatorics. To obtain limit objects for more general collections of patterns, apart from the additional technicalities in defining the associated densities for functions on compact nilspaces, there is also a more general corresponding exchangeability problem involved, which consists in describing the structure of the joint distribution of a sequence of random variables $(X_v)_{v\in \mb{Z}^\infty}$ (where $\mb{Z}^\infty=\bigoplus_{i\in \mb{N}}\mb{Z}$), assuming that this distribution is invariant under the action of the affine automorphism group of $\mb{Z}^\infty$, that is $\textrm{GL}(\mb{Z}^\infty)\ltimes \mb{Z}^\infty$. 
\end{remark}

\begin{appendix}
\section{Background results from measure theory}\label{App}

We begin with the proof of Lemma \ref{lem:pisysapprox}, which we restate here.
\begin{lemma}\label{lem:pisysapproxApp}
Let $(\Omega,\mc{A},\lambda)$ be a probability space, let $p\in [1,\infty)$, let $(\mc{B}_i)_{i=1}^n$ be a sequence of sub-$\sigma$-algebras of $\mc{A}$, and let $\mc{B}=\bigvee_{i=1}^n \mc{B}_i$. Let $\mc{R}$ denote the set of functions on $\Omega$ of the form $\omega\mapsto \prod_{i=1}^n f_i(\omega)$, where $f_i\in \mc{U}^\infty(\mc{B}_i)$ for all $i$. Then for every $f\in L^p(\mc{B})$, and every $\epsilon>0$, there is a finite linear combination $g$ of functions in $\mc{R}$ such that $\|f-g\|_{L^p}\leq\epsilon$.
\end{lemma}

\begin{proof}
In $L^p(\mc{B})$ the set of simple functions is everywhere dense \cite[Lemma 4.2.1]{Boga1}, so it suffices to prove the lemma assuming that $f$ is simple. Then by the triangle inequality for $\|\cdot\|_{L^p}$, it suffices to prove the lemma for indicator functions of sets in $\mc{B}$. In fact, it suffices to show that the sets of the form $\cap_{i\in [n]} B_i$ with $B_i\in \mc{B}_i$, form a semiring (to recall the notion of a semiring of sets see \cite[p.\ 166]{Bill1}). Indeed, if this holds then by \cite[Theorem 11.4]{Bill1} for every set $B\in \mc{B}$ and every $\epsilon>0$ there is a finite disjoint union $E$ of such finite intersections satisfying $\lambda(B\Delta E)\leq\epsilon$, which implies our result for indicator functions of sets in $\mc{B}$ as required.\\
\indent To check the semiring property, the nontrivial part is to check that if $A,B$ are sets of the above form and $A\subset B$, then there exist disjoint sets $C_1,\ldots,C_k$, each being of the above form and such that $B\setminus A=\bigcup_{j\in [k]} C_j$. To show this, we first prove the following basic case. Let $\mc{B}_1,\mc{B}_2$ be sub-$\sigma$-algebras of $\mc{B}$ and let $X_i,Y_i\in \mc{B}_i$ for $i=1,2$. Using the partition $Y_1^c \cup Y_2^c=(Y_1^c \cap Y_2^c) \sqcup (Y_1^c \cap Y_2) \sqcup (Y_1 \cap Y_2^c)$, where $Y_i^c=\Omega\setminus Y$, we obtain that
\[
(X_1\cap X_2)\setminus (Y_1\cap Y_2)
= \big( (X_1\setminus Y_1) \cap (X_2\setminus Y_2) \big) \sqcup \big( (X_1\setminus Y_1) \cap (X_2\cap Y_2) \big) \sqcup \big( (X_1\cap Y_1) \cap (X_2\setminus Y_2) \big).
\]
It follows from this and the fact that $\mc{B}_1, \mc{B}_2$ are $\sigma$-algebras  that sets of the form $X_1\cap X_2$ indeed form a semiring. By induction on $n$ we then deduce the general case, namely that sets of the form $A_1\cap\cdots\cap A_n$, $A_i\in \mc{B}_i$ for each $i$, form a semiring (using the equation above  with $X_1=\cap_{i=1}^{n-1} A_i$, $X_2=A_n$, $Y_1=\cap_{i=1}^{n-1} B_i$, $Y_2=B_n$ to reduce to the case $n-1$).
\end{proof}

Next we show that the meet of two sub-$\sigma$-algebras is indeed a sub-$\sigma$-algebra.
\begin{lemma}\label{lem:meetsubalg}
Let $(\Omega,\mc{A},\lambda)$ be a probability space, and let $\mc{B}_0,\mc{B}_1$ be sub-$\sigma$-algebras of $\mc{A}$. Then $\mc{B}_0\wedge \mc{B}_1$ is a sub-$\sigma$-algebra of $\mc{A}$.
\end{lemma}
\begin{proof}
We clearly have $\Omega$ and $\emptyset$ in $\mc{B}_0\wedge \mc{B}_1$. Let $A$ be in $\mc{B}_0\wedge \mc{B}_1$, and for $i=0,1$ let $B_i\in\mc{B}_i$ be such that $\lambda(A\Delta B_i)=0$. Then for $i=0,1$ we have $A^c\Delta B_i^c=A\Delta B_i$, so $\lambda(A^c\Delta B_i^c)=0$. Hence $\mc{B}_0\wedge \mc{B}_1$ is closed under taking complements. If $(A_n)_{n\in\mb{N}}$ is a sequence of sets in $\mc{B}_0\wedge \mc{B}_1$, then for each $n$ and $i=0,1$ there is $B_{n,i}\in \mc{B}_i$ such that $\lambda(A_n\Delta B_{n,i})=0$. We then have $(\bigcup_n A_n) \setminus (\bigcup_m B_{m,i})= \bigcup_n(A_n\cap\, \bigcap_m B_{m,i}^c)$ and each set $A_n\cap\, \bigcap_m B_{m,i}^c$ is in $\mc{A}$ and included in $A_n\setminus B_{n,i}$, so it is a null set, whence $\lambda \big((\bigcup_n A_n) \setminus (\bigcup_m B_{m,i})\big)=0$. Similarly $\lambda \big((\bigcup_m B_{m,i})\setminus (\bigcup_n A_n) \big)=0$. Hence $\mc{B}_0\wedge \mc{B}_1$ is closed under countable unions.
\end{proof}

The following result was stated as Lemma \ref{lem:nestexp}.

\begin{lemma}\label{lem:nestexpApp}
Let $(\Omega,\mc{A},\lambda)$ be a probability space, and let $\mc{B},\mc{B}'$ be sub-$\sigma$-algebras of $\mc{A}$ with $\mc{B}\subset_\lambda\mc{B}'$. Then for every integrable function $f:\Omega\to \mb{R}$ we have $\mb{E}(\mb{E}(f|\mc{B}')|\mc{B})=_\lambda\mb{E}(f|\mc{B})$, and also $\mb{E}(f|\mc{B}')=_\lambda\mb{E}(f|\mc{B}' \vee \mc{B})$. 
\end{lemma}

\begin{proof}
To prove the first equality, by definition of conditional expectation it suffices to show that for every set $A\in \mc{B}$ we have $\int 1_A f \ud\lambda=\int 1_A\, \mb{E}(f|\mc{B}')\ud\lambda$. Since $\mc{B}\subset_\lambda\mc{B}'$, there is a set $A'\in \mc{B}'$ such that $\|1_A-1_{A'}\|_{L^1}=\lambda(A\, \Delta\, A')=0$, which implies that $\int 1_A f \ud\lambda = \int 1_{A'} f \ud\lambda$. But the last integral equals $\int 1_{A'}\, \mb{E}(f|\mc{B}') \ud\lambda$ by definition of $\mb{E}( f|\mc{B}')$. Using again that $\|1_A-1_{A'}\|_{L^1}=0$, this last integral is seen to equal $\int 1_A\, \mb{E}( f|\mc{B}') \ud\lambda$, as required.

To prove the second equality, by definition again it suffices to show that for every set $A\in \mc{B}' \vee \mc{B}$ we have $\int 1_A\, f\ud\lambda = \int 1_A\, \mb{E}(f|\mc{B}') \ud\lambda$. By Lemma \ref{lem:pisysapproxApp}, the approximation of functions in $L^1(\mc{B})$ and $L^1(\mc{B}')$ by simple functions, and linearity of the integral, it suffices to show that for every $A\in \mc{B}'$ and $B\in \mc{B}$ we have $\int 1_A1_B\, f\ud\lambda = \int 1_A1_B\, \mb{E}(f|\mc{B}') \ud\lambda$. But then since $\|1_B-1_{B'}\|_{L^1}=0$ for some $B'\in\mc{B}'$, arguing as in the previous paragraph we have $\int 1_A1_B\, f\ud\lambda = \int 1_A1_{B'}\, f\ud\lambda =  \int 1_A 1_{B'}\,\mb{E}(f|\mc{B}')\ud\lambda = \int 1_A 1_B\,\mb{E}(f|\mc{B}')\ud\lambda$, as required. 
\end{proof}

\begin{proposition}\label{prop:condindepApp}
Let $(\Omega,\mc{A},\lambda)$ be a probability space, and let $\mc{B}_0,\mc{B}_1$ be sub-$\sigma$-algebras of $\mc{A}$. Then $\mc{B}_0\upmod \mc{B}_1$ holds if and only if, for every bounded measurable function $f:\Omega\to\mb{R}$, the following equation is satisfied for $i=0$ or, equivalently, for $i=1$\textup{:} 
\begin{equation}\label{eq:condindepApp}
\mb{E}( \,\mb{E}(f|\mc{B}_i)\, |\mc{B}_{1-i}) =_\lambda \mb{E}( f | \mc{B}_0\wedge \mc{B}_1).
\end{equation}
\end{proposition}
\begin{proof}
Throughout this proof let $\mc{B}=\mc{B}_0\wedge \mc{B}_1$. We first prove the necessity of \eqref{eq:condindepApp}. For $i=0,1$, for every bounded $\mc{B}_i$-measurable function $g$, by the second equality in Lemma \ref{lem:nestexpApp} we have $\mb{E}(g|\mc{B}_{1-i})=_\lambda \mb{E}(g|\mc{B}_{1-i} \vee \mc{B})$. By Theorem \ref{thm:condindep} we have $\mb{E}(g|\mc{B}_{1-i} \vee \mc{B})=_\lambda \mb{E}(g|\mc{B})$. The last two equalities imply that $\mb{E}(g|\mc{B}_{1-i})=_\lambda\mb{E}(g|\mc{B})$. Applying this to $g=\mb{E}(f|\mc{B}_i)$, we deduce that $\mb{E}( \mb{E}(f|\mc{B}_i) |\mc{B}_{1-i})=_\lambda \mb{E}(\mb{E}(f|\mc{B}_i)|\mc{B})$, and then since $\mc{B}\subset_\lambda \mc{B}_i$, by the first equality in Lemma \ref{lem:nestexpApp} the last expectation equals $\mb{E}(f|\mc{B})$, so we deduce \eqref{eq:condindepApp}.

To see the sufficiency of \eqref{eq:condindepApp}, we show that the last equation in Theorem \ref{thm:condindep} holds for every bounded $\mc{B}_i$-measurable function, which will suffice (using Remark \ref{rem:integtobded}). Thus, letting $f$ be any such function, since we have $f=_\lambda \mb{E}(f|\mc{B}_i)$, and also $\mb{E}(f|\mc{B}_{1-i} \vee \mc{B}) =_\lambda \mb{E}(f|\mc{B}_{1-i})$ as noted above, we have thus $\mb{E}(f|\mc{B}_{1-i} \vee \mc{B}) =_\lambda \mb{E}(f|\mc{B}_{1-i}) =_\lambda \mb{E}(\mb{E}(f|\mc{B}_i)|\mc{B}_{1-i})$, which by \eqref{eq:condindepApp} equals $\mb{E}( f | \mc{B}_0\wedge \mc{B}_1)=_\lambda\mb{E}( f | \mc{B})$. Hence the equation in Theorem \ref{thm:condindep} holds, so $\mc{B}_0,\mc{B}_1$ are indeed conditionally independent relative to $\mc{B}_0\wedge \mc{B}_1$.
\end{proof}
The following fact was stated as Lemma \ref{lem:intermeet}.
\begin{lemma}\label{lem:intermeetApp}
Let $(\Omega,\mc{A},\lambda)$ be a probability space, let $\mc{B}_0$, $\mc{B}_1$ be sub-$\sigma$-algebras of $\mc{A}$, and let $1\leq p<\infty$. Then $L^p(\mc{B}_0)\cap L^p(\mc{B}_1)= L^p(\mc{B}_0\wedge \mc{B}_1)$.
\end{lemma}
\begin{proof}
To see the inclusion $L^p(\mc{B}_0)\cap L^p(\mc{B}_1)\supset L^p(\mc{B}_0\wedge \mc{B}_1)$ we can argue starting with any $f\in L^p(\mc{B}_0\wedge \mc{B}_1)$ and using approximation by simple functions involving sets in $\mc{B}_0\wedge \mc{B}_1$. For the opposite inclusion, we can argue starting with a real-valued function $f\in L^p(\mc{B}_0)\cap L^p(\mc{B}_1)$ and showing that any set of the form $\{f> c\}$, $c\in \mb{R}$, is in $\mc{B}_0\wedge\mc{B}_1$, which implies that $f$ is $(\mc{B}_0\wedge\mc{B}_1)$-measurable.
\end{proof}
\noindent Let us now prove Lemma \ref{lem:halfdistrib}, restated as follows, which gives one half of a distributivity property for meet over join and shows that the other half can fail.
\begin{lemma}\label{lem:halfdistribApp}
Let $(\Omega,\mc{A},\lambda)$ be a probability space, and let $\mc{B}_1,\mc{B}_2,\mc{B}_3$ be sub-$\sigma$-algebras of $\mc{A}$. Then
\begin{equation}\label{eq:distrib}
(\mc{B}_1\vee \mc{B}_2)\wedge \mc{B}_3 \supset (\mc{B}_1\wedge \mc{B}_3)\vee (\mc{B}_2\wedge \mc{B}_3).
\end{equation}
The opposite inclusion does not hold in general.
\end{lemma}

\begin{proof}
Suppose that $A\in (\mc{B}_1\wedge \mc{B}_3)\vee (\mc{B}_2\wedge \mc{B}_3)$. Using Lemma \ref{lem:pisysapproxApp} we can approximate $1_A$ in $L^p$ by finite sums of rank-1 functions of the form $g_i h_i$, with $g_i$ being bounded $(\mc{B}_1\wedge \mc{B}_3)$-measurable and $h_i$ being bounded $(\mc{B}_2\wedge \mc{B}_3)$-measurable. Combining this with approximation by simple functions of each such $g_i$ and $h_i$, we deduce that for every $\epsilon>0$ there exist sets $B_j\in\mc{B}_1\wedge \mc{B}_3$ and $C_j\in \mc{B}_2\wedge \mc{B}_3$, $j\in [N]$ such that $\|1_A-\sum_{j\in [N]} \alpha_j 1_{B_j} 1_{C_j}\|_{L^p}\leq \epsilon$. From here, replacing each $B_j$ and $C_j$ by a $\mc{B}_1$-measurable set and a $\mc{B}_2$-measurable set respectively (modulo a null-set error), we deduce that $1_A$ is within $L^p$ distance $\epsilon$ of $L^p(\mc{B}_1\vee \mc{B}_2)$, and since this holds for every $\epsilon>0$, we conclude that $A$ is $\mc{B}_1\vee \mc{B}_2$-measurable modulo a null set. A similar argument starting from the last inequality shows that $A$ is $\mc{B}_3$-measurable modulo a null set. Hence $A\in (\mc{B}_1\vee \mc{B}_2)\wedge \mc{B}_3$ as required.

To see a counterexample for the opposite inclusion, we shall use partitions of $[3]=\{1,2,3\}$ with $\lambda$ the counting measure. In this case the operation $\wedge$ is just intersection and \eqref{eq:distrib} reduces to $(\mc{B}_1\vee \mc{B}_2)\cap \mc{B}_3 \supset (\mc{B}_1\cap \mc{B}_3)\vee (\mc{B}_2\cap \mc{B}_3)$. The following example shows that this inclusion can be a strict one: let $P_1=\big\{\{1\},\{2,3\}\big\}$, $P_2=\big\{\{1,2\},\{3\}\big\}$, $P_3=\{ \{1,3\},\{2\}\big\}$, and let $\mc{B}_i=\sigma(P_i)$. We then have that $\sigma(\mc{B}_1\cup \mc{B}_2)=2^{[3]}\supset \mc{B}_3$, so that  $\sigma(\mc{B}_1\cup \mc{B}_2)\cap \mc{B}_3 = \mc{B}_3$, whereas $\mc{B}_1\cap \mc{B}_3=\mc{B}_2\cap \mc{B}_3=\{\emptyset,[3]\}$.
\end{proof}
Next we prove Lemma \ref{lem:modlaw}, which we restate here.
\begin{lemma}\label{lem:modlawApp}
Let $(\Omega,\mc{A},\lambda)$ be a probability space, let $\mc{B}$ and $\mc{C}$ be sub-$\sigma$-algebras of $\mc{A}$ satisfying $\mc{B}\upmod \mc{C}$, and let $\mc{B}_1$ be a sub-$\sigma$-algebra of $\mc{B}$. Then $(\mc{C} \vee \mc{B}_1) \wedge \mc{B} =_\lambda (\mc{C} \wedge \mc{B}) \vee \mc{B}_1$.
\end{lemma}
\begin{proof}
The inclusion $(\mc{C} \vee \mc{B}_1) \wedge \mc{B} \supset_\lambda (\mc{C} \wedge \mc{B}) \vee \mc{B}_1$ follows from \eqref{eq:distrib} and the fact that $\mc{B}_1\wedge \mc{B}=_\lambda \mc{B}_1$ (using the fact that in general if $\mc{B}_1=_\lambda\mc{B}_2$ then $\mc{B}_1\vee \mc{B}_3 =_\lambda \mc{B}_2\vee \mc{B}_3$).

To see the opposite inclusion, let $f\in L^\infty((\mc{C}\vee\mc{B}_1)\wedge\mc{B})$ and fix any $\epsilon>0$. Since $f\in L^2(\mc{C}\vee\mc{B}_1)$, by Lemma \ref{lem:pisysapproxApp} there is a function of the form $f'=\sum_{i=1}^n c_ib_i$
where $c_i\in L^\infty(\mc{C})$ and $b_i\in L^\infty(\mc{B}_1)$, such that $\|f-f'\|_{L^2}\leq\epsilon$. Then, since $f\in L^\infty(\mc{B})$, we have $\epsilon\geq\|\mb{E}(f|\mc{B})-\mb{E}(f'|\mc{B})\|_{L^2}=\Big\|f-\sum_{i=1}^n \mb{E}(c_i|\mc{B})b_i\Big\|_{L^2}$. Since $\mc{B}\upmod \mc{C}$, we have that $\mb{E}(c_i|\mc{B})$ is $\mc{B}\wedge\mc{C}$-measurable for every $i$, so the last sum is $(\mc{C}\wedge \mc{B})\vee\mc{B}_1$-measurable. Since $\epsilon>0$ was arbitrary, we can let $\epsilon\to 0$ and we deduce that $(\mc{C} \vee \mc{B}_1) \wedge \mc{B} \subset_\lambda (\mc{C} \wedge \mc{B}) \vee \mc{B}_1$.
\end{proof}

\noindent We now turn to the topological properties of coupling spaces, and the proof of Proposition \ref{prop:coupspace}, which we restate as follows.

\begin{proposition}\label{prop:coupspaceApp}
Let $S$ be a finite set and $\varOmega=(\Omega,\mc{A},\lambda)$ be a Borel or standard pro- bability space. Then $\coup(\varOmega,S)$ is a non-empty convex compact second-countable Hausdorff space \textup{(}in particular it is a Polish space\textup{)}. Moreover $\coup(\varOmega,S)$ can be metrized in such a way that every ball is a convex set. 
\end{proposition}
\noindent Recall that a measure space $\varOmega$ is \emph{separable} if its measure algebra (or metric Boolean algebra) is separable as a metric space (see \cite[\S 1.12(iii)]{Boga1}) or, equivalently, if $L^1(\varOmega)$ is separable (see \cite[\S 7.14(iv)]{Boga2}). Every Borel or standard probability space is separable.

For $\mu\in \coup(\varOmega,S)$, let $\xi(\mu,\cdot)$ denote the map ${L^\infty(\varOmega)}^S \to \mb{C}$, $F\mapsto \xi(\mu,F)$. The following lemma implies that the map $\mu\mapsto \xi(\mu,\cdot)$ is an injection from $\coup(\varOmega,S)$ into the set of multilinear maps ${L^\infty(\varOmega)}^S \to \mb{C}$. In fact, the lemma tells us that this injectivity holds even when restricting the function $\xi(\mu,\cdot)$ to systems of measurable indicator functions.

By a \emph{measurable product-set} in $\Omega^S$ we mean a Cartesian product $R=\prod_{v\in S} B_v$ where $B_v\in \mc{A}$ for each $v\in S$. These measurable product-sets form a semiring of sets (see \cite[Proposition 2, p.\ 415]{R&F}), which we denote by $\mc{R}(\mc{A},S)$.
\begin{defn}\label{def:premeas}
Let $S$ be a set and let $\varOmega=(\Omega,\mc{A},\lambda)$ be a probability space. We denote by $\mc{H}$ the set of functions $h:\mc{R}(\mc{A},S)\to [0,1]$ that satisfy the following properties:
\begin{enumerate}[leftmargin=0.7cm]
\item Additivity on the semiring $\mc{R}(\mc{A},S)$: for every collection of pairwise disjoint sets $R_1,\ldots,R_n\in \mc{R}(\mc{A},S)$, we have $h(\bigsqcup_{i=1}^n R_i)=\sum_{i=1}^n h(R_i)$.
\item If $R\in \mc{R}(\mc{A},S)$ is of the form $p_w^{-1}(B_w)$ for some $B_w\in \mc{A}$, then $h(R)=\lambda(B_w)$.
\end{enumerate}
\end{defn}

\begin{lemma}\label{lem:coupcorresp}
Let $S$ be a finite set and let $\varOmega=(\Omega,\mc{A},\lambda)$ be a Borel or standard probability space. Let $\xi_0:\coup(\varOmega,S)\to [0,1]^{\mc{R}(\mc{A},S)}$ be the map sending $\mu$ to the function $R\mapsto \mu(R)$. Then $\xi_0$ is a bijection from $\coup(\varOmega,S)$ to $\mc{H}$.
\end{lemma}

\begin{proof}
It is clear from the definition of $\coup(\varOmega,S)$ that $\xi_0$ takes values in $\mc{H}$.

To show that $\xi_0$ is surjective,\footnote{This surjectivity can be deduced from similar results in the literature (see for instance \cite[Theorem 454D]{Fremlin4}), but we include a proof here for completeness.} we shall prove that every $h\in \mc{H}$ is a premeasure on the semiring of product sets $\prod_{v\in S}B_v$. Surjectivity will then follow from the Carath\'eodory extension theorem, since this theorem  yields a measure $\mu$ on the product $\sigma$-algebra $\mc{A}^S$ (generated by the semiring $\mc{R}(\mc{A},S)$), and by property (ii) we have $\mu\in\coup(\varOmega,S)$. 
Thus, let us fix any $h\in \mc{H}$. To prove that $h$ is a premeasure as claimed, the only non-trivial part is to show that if $R\in \mc{R}(\mc{A},S)$ is the pairwise disjoint union of sets $R_i=\prod_{v\in S} B_{v,i}\in \mc{R}(\mc{A},S)$, $i\in \mb{N}$, then $h(R)\leq \sum_{i\in \mb{N}} h(R_i)$. To show this, we first note that it suffices to prove it for $R=\Omega^S$ (since given any other product set $R$ we can obtain $\Omega^S$ as the disjoint union of $R$ and a finite number of other product sets, in such a way that the claim for $\Omega^S$ implies the claim for $R$). Now, since $\varOmega$ is standard or Borel, there exists a topology $\tau$ generating $\mc{A}$ and such that, on one hand, $\lambda$ is tight relative to $\tau$ (i.e.\ we can approximate the probability of any measurable set arbitrarily closely by the measure of some compact subset), and on the other hand every set $B_{v,i}$, $v\in S, i\in\mb{N}$ is open in $\tau$; see \cite[Definition 1-1 and Lemma 3-1]{dlR}, or \cite[(13.1) and (13.3)]{Ke}. Therefore, there exists a compact set $K_v\subset\Omega$ such that $\lambda_v(K_v)>1-\epsilon/|S|$. It then follows from properties (i) and (ii) that the compact set $K=\prod_{v\in S} K_v$ in $\tau^S$ satisfies $h(K)>1-\epsilon$. Since every $B_{v,i}$ is open, we have that all the rectangles $R_i$ form an open cover of $K$ in $\tau$, so there is a finite subcover. Applying property (ii) to this subcover, we conclude that $h(R)=1 \leq h(K)+\epsilon \leq \epsilon + \sum_{i\in \mb{N}} h(R_i)$. Letting $\epsilon\to 0$, we deduce that $h(R)\leq \sum_{i\in \mb{N}} h(R_i)$. Hence $h$ is indeed a premeasure.

Finally, to see that $\xi_0$ is injective, note that if $\xi_0(\mu_1)=\xi_0(\mu_2)$ then in particular $\mu_1$ and $\mu_2$ restrict to the same premeasure on $\mc{R}(\mc{A},S)$, so we have $\mu_1=\mu_2$ by the uniqueness of the Carath\'eodory extension \cite[p.\ 356]{R&F}.
\end{proof}
We can now establish the main topological properties of coupling spaces. 
\begin{proof}[Proof of Proposition \ref{prop:coupspaceApp}]
The product measure $\lambda^S$ shows that $\coup(\varOmega,S)$ is non-empty.

Let us equip $[0,1]^{\mc{R}(\mc{A},S)}$ with the product topology (where $[0,1]$ is equipped with the restriction of the standard topology on $\mb{R}$). Note that the map $\xi_0$ from Lemma \ref{lem:coupcorresp} is continuous from $\coup(\varOmega,S)$ to $\mc{H}$ equipped with the relative topology from $[0,1]^{\mc{R}(\mc{A},S)}$, indeed this follows readily from the definition of the product topology on $[0,1]^{\mc{R}(\mc{A},S)}$ and our choice of topology on $\coup(\varOmega,S)$.

It follows from the separability of $\varOmega$ that there is a sequence $(R_i)_{i\in\mb{N}}$ of sets $R_i \in \mc{R}(\mc{A},S)$ such that for every $R\in \mc{R}(\mc{A},S)$ and every $\epsilon>0$ there is $i$ such that $\mu(R\Delta R_i)\leq \epsilon$ for every $\mu\in \coup(\varOmega,S)$ (this uses that the functions from Definition \ref{def:multilin} are linear in each entry of $F$ and are bounded by the $L^1(\lambda)$-norm as explained after \eqref{eq:cupconst2}). Let $\pi$ denote the projection from $[0,1]^{\mc{R}(\mc{A},S)}$ to $[0,1]^{\mb{N}}$ consisting in deleting coordinates corresponding to sets $R\in \mc{R}(\mc{A},S)\setminus \{R_i: i\in \mb{N}\}$. Now $[0,1]^{\mb{N}}$ equipped with the product topology is a second-countable compact Hausdorff space, and $\pi$ is then continuous and surjective. Then we have that $\pi\co\xi_0$ is a continuous surjective map $\coup(\varOmega,S)\to [0,1]^{\mb{N}}$. Moreover $\pi\co\xi_0$ is also injective, for if $\pi\co\xi_0(\mu_1)=\pi\co\xi_0(\mu_2)$ then the density of the $R_i$ above implies that $\xi_0(\mu_1)=\xi_0(\mu_2)$, and then by injectivity of $\xi_0$ we have $\mu_1=\mu_2$. Finally note that the inverse of $\pi\co\xi_0$ is also continuous. Indeed, fix any system $F=(f_v)_{v\in S}$ of functions in $L^\infty(\varOmega)$ and $\epsilon>0$. Then approximating each $f_v$ by simple functions, using multilinearity of the functions $F\mapsto \xi(\mu,F)$ from Definition \ref{def:multilin}, and using the density of the $R_i$ above, we obtain a finite collection $C$ of such sets $R_i$, and some $\delta=\delta(\epsilon,F)>0$, such that if $\mu,\nu\in \coup(\varOmega,S)$ satisfy $|\pi\co\xi_0(\mu)_R-\pi\co\xi_0(\nu)_R|< \delta$ for all $R\in C$ then $|\xi(\mu,F)-\xi(\nu,F)|<\epsilon$. This implies the claimed continuity.

We have thus shown that $\pi\co\xi_0$ is a homeomorphism between $\coup(\varOmega,S)$ and $[0,1]^{\mb{N}}$, so $\coup(\varOmega,S)$ is indeed compact second-countable Hausdorff.

To see that $\coup(\varOmega,S)$ is convex (as a subset of the vector space of signed measures on $(\Omega^S,\mc{A}^S)$), note that for any probability measure $\nu$ on $\coup(\varOmega,S)$, the map $\mu_0:\mc{A}^S\to [0,1]$, $B\mapsto \int_{\coup(\varOmega,S)} \mu(B) \ud\nu(\mu)$ is a measure, and for each $v\in S$ the image of $\mu_0$ under $p_v$ is $\lambda$ (since $\mu\co p_v^{-1}=\lambda$ for each $\mu$ in the integral), so this convex combination $\mu_0$ is in $\coup(\varOmega,S)$.

Finally, let $d^*$ denote the metric $(x,y)\mapsto \sum_{i\in\mb{N}} 2^{-i}|x_i-y_i|$ on $[0,1]^{\mb{N}}$, and define the metric $d$ on $\coup(\varOmega,S)$ by $d(\mu,\mu')=d^*(\pi\co\xi_0(\mu),\pi\co\xi_0(\mu'))$. It is readily checked that $\pi\co\xi_0$ takes convex combinations in $\coup(\varOmega,S)$ to convex combinations in $[0,1]^{\mb{N}}$. Then, a straightforward argument shows that if $\mu_0$ is a convex combination $\int_{\coup(\varOmega,S)}\mu \ud\nu(\mu)$, then $d(\mu_0,\mu')\leq \int_{\coup(\varOmega,S)}d(\mu,\mu') \ud\nu(\mu)$. This implies that balls in the metric $d$ are convex.
\end{proof}
\noindent Recall that $\theta$ is a \emph{mod 0 isomorphism} from $(\Omega,\mc{A},\lambda)$ to $(\Omega',\mc{A}',\lambda')$ if there are measurable sets $\Omega_0\subset \Omega$, $\Omega_0'\subset \Omega'$ with $\lambda(\Omega)=\lambda'(\Omega')=1$ such that $\theta$ is a measure-preserving bijection $\Omega_0\to\Omega_0'$ \cite[Definition 9.2.1]{Boga2}. The last result of this appendix is the fact that the property of a probability space $\varOmega$ having a cubic coupling structure is a measure-theoretic invariant, in the sense that any mod 0 isomorphism from this space to another probability space $\varOmega'$ carries the sequence of measures forming the original cubic coupling to a sequence of measures forming a cubic coupling on $\varOmega'$.
\begin{proposition}
Let $\varOmega=(\Omega,\mc{A},\lambda)$, $\varOmega'=(\Omega',\mc{A}',\lambda')$ be probability spaces, suppose that $\theta:\Omega_0\to\Omega'_0$ is a mod 0 isomorphism of $\varOmega,\varOmega'$, and let $\big(\varOmega,(\mu^{\db{n}})_{n\geq 0}\big)$ be a cubic coupling. For each $n$ let $\nu^{\db{n}}:=\mu^{\db{n}}\co (\theta^{\db{n}})^{-1}$. Then $\big(\varOmega',(\nu^{\db{n}})_{n\geq 0}\big)$ is a cubic coupling.
\end{proposition}

\begin{proof}
Let us first note that each measure $\nu^{\db{n}}$ is a coupling in $\coup(\varOmega',\db{n})$. Indeed, for every $A'\in\mc{A}'$, we have by definition $\nu^{\db{n}}(p_v^{-1}(A'))=\mu^{\db{n}}\co(\theta^{\db{n}})^{-1}\big(p_v^{-1}(A'\cap\Omega'_0)\cap {\Omega'_0}^{\db{n}}\big)$, and by assumption $A'\cap\Omega'_0=\theta(A)$ for some measurable set $A\subset \Omega_0$. Hence we have  $p_v^{-1}(A'\cap\Omega'_0)\cap {\Omega_0'}^{\db{n}}=\theta^{\db{n}}(p_v^{-1}(A)\cap \Omega_0^{\db{n}})$, so the last measure is $\mu^{\db{n}}(p_v^{-1}(A)\cap \Omega_0^{\db{n}})$. This equals $\mu^{\db{n}}(p_v^{-1}(A))$ since $\lambda(\Omega_0)=1$ and each marginal of $\mu^{\db{n}}$ is $\lambda$. Hence $\nu^{\db{n}}(p_v^{-1}(A'))=\mu^{\db{n}}(p_v^{-1}(A))=\lambda(A)=\lambda'(A')$, as required.

It remains to show that the sequence $(\nu^{\db{n}})_{n\geq 0}$ satisfies the axioms in Definition \ref{def:cc}. The consistency and ergodicity axioms are seen to follow in a straightforward way from the same axioms for the original measures $\mu^{\db{n}}$. Finally, note that the conditional independence axiom involves just conditional expectations and sub-$\sigma$-algebras (involved in the notion of conditionally independent index sets, via Lemma \ref{lem:botsuff}). Via the map $\theta^{\db{n}}$, the sub-$\sigma$-algebras of the trace $\sigma$-algebra $\mc{A}^{\db{n}}|_{\Omega_0^{\db{n}}}$ are in bijection with the sub-$\sigma$-algebras of the trace $\sigma$-algebra ${\mc{A}'}^{\db{n}}|_{{\Omega_0'}^{\db{n}}}$. Moreover, since $\theta^{\db{n}}$ is measure-preserving (by definition of $\nu^{\db{n}}$), the conditional expectations corresponding to index sets $T\subset \db{n}$ satisfy $\mb{E}(f\co \theta^{\db{n}}|\mc{A}^{\db{n}}_T)=\mb{E}(f|{\mc{A}'}^{\db{n}}_T)\co \theta^{\db{n}}$ for every measurable $f:{\Omega_0'}^{\db{n}}\to\mb{C}$, by \eqref{eq:exprel}. Thus the measures $\nu^{\db{n}}$ inherit the conditional independence axiom from the measures $\mu^{\db{n}}$.
\end{proof}

\end{appendix}


\begin{thebibliography}{1}
\bibitem{Aldous} D. J. Aldous, \emph{Representations for partially exchangeable arrays of random variables}, J. Multivariate Anal. \textbf{11} (1981), no. 4, 581--598.
\bibitem{Aldous2} D. J. Aldous, \emph{Exchangeability and related topics}. In \'Ecole d'\'et\'e de probabilit\'es de Saint-Flour, XIII--1983, volume 1117 of Lecture Notes in Math., pages 1--198. Springer, Berlin, 1985.
\bibitem{Austin} T. Austin, \emph{On exchangeable random variables and the statistics of large graphs and hypergraphs}, Probab. Surv. \textbf{5} (2008), 80--145.
\bibitem{Austin2} T. Austin, \emph{On the geometry of a class of invariant measures and a problem of Aldous}, \href{http://arxiv.org/abs/0808.2268}{arXiv:0808.2268}.
\bibitem{Bill1} P. Billingsley, \emph{Probability and Measure}, Third Edition. Wiley Series in Probability and Mathematical Statistics. A Wiley-Interscience Publication. John Wiley \& Sons, Inc., New York, 1995.
\bibitem{Boga1} V. I. Bogachev, \emph{Measure Theory}, Vol. I. Springer-Verlag, Berlin, 2007.
\bibitem{Boga2} V. I. Bogachev, \emph{Measure Theory}, Vol. II. Springer-Verlag, Berlin, 2007.
\bibitem{BTZ} V. Bergelson, T. Tao, T. Ziegler, \emph{An inverse theorem for the uniformity seminorms associated with the action $\mb{F}_p^\infty$}, Geom. Funct. Anal. \textbf{19} (2010), no. 6, 1539--1596.
\bibitem{CamSzeg} O. A. Camarena, B. Szegedy, \emph{Nilspaces, nilmanifolds and their morphisms}, \href{http://arxiv.org/abs/1009.3825}{arXiv:1009.3825}.
\bibitem{Cand:Notes1} P. Candela, \emph{Notes on nilspaces: algebraic aspects}, Discrete Analysis, 2017, Paper No. 15, 59 pp.
\bibitem{Cand:Notes2} P. Candela, \emph{Notes on compact nilspaces}, Discrete Analysis, 2017, Paper No. 16, 57 pp.
\bibitem{CGS} P. Candela, D. Gonz\'alez-S\'anchez, B. Szegedy, \emph{On nilspace systems and their morphisms},  Ergodic Theory Dynam. Systems \textbf{40} (2020), No. 11, 3015--3029.
\bibitem{D&J} P. Diaconis, S. Janson, \emph{Graph limits and exchangeable random graphs}, Rend. Mat. Appl. (7) \textbf{28} (2008), no. 1, 33--61.
\bibitem{Eisner&al} T. Eisner, B. Farkas, M. Haase, R. Nagel, \emph{Operator theoretic aspects of ergodic theory}, Graduate Texts in Mathematics, 272. Springer, Cham, 2015.
\bibitem{dF} B. de Finetti, \emph{Funzione caratteristica di un fenomeno aleatorio}, Mem. R. Acc. Naz. Lincei, \textbf{4}(6) (1930), 86--133.
\bibitem{Fra} N. Frantzikinakis, \emph{Some open problems on multiple ergodic averages}, Bull. Hellenic Math. Soc. \textbf{60} (2016), 41--90. 
\bibitem{Fra2} N. Frantzikinakis, \emph{The structure of strongly stationary systems}, J. Anal. Math. \textbf{93} (2004), 359--388. 
\bibitem{Fremlin4} D. H. Fremlin, \emph{Measure theory, vol. 4, 
Topological measure spaces}, Part I, II, Corrected second printing of the 2003 original. Torres Fremlin, Colchester, 2006.
\bibitem{FurstSzem} H. Furstenberg, \emph{Ergodic behavior of diagonal measures and a theorem of Szemer\'edi on arithmetic progressions}, J. Analyse Math. \textbf{31} (1977), 204--256.
\bibitem{Furst} H. Furstenberg, \emph{Recurrence in ergodic theory and combinatorial number theory}, M. B. Porter Lectures. Princeton University Press, Princeton, N.J., 1981.
\bibitem{GGY} E. Glasner, Y. Gutman, X. Ye, \emph{Higher order regionally proximal equivalence relations for general minimal group actions}, Adv. Math. \textbf{333} (2018), 1004--1041.
\bibitem{GSz} W. T. Gowers, \emph{A new proof of Szemer\'edi's theorem}, GAFA \textbf{11} (2001), 465--588.
\bibitem{GHFA} W. T. Gowers, \emph{Generalizations of Fourier analysis, and how to apply them}, Bull. Amer. Math. Soc. \textbf{54} (2017), no. 1, 1--44.
\bibitem{GW} W. T. Gowers, J. Wolf, \emph{The true complexity of a system of linear equations}, Proc. Lond. Math. Soc. (3) \textbf{100} (2010), no. 1, 155--176.
\bibitem{GreenICM} B. Green, \emph{Approximate algebraic structure}, Proceedings of the International Congress of Mathematicians--Seoul 2014. Vol. 1, 341--367, Kyung Moon Sa, Seoul, 2014. 
\bibitem{GTprimes} B. Green, T. Tao, \emph{The primes contain arbitrarily long arithmetic progressions}, Ann. of Math. (2) \textbf{167} (2008), no. 2, 481--547.
\bibitem{GTorb} B. Green, T. Tao, \emph{The quantitative behaviour of polynomial orbits on nilmanifolds}, Ann. of Math. (2) \textbf{175} (2012), no. 2, 465--540. 
\bibitem{GMV1} Y. Gutman, F. Manners, P. P. Varj\'u, \emph{The structure theory of nilspaces I},  J. Anal. Math. \textbf{140} (2020), no. 1, 299--369. 
\bibitem{GMV2} Y. Gutman, F. Manners, P. P. Varj\'u, \emph{The structure theory of nilspaces II: Representation as nilmanifolds}, Trans. Amer. Math. Soc. \textbf{371} (2019), no. 7, 4951--4992. 
\bibitem{GMV3} Y. Gutman, F. Manners, P. P. Varj\'u, \emph{The structure theory of nilspaces III: Inverse limit representations and topological dynamics}, Adv. Math. \textbf{365} (2020), 107059. 
\bibitem{GTZ} B. Green, T. Tao, T. Ziegler, \emph{An inverse theorem for the Gowers $U^{s+1}[N]$-norm}, Ann. of Math. \textbf{176} (2012), no. 2, 1231--1372.
\bibitem{Hoover} D. N. Hoover, \emph{Relations on probability spaces and arrays of random variables}, preprint, Institute for Advanced Study, Princeton, NJ, 1979.
\bibitem{Host} B. Host, \emph{Ergodic seminorms for commuting transformations and applications}, Studia Math. \textbf{195} (2009), no. 1, 31--49.
\bibitem{HKbook} B. Host, B. Kra, \emph{Nilpotent structures in ergodic theory},  Mathematical Surveys and Monographs, 236. American Mathematical Society, Providence, RI, 2018.
\bibitem{HK} B. Host, B. Kra, \emph{Nonconventional ergodic averages and nilmanifolds}, Ann. of Math. (2) \textbf{161} (2005), no. 1, 397--488.
\bibitem{HKparas} B. Host, B. Kra, \emph{Parallelepipeds, nilpotent groups, and Gowers norms}, Bull. Soc. Math. France \textbf{136} (2008), no. 3, 405--437.
\bibitem{KalCut} S. Kalikow, R. McCutcheon, \emph{An outline of Ergodic Theory}, Cambridge studies in advanced mathematics, Cambridge University Press, 2010.
\bibitem{Kal} O. Kallenberg, \emph{Symmetries on random arrays and set-indexed processes}, J. Theoret. Probab. \textbf{5} (1992), no. 4, 727--765.
\bibitem{Ke} A. S. Kechris, \emph{Classical descriptive set theory}, Graduate Texts in Mathematics, \textbf{156}. Springer-Verlag, New York, 1995.
\bibitem{Kell} H. G. Kellerer, \emph{Duality theorems for marginal problems}, Z. Wahrsch. Verw. Gebiete \textbf{67} (1984), no. 4, 399--432.
\bibitem{Lovasz} L. Lov\'asz, \emph{Large networks and graph limits}, 
American Mathematical Society Colloquium Publications, \textbf{60}. American Mathematical Society, Providence, RI, 2012.
\bibitem{Meyer} P. A. Meyer, \emph{Probability and potentials}, Blaisdell Publishing Co. Ginn and Co., Waltham, Mass.-Toronto, Ont.-London 1966.
\bibitem{Munkres} J. R. Munkres, \emph{Topology, Second Edition}, Prentice Hall, Inc., Upper Saddle River, NJ, 2000.
\bibitem{Pachl} J. K. Pachl, \emph{Disintegration and compact measures}, Math. Scand. \textbf{43} (1978/79), no. 1, 157--168.
\bibitem{Prat} L. Pratelli, \emph{Sur le lemme de mesurabilit\'e de Doob},  S\'eminaire de Probabilit\'es, XXIV, 1988/89, 46--51, Lecture Notes in Math., \textbf{1426}, Springer, Berlin, 1990. 
\bibitem{R&F} H. L. Royden, P. M. Fitzpatrick, \emph{Real Analysis, Fourth Edition}, Prentice Hall, 2010.
\bibitem{dlR} T. de la Rue, \emph{Espaces de Lebesgue},  S\'eminaire de Probabilit\'es, XXVII, 15--21, 
Lecture Notes in Math., \textbf{1557}, Springer, Berlin, 1993.
\bibitem{Szegedy:HFA} B. Szegedy, \emph{On higher order Fourier analysis},  \href{http://arxiv.org/abs/1203.2260}{arXiv:1203.2260}.
\bibitem{Szegedy:Lim} B. Szegedy, \emph{Limits of functions on groups},  Trans. Amer. Math. Soc. \textbf{370} (2018), no. 11, 8135--8153.
\bibitem{Weiss} B. Weiss, \emph{Actions of amenable groups}, Topics in dynamics and ergodic theory, 226--262, London Math. Soc. Lecture Note Ser., \textbf{310}, Cambridge Univ. Press, Cambridge, 2003. 
\bibitem{Yan} C. H. Yan, \emph{Decomposition of Lebesgue spaces}, Adv. Math. \textbf{135} (1998), no. 2, 330--350.
\bibitem{Yan2} C. H. Yan, \emph{The theory of commuting Boolean sigma-algebras}, Adv. Math. \textbf{144} (1999), no. 1, 94--116.
\bibitem{Ziegler} T. Ziegler, \emph{Universal characteristic factors and Furstenberg averages}, J. Amer. Math. Soc. \textbf{20} (2007), no. 1, 53--97.
\end{thebibliography}
\end{document}